\documentclass[12pt]{amsart}
\usepackage{longtable}
\usepackage{amsfonts,amssymb,amscd,amsmath,epsfig}
\usepackage{latexsym}
\usepackage{tocvsec2}
\usepackage[scr]{rsfso}
\usepackage{bbm}

\usepackage[]{youngtab}
\usepackage{pifont}
\newcommand{\hollowstar}{\text{\ding{80}}}

\usepackage{chngcntr}
\counterwithin*{equation}{section}

\usepackage{appendix}

\usepackage[%dvipdfm,
hyperindex=true,pagebackref=true,bookmarks=true,
colorlinks=true,linkcolor=blue,citecolor=red]{hyperref}

\begin{document}

%MACOS FOR LECTURES ON DAHA
\renewcommand{\tilde}{\widetilde}
\renewcommand{\hat}{\widehat}

\newcommand{\BR}{{\mathbb R}}
\newcommand{\BQ}{{\mathbb Q}}
\newcommand{\BC}{{\mathbb C}}
\newcommand{\BP}{{\mathbb P}}
\newcommand{\BZ}{{\mathbb Z}}
\newcommand{\BN}{{\mathbb N}}
\newcommand{\BS}{{\mathbb S}}

\newcommand{\cH}{{\mathcal H}}
\newcommand{\cA}{{\mathcal A}}
\newcommand{\cB}{{\mathcal B}}
\newcommand{\ccF}{{\mathfrak F}}
\newcommand{\cD}{{\mathcal D}}
\newcommand{\cL}{{\mathcal L}}
\newcommand{\cF}{{\mathcal F}}
\newcommand{\cP}{{\mathcal P}}
\newcommand{\cX}{{\mathcal X}}
\newcommand{\cY}{{\mathcal Y}}
\newcommand{\cS}{{\mathcal S}}
\newcommand{\cSol}{\hbox{$\mathcal Sol$}}
\newcommand{\cT}{\hbox{$\mathcal T$}}

\newcommand{\Z}{{\mathbb Z}}
\newcommand{\Q}{{\mathbb Q}}
\newcommand{\N}{{\mathbb N}}
\newcommand{\C}{{\mathbb C}}
\newcommand{\R}{{\mathbb R}}
\newcommand{\X}{{\mathbb X}}
\newcommand{\Y}{{\mathbb Y}}

\newcommand{\CH}{{\mathcal H}}
\newcommand{\CA}{{\mathcal A}}

\def\HH{\mbox{${\mathcal H}$\kern-5.2pt${\mathcal H}$}}

\newcommand{\binomial}[2]{\genfrac{(}{)}{0pt}{}{ #1 }{ #2 }}
\newcommand{\qbinomial}[2]{\genfrac{[}{]}{0pt}{}{ #1 }{ #2 }_q }
\newcommand{\qbinom}[3]{\genfrac{[}{]}{0pt}{}{ #1 }{ #2 }_{ #3 } }

%%SPECIAL SEC 1.0

\def\der{\partial}
\def\tensor{\otimes}
\def\gam{\gamma} \def\Gam{\Gamma}
\def\del{\delta} \def\Del{\Delta}
\def\kap{\kappa}
\def\lam{\lambda} \def\Lam{\Lambda}
\def\Comp{{\mathbb C}}
\def\sM{{\mathcal M}}

\newtheorem{theorem}{Theorem}[section]
\newtheorem{maintheorem}[theorem]{Main Theorem}
\newtheorem{proposition}[theorem]{Proposition}
\newtheorem{definition}[theorem]{Definition}
\newtheorem{lemma}[theorem]{Lemma}
\newtheorem{corollary}[theorem]{Corollary}
\newtheorem{notation}[theorem]{Notation}
\newtheorem{remark}[theorem]{Remark}
\newtheorem{example}[theorem]{Example}

\newtheorem{theorem }{Theorem}[section]
\newtheorem{maintheorem }[theorem]{Main Theorem}
\newtheorem{proposition }[theorem]{Proposition}
\newtheorem{definition }[theorem]{Definition}
\newtheorem{lemma }[theorem]{Lemma}
\newtheorem{corollary }[theorem]{Corollary}
\newtheorem{notation }[theorem]{Notation}
\newtheorem{remark }[theorem]{Remark}
\newtheorem{example }[theorem]{Example}

\newtheorem{ maintheorem }[theorem]{Main Theorem}
\newtheorem{ theorem}{Theorem}[section]
\newtheorem{ proposition}[theorem]{Proposition}
\newtheorem{ definition}[theorem]{Definition}
\newtheorem{ lemma}[theorem]{Lemma}
\newtheorem{ corollary}[theorem]{Corollary}
\newtheorem{ notation}[theorem]{Notation}
\newtheorem{ remark}[theorem]{Remark}
\newtheorem{ example}[theorem]{Example}

\newtheorem{thm}{Theorem}[section]
\newtheorem{prop}[thm]{Proposition}
\newtheorem{lem}[thm]{Lemma}
\newtheorem{cor}[thm]{Corollary}
\newtheorem{conj}[thm]{Conjecture}
\newtheorem{con}[thm]{Conjecture}
\newtheorem{dfn}[thm]{Definition}
\newtheorem{df}[thm]{Definition}
 \newcommand{\rem}{{\bf Comment.\ }}
 \newcommand{\rmk}{{\bf Comment.\ }}
 \newcommand{\exmp}{{\bf Example.\ }}
 \newcommand{\ex}{{\bf Example.\ }}
 \newcommand{\prob}{{\bf Problem.\ }}

\newtheorem{note}{Note} 
\renewcommand{\thenote}{}
\newtheorem*{acka}{Acknowledgments}
\newtheorem{ack}{Acknowledgments}
\renewcommand{\theack}{}
\renewcommand{\appendixname}{\bf Appendix}
\renewcommand{\proof}{{\em Proof.\ }}

\hyphenation{
ap-pen-dix as-ymp-tot-ic at-trib-uted at-trib-ut-able
Bry-li-n-sky com-mu-ta-tion de-ge-ne-rate
de-riv-a-tive dis-trib-ute equi-vari-ant ex-tra-or-di-nary  
geo-met-ric griev-ance griev-ous grad-ed ho-lo-no-my ho-mo-thetic
in-fin-ite-ly in-fin-i-tes-i-mal Ha-rish Cha-n-dra mul-ti-plic-able 
non-euclid-ean non-iso-mor-phic non-smooth par-a-digm 
par-a-bol-ic pa-rab-o-loid pa-ram-e-trize phe-nom-e-non 
post-script pseu-do-dif-fer-en-tial pseu-do-fi-nite 
qua-drat-ics quad-ra-ture Han-kel rec-tan-gle semi-def-i-nite 
set-up wide-spread Euler-ian Feb-ru-ary Gauss-ian Grothen-dieck 
Hamil-ton-ian Her-mi-t-ian her-mi-t-ian Jan-u-ary 
Japan-ese Ka-shi-wa-ra Kor-te-weg Le-gendre No-vem-ber Rie-mann-ian 
Sep-tem-ber Za-mo-lo-d-chi-kov Kni-zh-nik quan-tum Op-dam
Mac-do-nald Ca-lo-ge-ro Su-ther-land Mo-ser 
Ol-sha-net-sky  Pe-re-lo-mov in-de-pen-dent ope-ra-tors 
cy-clo-to-mic ra-tio-nal de-gen-er-a-tion 
in-ter-est-ing de-for-ma-tions de-for-ma-tion pro-ce-dure 
fol-lows ope-ra-tors  pre-serve suf-fices ap-proach 
for-mu-las con-sider its com-ple-tion cor-re-spond-ing 
au-to-mor-phism be-cause pro-por-tional fi-nal-ly let-ting 
equi-v-a-lence ge-n-er-al-ized Mac-do-nald iden-ti-ties 
cor-re-s-pond sub-dia-grams par-ti-tion na-t-u-ral-ly 
or-dered stan-dard de-for-ma-tion ar-gu-ment com-bined 
sphe-r-i-cal rep-re-sen-ta-tions tri-go-no-me-t-ric
ge-n-er-al-ly speak-ing pri-m-it-ive ir-re-du-cible 
sum-ma-tion  rep-re-sen-ta-tives pro-por-ti-o-na-li-ty
ultra-sphe-ri-cal Ro-gers}

\def\ffor{\quad\hbox{ for }\quad}
\def\wwhen{\quad\hbox{ when }\quad}
\def\wwhere{\quad\hbox{ where }\quad}
\def\aand{\quad\hbox{ and }\quad}
\def\for{\  \hbox{ for } \ }
\def\iif{ \ \hbox{ if } \ }
\def\when{ \ \hbox{ when } \ }
\def\where{\  \hbox{ where } \ }
\def\and{\  \hbox{ and } \ }
\def\and{\  \hbox{ and } \ }
\def\oor{\  \hbox{ or } \ }
\def\proof{{\em Proof. \  }}

\def\equal{\stackrel{\,\mathbf{def}}{= \kern-3pt =}}

\def\la{\lambda}
\def\La{\Lambda}
\def\om{\omega}
\def\Om{\Omega}
\def\Th{\Theta}
\def\th{\theta}
\def\al{\alpha}
\def\be{\beta}
\def\ga{\gamma}
\def\ep{\epsilon}
\def\up{\upsilon}
\def\Up{\Upsilon}
\def\de{\delta}
\def\De{\Delta}
\def\ka{\kappa}
\def\kapp{\hbox{\bf \ae}}
\def\si{\sigma}
\def\Si{\Sigma}
\def\Ga{\Gamma}
\def\ze{\zeta}
\def\io{\iota}
\def\bio{b^\iota}
\def\aio{a^\iota}
\def\twio{\tilde{w}^\iota}
\def\hwio{\hat{w}^\iota}
\def\gio{\g^\iota}
\def\Bio{B^\iota}

\def\del{\delta}
\def\pa{\partial}
\def\vp{\varphi}
\def\ve{\varepsilon}
\def\inf{\infty}

\def\vph{\varphi}
\def\vps{\varpsi}
\def\vPh{\varPhi}
\def\vep{\varepsilon}
\def\vpi{{\varpi}}
\def\vth{{\vartheta}}
\def\vsi{{\varsigma}}
\def\vrh{{\varrho}}

\def\bph{\bar{\phi}}
\def\bsi{\bar{\si}}
\def\bvp{\bar{\varphi}}

\newcommand{\bS}{{\mathbf S}}
\newcommand{\bH}{{\mathbf H}}
\newcommand{\bF}{{\mathbf F}}
\newcommand{\bE}{{\mathbf E}}

\def\tal{\tilde{\alpha}}
\def\tbe{\tilde{\beta}}
\def\tde{\tilde{\delta}}
\def\tpi{\tilde{\pi}}
\def\txi{\tilde{\xi}}
\def\tPi{\tilde{\Pi}}
\def\tPhi{\tilde{\Phi}}
\def\tV{\tilde{V}}
\def\tJ{\tilde{J}}
\def\tla{\tilde{\lambda}}
\def\tga{\tilde{\gamma}}
\def\tGa{\tilde{\Gamma}}
\def\tvs{\tilde{{\varsigma}}}
\def\tu{\tilde{u}}
\def\tU{\tilde{U}}
\def\tw{\widetilde w}
\def\tW{\widetilde W}
\def\tB{\tilde B}
\def\tv{\tilde v}
\def\tV{\tilde V}
\def\tz{\tilde z}
\def\tb{\tilde b}
\def\ta{\tilde a}
\def\tih{\tilde h}
\def\trh{\tilde {\rho}}
\def\tx{\tilde x}
\def\tf{\tilde f}
\def\tg{\tilde g}
\def\tG{\tilde G}
\def\tk{\tilde k}
\def\tl{\tilde l}
\def\tL{\tilde L}
\def\tD{\tilde D}
\def\tR{\tilde R}
\def\tP{\tilde P}
\def\tH{\tilde H}
\def\tp{\tilde p}

\def\hH{\hat{H}}
\def\hh{\hat{h}}
\def\hR{\hat{R}}
\def\hY{\hat{Y}}
\def\hX{\hat{X}}
\def\hP{\hat{P}}
\def\hT{\hat{T}}
\def\hV{\hat{V}}
\def\hG{\hat{G}}
\def\hF{\hat{F}}
\def\hw{\widehat{w}}
\def\hW{\widehat{W}}
\def\hu{\hat{u}}
\def\hs{\hat{s}}
\def\hv{\hat{v}}
\def\hb{\hat{b}}
\def\hB{\widehat{B}}
\def\hze{\hat{\zeta}}
\def\hsi{\hat{\sigma}}
\def\hrh{\hat{\rho}}
\def\hth{\hat{\theta}}
\def\hy{\hat{y}}
\def\hx{\hat{x}}
\def\hz{\hat{z}}
\def\hg{\hat{g}}
\def\he{\hat{e}}
\def\hE{\widehat{E}}

\def\B{\mathbf{B}}
\def\I{\mathbf{I}}
\def\P{\mathbf{P}}
\def\G{\mathbf{G}}
\def\S{\mathbf{S}}
\def\F{\mathbf{F}}
\def\one{\mathbf{1}}
\def\Sn{\mathbf{S}_n}
\def\0{\mathbf{0}}
\def\H{\mathbf{H}}
\def\V{\mathbf{V}}

\def\f{\mathcal{F}}
\def\çF{\mathcal{F}}
\def\o{\mathcal{O}}
\def\t{\mathcal{T}}
\def\r{\mathcal{R}}
\def\l{\mathcal{L}}
\def\m{\mathcal{M}}
\def\k{\mathcal{K}}
\def\n{\mathcal{N}}
\def\d{\mathcal{D}}
\def\p{\mathcal{P}}
\def\cP{\mathcal{P}}
\def\a{\mathcal{A}}
\def\h{\mathcal{H}}
\def\c{\mathcal{C}}
\def\y{\mathcal{Y}}
\def\e{\mathcal{E}}
\def\v{\mathcal{V}}
\def\z{\mathcal{Z}}
\def\x{\mathcal{X}}
\def\s{\mathcal{S}}
\def\g{\mathcal{G}}
\def\u{\mathcal{U}}
\def\w{\mathcal{W}}
\def\i{\mathcal{I}}
\def\j{\mathcal{J}}
\def\b{\mathcal{B}}

\def\lan{\langle}
\def\llb{(\!(}
\def\ran{\rangle}
\def\rrb{)\!)}
 \def\dim{{\hbox{\rm dim}}_{\mathbb C}\,}
\def\lng{\hbox{\rm{\tiny lng}}}
\def\sht{\hbox{\rm{\tiny sht}}}
\def\sph{\hbox{\rm{\tiny sph}}}
\def\inv{\hbox{\rm{\tiny inv}}}

\def\br#1{\langle #1 \rangle}

\def\rank{\hbox{rank}}
\def\gl{\mathfrak{gl}_N}
%\def\sgn{\hbox{sgn}}
%\font\germ=eufb10 %at 12pt 
%\def\mathfrak#1{\hbox{\germ #1}}

\newcommand{\Aut}{\operatorname{Aut}}
\newcommand{\Hom}{\operatorname{Hom}}
\newcommand{\End}{\operatorname{End}}
\newcommand{\Ind}{\operatorname{Ind}}
\newcommand{\ad}{\operatorname{ad}}
\newcommand{\pr}{\operatorname{pr}}
\newcommand{\aweyl}{\tilde{\mathbb S}_n}
\newcommand{\hec}{{\mathcal H}^t_n}
\newcommand{\Func}{{\mathcal F}({\mathbb C}^n,{\mathcal H}^t_n)}
\newcommand{\tr}{\operatorname{tr}}
\newcommand{\Out}{\operatorname{Out}}
\newcommand{\Rad}{\operatorname{Rad}}
\newcommand{\Spec}{\operatorname{Spec}}
\newcommand{\id}{\operatorname{id}}
\newcommand{\Int}{\operatorname{Int}}
\newcommand{\ct} {\operatorname{ct}}

\newcommand{\rat}{{\mathbb Q}}
\newcommand{\real}{{\mathbb R}}
\newcommand{\cplx}{{\mathbb C}}
\newcommand{\zint}{{\mathbb Z}}

\newcommand{\sq}{\phantom{1}\hfill$\qed$}
\newcommand{\Rea}{\Re}
\newcommand{\Ima}{\Im}

\newcommand{\st}{\bowtie}
\newcommand{\modd}{\mbox{\,mod\,}}
\newcommand{\lr}{\langle}
\newcommand{\rr}{\rangle}
\newcommand{\eps}{\varepsilon}
\newcommand{\phk}{\phi^{(k)}}
\newcommand{\psk}{\psi^{(k)}}
\newcommand{\Res}{\mbox{Res}\;}
\newcommand{\sgn}{\mbox{sgn}}
\newcommand{\mn} {\left\{ \begin{array}{c}m\\
n\end{array}\right\}}

\def\sX{\mathscr{X}}
\def\sH{\mathscr{H}}
\def\sY{\mathscr{Y}}
\def\TT{\mathfrak{T}}
\def\JJ{\mathfrak{J}}
\def\HH{\mathfrak{H}}
\def\FF{\mathfrak{F}}
\def\GG{\mathfrak{G}}
\def\CC{\mathfrak{C}}
\def\LL{\mathfrak{L}}

\def\BB{\mathfrak{B}}
\def\AA{\mathfrak{A}}
\def\ZZ{\mathfrak{Z}}
\def\HH{\hbox{${\mathcal H}$\kern-5.2pt${\mathcal H}$}}
\def\HHH{\hbox{${\mathbb H}$\kern-4.2pt${\mathbb H}$}}
\def\tHH{\widetilde{\HH\ }}

\font\smm=msbm10 at 12pt 
\def\symbol#1{\hbox{\smm #1}}
\def\lsmash{{\symbol n}}
\def\rsmash{{\symbol o}}
\def\#{\sharp}

\font\tenbf=cmbx10
\font\tenrm=cmr10
\font\tenit=cmti10
\font\ninebf=cmbx9
\font\ninerm=cmr9
\font\nineit=cmti9
\font\eightbf=cmbx8
\font\eightrm=cmr8
\font\eightit=cmti8
\font\sevenrm=cmr7
\font\sevenbf=cmbx7

%END MACROS

%\par
%{\centering
%Dedicated with admiration to Yuri Ivanovich Manin  \\
%on the occasion of his 80th birthday
%\medskip
%\par}

\title [Instanton slices and their superpolynomials]
{Instanton slices and their superpolynomials}
\author[Ivan Cherednik]{Ivan Cherednik}% $^\dag$}
%\date{February 2, 2014}

%{%\vskip -0.8cm
%\centering 
%{\sf\em Dedicated to the memory of Ian Macdonald
%\medskip\par}
%\vskip 0.3cm}

\begin{abstract} 
The key theorem is a connection between motivic
superpolynomials of plane curve singularities
in any ranks with superpolynomials of the corresponding
instanton slices, Nekrasov-type instanton sums 
with conductors. In this case, instanton slices are
related to Compactified Jacobians, but they form a much
wider class and, generally, have no connection to
plane curve singularities. Thus, we extend significantly 
theory of affine Springer fibers in type $A$.  
For instance, we obtain a motivic 
interpretation (counting $\mathbb{F}_q$-points of some stacks)
of superpolynomials for hyperbolic knots K12n242, K12n725,
among many others. New formulas for Nekrasov's
instanton sums as limits of superpolynomials
of families of plane curve singularities are another 
application, with an
impressive connection to nonsymmetric Macdonald polynomials,

The main conjecture states that instanton superpolynomials
for arbitrary Young diagrams coincide with the 
corresponding (new) DAHA
superpolynomials if and only if the latter are positive or
(equivalently) the former are superdual. Only very few such
diagrams come from plane curve singularities. Our DAHA superpolynomials
are parallel to the EHA ones, but we used the nonsymmetric
approach. The connection
with Khovanov-Rozansky polynomials of the corresponding 
Coxeter knots is expected. Among applications:\, 
the Volume Conjecture reformulated in terms of instanton slices,
topological invariance of unibranch motivic 
superpolynomials for semigroups with up to $4$ generators,
a mixed-characteristic conjecture, upgraded 
Weak Riemann Hypothesis,  and superpolynomials for  Gorenstein
curve singularities and multibranch ones.
\end{abstract}

%\thanks{$^\dag$ \today.
%\ \ \ Partially supported by NSF grant
%DMS--1901796}

\address[I. Cherednik]{Department of Mathematics, UNC
Chapel Hill, North Carolina 27599, USA\\
chered@email.unc.edu}

\newcommand{\hga}{\hat{\ga}}
 \def\sht{\raisebox{0.4ex}{\hbox{\rm{\tiny sht}}}}
 \def\bysame{{\bf --- }}
\let\oldt\~   %%%!!!
 \def\~{{\bf --}}
 \def\hga{{\hat{\gamma}}}
 \def\rr{{\mathsf r}}
 \def\ss{{\mathsf s}}
 \def\tt{{\mathsf t}}
 \def\mm{{\mathsf m}}
\def\nn{{\mathsf n}}
 \def\pp{{\mathsf p}}
 \def\ll{{\mathsf l}}
 \def\aa{{\mathsf a}}
 \def\bb{{\mathsf b}}
 \def\cc{{\mathsf c}}
 \def\cn{{\mathbbm n}}
 \def\NS{\hbox{\tiny\sf ns}}
 \def\ssum{\hbox{\small$\sum$}}
\def\sC{\mathscr{C}}
\def\sH{\mathscr{H}}

\newcommand{\comment}[1]{}
\renewcommand{\tilde}{\widetilde}
\renewcommand{\hat}{\widehat}
\renewcommand{\V}{\mathbb{V}}
\renewcommand{\S}{\mathbb{S}}
\renewcommand{\F}{\mathbb{F}}
\newcommand{\q}{\mathcal{Q}}
\newcommand{\dagx}{\hbox{\tiny\mathversion{bold}$\dag$}}
\newcommand{\ddagx}{\hbox{\tiny\mathversion{bold}$\ddag$}}
\newtheorem{conjecture}[theorem]{Conjecture}
\newcommand*\toeq{
\raisebox{-0.15 em}{\,\ensuremath{
\xrightarrow{\raisebox{-0.3 em}{\ensuremath{\sim}}}}\,}
}
\newcommand{\unknot}{\hbox{\tiny\!\raisebox{0.2 em}
{$\bigcirc$}}\!}
\newcommand{\mmu}{\hbox{\mathversion{bold}$\mu$}}
\newcommand{\lla}{\hbox{\mathversion{bold}$\lambda$}}
\newcommand{\dde}{\hbox{\mathversion{bold}$\delta$}}

\newcommand\rightthreearrow{\hbox{\tiny
        $\mathrel{\vcenter{\mathsurround0pt
         \ialign{##\crcr
         \noalign{\nointerlineskip}$\rightarrow$\crcr
         \noalign{\nointerlineskip}$\rightarrow$\crcr
         \noalign{\nointerlineskip}$\rightarrow$\crcr
                }}}$ }}
\newcommand\rightfourarrow{\hbox{\tiny
        $\mathrel{\vcenter{\mathsurround0pt
         \ialign{##\crcr
         \noalign{\nointerlineskip}$\rightarrow$\crcr
         \noalign{\nointerlineskip}$\rightarrow$\crcr
         \noalign{\nointerlineskip}$\rightarrow$\crcr
         \noalign{\nointerlineskip}$\rightarrow$\crcr
                }}}$ }}

\newcommand\rightdotsarrow{\hbox{\small
        $\mathrel{\vcenter{\mathsurround0pt
         \ialign{##\crcr
         \noalign{\nointerlineskip}$\,\rightarrow$\crcr
         \noalign{\nointerlineskip}$\cdots$\crcr
         \noalign{\nointerlineskip}$\,\rightarrow$\crcr
                }}}$ }}
\newcommand\rightdotsarrowtiny{\hbox{\tiny
        $\mathrel{\vcenter{\mathsurround0pt
         \ialign{##\crcr
         \noalign{\nointerlineskip}$\,\rightarrow$\crcr
         \noalign{\nointerlineskip}$\cdots$\crcr
         \noalign{\nointerlineskip}$\,\rightarrow$\crcr
                }}}$ }}

\newcommand*{\vect}[1]{\overrightarrow{\mkern0mu#1}}
\newcommand{\twoone}
{\hbox{\rm
$\circ\!$\raisebox{-2.6pt}{$\rightarrow$}
\raisebox{-2.6pt}{$\!\!\circ\!\!\rightarrow$}
\kern-33pt\raisebox{+2.6pt}{$\rightarrow$}
\kern+14pt}}
\newcommand{\twoonetiny}
{\hbox{\tiny
$\circ\!$\raisebox{-2.pt}{$\rightarrow$}
\raisebox{-2.pt}{$\!\!\circ\!\!\rightarrow$}
\kern-23pt\raisebox{+2.pt}{$\rightarrow$}
\kern+10pt}}

\newcommand{\twotwo}
{\hbox{\rm $\circ\!\rightrightarrows\!
\raisebox{2.5pt}{\hbox{\small $\circ$}}
\kern-5.5pt\raisebox{-2.5pt}{\hbox{\small $\circ$}}
\!\rightrightarrows\,$}}
\newcommand{\twotwotiny}
{\hbox{\tiny $\circ\!\rightrightarrows\!
\raisebox{2.pt}{\hbox{\tiny $\circ$}}
\kern-4.2pt\raisebox{-2.pt}{\hbox{\tiny $\circ$}}
\!\rightrightarrows$}}
\newcommand{\tax}{\hbox{\sf[r,s]}}
\newcommand{\lxi}{\raisebox{0.5pt}{${}^\xi$}\!}
\vskip -0.0cm

\newcommand{\qbin}[2]{\begin{bmatrix}{#1}\\ {#2}\end{bmatrix}_q}

%\par
%{\centering
%\medskip
%\par}
%\vskip -0.0cm
\maketitle
\smallskip

{\footnotesize
\centerline{{\bf MSC} (2010): 14H50, 17B22, 17B45, 20C08,
20F36, 33D52, 30F10, 55N10, 57M25}
}
\smallskip

\vskip -0.cm
\renewcommand{\baselinestretch}{1.2}
{\vbadness=10000\textmd
{ 
\tableofcontents}
}
\renewcommand{\baselinestretch}{1.2}
%}
%$\mathfrak{a,b,c,d,e,f,g,h,i,j,k,l,m,n,o,p,q,r,s,t,e^u,e^v,e^w}$
\vfill\eject

\setcounter{section}{0}
\setcounter{equation}{0}
\section{\sc Introduction}
\vbadness=3000
\hbadness=3000

\subsection{\bf Overview, instanton slices} 
The $1${\footnotesize st} major result
of this paper is Theorem \ref{thm:mainred}, connecting
the motivic superpolynomials $\h^{mot}(q,t,\aa; \mathbbm{n})$
of plane curve singularities from \cite{ChP1,ChP2,ChQ}
with  superpolynomials
of {\sf\em instanton slices}. It develops the corresponding
theorems from \cite{ChG,ChQ}. {\sf\em Instanton slices} 
are Quot-stacks $\mathfrak{S}^{(\mathbbm{n})}_{\mathscr{C}}$
of submodules $\m$ 
in $\a^\mathbbm{n}$ for $\a=\F[[x,y]]$ containing specific 
submodules $\mathscr{C}$, called 
{\sf\em conductors}. 
In Theorem \ref{thm:mainred}:\ 
$\mathscr{C}=\c^{\mathbbm{n}}$ for
the preimages $\c$ in $\a$ of the conductors of
unibranch plane curve singularities. 

The base field  $\F$ is
$\F_q$ in our definition of instanton superpolynomials:\,
$\mathscr{H}^{inst}_{\mathscr{C}}\equal
\sum_{\m}t^{deg(\m)}(1+\aa q)
\cdots (1+\aa q^{\varrho(\m)-1})$, 
summed over $\mathfrak{S}^{(\mathbbm{n})}_{\mathscr{C}}$.
Here $deg(\m)=dim_\F\bigl(\a^\mathbbm{n}/\m\bigr)$ and
$\varrho(\m)=dim_\F\bigl(\m/\mathfrak{m}_\a\m\bigr)$ for
the maximal ideal $\mathfrak{m}_\a\subset \a$,
the number of generators of $\m$.
\medskip

An impressive development is that instanton
slices produce superpolynomials associated with 
%non-algebraic
%cable $C\!ab(11,2)T(3,2)$, well-known 
{\sf\em hyperbolic knots} K12n242, quite famous, 
K12n725, and many more of them. The key application is our 
Instanton Volume Conjecture 
for such knots in (\ref{instvol}) and Appendix 
\ref{sec:inst-count}. %The knot K12n242 is quite famous.
Hyperbolic knots are the key in low-dimensional 
and geometry. 
The prior theory of DAHA and motivic superpolynomials
was restricted to iterated torus links and algebraic links; 
see \cite{ChQ} for its current  status.
 
%For instance, we can now reformulate the 
%{\sf\em Volume Conjecture} as the coincidence of
%hyperbolic\,  $vol(S^3\setminus K)$ with point-counting of some
%stacks over $\F_{p^m}$ in the limit  $m\to \infty$.
%Counting
%$\F_{q}$-points is a classical substitute for volumes
%and other geometric invariants of  {\sf\em algebraic} varieties.
%Now we can do some {\sf\em hyperbolic} volumes in a similar manner.
\medskip

The key for the Volume Conjecture is
a powerful and elegant
Conjecture \ref{conj:main}, which  connects
$\sH^{inst}_\sC$ for $\sC=I_\la^\cn$ with new
DAHA superpolynomials $\mathfrak{H}^{daha}_\la$,
introduced in Definition \ref{def:frakH}.
Here and throughout this paper,  $I_\la\subset \a$
is the monomial ideal associated with a Young diagram 
$\la$. The superduality 
$\sH^{inst}_\sC$ for $\cn=1$ or, equivalently,
the positivity of  $\mathfrak{H}^{daha}_\la$,
are conjectured to be necessary and sufficient for
their coincidence. We conjecture the former to be
always positive, and the latter always superdual.
Also, if this coincidence holds for $\cn=1$, then we conjecture
that it holds for any ranks $\cn$. Conjecture \ref{conj:supd}
provides the corresponding conditions for $\la$, 
intriguingly nontrivial.
 
\medskip

The rationale is that $\h^{daha}$ are {\sf\em superdual} 
for iterated torus links due to \cite{GoN,ChD1,ChD2}
(conjectured in the original \cite{CJ}), i.e. invariant with 
respect to $q\leftrightarrow t^{-1}$, and this 
is expected for $\mathfrak{H}^{daha}_\la(\cn=1)$ too.
{\sf\em Superduality} is a fundamental
feature of superpolynomials. Upon their conjectural
coincidence with the Galkin-St\"ohr $L$-functions of plane 
curve singularities, it becomes the {\sf\em
Hasse-Weil functional equation}.
%This is  $\h(q,t)=L(\frac{q}{t},t)$ in \cite{ChW} and
%``Quot vs Hilb" in \cite{KTr}. 

\medskip
Our definition of $\mathfrak{H}^{daha}_\la$
is parallel to the
construction of the {\sf\em EHA superpolynomials} in
\cite{GL2} (arbitrary $\la$, $\cn=1$) and Definition 4.23
for torus links in \cite{KTs}.
 The main difference is 
that our construction
is {\sf\em nonsymmetric}. Also, there 
are diagrammatic deviations 
and we can do any minuscule $\om_\cn$.  
Theorem \ref{thm:trans} and Theorem \ref{thm:sharp-flat} 
demonstrate the potential of the
nonsymmetric approach. The 1{\footnotesize st} one is that 
$\mathfrak{H}_\la= \mathfrak{H}_{\la^{tr}}$
up to some $q^\bullet t^\bullet$-factor, where  
$\la^{tr}$ is the transpose of  $\la$. It is
used to find certain Nekrasov's instanton sums,
the answer appeared the evaluation formula for {\sf\em nonsymmetric}
Macdonald polynomials from \cite{Ch4} in the case of
rectangles; see Theorem \ref{thm:daha-free} (with an impressive
proof).

%In spite of these differences, 
We conjecture that
our $\mathfrak{H}_\la$ for $\cn=1$ 
coincide with the EHA ones. We are very grateful to
Pavel Galashin and Oscar Kivinen 
who helped us to understand their constructions.
%We note that {\sf\em superduality}
%holds for the EHA superpolynomials according to \cite{GL2}.
%\smallskip

Our DAHA-definition is  $\{ W_\la[\cn]\, X_{\om_\cn}\}$
for the DAHA {\sf\em coinvariant} $\{\cdot \}$ for
the roots systems $A_{N-1}$ or $gl_N$ subject to the $N$-stabilization.
To obtain   $W_\la$,
we move along the border of the complement of
$\la$ in $\Z_+^2$ from the top
to the bottom, and put $X$ for each step 
left, and $Y$ for each step down. 
%It must be:\  
%$W_\la=Y\cdots X$ because
%the first step is always $X$ and the last one is always $Y$.
The word  $W_\la[\cn]$ is when  $X,Y$ are replaced 
 by $X_{\om_\cn}, Y_{\om_\cn}$. 
%This is parallel 
%to \cite{GL2}, but with obvious differences. 

%The rest of the paper is with (many) examples and various
%applications, including the product formula (\ref{prod-free}) 
%for ``free instanton sums" with $\aa$ and  for any $\cn$. 

\medskip
{\bf Free $\mathscr{H}^{inst}$ as limits of  $\h^{mot}$.}
In the absence of $\mathscr{C}$, 
$\mathscr{H}^{inst}$  are ``free"
punctual Nekrasov-type
instanton sums for modules $\m$ of rank $\cn$
supported at $(0,0)\in \mathbb{A}_\F^2$,
enhanced by adding $\varrho(\m)$ as above, the number of
generators of $\m$. Restricting the length of the corresponding
sets of Young diagrams by $\ell$, Theorem \ref{thm:daha-free}
gives (with some reservations):
\vskip 0.1cm

\centerline{
\(
\mathscr{H}^{inst}_{\text{\scalebox{.7}{$\le$\,}} \ell} 
(q,t,\aa; \cn)
=\text{\Large $\prod_{i=1}^ \ell\,\prod_{j=1}^ \mathbbm{n}
\frac{1+\aa q^{\mathbbm{n}(i+1)-j} t^i}
{1-q^{\mathbbm{n}i-j} t^i}$}.
\)
}

\vskip 0.1cm
The 
proof is remarkable;  the answer appeared the evaluation
formula  for rectangle 
{\sf\em nonsymmetric} Macdonald polynomials  from \cite{Ch4}.
We use that
$\mathscr{H}^{inst}_{\text{\scalebox{.7}{$\le$\,}} \ell}$
can be obtained 
as limits of motivic superpolynomials, which is based
on Theorem \ref{thm:mainred}. Alternatively, one
can use 
\cite{NY1,NY2}. 

\smallskip

Throughout this paper, $\r\subset \F[[x,y]]\subset \o=\F_q[[z]]$ 
is the ring of a plane curve singularity generated by $x,y$. 
``Free" $\sH^{inst}$ occur when the corresponding 
conductors $\mathfrak{c}$ lifted to $\c\subset \a$
tend to zero. For instance, the family
$\r=\F_q[[x=z^{m\ss+1},y=z^\ss]]$ as $m\to\infty$ gives
$\mathscr{H}^{inst}_{\text{\scalebox{.7}{$\le$\,}} \ell}[\cn]$
for $\ell=\ss-1$; see formula (\ref{lim2p+1}) for $\ell=1$.
%\medskip

The most general version of Theorem \ref{thm:mainred}
 is Theorem \ref{thm:gen-sing}, where the
curve singularities are arbitrary Gorenstein and there are
no restrictions for the isolated singularities (can be not only
surface ones). The limits then must
be over sufficiently {\sf\em ample} Gorenstein
curve singularities inside a
given isolated one. Presenting instanton sums for isolated  
surface singularities via curve singularities 
 can be a valuable tool. There is significant progress 
for the surface $ADE$-singularities, 
but only for their Euler characteristic (when $q=1$).
See \cite{To,GNS,Nak}.

\subsection{\bf Conductors and Gr\"obner cells}
Theorem \ref{thm:cond}
is devoted to the calculation of conductors of algebraic
$2$-cables. It extends 
Theorem 3.4,(ii) in \cite{ChG}, where it was proven that
the (lifted) conductors $\c[\rr,\ss,\upsilon, p]$ are monomial 
for torus knots
and cables $C\!ab(\upsilon \rr\ss+p,\upsilon)T(\rr,\ss)$
for $p=1$. Generally, $gcd(\rr,\ss)=1=gcd(\upsilon,p)$.
It is likely now that such $\c$ for algebraic knots
can be monomial {\sf\em only} in such cases,
supported by  Theorem \ref{thm:cond} for algebraic
$2$-cables and examples of algebraic 
$3$-cables. 

We calculate in this theorem the 
limits of $\c[\rr,\ss,\upsilon, p]$ as $p\to \infty$, which
are not monomial (unless $\upsilon=1$). We found them
(for $\cn=1$) 
for the family $\ss=2,\rr=3, \upsilon=2$ in Corollary 
\ref{cor:lim-up2}, and used computers 
for $\ss=2, \rr=3, \upsilon=3$. See also Theorem
\ref{thm:explicit-semiroot-conductor} (any algebraic $2$-cables):\, 
their lifted conductors 
are ``combinatorial":\, topological invariants.

\vskip 0.2cm
The combinatorics of the Young diagrams $\la(\c)$ corresponding to
$\c[\rr,\ss,\upsilon,p=1]$ is very rich. Such $\la$ are formed by 
the boxes in rectangles
$\upsilon\rr\times \upsilon\ss$ above the diagonal ($\upsilon=1$
for torus knots); those touched it are {\sf\em included}. 
We will also consider $\la^{\boldsymbol{\flat}}$ where the latter 
boxes are 
{\sf\em excluded}. They correspond to\, {\sf\em non-algebraic}\, 
 $C\!ab(\upsilon \rr\ss-1,\upsilon)T(\rr,\ss)$,
so $\h^{mot}$ cannot be used. However, the following $3$
constructions perfectly work:\ (i)\, $\h^{daha}$ for  
$C\!ab(\upsilon \rr\ss-1,\upsilon)T(\rr,\ss)$, \,
(ii)\, $\mathfrak{H}^{daha}$ for $\la^{\boldsymbol{\flat}}$, and\, 
(iii)\,
$\sH^{inst}$ for
the conductor $\c=I_{\la^{\boldsymbol{\flat}}}$, which is 
{\sf\em not} from plane curve singularities.

The diagrams $\la$ for 
$[\rr,\ss,\upsilon, p=1]$ and related Dyck paths  were studied in
quite a few papers, though without connections to the 
{\sf\em conductors} 
and Gr\"obner cells (from \cite{ChG}). 
Let us mention \cite{GM,GMV1,GMV2,GMO},
and \cite{BHM, CGM,CaM} devoted to the {\sf\em Shuffle Conjecture}.
By the way, the superduality can be seen in some cases directly
from the formulas there and those for the EHA superpolynomials. 

Our Conjecture \ref{conj:main},(i) is a very general
version of the Shuffle Conjecture. It is for any plane 
curve singularities and for
arbitrary ranks $\cn$. Its Part  $(ii)$ is for 
{\sf\em instanton slices} associated with Young diagrams. See also 
the end of Section \ref{sec:polynom} concerning the usage
of nonsymmetric Macdonald polynomials in the theory of 
$\nabla$-operator. 
\vskip 0.2cm

{\bf Cells \mathversion{bold}{$G_{\mu\subset\la}$}.}
The usage of Gr\"obner cells is very essential here.
In the case of monomial $\sC=I_\la$ and
for $\cn=1$, 
they are  formed by ideals $\i\subset \a$
such that $\la(\i)=\mu$, so  one-to-one with the generalized Dyck paths.
 Here $\la(\i)$ is the Young diagram of the
monomial ideal $\i^0$ generated by the leading term of  $f\in \i$.
The latter are minimal monomials with respect to the 
Gr\"obner-type ordering:\ 
$1 \prec x \prec y$, which means that $x^m\prec y$ for  $m>0$.
For  $\cn>1$,
we use the order $1\prec x \prec e_2/e_1\prec\cdots
\prec e_\cn/e_{\cn-1}\prec y$,
which means that $x^me_1\prec e_2,\ x^m e_\cn\prec y e_1$ and so on
for a (fixed)  basis $\{e_1,e_2,\ldots, e_\cn\}$ in $\a^n$. 

\comment{
What
happens if we make it $e\prec x\prec y$? A full answer is
in formula (\ref{sum-second}) and in Appendix \ref{app:sec-Gr}
(joint with ChatGPT).
}

Gr\"obner cells are, generally, 
very different from the {\sf\em Piontkowski cells}
in {\sf\em Compactified Jacobians}; see \cite{ChG}. 
For instance, Piontkowski cells are affine spaces 
in the case of $[\rr,\ss,\upsilon, p=1]$ due to \cite{Pi}
and recent \cite{GMO}; 
however some of them can be empty.
Explicit Piontkowski-style formulas were used in
\cite{GMO} to justify that nonempty ones are one-to-one with
{\sf\em Dyck paths} above the diagonal in the
rectangle $\upsilon \rr \times \upsilon\ss$. 
This  gives (conditionally) the 
Euler characteristic of the corresponding Compactified Jacobian
(when $\aa=0,q=1$ and $\cn=1$).
This formula is a special case of Corollary 5.48 from \cite{KTs} 
(for {\sf\em any} algebraic knots). Also, Corollary 5.47 gives
that $\h^{mot}(q,t\!=\!1,\aa\!=\!0)$ is a $q$-polynomial.
%\medskip

In our approach,  $|G_{\mu\subset \la}|=1$ modulo 
a sum of nontrivial divisors of $(q-1)^2$, which formally
vanish at $q=1$ for any $\la$
and sufficiently general  $q$. Indeed, $I_\mu\in G_{\mu\subset \la}$
is a unique fixed  point for  the  
action $(\F_q^\ast)^2\ni (u,v):\, (x,y)\mapsto
(ux,vy)$, and  $ G_{\mu\subset \la}$ becomes a 
disjoint union of $I_\mu$ and non-empty
orbits of such an action. In particular,
Theorem \ref{thm:echar1} states that
if $\sH^{inst}_\la(q,t=1,\aa=0)$
is a $q$-polynomial, its value at $q=1$, the Euler characteristic,
equals the number of $\mu\subset \la$.
% a powerful extension
%of the formulas above to any $\la$.
%this argument  can be extended to any rank $\cn$. 
%See 
%by adding the action of $\F_q^{\cn-1}$.

%compare 
%with \cite{HS}. This direction is challenging beyond torus
%knots; we provide examples. 

%\vskip 0.2cm
\subsection{\bf Concluding remarks}
The intersection of  DAHA superpolynomials
from \cite{CJ,GoN,CJJ,ChD1,ChD2} with  new 
DAHA superpolynomials $\mathfrak{H}_\la$ 
is for 
torus knots $T(\rr,\ss)$ 
and  cables (satellites) 
$C\!ab(\upsilon\rr\ss\pm 1,\upsilon)T(\rr,\ss)$,
where  $\rr>\ss>0$ and $gcd(\rr,\ss)=1$.
These are Theorems \ref{thm:Htor}, \ref{thm:sharp-flat}
for $T(\rr,\ss)$, and for  $[\rr,\ss,\upsilon, p=\pm 1]$. 
We focus on {\sf\em knots} in this paper;
 $\mathfrak{H}^{daha}_\la$ are studied only for
{\sf\em
Coxeter knots}, the ones from \cite{ObR2, GL2} and some 
prior papers. Potentially this can be extended to links; 
Theorem \ref{thm:mainred} includes 
{\sf\em multibranch} 
singularities for $\cn=1$.

The lifted conductors $\c$ 
of plane curve singularities are generally 
{\sf\em non-monomial}, but their calculation is very much
combinatorial. The generators are described in terms
of the so-called {\sf\em semiroots} \cite{AM,Cam}. See Theorem  
\ref{thm:general-semiroot-conductor}.
From this perspective,
the case of monomial conductors  
seems very special. However, that fact that
it serves some hyperbolic knots, including important ones,
is quite a motivation.

%They already provide superpolynomials for
%quite a few hyperbolic knots, which are very difficult to
%obtain topologically.
\Yboxdim7pt
%including  K12n242 for the diagram $\la=\yng(3,2)$.

As we mentioned, the 
{\sf\em EHA-superpolynomials} were introduced
in \cite{GL2} for arbitrary $\la$. For torus links,
they are from Definition 4.23 in \cite{KTs}, which
match uncolored DAHA superpolynomials for torus
knots from
\cite{CJ} (in type $A$), and those for links in \cite{ChD2}.
See Theorem \ref{thm:Htor}. 
Conjecturally, uncolored
$\mathfrak{H}^{daha}_\la$ 
from Definition \ref{def:frakH} for  $\mathbbm{n}=1$ 
coincide with the EHA-superpolynomials for any $\la$.
See also \cite{KTs} for the case of arbitrary 
algebraic knots upon $t=1$ and the relation 
to \cite{ChD1}.

\vskip 0.2cm
The approach of these papers is via
{\sf\em Elliptic Hall algebras}, which are stable spherical DAHA in
type $A$ due to \cite{SV}. We used {\sf\em the whole} DAHA, and 
that the {\sf\em nonsymmetric} Macdonald polynomials 
are simply $X_{\om_r}$ for minuscule $\om_r$.  
%Accordingly, the 
%{\sf\em DAHA coinvariant} is the key for us, or coinvariants
%(there are various options).
%used vs. their approach based on a certain
%{\sf\em plethystic} formulas. 
Generally, DAHA-Jones polynomials
from \cite{CJ,CJJ, ChD1,ChD2} are for any dominant weights,
and superpolynomials are their stabilization in type $A$.
The stabilization was conjectured
in \cite{CJ} for $BCD$ and, later, for the 
``E-series" \cite{DG,ChE}.

\medskip
Thus, $\mathfrak{H}^{daha}_\la$ can be considered a special
DAHA chapter, which utilizes exceptional features
of minuscule weights $\om_m$:\, the formula 
$E_{\om_m}=X_{\om_m}$, which is a link from our superpolynomials
to Shuffle Conjecture, 
and simple action of the projective 
$PSL_2(\Z)$ on $X_{\om_m}$ and $Y_{\om_m}$. 
%Hyperbolic knots are quite a motivation.
\medskip

%We used some elements of it.
An interesting new
phenomenon is that the 
Connection Conjecture $\mathfrak{H}^{daha}_\la\leftrightarrow
\sH^{inst}_\la$,
% similar to $\h^{daha}=\h^{mot}$ 
%from \cite{ChQ} and prior papers,  
holds only for 
sufficiently general $\la$.
For instance, Theorem  \ref{thm:2-row} states that the condition
$b\ge 2a-1$ is necessary and sufficient for the
coincidence among the two-row diagrams $(b,a)$. 
\vskip 0.2cm

The case of $2$-row diagrams
appeared already challenging combinatorially; we provide
full justifications, including the corresponding dimension
formulas. The intersection of the $\la$-theory with
the classical theory of plane curve singularities and algebraic
knots is only for diagrams $(2\pp,\pp)$
and $(2\pp-1,\pp-1)$, which correspond to $T(3\pp\pm 1, 3)$ .
So there are  just $2$ of them when $a\ge 1$ is fixed.
% All other superpolynomials
%are for {\sf\em hyperbolic knots} in this case. 
The simplest superdual $\sH_\la^{inst}$ is for $(3,2)$, 
exactly the case of the famous K12n242.
\medskip

Conjecture \ref{conj:supd} gives general conditions for 
$\la$ with  superdual $\sH^{inst}_\la$ and (equivalently) positive  
$\mathfrak{H}^{daha}_\la$, and (equivalently) their coincidence,
We (with ChatGPT) 
checked it systematically, but it is not impossible that 
 corrections may occur beyond what we can reach numerically.
It was fully checked 
for $3$-row diagrams and quite significantly 
for $\la$ with up to $4$ corners. The top Young diagrams we
considered were like  $\la=(5,5,4,3,3,2,1)$, to give an
idea of the range of our calculations.
%\medskip

{\bf Some perspectives.}
One can expect some physics interpretation of the conductors $\sC$ in 
$\mathscr{H}^{inst}_\mathscr{C}$; some kind of  
bound state? Anyway, {\sf\em conductors}
result in  {\sf\em superduality}, fundamental 
in related physics. 

Mathematically, 
$\mathfrak{S}^{(\mathbbm{n})}_{\mathscr{C}}$ provide
a robust generalization of {\sf\em Compactified Jacobians} of
plane curve singularities, with expected implications 
in the theory of 
affine Springer fibers (type $A$), Hitchin systems,
Kac-Moody algebras and beyond. Conjecture  \ref{conj:mixed}
on mixed characteristic hints on even greater perspective. 
Generally,  the definition of
$\mathscr{H}^{inst}_\mathscr{C}$ is very relaxed, but superduality
requires very special conductors $\mathscr{C}$. 

\medskip
In topology, 
the passage from (any) diagrams $\la$ to knots is
due to \cite{GL1,GL2,ObR1,ObR2}; the corresponding knots
are called {\sf\em Coxeter knots}.
The definition is mostly due to Oblomkov-Rozansky
(with some prior related constructions).
However \cite{GL2} is more suitable for our paper,
and it contains an interesting connection with, what they call,
{\sf\em monotone paths} in $T^2$, connected with {\sf\em postroid} 
paths from \cite{GL1}.
Generally, the place of Coxeter links
in knot theory
remains unclear. 

\medskip
Algebraically,
theory of ``good" superpolynomials 
of instanton slices for 
non-monomial $\sC$
is wide open  beyond $I_\la$ and lifted conductors. 
%Some of non-monomial $\sC$  are related to 
%hyperbolic knots:\, the corresponding 
%HOMFLY-PT polynomials match. 
Recall that the conductors
are generally non-monomial for algebraic knots.
%be obtained from any $I_\la$ or lifted conductors $\c$ 
%(which give only cables).
Interestingly, we found a relatively small
example of 
instanton superpolynomial with non-monomial
$\sC$ that seems beyond reduced KhR-polynomials. 
It does not match any of the 
HOMFLY-PT polynomials with  $16$ crossings.
See the ``challenge" around $(\ref{chall})$.
%\medskip

Superpolynomials $\h_K$
contain a lot of important geometric information 
about the corresponding 
 $S^3\setminus K$ and $3$-folds. For instance, our
$\rho(q,t)$  defined in (\ref{rho-zun}) via the 
Heegaard-Floer specialization $\aa\mapsto -t/q$ is
a refinement of the Casson invariant
$\la(S^3_{+1}(K))$ of the $+1$-surgery on $K$, with direct
links to instanton Flower homology and the Neumann-Wahl conjecture
\cite{NeO}. The Instanton
Volume Conjecture in Theorem \ref{thm:volconj} appeared
a comprehensive test of our theory, collecting almost all
its aspects and connecting hypergeometric and arithmetic
geometry.

The connection to $p$-adic theory is exciting too. 
The case of mixed characteristic is   
for $\Z_p[[x]]$ instead of $\F_q[[x,y]]$. For instance, K12n242 has
its counterpart in $\Z_p[[[x]]$. 
%This is  connected with Iwasawa theory. 
Our conjecture is that the corresponding
superpolynomials always coincide; it was checked for $3$-row diagrams
by ChatGPT (preliminary, for $4$ rows):\, Conjecture \ref{conj:mixed} and
Appendix \ref{app:mot-top}. Concerning other connections to Number Theory,
let us mention ``Weak Riemann Hypothesis" from Section \ref{sec:rh-hyper}.

\subsection{\bf Summary of main results}
\ \ 

(1) Main Theorem \ref{thm:mainred}:\, the passage from Compactified
Jacobians to instanton slices $\mathfrak{S}_\c$ and the formula
for $\h_\r^{mot}$ in terms of 
$\sH^{inst}_\c$ for the lifted conductors $\c$; generalizations to
any ranks $\cn$.
%\vskip 0.2cm

(2) The definition of 
$\mathscr{H}^{inst}_\mathscr{C}$ via counting 
$\F_q$-points of instanton slices 
$\mathfrak{S}^{(\mathbbm{n})}_\sC$ (some stacks),
%superpolynomials presumably serving hyperbolic knots the connection
%of $\sH^{inst}_\la$ for $\sC=I_\la$ with the corresponding
% hyperbolic knots,
a  surprising link between arithmetic geometry and hyperbolic geometry
due to our Connection Conjectures.
%\vskip 0.2cm

(3) Theory of DAHA-superpolynomials $\mathfrak{H}^{daha}_\la$, including 
their coincidence for  $\la$ and its transpose
 in Theorem \ref{thm:trans}, and the Young diagrams 
related to rectangles $\up \rr\times \up\ss$ in 
Theorem \ref{thm:sharp-flat}.
%$\la$ and many examples.
%\vskip 0.2cm

(4) Main Conjecture \ref{conj:main}\,:\,
$\sH^{inst}_\la (q,t,\aa;\,\cn)=
\mathfrak{H}^{daha}_\la(t q^\cn,1/q,\,\aa)[\cn]$,
%$\sH^{inst}_\la(q,t,\aa)=$ %for $\sC=I_\la^\cn$ with 
%$t^{\cn}\,\mathfrak{H}^{daha}_\la(q, 1/(q^nt),\aa)$ 
assuming the superduality for $\cn=1$ of the former or 
(equivalently) the positivity of the latter. It becomes (1)
for lifted conductors $\c$.  
%\vskip 0.2cm

(5) Combinatorial Conjecture \ref{conj:supd}, which gives
conditions for $\la$ ensuring the Main Conjecture, its proof in the
case of $2$-row diagrams in Theorem  \ref{thm:2-row} and various 
other aspects.

(6) DAHA and Instanton Volume Conjectures for knots $K_\la$,
including many interesting ones, which is well-checked numerically 
for K12n242 and K12n725 and  seems a ground-breaking
development. 
 
%\vskip 0.2cm
(7) Product formula for enhanced Nekrasov's instanton sums as
a limit of $\h_\r^{mot}$; an impressive proof 
in the DAHA setting in Theorem \ref{thm:daha-free}, linking
it to nonsymmetric Macdonald polynomials.
%\vskip 0.2cm

(8) Mixed-Characteristic Conjecture  \ref{conj:mixed}, which
establishes a connection between $\sH^{inst}_\la$ and
superpolynomials for  $\Z_p[[x]]$; for instance, this gives 
a counterpart of $\sH^{inst}$ of  K12n242 via  $\Z_p$.
%\vskip 0.2cm

(9) Motivic superpolynomials for irreducible $\r$ 
with up to $3$ generators of the 
corresponding valuation semigroup $\Ga$ are
topological invariants; Appendix \ref{app:mot-top} (joint with
ChatGPT) and Corollary \ref{cor:3gen-top}.

(10) More special directions: \ arbitrary one-dimensional
Gorenstein rings and 
non-algebraic cables, the multibranch part of the Main Theorem 
 in $\cn=1$, and extended
Weak Riemann Hypothesis.

%\vskip 0.2cm
%Also, there are $2$ other Appendices written by ChatGPT (under the
%supervision of the author), which contributed a lot to this paper.

\comment{
A general definition is expected to
be {\sf\em Schubert-style}, like this one:

\noindent 
{\it Given a flag of modules 
$\mathscr{F}=\{M_0\subset M_1\subset \cdots
\subset M_k\subset \a^\cn\}$, and a sequence 
of numbers $\{d_0\ge  d_1\ge \cdots\ge d_k\ge 0\}$,
the corresponding instanton slice is 
$\mathfrak{S}_\mathscr{F}\equal \{\m\subset \a^\cn\, \mid\, 
dim_{\F}\bigl(\m/(\m \cap M_i)\bigr)=d_i\}$. }
}

\setcounter{equation}{0}
\section{\sc Double Hecke algebras}\label{sec:daha}
\subsection{\bf Affine root systems}\label{se:roots}
Let 
$R=\{\al\}\subset \R^\nn$ be a reduced irreducible
root system. Here and further  $\al_i$ for $1\le i\le \nn$ 
are simple roots,  
$Q=\oplus_{i=1}^\nn \Z\al_i$ is the root lattice.
The weight lattice is $P=\oplus_{i=1}^\nn \Z\om_i$ for the
fundamental weights $\om_i$; the  dominant weights
are those from $P_+=\oplus_{i=1}^\nn \Z_+\om_i$ for 
$\Z_{+}=\{m\in\Z, m\ge 0\}$. The inner product in $\R^\nn$
will be normalized by the condition 
$(\vth,\vth)=2$ for the maximal {\sf\em short} root $\vth$.
This normalization is better compatible with the passage to
Quantum Group invariants for $B,C,F,G$.
In the simply-laced case, $\vth=\th$
for the maximal root $\th$.
See,  e.g., \cite{Bo} and \cite{C101}.

We will mostly need minuscule weights and
type $A$ below. However, let us provide
the full DAHA construction
from \cite{CJ,CJJ,ChD1,ChD2,ChQ} for any dominant weights 
and iterated torus knots. It is quite uniform. We 
extend it in this paper to Coxeter knots in the sense of 
\cite{ObR2,GL2}, generally hyperbolic, but this will be done
only for {\sf\em minuscule} weights. Let us mention here
\cite{AS}, which triggered \cite{CJ}, and \cite{GS},
where {\sf\em superduality} was predicted (for any colors).

Vectors $\ \tal=[\al,j]\in \R^{\nn+1}$
for $\al \in R, j \in \Z $ form the
{\sf\em affine root system\,}
$\tR \supset R$, where $\al\in R$ are identified with
$[\al,0]$. Generally, $\tal=[\al,\nu_{\al} j]$, where
$\nu_\al=1$ for short roots and $\nu_\al=2,3$ for long 
ones ($3$ is for $G_2$). The set
$\tR_+$ of affine positive roots is %%%% tR_+BOOK!!!
$R_+\cup \{[\al,j],\ \al\in R, \ j > 0\}$. Adding
$\al_0 \equal [-\vth,1]$ to simple $\{\al_i, i>0\}$
is the passage from the standard Dynkin
diagram to the corresponding extended one.

In the case of $A_\nn$, let $N=\nn+1$, %$\mathbb{R},\mathbb{X},\mathbb{F}$
$R\!=\!\{\al\!=\!\mathbbm{e}_i\!-\!\mathbbm{e}_j, i\!\neq\! j\}$
for the basis $\{\mathbbm{e}_i, 1\le N\}\in \R^{N}$,
orthonormal with respect to the usual euclidean form
$(\mathbbm{e}_i,\mathbbm{e}_j)=\de_{ij}$. The Weyl group is $W=\S_{N}$, 
generated by $s_\al$ for the set of positive  roots
$R_+=\{\mathbbm{e}_i-\mathbbm{e}_j, i<j\}$.
The simple roots are $\al_i=\mathbbm{e}_i\!-\!\mathbbm{e}_{i+1},$ 
and $\al_0=[-\th,1]$
for $\th=\mathbbm{e}_{1}-\mathbbm{e}_{N}$. The fundamental weights 
and $\rho$ are:
%the weight lattice is
%$P=\oplus^\nn_{i=1}\Z \om_i$,
%where $\{\om_i\}$ are fundamental weights, satisfying
%$ (\om_i,\al_j)=\de_{ij}$. For $A_\nn$:\,
\begin{align}%\label{omviavep}
&\om_i=\mathbbm{e}_1+\cdots+\mathbbm{e}_i-\frac{i}{N}
(\mathbbm{e}_1+\cdots+\mathbbm{e}_{N})
\for i=1,\ldots,N-1,\\
\rho=&\om_1\!+\cdots+\!\om_{N-1}=\frac{1}{2}\bigl(
(N\!-\!1)\mathbbm{e}_1+(N\!-\!3)\mathbbm{e}_2+\cdots+(1\!-\!N)
\mathbbm{e}_N\bigr).\notag
\end{align}

We will consider  below the root system $gl_N$, when
$\om_i=\mathbbm{e}_1+\ldots \mathbbm{e}_i$ for $1\le i\le N$. 
A simple but important fact from \cite{CJ} and
further papers is that the passage
from $A_{N-1}$ to $gl_N$ does not change the
DAHA-Jones polynomials up to some factor $q^\bullet t^\bullet$,
and does not change the corresponding DAHA-superpolynomials.
We will use this later. 

%\smallskip
\vskip 0.2cm

{\bf Affine Weyl groups.} 
Given $\tal=[\al,j]\in \tR,  \ b \in P$, let
\begin{align}
&s_{\tal}(\tz)\ =\  \tz-(z,\al^\vee)\tal,\
\ b'(\tz)\ =\ [z,\ze-(z,b)]
\label{ondon}
\end{align}
for $\tz=[z,\ze]\in \R^{\nn+1}$.
The
{\sf\em affine Weyl group\,}
$\tW=\lan s_{\tal}, \tal\in \tR_+\ran$
is the semidirect product $W\lsmash Q$ of
its subgroups $W=$ $\lan s_\al,
\al \in R_+\ran$ and $Q$, where $\al$ in the latter are identified with
\begin{align*}
& s_{\al}s_{[\al,\,1]}=\
s_{[-\al,\,1]}s_{\al}\text{ for any }
\al\in R.
\end{align*}

The {\sf\em extended Weyl group\,} $ \hW$ is $W\lsmash P$, where
the action is
\begin{align}
&(wb)([z,\ze])\ =\ [w(z),\ze-(z,b)] \for w\in W, b\in P.
\label{ondthr}
\end{align}
It is isomorphic to $\tW\rsmash \Pi$ for $\Pi\equal P/Q$.
The latter group consists of $\pi_0=$\,id\, and the
images $\pi_r$ of $\om_r$ in $P/Q$ for {\sf\em minuscule weights}.
From now on, $O$ will be the set of such $r$;\, $O'\equal
O\setminus \{0\}$.
 
Note that
$\pi_r^{-1}$ is $\pi_{r^\varsigma}$,  where $\varsigma$ is
the standard involution of the {\sf\em non-affine\,}
Dynkin diagram of $R$, which is
$\al_i\mapsto \al_{\nn+1-i}$ for $A_\nn$.
Generally, we set
$\varsigma(b)=-w_0(b)=b^{\,\iota}\,$, where $w_0$ is the
longest element
in $W$. For $A_\nn$:\,  $w_0$
 sends $ \{1,2,\ldots,\nn+1\}$
to $\{\nn+1,\ldots,2,1\}$, and $\Pi=\Z_{\nn+1}$.
\vskip 0.2cm

The group $\Pi$
is naturally identified with the subgroup of $\hW$ of 
elements of length zero; the {\sf\em length\,} is defined as
follows:
\begin{align*}
&l(\hw)=|\La(\hw)| \for \La(\hw)\equal\tR_+\cap \hw^{-1}(-\tR_+).
\end{align*}
One has $\om_r=\pi_r u_r$, 
where $u_r$ is the
element $u\in W$ of {\sf\em minimal\,} length such that
$u(\om_r)\in P_-$,
equivalently, $w=w_0u$ is of {\sf\em maximal\,} length such that
$w(\om_r)\in P_+$. This is for any $1\le r\le \nn$, but we will
need it only for $r\in O$. Then 
$u_r$ are very explicit.
Let $w^r_0$ be the longest element
in the subgroup $W_0^{r}\subset W$ of the elements
preserving $\om_r$; this subgroup is generated by simple
reflections. Generally,
\begin{align}\label{ururstar}
u_r = w_0w^r_0 \hbox{\,\, and\,\, } (u_r)^{-1}=w^r_0 w_0=
u_{r^\iota} \for 1\le r\le \nn.
\end{align}
\smallskip

Setting $\hw = \pi_r\tw \in \hW$ for $\pi_r\in \Pi,\, \tw\in \tW,$
\,$l(\hw)$ coincides with the length of any reduced decomposition
of $\tw$ in terms of the simple reflections
$s_i,\, 0\le i\le \nn.$ Thus, $\Pi$ is, indeed, a subgroup of
$\hW$ of the elements of length $0$. We will
also use below the {\sf\em partial lengths} 
denoted by $l_\nu(\hw)$, where we
count long and short $s_i$ separately. 

%\smallskip

\subsection{\bf Definition of DAHA}
For the irreducible reduced root system $R$ and the
definitions above,
let $\mm$ be the smallest natural number
such that  $(P,P)=(1/\mm)\Z.$  Thus,
$\mm=|\Pi|$ unless 
$\mm=2 \for D_{2k}$ and $\ \mm=1 \for B_{2k},C_{k}.$ 
Recall that there are one or two different lengths
$\nu$ of roots $\al\in R$ and those in $\tR$:\, 
$\nu=\nu_{\sht}=1,\ \nu=\nu_{\lng}.$

\vskip 0.2cm

{\sf\em Double affine Hecke algebras},  DAHA,  depend
on the parameters
$q, t_\nu$.  These algebras are defined
over the ring
$\Z_{q,t}\equal\Z[q^{\pm 1/\mm},t_\nu^{\pm 1/2}]$
formed by
polynomials in terms of $q^{\pm 1/\mm}$ and
$\{t_\nu^{\pm1/2}\}.$ 

It will be convenient to use parameters
$k_\nu$ (one for $ADE$ or two otherwise). We set:
$t_\nu =q_\nu^{k_\nu}=q^{\nu k_\nu}$ for $q_\nu=q^\nu$,
and $ t_{\tal}= t_{\nu_\al}=q_\al^{k_\al} ,
q_{\tal}\equal q^{\nu_\al},$
for $\tal=[\al,\nu_\al j] \in \tR,\ 0\le i\le \nn$.
Also, $t_i=t_{\al_i}=q_i^{k_i}$.

Let 
$\rho_k\equal \frac{1}{2}\sum_{\al\in R_+}k_\al \al$.
Using $\rho_\nu\equal \frac{1}{2}\sum_{\nu_\al=\nu}\al$,
we have: $\rho_k=\sum_\nu k_\nu \rho_\nu=k_{\sht}\rho_{\sht}+
k_{\lng}\rho_{\lng}$. 
The standard argument based on the application of $s_i$ for $i>0$
to  $\rho_\nu$ gives that $(\rho_\nu, \al_i^\vee)=1$ for $\nu_i=\nu$ 
and $0$ otherwise (for  $i>0$). In particular, 
 $\rho_k=\sum_{i=1}^\nn k_i\om_i$. 
\vskip 0.2cm

For pairwise commutative $X_{\om_1},\ldots,X_{\om_\nn},$ let
\begin{align}
& X_{\tb}\ \equal\ q^{ j}\prod_{i=1}^\nn X_{\om_i}^{l_i} 
\iif \tb=[b,j],\ \hw(X_{\tb})\ =\ X_{\hw(\tb)},
\label{Xdex}\\
&\hbox{where\ } b=\sum_{i=1}^\nn l_i \om_i\in P,\ \, j \in
\frac{1}{ \mm}\Z \ \text{\  and \ } \hw\in \hW.
\notag \end{align}
For instance, one has: $X_{\al_0}=qX_\vth^{-1}$. 
From now on,  $X_i\equal X_{\om_i}$.

Also, let 
 $X_b=q^{(x,b)}$ and $X_i=q^{(x,\om_i)}$, where
$x\in \R^\nn$ and the pairing $(\cdot,\cdot)$ is as above.
 If the root system $gl_N$ is used, 
we set $\mathbb{X}_i=X_{\mathbbm{e}_i}$, where
$\om_i=\mathbbm{e}_1+\ldots\mathbbm{e}_i$. The same notation
will be used for $Y$.
\medskip

Recall that 
$\om_r=\pi_r u_r$ for $r\in O'$. One has:
$\pi_r^{-1}=\pi_{\varsigma(i)}$, where
$\,\varsigma\,$ is the action of $-w_0$ on roots and weights;
we set $X_b^\varsigma=X_{b^\varsigma}.$ 
%Finally, we set $m_{ij}=2,3,4,6$
%when the number of links between $\al_i$ and $\al_j$ in the affine 
%Dynkin diagram is $0,1,2,3$. 

\begin{definition}\label{def:daha}
The double affine Hecke algebra $\HH\ $
is generated over $\Z_{q,t}$ by 
$\h=\lan  \Pi, T_i,\ 0\le i\le \nn \ran$, subject to the homogeneous
Coxeter relations and the quadratic relations 
$(T_i-t_{i}^{1/2})(T_i+t_{i}^{-1/2})=0$, and by
pairwise commutative $\{X_b, \ b\in P\}$ satisfying
(\ref{Xdex}). Let $T_{\hw}=T_{\hat{u}} T_{\hat{v}}$ if 
$l(\hw)=l(\hat{u})+l(\hat{v})$,
and impose the following ``cross-relations":

(i)\ \ \  $T_iX_b \ =\ X_b X_{\al_i}^{-1} T_i^{-1}$ if
$(b,\al^\vee_i)=1,\,
0 \le i\le \nn$;

(ii)\ \ $T_iX_b\ =\ X_b T_i\ $ if $\ (b,\al^\vee_i)=0
\for 0 \le i\le \nn$;

(iii)  $\pi_rX_b \pi_r^{-1} = X_{\pi_r(b)} =
X_{ u^{-1}_r(b)}
 q^{(\om_{\varsigma(r)},b)} \text{ for }  r\in O'$. \sq
\label{double}
\end{definition}

\noindent
The action of $\hW$ in $\R^{\nn+1}$ is used 
in $(iii)$. Namely:\,
 $\pi_r(b)=\om_r u_r^{-1}(b)=
[u_r^{-1}(b),-(\om_r,u_r^{-1}(b))]$, where $-(\om_r,u_r^{-1}(b))=
(b,-u_r(\om_r))=(b,\om_{\varsigma(r)})$ and
$u_r^{-1}=u_{\varsigma(r)}$.  We will use below that
\begin{align}\label{uiomr}
&u_r(\om_r)=w_0(\om_r)=-\om_{\varsigma(r)},\,
X_r\pi_r=q^{(\om_r,\om_r)}\pi_r X_{\varsigma(r)}^{-1}\ \, (r\in O').
\end{align}

The following {\sf\em pairwise commutative} elements
are the key:
\begin{align}
& Y_{b}\equal
\prod_{i=1}^\nn Y_i^{l_i} \iif
b=\sum_{i=1}^\nn l_i\om_i\in P,\  b\in P,\ 
Y_i\equal T_{\om_i}=\pi_i T_{u_i}.
\label{Ybx}
\end{align}
They can be used instead of $T_0$ and $\{\pi_r\}$ in the
definition of DAHA. This makes the definition of DAHA
$X\leftrightarrow Y$-symmetric, where this  symmetry
is a formalization of {\sf\em Fourier Transform} in many
theories. 

\comment{
When acting in the polynomial representation
(see below), they are called {\sf\em difference
Dunkl operators}. We arrive at the presentation
$\HH\!=\!\lan X_b,T_w,Y_b,q^{\pm 1/\mm},
t_\nu^{\pm 1/2}\ran,\, b\in P,\, w\in W$.
The cross-relations of $\{Y_b\}$ with $\{T_i, X_b\}$ 
result from those for $T_0$ and the relations in 
$\h_X\equal \lan T_i X_b\ran$, where
$1\le i \le \nn,\, b\in P$. 
The algebra $\h_X$ is isomorphic
to $\h=\h_Y$ under the map $X_b\mapsto Y_b^{-1}$, $T_w\mapsto T_w$. 
%\smallskip
}
\vskip 0.2cm

\subsection{\bf  Automorphisms and involutions}\label{sect:Aut}
We add $q^{1/(2\mm)}$  to $\Z_{q,t}$ 
and treat it and $t_\nu^{1/2}$ as
central elements. 
The following formulas define 
automorphisms of $\HH\,$
(see \cite{C101}, (3.2.10)--(3.2.15))\,:
\begin{align}\label{tauplus}
& \tau_+:\  X_b \mapsto X_b, \ T_i\mapsto T_i\, (i>0),
\ Y_{\om_r} \mapsto X_{\om_r}
Y_{\om_r} q^{-\frac{(\om_r,\om_r)}{2}},
\\
& \tau_+:\ T_0\mapsto  q^{-1}\,X_\vth T_0^{-1},\ \,
\pi_r \mapsto q^{-\frac{(\om_r,\om_r)}{2}}X_{\om_r}\pi_r\
(r\in O'),\notag\\
& \label{taumin}
\tau_-:\ Y_b \mapsto \,Y_b, \ \, T_i\mapsto T_i\, (i\ge 0),\
\ X_{\om_r} \mapsto Y_{\om_r} X_{\om_r}
 q^\frac{(\om_r,\om_r)}{ 2},\\
&\tau_-(X_{\vth})= 
q T_0 X_\vth^{-1} T_{s_{\vth}}^{-1};\ \
\si\equal \tau_+\tau_-^{-1}\tau_+\, =\,
\tau_-^{-1}\tau_+\tau_-^{-1},\notag\\
&\si(X_b)\!=\!Y_b^{-1},\   \si(Y_b)\!=\!
T_{w_0}^{-1}X_{b^{\,\varsigma}}^{-1}T_{w_0},\ \si(T_i)\!=\!T_i (i>0),
\label{taux}
\end{align}
where automorphism $\si^{-1}$
is the 
{\sf\em DAHA-Fourier transform}. 

In the case of DAHA of
type $A_{N-1}$ or type $gl_N$, the extended ring is
$\Z[q^{\pm \frac{1}{2N}},t^{\pm 1/2}]$,
formed by
polynomials in terms of $q^{\pm \frac{1}{2N}}$ and
$t^{\pm1/2}.$ All fundamental weights are
minuscule in this case, which simplifies 
the usage of the automorphisms and anti-involutions.
\vskip 0.2cm

{\bf Further properties.}
The action of $\si$  on $Y_r$ for $r\in O'$ is as follows: 
 $\si(Y_r)=\tau_{-}^{-1}\tau_+\tau_-^{-1}(Y_r)=
q^{-(\om_r,\om_r)}Y_r^{-1}X_r Y_r$  Also,
$\si(\pi_r)=T_{u_r}^{-1}X_{\varsigma(r)}^{-1}$ and 
$\si(Y_r)=T_{u_r}^{-1}X_{\varsigma(r)}^{-1}T_{u_r}$.
See formulas
(3.2.16) and (3.2.22) from \cite{C101}.
Let us provide the justification. First,
$\si(\pi_r)=T_{w_0}^{-1}X_{\varsigma(r)}^{-1}T_{w_0}T_{u_r}^{-1}=
T_{u_r}^{-1}X_{\varsigma(r)}^{-1}.$ Second, we represent 
$w_0=v u_{r}$ for  
$v=w_0 u_{\varsigma(r)}$. Then, $T_{w_0}=T_v T_{u_r}$, where
$v(\om_{\varsigma(r)})=
\om_{\varsigma(r)}$, which gives that
 $T_v$ commutes with $X_{\varsigma(r)}$. 
%See (\ref{ururstar}) and  
%formula (\ref{diamsi}) below. 
\vskip 0.2cm

Concerning the conjugation by  $T_{w_0}$, one has: 
$T_{w_0}^{-1}T_iT_{w_0}=T_{\varsigma(i)}$
for $i>0$, $T_{w_0}^{-1}T_0T_{w_0}=T_0$ and
 $T_{w_0}^{-1}\pi_r T_{w_0}=\pi_{\varsigma(r)}$. 
Generally,
$\si^2(H)=T_{w_0}^{-1} \varsigma(H) T_{w_0}$, where 
the involution $\varsigma$ is  naturally
extended to an automorphism
of $\HH\ni H$:
$$
X_b\mapsto X_{b^\varsigma},\, Y_b\mapsto Y_{b^\varsigma},\,
T_i\mapsto T_{i^\varsigma},\, \pi_r\mapsto \pi _{r^\varsigma},\ 
b\!\in\! P,\, i\!\ge\! 0,\, r\!\in\! O'.
$$

We obtain that {\sf\em projective\,} $PSL_2(\Z)$ due to Steinberg 
acts in $\HH$; it is  
generated by $\tau_{\pm}$ with the
relations $\tau_+\tau_-^{-1}\tau_+=\si=
\tau_-^{-1}\tau_+\tau_-^{-1}$.  The notation will
be $PSL_{\,2}^{\wedge}(\Z)$; it is isomorphic to the braid 
group $B_3$. Note that
$\si \tau_{\pm}=\tau_{\mp}^{-1}\si$.
\vskip 0.2cm

The following important automorphism conjugates $q$ and $t$ and
sends
\begin{align}\label{aut-eta}
&\eta: X_b\mapsto X_b^{-1}, Y_b\mapsto 
T_{w_0} Y_{w_0(b)}T_{w_0}^{-1}, T_i\mapsto T_{i}^{-1} (0\le i\le n),\\
\eta(\pi_r)&=\pi_r \text{ for } r\in O', \text{  and  } \eta(Y_r)=
T_{u_{r^\varsigma}} Y_{r^\varsigma}^{-1} T_{u_{r^\varsigma}}^{-1}\sim
X_r^{-1}Y_r X_r, 
\notag
\end{align}
where $u_{r^\varsigma}=u_r^{-1}$, and $\sim$ means ``up to 
$q^\bullet t^\bullet$". The conjugation sends $q\mapsto q^\star=q^{-1}, 
t_\nu \mapsto t_\nu^\star=t_\nu^{-1}$ and  
for their fractional powers. The previous automorphisms
fixed $\ t_\nu,\ q$
(and their fractional powers).
\vskip 0.2cm

Let us justify the last relation. Using (\ref{uiomr}), we go
from $T_{w_0}$ to $T_{u_{r^\varsigma}}$. Next,
we apply the anti-involution $\vph$ (below) to $\pi_r^{-1}X_r\pi_r\sim
X_{r^\varsigma}$:\, 
 $\vph(\pi_r)Y^{-1}_r (\vph(\pi_r))^{-1}\sim Y_{r^\varsigma}$ and
 $\vph(\pi_r)^{-1}Y^{-1}_{r^\varsigma}\vph(\pi_r)=Y_r$. Finally,
$X_r=\vph(\pi_r)^{-1}T_{u_{\varsigma(r)}}^{-1}$, 
$T_{u_{\varsigma(r)}}=X_r^{-1}(\vph(\pi_r))^{-1}$, and we
arrive at
$T_{u_{r^\varsigma}} Y_{r^\varsigma}^{-1} T_{u_{r^\varsigma}}^{-1}\sim
X_r^{-1}Y_r X_r$. One can use here $\si$ instead of $\vph$. 
\vskip 0.2cm

Let us list the matrices corresponding to the automorphisms 
above upon the natural projection
onto $PGL_2(\Z)$.
The matrix {\footnotesize
$\begin{pmatrix} \al & \be \\ \ga & \de\\ \end{pmatrix}$}
will then represent the map $X_b\mapsto X_b^\al Y_b^\ga,
Y_b\mapsto X_b^\be Y_b^\de$ for $b\in P$ when
$\,t^{\frac{1}{2}}=1=q^{\frac{1}{2(\nn+1)}}$.
 One has:
\smallskip

\centerline{
$\!\!\!\tau_+\rightsquigarrow$
{\tiny
$\begin{pmatrix}1 & 1 \\0 & 1 \\ \end{pmatrix}$},\
$\tau_-\rightsquigarrow$
{\tiny
$\begin{pmatrix}1 & 0 \\1 & 1 \\ \end{pmatrix}$},\
$\si\rightsquigarrow$
{\tiny
$\begin{pmatrix}0 & 1 \\-1 & 0 \\ \end{pmatrix}$},\
$\eta\rightsquigarrow$
{\tiny
$\begin{pmatrix}-1 & 0 \\ 0 & 1 \\ \end{pmatrix}$}.\
}
\vskip 0.1cm

\noindent
So $\eta$ is from   $PGL_{\,2}^{\wedge}(\Z)$,
{\em projective\,} $PGL_2(\Z)$.

\vskip 0.2cm

{\bf Anti-involutions.} The following one is
the key in the definition of {\sf\em coinvariant}.
It fixes $\ t_\nu,\ q$
and their fractional powers, and sends:
\begin{align}
&\vph:\ 
X_b\mapsto Y_b^{-1},\, Y_b\mapsto X_b^{-1},\,
T_w\mapsto T_{w^{-1}}\, (w\in W);
\label{starphi}
\end{align} 
In particular, $\pi_r\mapsto 
T_{u_r}^{-1}X_r^{-1}, T_0\mapsto (X_\vth T_{s_\vth})^{-1}.$
For $b\in P$:
\begin{align}\label{vphtaupm}
\vph \tau_+\vph\!=\!\tau_-,\, \vph \si\!=\!\si^{-1}\vph,\,
\vph\si^{-1}(Y_b)\!=\!Y_b, \,\vph\bigl(\tau_+^{-1}(Y_b)\bigr)\!=\!
\tau_+^{-1}(Y_b),
\end{align}
which is direct from the definitions.
Also, for $i\ge 0$ and $r\in O$:
\begin{align}
&\vph(\tau_+(T_i))=\tau_+(T_i),\ 
\vph(\tau_+(\pi_r))= (\tau_+(\pi_r))^{-1}=\tau_+(\pi_{r^\varsigma}).
\label{vphtau}
\end{align} 
For the sake of completeness, let us justify (\ref{vphtau}).
We need  to check the first formula only for $i=0$:\,
$\tau_+(T_0)=q^{-1}X_\vth T_{s_\vth} Y_\vth^{-1}$ is obviously
$\vph$-invariant. For the 2{\footnotesize nd}:  $\pi_r=Y_r T_{u_r}^{-1}=
\pi_{r^\varsigma}^{-1}=T_{u_{r^\varsigma}}Y_{r^\varsigma}^{-1}$.
Applying $\vph$, we obtain the identities 
$T_{u_{r^\varsigma}}^{-1}X_r^{-1}=X_{r^\varsigma}T_{u_r}$, \, 
$X_r T_{u_{r^\varsigma}}=T_{u_r}^{-1} X_{r^\varsigma}^{-1}$ \, and
\begin{align*}
\tau_+(\pi_r)=&\,
q^{-\frac{(\om_r,\om_r)}{2}}X_r\pi_r\,=
q^{-\frac{(\om_r,\om_r)}{2}}X_rY_rT_{u_r}^{-1}\\
=&\,
q^{\frac{(\om_r,\om_r)}{2}}\pi_r X_{r^\varsigma}^{-1}
=q^{\frac{(\om_r,\om_r)}{2}}
Y_rT_{u_r}^{-1}X_{r^\varsigma}^{-1}=
q^{\frac{(\om_r,\om_r)}{2}}
Y_r X_r T_{u_{r^\varsigma}}.
\end{align*}
Therefore  $\vph(X_rY_rT_{u_r}^{-1})=
T_{u_{r^\varsigma}}^{-1}X_r^{-1}Y_r^{-1}$ and we obtain the required.
See formula (3.2.12) in \cite{C101}.

\comment{
\vskip 0.2cm
The  following anti-involution $\star$ results directly from
the group nature of the DAHA relations. Let 
$
H^\star= H^{-1} \for H\in \{T_{\hw},X_b, Y_b, \pi_r, q, t_\nu\}.
$
To be exact, it is naturally extended to the fractional
powers of $q,t$:
$$
\star:\ t^{\frac{1}{2}}_\nu \mapsto t_\nu^{-\frac{1}{2}},\
q^{\frac{1}{2\mm}}\mapsto  q^{-\frac{1}{2\mm}}.
$$
It commutes with any (anti-)automorphisms of $\HH$. 
This anti-involution is standard in the theory 
of the DAHA polynomial representation $\mathscr{X}$ (below).
In this paper, the {\sf\em anti-involutions} $\Diamond_l$ 
are key, where  $l\in \Z$. The other DAHA automorphisms
are mostly needed to
introduce them and justify their main properties. They  
preserve $q,t_\nu$ and send:
\begin{align}\label{diamondef}
&  \Diamond: X_b\!\mapsto\! T_{w_0}^{-1}X_{-w_0(b)}T_{w_0},\,  
Y_b\!\mapsto\! Y_b,\, T_w\!\mapsto\! T_{w^{-1}},\, 
\pi_r\!\mapsto\! T_{w_0}^{-1}
\pi_rT_{w_0},
% \\
%\Diamond_l=& q^{l x^2/2}\circ\Diamond\,\circ q^{-lx^2/2}:
%X_a\mapsto X_a^{\Diamond},\ \,
%Y_b\mapsto q^{l x^2/2}Y_b q^{-lx^2/2}=\tau_+^l(Y_b),\notag
\end{align}
where $b\in P, w\in W, r\in O$. Here, formally $\Diamond(q^{lx^2/2})=
q^{lx^2/2}$; we use
that $x^2$ is $W$-invariant and $\varsigma$-invariant.
Thus, $\Diamond_l$ is the composition
$\tau_+^l\circ \Diamond$. We note that
$\Diamond=\vph\si^{-1}$, 
$\Diamond\circ \tau_{\pm}=\tau_{\pm}^{-1}\circ \Diamond$
and $\Diamond\circ \si=\si^{-1}\circ\Diamond.$

Chapter 3 of \cite{C101} is actually the theory of 
$\vph,\star,\Diamond_{\pm 1}$ and the corresponding
symmetric forms in the polynomial representation and its
Fourier-dual, which is the space generated by delta-functions
at the points $\pi_b(-\rho_k)=b-u_b^{-1}(\rho_k)$ for $b\in P$.
\vskip 0.2cm

Let us provide the counterpart of the symmetries
from (\ref{vphtau}) for $\Diamond$:
\begin{align}
&\Diamond(\si(T_i))=\si(T_i) (i\ge 0),\ 
\Diamond(\si(\pi_r))= (\si(\pi_r))^{-1}=\si(\pi_{r^\varsigma}).
\label{diamsi}
\end{align} 

The first relation is not immediate only for 
$\si(T_0)=T_{s_\vth}^{-1}X_\vth^{-1}.$ One has:
$\Diamond(\si(T_0))=T_{w_0}^{-1}X_\vth^{-1} T_{w_0}T_{s_\vth}^{-1}=
T_{s_\vth}^{-1}X_\vth^{-1} T_{s_\vth}T_{s_\vth}^{-1}=
T_{s_\vth}^{-1}X_\vth^{-1}.$ We use that
$T_{w_0}^{-1}X_\vth^{\pm 1} T_{w_0}=
T_{s_\vth}^{-1}X_\vth^{\pm 1} T_{s_\vth}$, which follows
from (3.2.22) in \cite{C101}, and can be check directly using
that $w_0=u s_\vth$ for $u$ such that $u(\vth)=\vth$; indeed,
$w_0(\vth)=-\vth=s_\vth(\vth)$. We obtain that
$T_{w_0}^{-1}X_\vth^{\pm 1} T_{w_0}=T_{s_\vth}^{-1}T_u^{-1}
X_\vth^{\pm 1} T_u T_{s_\vth}$, where $T_u$ commutes with 
any polynomial of $X_\vth$. 

The second equality is justified as follows. One has: 
$\Diamond(\si(\pi_r))=\Diamond(T_{u_r}^{-1}X_{\varsigma(r)}^{-1})=
T_{w_0}^{-1}X_r^{-1}T_{w_0} T_{u_{\varsigma(r)}}^{-1}=
T_{u_{\varsigma(r)}}^{-1} X_{r}^{-1}$ due to 
$T_{w_0}^{-1}X_{\varsigma(r)}^{\pm 1}
T_{w_0}=T_{u_r}^{-1}X_{\varsigma(r)}^{\pm 1}T_{u_r}.$ 
Alternatively, one can use here and above 
that $\Diamond
=\vph\si^{-1}$.
}

%\comment{ %AHA ANTI-INVOLUTION

\vskip 0.2cm
{\bf Relation to AHA.}
Let us discuss the 
anti-involution $\dag$ of DAHA, extending
the standard AHA anti-involution
for $\h=\h_Y$ generated by $T_i$ for $i\ge 0$ 
and $\pi_r$ for $r\in O'$. The defining relations are 
$T_i^\dag = T_i$ for $i\ge 0$, $\pi_r^\dag=\pi_r^{-1}$ for $r\in O$,
and $X_b^\dag= X_b$ for $b\in P$. Use directly 
Definition \ref{def:daha} to check that this is an anti-involution.

It is related to 
the anti-involution $\Diamond$, one of the key in DAHA theory:
\begin{align}\label{diamondef}
&  \Diamond: X_b\!\mapsto\! T_{w_0}^{-1}X_{-w_0(b)}T_{w_0},\,  
Y_b\!\mapsto\! Y_b,\, T_w\!\mapsto\! T_{w^{-1}},\, 
\pi_r\!\mapsto\! T_{w_0}^{-1}
\pi_rT_{w_0},
% \\
%\Diamond_l=& q^{l x^2/2}\circ\Diamond\,\circ q^{-lx^2/2}:
%X_a\mapsto X_a^{\Diamond},\ \,
%Y_b\mapsto q^{l x^2/2}Y_b q^{-lx^2/2}=\tau_+^l(Y_b),\notag
\end{align}
where $b\in P, w\in W, r\in O$. 
Namely,  $\dag=\varsigma\circ T_{w_0}\circ \Diamond$,
where by $T_{w_0}$, we mean the conjugation 
$T_{w_0} (\ldots) T_{w_0}^{-1}$,
and  $b^\varsigma\equal -w_0(b)$ for $b$, 
which is applied to the indices of 
$X_b, Y_b$. Accordingly, 
$T_i^\varsigma =T_{i^\varsigma}$ for $\varsigma$
acting naturally on the indices of $\{\al_i\}$
(in the extended Dynkin diagram). 

\comment{
For instance, the
relation $\Diamond(Y_b)=Y_b$ from (\ref{diamondef}) gives  that 
$Y_b^\dag= T_{w_0} Y_{b^\varsigma} T_{w_0}^{-1}$ 
for $b\in P$. Let us employ this formula to 
$T_0=Y_\vth T_{s_\vth}^{-1}$. 
One has: $\Diamond(T_0)=T_{s_\vth}^{-1} T_0T_{s_\vth}= 
T_{w_0}^{-1} T_0 T_{w_0}$. Then we conjugate by $T_{w_0}$ and, indeed,
obtain that $T_0^\dag=T_0$. 
We use here formulas (3.2.22) from \cite{C101}, including 
the commutativity of $T_{w_0}$ with 
$T_{s_\vth}$ and $T_0$.
Note that 
$T_{w_0}^{-1}T_i T_{w_0}=
T_{i^\varsigma}$ for $i>0$. 
}

Harmonic analysis of the regular representation
of AHA or its spherical part is, generally, 
the decomposition of the canonical coinvariant
$\langle T_{\hw}\rangle=\de_{\hw,id}$
in terms of the characters of irreducible 
representations. Any ``traces"  $\xi$ satisfying 
$\xi(AB)=\xi(BA)$ for $A,B\in \HH$ are
 $\dag$-invariant. Generally, DAHA $\dag$-coinvariants
are any functionals vanishing at $T_{\hw}X_b-
X_b T_{\hw^{-1}}$ for arbitrary $\hw\in \hW$ and $b\in P$. 

The anti-involution $\dag$ is the simplest.
The main example in this paper,
{\sf\em the coinvariant}, is by definition
$\vph$-invariant and acting via the projection of $\HH$
onto the polynomial representation.
These properties determine it uniquely up to proportionality
(for any $q,t$, not only generic). 
In DAHA theory, the spherical
AHA representation becomes the {\sf\em polynomial 
representation}, which we will introduce now.

%} %AHA ANTI-INVOLUTION

\subsection{\bf  Polynomial representation}\label{sec:polynom}
The theory of polynomial representations is based on the 
{\sf\em PBW Theorem}.
We need it for the following order of the generators:

\begin{theorem}[{\sf PBW for DAHA}]\label{PBWDAHA}
Every element in $\HH$ can be uniquely written in the form
\begin{equation}\label{pbwcdaha}
\sum_{a,w,b}C_{a,w,b}\,X_{a}T_{w}Y_b \text{ for } 
C_{a,w,b}\in \C_{q,t},\ a,b\in P,\, w\in W, \text{ where }
\end{equation}
 $\C_{q,t}=
\C[q^{\pm \frac{1}{2\mm}},t_\nu^{\pm 1/2}]$. In fact,
$\Z_{q,t}=\Z[q^{\pm \frac{1}{2\mm}},t_\nu^{\pm 1/2}]$ is
 sufficient.
 \sq
%\vskip -1.1cm\sq
\end{theorem} 
\vskip 0.5cm

This theorem readily results in the definition of the 
{\sf\em polynomial 
representation} of $\HH$ in  
$\mathscr{X}\equal\C_{q, t}[X_b]=\C_{q,t}[X_{\om_i}^{\pm 1}]$. 
%where, recall,  
%$\C_{q,t}$ is the field 
%of rational functions in terms of $q^{1/m},t^{1/2}$
%As a linear space, $\mathscr{X}$ is generated
%by $\{X_{b}\,|\,b\in P\}$. 
It is Ind$_\h^{\text{\footnotesize $\HH$}}\,\C_+$, 
where $\C_+$ is the $1$-dimensional 
module of $\h_Y$ such that $T_{\hw}\mapsto t^{l(\hw)/2}\equal
\prod_\nu t_\nu^{l_\nu(\hw)/2}$, and
$X_{b}$ act by multiplication. For instance, one has
for $b\in P_+$\,:\, 
$t^{l(b)/2}=\prod_\nu t_\nu^{(\rho^\vee_\nu,b)}=
\prod_\nu q^{k_\nu (\nu\frac{\rho_\nu}{\nu},b)}=
q^{(\rho_k,b)}.$ 

\vskip 0.2cm
Then
$T_i (i\ge 0)$ and $\pi_r (r\in O')$ act in $\mathscr{X}$
as follows: 
\begin{eqnarray}\label{pitpolyn}
&\pi_{r}\mapsto \pi_{r},\ \, 
T_{i}\mapsto t_i^{1/2}s_{i}+\dfrac{t_i^{1/2}-t_i^{-1/2}}
{X_{\al_{i}}-1}(s_{i}-1)\text{ for } t_i=t_{\al_i}. 
\end{eqnarray}
Recall that 
$s_{0}(X_{b})=X_{b}X_{\vth}^{-(b, \vth)}q^{(b, \vth)}$
for maximal short $\vth$. 
The images of $T_i$ for $i>0$ are 
{\sf\em Demazure-Lusztig operators}. 
\vskip 0.2cm

{\bf The coinvariant.} We begin with an arbitrary anti-involution
  $\varkappa$ of $\HH$ that fixes
 fractional powers of $q,t_\nu$. Let $\varrho$ be any 
$\C_{q,t}$-linear map  $\HH \to \C_{q,t}$,
such that  $\varrho(H^\varkappa)=\varrho(H)$.
Then the pairing
$\{A,B\}_\varrho\equal \varrho(
A^{\varkappa}\, B)
\,=\,\{B,A\}_\varrho$ is $\kappa$-invariant:\ 
$\{HA,B\}_\varrho=\{A,H^\varkappa B\}_\varrho.$ 

%We will stick here to $\mathscr{X}$; later
%$Y$-induced modules will be considered. 

\begin{definition}\label{def:coinv}
A {\sf\em Shapovalov-type
anti-involution} $\varkappa$ of $\HH$ with respect to $Y$
is such that $T_w^{\varkappa}\!=\!T_{w^{-1}}$ for $w\in W$, and 
the following property holds: for any $H\in \HH$,\,
the decomposition $H=\!\sum c_{awb}\! Y_a^{\varkappa} 
T_{w} Y_b$ exists and 
is unique. Then the corresponding $\varrho$ is unique
if its restriction to $\h_Y$, generated by
$Y_b$ and $T_w$, is given. Namely,
$\varrho(H) =
\,\sum c_{awb}\, \varrho(Y_a)
\varrho(T_{w}) \varrho(Y_b)$. 
See \cite{ChM}, Section 2.6. \sq
\end{definition}

\vskip 0.2cm

The {\sf\em coinvariant} $\{\cdot \}$ in this paper will be 
for $\vph$, which is of Shapovalov-type. It sends
$T_{i}\mapsto t_{i}^{1/2}$ for $i\ge 0$,
and $\pi_r\mapsto 1$. 

Explicitly, we represent
$
H=\sum_{a,w,b} c_{a,w,b}\, X_a T_{w} Y_b \for w\in W,
a,b\in P
$
due to  Theorem \ref{PBWDAHA}, and the
{\sf\em coinvariant\,} is:\,
\begin{align}\label{evfunct}
\{\,\}:\ X_a \ \mapsto\  q^{-(\rho_k,a)},\
Y_b \ \mapsto\  q^{(\rho_k,b)},\
T_i \ \mapsto\  t^{1/2}.
\end{align}
Equivalently, $\{H\}= H(1)(q^{-\rho_k})$, where $H\in \HH$
acts on $1$ in $\mathscr{X}$. Recall that
$q^{-\rho_k}=t^{-\rho}$ in the $ADE$-cases.

The defining symmetry of the coinvariant, which is
$\{\,\vph(H)\,\}\,=\,\{\,H\,\}$, 
readily follows from (\ref{evfunct}).
See \cite{ChD2} for its further properties.
We will use that
\begin{align}\label{eta-vph}
&\{\,\eta(H)\,\}=\{H\}^\star \text{ for }
q^\star=q^{-1}, t^\star=t^{-1}, \text{ and } 
\{\,\varsigma(H)\,\}=\{\,H\,\},
\end{align}
where $\eta$ is from (\ref{aut-eta}), and
$\varsigma$ is naturally extended to $\HH$:
\begin{align}\label{varsigmaXY}
\varsigma(X_b)\!=\!X_{\varsigma(b)}, \
\varsigma(Y_b)\!=\!Y_{\varsigma(b)}, \
T_i^\varsigma\!=\!T_{\varsigma(i)},\ 1\le i\le \nn.
\end{align}

The main application of $\mathscr{X}$ is the theory of
Macdonald polynomials. Originally, the symmetric
Macdonald polynomials $P_b$ for $b\in P_+$ were defined.
The key development was the theory of {\sf\em 
non-symmetric Macdonald polynomials} $E_b$ for $b\in P$
due to Heckman, Opdam, Macdonald, and the author; see 
\cite{Ch4,C101}. The Macdonald conjectures for $\{E_b\}$
became  entirely conceptual and with simple justifications.
The passage from $E_b$ to $P_b$  is via the $t$-symmetrization.
The key tool in the theory of  $E$-polynomials is the usage 
of {\sf\em intertwining operators}. For instance,
the {\sf\em HHL-formula}
(Haglund-Haiman-Loehr) was an impressive demonstration of
the power of the non-symmetric approach. 
\vskip 0.2cm

Generally, the polynomials $\{E_b, b\in B\}$
are unique (up to proportionality) eigenfunctions of
the operators $\{L_f\equal f(Y_1,\ldots, Y_n),
f\in \Q[X]\}$
acting in $\Q_{q,t}[X]:$
\begin{align}
&L_{f}(E_b)\ =\ f(q^{-b_\#})E_b\, \hbox{\ for\ }\,
b_\#\equal b- u_b^{-1}(\rho_k),
\label{Yone} 
\end{align}
where 
$X_a(q^{b})=
q^{(a,b)}$, and
$u_b$ is the
element $u\in W$ of {\sf\em minimal\,} length such that
$u(b)\in P_-$. Their coefficients are in $\Q(q,t_\nu)$. 

We need $E_b$ only for $b\in P_+$ in this paper.
Then $E_b=X_b$ modulo lower terms, and 
\begin{align}\label{nontosymdomx}
&\Pi_R P_{b}^\circ=\mathscr{P}_+(E_{b}^\circ)\for
\mathscr{P}_+\!\equal\!\sum_{w\in W}
t_{\sht}^{l_{\sht}(w)/2}
t_{\lng}^{l_{\lng}(w)/2}T_w,
\end{align}  
for the  
Poincar\`e polynomial $\Pi_R=\mathscr{P}_+(1)$,
$E_b^\circ=E_b/E_b(q^{-\rho_k})$, and
 $P_b^\circ=P_b/P_b(q^{-\rho_k})$. See, for instance, formula (3.4) 
in \cite{ChD2}. 
\vskip 0.2cm

{\bf EHA and $\nabla$-operator.}
Let us comment on the relation of the theory of
$\nabla$ to $E$-polynomials, 
especially to the formulas for 
the action of $PSL^\wedge_2(\Z)$ on $E_{\om_r}=X_r$ for
minuscule  $\om_r$. Generally, this is correspondence
between  the DAHA superpolynomials from \cite{CJ} and further papers 
and  the EHA-superpolynomials from \cite{KTs,GL2}. 
\medskip

Following the connection  of EHA, {\sf\em Elliptic Hall
Algebra}, to spherical $\HH$ due to \cite{SV} and other papers,
let $E^{(k)}_{m,n}\equal \hga_{n,m}(X_{\om_k})$ in
 $\HH$ for the root system  $gl$, where
$X_{\om_k}=\mathbb{X}_1\cdots \mathbb{X}_k$ for $\mathbb{X}_i=
X_{\mathbbm{e}_i}$ as above (the same for $Y_{\om_k}$).
These  $E^{(k)}_{m,n}$  are products 
of $X_{\om_k}$ and $Y_{\om_k}$ up to some powers of $q$.
Replacing  $X_{\om_k}$ by the corresponding 
{\sf\em monomial symmetric} functions, we arrive at
$P_{m,n}^{(k)}$ from the theory of EHA.
\medskip

Obviously, $\tau_-(E^{(k)}_{m,n}) =E^{(k)}_{m+n,n}$, which is
a counterpart of
the formula $\nabla P_{m,n} \nabla^{-1}=P_{m+n,n}$ for the
operator $\nabla$ in the theory of {\sf\em modified Macdonald
polynomials} $\tilde{H}_\mu$ for Young diagrams $\mu=(\mu_i)$.
Its main property (actually, the
definition) is that $\nabla( \tilde{H}_\mu)=q^{c}t^{c'} 
\tilde{H}_\mu$ for $c=\sum_i (i-1)\mu_i$ and $c'$ defined 
for the transpose of $\mu$ (using the standard notation).

This matches our formula for the action
of $\tau_-$ in $\mathscr{X}$:
$\dot{\tau}_-(E_b)=q^{-(b,b)/2}t^{-(b,\rho)}E_b$, where  $b\in P_+$.
 Our factor
$q^{-(b,b)/2}$ is close to $q^{c}t^{c'}$ for the Young diagram
$\mu$ corresponding to $b$. Also, \cite{CJ} is connected
with \cite{Neg}, more specifically, with his $D$-operators.

Thus,
 $\dot{\tau}_-$ is a counterpart
of $\nabla$, which observation can be of importance for
the {\sf\em Shuffle Conjecture}  and its variants. 
There is a relation to \cite{CaM, Mel2}, at least at the
level of formulas. 
Concerning this conjecture  for the rectangles
$\upsilon\rr\times \upsilon\ss$ (now a theorem),
 see \cite{BHM,GMO}. This is 
the case of algebraic cables 
$C\!ab(\upsilon\rr\ss+1,\upsilon)T(\rr,\ss)$, 
when  the conductors
$\c$ of the corresponding plane curve singularities
are monomial, $I_\la$\, for the Young diagrams $\la$ from
\cite{BHM,GMO}. See Theorem 
\ref{thm:cond} below.

The Shuffle Conjectures 
always come with explicit combinatorial formulas, which are
closely related to parking functions and generalized Catalan
numbers; see, for example, \cite{GMV1,GMV2}.
We note that {\sf\em superduality}, which is under 
 $q\leftarrow t^{-1}$ in DAHA parameters,  can be (sometimes)
seen directly from such formulas.
For us, the {\sf\em superduality}
is a general theorem (or a conjecture).

Let us note that, the {\sf\em superduality} under
$q\leftrightarrow t^{-1}$ of DAHA superpolynomials
requires the $q\leftrightarrow t$ symmetry
of  the polynomials $\tilde{H}_\mu$, i.e. some symmetric theory.
This is for general $\mu$;
we expect that the symmetric theory  is not really needed
for  minuscule $\om_r$ due to $E_{\om_r}=X_{\om_r}$.

\setcounter{equation}{0}
\section{\sc DAHA-Jones theory}
\subsection{\bf Iterated torus knots}\label{sec:ITER-KNOTS}
Torus knots $T(\rr,\ss)$ are defined for any nonzero
integers 
assuming that $gcd(\rr,\ss)=1$. In the standard torus $T^2$,
it is a path with $\rr$ horizontal turns (along the torus)
and $\ss$ vertical ones.  
There are obvious symmetries:
 $\,T(\rr,\ss)=T(\ss,\rr)=T(-\rr,-\ss)$, where 
we use ``$=$" for the {\sf\em ambient isotopy equivalence}.
Also, $\,T(\rr,\ss)=\unknot\,$\, if $\rr=1$ or $\ss= 1$, and
changing $\rr\mapsto -\rr$
results in the {\sf\em mirroring} of $T(\rr,\ss)$. 
\medskip

Generally, the DAHA
construction below is for any iterated torus links. For instance,
$\rr,\ss$ can be negative and the condition $gcd(\rr,\ss)=1$ 
can be omitted (which we will not do). The motivic theory is
only for {\sf\em algebraic links}, which requires strict 
positivity of  $\rr,\ss$, and the (strict) positivity of linking
numbers if multibranch plane curve singularities are
considered. See \cite{EN,ChD1,ChD2,ChW}.
\medskip

Following  \cite{ChD1}, 
the DAHA construction of invariants of iterated torus {\sf\em knots}
is based one their {\sf\em $\tax$-presentations}. Algebraically, such
a presentation is a sequence of  pairs of integers $[\rr_i,\ss_i]$
from

\centerline{
$
\vec{\rr}=\{\rr_1,\ldots \rr_\ell\}, \
\vec{\ss}=\{\ss_1,\ldots \ss_\ell\} \hbox{\, such that\,}
gcd(\rr_i,\ss_i)=1,
$
}

\noindent
where $\ell$ is the  {\sf\em length} of the iteration.
For algebraic knots, $[\rr_i,\ss_i]$ will 
be interpreted below as 
{\sf\em characteristic pairs\,} or {\sf\em Newton's pairs\,}
of the corresponding plane curve singularities. They  
are not topological invariants. The following sequences are.
%The necessary and sufficient
%condition for being algebraic is $\rr_i,\ss_i>0$,
%which will be imposed in this paper.
%\smallskip

\vskip 0.2cm
{\bf Cabling parameters.} 
The {\sf\em cabling\,} $C\!ab(\aa,\bb)(K)$ of 
any oriented
knot $K$ in (oriented) $S^3$ is defined as follows;
see, e.g., \cite{EN} and references therein.
We consider a small $2$\~dimensional torus
around $K$ and put there the torus knot $T(\bb,\aa)$
in the direction of $K$,
which is $C\!ab(\aa,\bb)(K)$. The parameter
$\aa$ is the number of
turns around $K$, $\bb$ is that along $K$.
For instance,\, $C\!ab(a,0)K=\unknot\ $ and $C\!ab(a,1)K=K$.

This procedure depends on
the order of numbers $\aa$ and $\bb$, the orientation of $K$,
and its {\sf\em framing}. We take the natural
zero-framing  when  the parallel copy of the knot has
zero linking number with a given knot.

Using this operation, iterated torus knots (also called
satellite knots)
are as follows.
Given two sequences
$\vec\aa=(\aa_1,\ldots,\aa_\ell),\vec\rr=(\rr_1,\ldots,\rr_{\ell})\}$, 
called  {\sf\em cable presentation}, 
we begin with  $T(\rr_1,\aa_1)=C\!ab(\aa_1,\rr_1)(\unknot\,)$
(note that we transpose $\aa$ and $\rr$ here), and then set:
\begin{align}\label{Knotsiter} 
\t(\vec\rr,\vec\ss)=C\!ab(\vec{\aa},\!\vec{\rr})(\unknot\,)=
\Bigl(C\!ab(\aa_\ell,\rr_\ell)\cdots \bigl(C\!ab(\aa_2,\rr_2)
(T(\rr_1,\!\aa_1))\bigr)\Bigr),
\end{align}
where the connection with the $\tax$-presentation is as follows:
\begin{align}\label{Newtonpair}
\aa_1=\ss_1,\,\aa_{i}=\aa_{i-1}\rr_{i-1}\rr_{i}+\ss_{i}\,\
(i=2,\ldots,\ell).
\end{align}
See,  e.g., \cite{EN}. The $\tax$-presentation will be the
one used later. However, to prove the topological invariance
of the DAHA construction (a theorem) the passage to the 
cable presentation is necessary. 

\Yboxdim7pt
Recall that knots are considered up to
{\sf\em ambient isotopy\,}. 
Generally, they are colored by dominant weights, which  
will be Young diagrams in what will follow; uncolored knots
are those for $\om_1=\yng(1)\,$. 
\medskip

{\bf Algebraic knots.}
They are defined for
(isolated) plane curve singularities 
presented by polynomial equations $F(x,y)=0$ in a neighborhood 
of $0=(x=0,y=0)$. The corresponding {\sf\em algebraic knots\,}
are their intersections
with a small $3$-dimensional sphere in
$\C^2$ around $0$.

Accordingly, $\rr_i$ and $\ss_i$, the
{\sf\em Newton's pairs},  must be all positive, which is
necessary and sufficient. The corresponding equation is 
\begin{align}\label{yxcurve}
y = c_1\,x^{\ss_1/\rr_1}
(1+c_2\,x^{\ss_2/(\rr_1\rr_2)}
\bigl(1+c_3\,
x^{\ss_3/(\rr_1\rr_2\rr_3)}
\Bigl(\ldots\Bigr)\bigr)) \hbox{\, at\, } 0,
\end{align}
where the parameters $c_i\in \C$ must be
sufficiently general.

\comment{
\begin{proposition}\label{PROP:stab-values}
Given two Young diagrams $\la$ and $\mu$,\,
the values $P_\la(q^{\mu+\rho_k})$ are $a$\~stable.
Namely,  
there exists a universal rational function in terms
of $\,q,t,a\,$ times a certain  
power of $a^{1/2}$ and power of $t^{1/2}$ 
such that 
 $P_b(q^{c+\rho_k})$ for $A_\nn$ is its value 
at  $a=-t^{\nn+1}$.
Here $n\ge 0,\, \la=\la(b),\mu=\la(c)$ are the Young
diagrams for  $b,c\in P_+$ for $A_\nn$,
$\nn+1$ is no smaller 
than the number of rows in $\la$ and in $\mu$,
 and $P_{m\om_{\nn+1}}=1$ in $A_{\nn}$ by definition. 
Also, the $\aa$-stabilization holds for 
$P_\la^\circ(q^{\mu+\rho_k})$. \sq
\end{proposition}
}

\subsection{\bf Definition and key features}
The {\sf\em DAHA-Jones polynomials} below
will be uniformly defined for any
affine root systems  $\tilde{R}$ defined above.
See also \cite{CJJ} for this construction in the case of $C^\vee C_1$.
We will mostly need $R=A_{n}$ or $R=gl_{\nn+1}$ in this paper.

The {\sf\em combinatorial data} will be $\{\vec{\rr},\vec{\ss}\}$,
and $b\in P_+$, which determines the color. We assume that 
$gcd(\rr,\ss)=1$ (they can be negative).

The notation $H\!\!\Downarrow\, \equal H(1)$ will be used,
 where the
action of $H\in \HH$ is in the polynomial representation.
Torus knots $T(\rr,\ss)$ are represented 
by the matrices $\ga[\rr,\ss]=\ga_{\rr,\ss}\in PSL_{\,2}(\Z)$
with the
first column $(\rr,\ss)^{tr}$ ($tr$ is the transposition). 
Let $\hat{\ga}_{\rr,\ss}\in PSL_{\,2}^{\wedge}(\Z)$ be
the pullback of $\ga_{\rr,\ss}$ to projective
$PSL_{\,2}(\Z)$. Using  $\rr/\ss=[a_1;a_2,\cdots,a_{2p}]$
(notice $2p$),  $\hga=
\tau_+^{a_1}\tau_-^{a_2}\cdots \tau_-^{a_{2p}}.$ This holds as such
with negative $a_1$ upon $\rr\mapsto -\rr$ if $\rr>\ss$.
See examples before Theorem \ref{thm:Htor}. We will use this
presentation, though any $\hga_{\rr,\ss}$ can be taken.
Accordingly,
$\hga_i$ will be $\hga[\rr_i,\ss_i]$. 

The following is formula (2.12) from \cite{ChD1}:
\begin{align}\label{jones-ditx}
& \j_{\vec\rr,\vec\ss}
(b;q,t)\! =\!
\Bigl\{\hat{\ga_{1}}\Bigr(
\cdots\Bigl(\hat{\ga}_{\ell-1}
\Bigl(\bigl(\hat{\ga}_\ell(E_b)/
E_b(t^{-\rho})\bigr)\!\Downarrow
\Bigr)\!\Downarrow\Bigr) \cdots\Bigr)\Bigr\}.
\end{align}
Here $b\in P_+$ and $E_b$ are
{\sf\em nonsymmetric Macdonald polynomials} $E_b$.
They can be replaced by their $t$-symmetrization,
the corresponding {\sf\em symmetric} Macdonald polynomials $P_b$,
which will not change  $\j$ due to the $T_w$-invariance of
the coinvariant $\{\cdots\}$. We need only $E_b$ in this paper.
\smallskip

{\bf Polynomiality etc.}
The following theorem and related statements are
mostly from \cite{CJ,CJJ,ChD1}.

\begin{theorem}\label{THM-integr-Jones}
The invariants  $\j\,$ 
are {\sf\em polynomials} in terms of $q,t$
up to a factor $q^\bullet t^\bullet$,
where the powers $\bullet$ can be rational.
Modulo such factors,
they do not depend
on the particular choice of the lifts \,$\ga\in PSL_2(\Z)$
to $\hat{\ga}\in PSL_{\,2}^{\wedge}(\Z)$, though we will
always use specific $\hga_{\rr,\ss}$ introduced above.

Up to the $q^\bullet t^\bullet$\~equivalence, the following
{\sf\em hat-normalization} $\hat{\j}\,$ 
is well-defined. It must be
 a $q,t$\~polynomial not divisible by
$\,q\,$ and by $\,t\,$ with integral
coefficients  such that 
their total $gcd$ is $1$, and with positive 
coefficient of the minimal pure
power of $t$. For algebraic knots colored
by rows $m\om_1$, this normalization becomes
$\hat{\j}(q\!=\!0,t\!=\!0)=1$. \sq
\end{theorem}

%\vskip -1.5cm

{\bf Topological symmetries.}
The polynomial $\hat{\j}$ defined in
Theorem \ref{THM-integr-Jones}\,
depends
only on the topological type of
the  corresponding iterated torus knot.
For instance, the pairs with  $\rr_i=1$ can be omitted,
 and the transposition
$[\rr_1,\ss_1]\mapsto $ $[\ss_1,\rr_1]$  (only
for $i=1$) does not influence
$\hat{\j}$. Also, let us mention that 
{\sf\em mirroring\,} of iterated torus knots 
results in  
%changing $\aa_i\mapsto -\aa_i$ in (\ref{Newtonpair}), and 
$q\mapsto q^{-1}, t\mapsto t^{-1}, \aa\mapsto \aa^{-1}$
in the DAHA-Jones polynomials
and superpolynomials (later), possibly up to $q^\bullet t^\bullet$.
\vskip 0.2cm

The justification of such symmetries 
for torus knots is in  Theorem 1.2 from
\cite{CJJ}. It suffices to check that 
DAHA-Jones polynomials coincide for $T(\rr,\ss)$,  
$T(\ss,\rr)$ and $T(-\rr,-\ss)$, and that they
are trivial for $T(\rr,1)$. We used in \cite{CJJ}
the anti-involutions and automorphisms above, and
the invariance of the coinvariant with respect to
the involution $\eta$.

For {\sf\em iterated} torus knots, it
suffices to check additionally
that the $\j$\~polynomials for
$C\!ab(m\rr+\ss,\rr)T(m,1)$ and $T(m\rr+\ss,\rr)$ 
are the same since $T(m,1)=\unknot$\,. The corresponding
DAHA fact is the commutativity of  $\tau_-^m$
with $\Downarrow$, which simply means that $\tau_-$ acts
in the polynomial representation. 
See \cite{ChD2} concerning the iterated torus links. 

We note that the elements
$\hat{\ga}_{1}\Bigr(
\cdots\Bigl(\hat{\ga}_{\ell-1}
\Bigl(\bigl(\hat{\ga}_\ell(E_b)/
E_b(t^{-\rho})\bigr)\!\Downarrow
\Bigr)\!\Downarrow\Bigr) \cdots\Bigr)$, refined
{\sf\em knot operators}, 
are invariants of the corresponding iterated torus 
knots considered in $T^2$, in
$T^2\times [0,\ep]$ for small $\ep>0$ to be exact. One level down,
these operators applied to $1$ (in the polynomial representation) 
are invariants of such knots considered 
in the solid torus. Finally, applying the coinvariant $\{\cdots\}$ 
results in $\j$-polynomials, which are 
invariants of the corresponding knots in $S^3$. 
\smallskip

{\bf Specialization $q\!=\!1$.} It is worth mentioning,
that 
\begin{align}\label{q-1-prod}
&\j_{\,\vec\rr,\,\vec\ss}
\,\bigl(b;\,q\!=\!1,t
\bigr)\!=\!
\hbox{\small$\prod$}_{p=1}^\nn \j_{\,\vec\rr,\,\vec\ss}\,
(\om_p;\,q\!=\!1,t)^{\be_p}
\end{align}
for $b=\hbox{\small$\sum$}_{p=1}^\nn \be_p \om_p\in P_+$.
The $\j$\~polynomials, those without the hat-normalization,
 are used
here; see formula (2.18) 
in \cite{ChD1}.

The same formula at $q\!=\!1$ holds for superpolynomials
$\h$ defined below. Namely, 
$\h_K(b;\,q\!=\!1,t,\aa)\!=\!
\hbox{\small$\prod$}_p \h_K
(q\!=\!1,t,\aa; \om_p)^{\be_p}$ for iterated torus knots $K$, where
$\be_p$ is the number of $p$-columns in the diagram
$dgrm(b)$ representing $b\in P_+$.

\subsection{\bf DAHA superpolynomials}\label{sec:sup-def}
Following \cite{CJ,GoN,CJJ,ChD1,ChD2}, the construction from
Theorem \ref{THM-integr-Jones} and other statements above can
be extended to the {\sf\em DAHA-superpolynomials\,},
the result of the stabilization  of
$\hat{\j}^{A_\nn}$, which are hat-normalized
$\j$-polynomials for $A_\nn$.

The $\aa$\~stabilization for torus knots in type $A$
was announced in \cite{CJ} with the reference to  Lemma 4.4 in
\cite{SV}; see also formula (4.1) there.  Its full proof was 
published in \cite{GoN} based on 
the same lemma (stabilization of $D_m$). The same argument 
applied to iterated torus knots provides
the $\aa$-stabilization
for $\mathfrak{H}^{daha}$-superpolynomials (below).
 
%Both approaches use \cite{SV}.
% later, the $a$\~stabilization
%was extended to arbitrary iterated knots and links.
The superduality conjecture was stated in \cite{CJ}
(let us also mention \cite{GS})
and proven in \cite{GoN} for torus knots; see also
\cite{CJJ} for an approach based on the
generalized level-rank duality. The justifications
of the $\aa$\~stabilization and
the superduality was then extended to
arbitrary iterated torus knots in \cite{ChD1},
and iterated torus links in \cite{ChD2}.
%\smallskip

\vskip 0.2cm
{\bf Definition.}
The sequences $\vec\rr,\,\vec\ss$ of length $\ell$
are as above. We will
use the DAHA-Jones polynomials $\hat{\j}$, 
those under the {\sf\em hat-normalization}.
As above, $\mu=dgrm(b)$ is the Young diagram
for $b=\hbox{\small $\sum$}_{p=1}^\nn \be_p \om_p
\in P_+$; namely,  $\be_p$ is the
number of $p$-columns in $\mu$.

\smallskip

We will consider $P_+\ni b=$ $\sum_{i=1}^\nn \be_i \om_i$
for $A_\nn $ as
(dominant) weights for {\sf\em any\,} $A_\mm$ 
with $\mm\ge \nn-1$, setting $\om_{\nn}=0$ for 
$A_{\nn-1}$ (when $\mm=\nn-1$).
%The integral form of $P_b$ in (\ref{jones-hat})
%and (\ref{jones-bar}) will be $J_{\la}$ from
%(\ref{P-arms-legs})
%for $\la=\la(b)$ in the next theorem.
See \cite{CJ,GoN,CJJ,ChD1}.

%\smallskip

\begin{theorem}\label{STABILIZ}\cite[Theorem 2.1]{ChD1}
Given $b\in P_+$ and $(\vec\rr,\,\vec\ss)$ as above,
there exists a
unique polynomial $\h_{(\vec\rr,\,\vec\ss)}(q,t,\aa;\mu)$
in $\Z[q,t^{\pm 1},\aa]$ for $\mu=dgrm(b)$ 
such that for any  $\mm\!\ge\! \nn\!-\!1$ ,
\begin{align}\label{jones-sup-hat}
&\h(q,t,\aa\!=\!-t^{\mm+1})\,=\,
\hat{\j}^{A_\mm}(q,t).
\end{align}
One sufficiently large
$\,\mm$\, fixes $\,\h\,$
uniquely. Here $\h(\aa\!=\!0)$
is automatically 
hat-normalized. In particular,
$\h(q\!=\!0,t\!=\!0,\aa\!=\!0)=1$ for $b=m \om_1$, when the
corresponding $\mu$ is the row with $m$ boxes.   
 \sq
\end{theorem}

The polynomials $\h$ depend only on the isotopy
class of the corresponding iterated torus links.
All symmetries for the $\j$\~polynomials
hold for $\h$, including the product
formula at $q\!=\!1$ from (\ref{q-1-prod}).
\vskip 0.2cm

We conjectured in \cite{ChD1} that
%\vskip -0.5cm
\begin{align}\label{deg-a-jj}
\hbox{deg}_\aa\h
=&
\min\{\,\rr_1,\ss_1\}\, \rr_2 \cdots
\rr_{\ell}\,|\mu| -|\mu|,
\end{align}
where $|\mu|$ is the number of boxes in $\mu$.
A somewhat weaker statement was  justified. This is
compatible with the specialization $q\!=\!1$, which gives that
the $a$\~degree is no smaller than that in (\ref{deg-a-jj}).

The right-hand side of this
formula is the multiplicity
of the corresponding singularity generalized to the colored
case. For instance, $deg_\aa=|\mu|(\ss-1)$ for torus knots
$T(\rr,\ss)$ colored by $\mu$, where $\ss<\rr$. 
\vskip 0.2cm

{\bf Superduality.} 
This is one of the key properties
of superpolynomials. It was conjectured in \cite{GS,CJ} and
proven in the DAHA setting in \cite{GoN,ChD1} that  
\begin{align}\label{iter-duality}
\h(\mu; q,t,\aa)=
q^{\bullet}t^{\bullet}
\h(\mu^{tr}; t^{-1},q^{-1},\aa),
\end{align}
where $\mu^{tr}$ is the transposition of $\mu$.
Its extension to iterated torus links is in \cite{ChD2}.
For instance, 
$\h(1^m; q,t,\aa)=
q^{m\de}t^{m^2\de}
\h((m);t^{-1},q^{-1},\aa)$ for the $m$-column $1^m$
and the $m$-row $(m)$. Here $\de$ is the {\sf\em genus} of the
corresponding knot; $\de=(\rr-1)(\ss-1)/2$ for $T(\rr,\ss)$,
where $\rr,\ss>0$. 

\vskip 0.2cm
{\bf The case of trefoil}.
Let us calculate $\h_{3,2}$ for uncolored trefoil.
We begin with the case of $A_1$. 

In this case,
$\HH$ is generated by $X^{\pm1},Y^{\pm1},T$ subject to  
group relations $TXTX=1=TY^{-1}TY^{-1}, Y^{-1}X^{-1}YXT^2=q^{-1/2}$
and the quadratic one 
$(T-t^{1/2})(T+t^{-1/2})=1$. The action of $\tau_{\pm}$ is:
\begin{align*}
&\tau_+\!:\!
Y\!\mapsto\! q^{-\frac{1}{4}}XY,\ X\!\mapsto\!X,\ T\!\mapsto\!T,\ 
\tau_-\!:\! X\!\mapsto\! q^{\frac{1}{4}}YX,\  
Y\!\mapsto\!Y,\  T\!\mapsto\!T.
\end{align*}
The polynomial representation is given by the formulas
$T\mapsto t^{1/2}s\!+\! 
\frac{t^{1/2}\!- t^{-1/2}}{X^2-1}(s\!-\!1)$, 
$X\mapsto X,\ \, Y\mapsto spT$, where $ s(X)=X^{-1},\,
\, p(X)=q^{1/2}X.$
\vskip 0.2cm

By $\,\sim$\,, 
we mean ``\,\,up to $q^{\bullet}t^{\bullet}$\,\,"
in the following calculation: 
\begin{align*}
&\j_{3,2}\!=\!
\{\tau_+\tau_-^2(X)\}_{ev}\!\sim\!\{(XY)(XY)X(1)\}_{ev}\!\sim\!
\{Y(X^2)\}_{ev}\\
&=t^{-\frac{1}{2}}q^{-1}X^2-
t^{\frac{1}{2}}+t^{-\frac{1}{2}}|_{X^2\mapsto t^{-1}}\sim
1+qt-qt^2,
\end{align*}
where we use
the {\sf\em nonsymmetric} Macdonald polynomial $E_1=X$.
One has:  $Y(X)\!=\!(qt)^{-\frac{1}{2}}X$.
Recall that 
$\{H\}\!\equal\! H(1)(X\!\mapsto\! t^{-\rho})$. Thus,
$\hat{\j}_{3,2}(q\mapsto t)= 1\!+\!t^2\!-\!t^3.$

\vskip 0.2cm

For $gl_{\nn+1}$, the corresponding
$\HH$ is generated by pairwise commutative 
$\mathbb{X}_i=X_{\mathbbm{e}_i}$,  pairwise commutative
$\mathbb{Y}_i=Y_{\mathbbm{e}_i}$, and $T_k$,  
where 
$1\le i\le \nn+1$ and
$1\le k \le \nn$. One has:   
 $\tau_+(\mathbb{Y}_1)\!=\!q^{\bullet} 
\mathbb{X}_1 \mathbb{Y}_1$,\  $\tau_-(\mathbb{X}_1)\!=\!q^{\bullet}
\mathbb{Y}_1 \mathbb{X}_1$, and so on. 
The action of $\mathbb{Y}_1$ in the polynomial
representation is by
the formula 
$\mathbb{Y}_1\!=\!\pi T_{\nn}\cdots T_1$, where 
$\pi: \mathbb{X}_1\!\mapsto\! \mathbb{X}_2, 
\mathbb{X}_2\!\mapsto\! \mathbb{X}_3,
\ldots, \mathbb{X}_\nn\!\mapsto\! q^{-1} \mathbb{X}_1$. Here $T_k$ are
those for $A_\nn$, and $\mathbb{X}_{\al}=\mathbb{X}_i 
\mathbb{X}_j^{-1}$ 
for
$\al=\mathbbm{e}_i-\mathbbm{e}_j$. For any fundamental weights,
$X_m=X_{\om_m}=\mathbb{X}_1\cdots\mathbb{X_m},\,
Y_m=Y_{\om_m}=\mathbb{Y}_1\cdots\mathbb{Y_m}$.
%The formulas for arbitrary $Y_i$ are longer, but
%almost equally simple \cite{C101}. 
%for any  up to $q^\bullet t^\bullet$. % but we need only $X_1, Y_1$.

Similar to $A_1$,
the polynomial $E_1=\mathbb{X}_1$ will be used.
We obtain that  $\hat{\j}^{A_\nn}_{3,2}
=1+qt-qt^{\nn+1}$ and 
$\h_{3,2}=1+qt+aq$. Indeed, the values of the latter 
at $a=-t^{\nn+1}$ are $\hat{\j}^{A_\nn}_{3,2}$.

Note that the relations
$\h_{3,2}(\aa\!\mapsto\! -t)=1$ and    
$\h_{3,2}(\aa\!\mapsto\! -t^{2})=1+qt-qt^2$ (only two)
are sufficient to fix $\h_{3,2}$ uniquely using
that $deg_a=1$. Generally, 
$deg_\aa \h_{\rr,\ss}(q,t,\aa; \mu)\!=\!
|\mu|\bigl(\min(\rr,\ss)\!-\!1\bigr)$. 
\vskip 0.2cm

The DAHA construction 
in the case of {\sf\em uncolored}  $T(2p+1,2)$ is practically
the same; we will need to calculate  $\{YX^{p+1}\}$ or
$\{\mathbb{Y}_1\mathbb{X}_1^{p+1}\}$. Indeed,
$\j_{2p+1,2}=\{\tau_+^p \tau_-^2(X)\}\sim \{X^p Y X^p Y X\}
\sim \{Y X^{p+1}\}$  for $A_1$. 

The action of the corresponding
$\hat{\ga}$ on any $X_m$ is given by 
the same formula as above upon adding $m$ (up to $\sim$). Namely,
adding the  color  $1^m$ is the switch $X_1\mapsto X_m$, and we
arrive at $\j_{2\pp+1,2}(1^m)\sim  \{Y_m X_m^{\pp+1}\}$.
 The only difference is that the
action of  $Y_m$ in $\mathscr{X}$ is given by somewhat  more involved 
formulas, but very explicit. 
%Recall that $X_m=X_{\om_m}=
%\mathbb{X}_1\cdots \mathbb{X}_m$ for $gl_{\nn+1}$, and the same
%formulas are for $Y_m$.
%$\{\mathbb{Y}_1(\mathbb{X}^{\pp+1}\}$. 
\medskip

This calculation will be the simplest particular 
case of the general definition of the {\sf\em new}
DAHA  superpolynomials
$\mathfrak{H}^{daha}_\la$ for arbitrary Young diagrams $\la$.
The torus knots $T(2\pp+1,2)$ will then correspond to  
$\la=(\pp)$, the $\pp$-row diagrams.  
We will use that $E_m=X_m$ and employ 
the action of $\tau_{\pm}$ as above.  The key will be that 
$\ga_{\rr,\ss}(X_m)$ is obtained from
a very explicit (uniform) word in terms of 
$X, Y$ by the substitution $X\mapsto X_m$ and $Y\mapsto Y_m$.
For positive $\rr,\ss$, it is given in terms of the continued
fraction of $\rr/\ss$.

\comment{
Also, $\h^{daha}$ for $T(2p+1,2)$ can be obtained using 
(\ref{jones-ditx}) for 
{\scriptsize
$\ga_1=\begin{pmatrix} 1& *\\p & *\end{pmatrix}$ and 
$\ga_2=\begin{pmatrix} 2& *\\1 & *)\end{pmatrix}$.}
This is related Theorem
 \ref{thm:sharp-flat},(i) below:\ the connection of
(\ref{jones-ditx}) and some $\mathfrak{H}^{daha}_\la$. 
}
\vskip 0.2cm

{\bf HOMFLY-PT polynomials.}
The definition of $H\!O\!M(t,a;\mu)$
is especially simple in the uncolored case, which is for 
$\mu=\yng(1)=\om_1$.
The following {\sf\em skein  relation} is sufficient to define
them (the reduced ones):
\smallskip

\centerline{\small
$
a^{1/2}
H\!O\!M(\nwarrow\kern-10pt\nearrow\kern-10.5pt
\nearrow\kern-11pt\nearrow
\kern-11.5pt\nearrow)
\!-\!a^{-1/2}
H\!O\!M(\nearrow\kern-10.2pt\nwarrow\kern-10.7pt
\nwarrow\kern-11.2pt\nwarrow
\kern-11.7pt\nwarrow)\!=\!
(t^{1/2}\!-\!t^{-1/2})
H\!O\!M(\uparrow\uparrow),\ H\!O\!M(\bigcirc)=1.
$
}

\medskip
The coincidence of the hat-normalization of the (reduced)
$H\!O\!M(t,a)$ with $\h(q=t,t,a=-\aa)$ holds
for any colored iterated torus links; 
the parameter $\sqrt{a}$ is mostly used in
topology.
% Spherical $P_\la^\circ$ can be used for knots;
%moreover, they can be replaced without any other
%changes by nonsymmetric $E_\al^\circ$ (see above). 
This coincidence is due to the author 
for torus knots, Morton-Samuelson (iterated torus knots), and
Cherednik-Danilenko (iterated torus links); see \cite{CJ, MoS,ChD2}.
The uncolored {\sf\em superduality} for  $H\!O\!M$
becomes $t^{\frac{1}{2}}\to -t^{-\frac{1}{2}}, 
a^{\frac{1}{2}}\to a^{-\frac{1}{2}}$;
it is obviously compatible with 
the skein relation above. 

We constantly calculate HOMFLY-PT polynomials for any DAHA 
superpolynomials, including 
$\mathfrak{H}^{daha}_\la$ defined below. 
They always matched those available online (for knots up to $16$ crossings).
Also, we do this (formally) for  HOMFLY-type reductions of 
$\mathscr{H}^{inst}_\mathscr{C}$ (below).
%Some of them (relatively small) seem beyond these tables.

\vskip 0.2cm

\subsection{\bf New approach}\label{sec:coxeter}
We will begin with the construction of 
{\sf\em Coxeter knots} following \cite{GL2}.
They are essentially the same as in \cite{ObR2};
let us also mention here \cite{GL1,ObR1}. However
we need the approach with Young diagrams from \cite{GL2}.
For instance, the conditions for $\la$ that provide the
positivity of $\mathfrak{H}^{daha}_\la$ can be stated only
in terms of $\la$.

\vskip 0.2cm

{\bf Coxeter knots.}
Let $\la=(\la_1\ge \la_2\ge \cdots \ge \la_\ell>0)$. For this
Young diagram, we form the sequence of $(i,j)$-coordinates 
of the corners ($m$ of them) 
of the complement $\overline{\la}=\Z^2_+\setminus \la$ 
ordered from the bottom to the top:\, 
$(i_1\!=\!\ell,j_1\!=\!\la_\ell),\ldots, (i_m, j_m\!=\!\la_1)$.
Generally, $j=\la_{i}$. We 
follow the boxes in the border of $\overline{\la}$. Here 
$(i,j)=(0,0)$ for the $1${\footnotesize st} box
in $\Z_+^2$.
\smallskip

Accordingly, we create an ordered  product of powers
$(L_{u_i})^{v_i}$ of formal symbols
$L_u$ for $i=1,\ldots,m$  placed left-to-right. Here:\, 
$(u_1\!=\!j_1\!=\!\la_{\ell},v_1\!=\!i_1\!-\!i_2), 
(u_2\!=\!j_2\!+\!1, v_2\!=\!i_2\!-\!i_3),\ldots,
(u_m\!=\!j_m\!+\!1, v_m\!=\!i_m\!+\!1).$ Generally, 
$(u_k=j_k+1,v_k=i_k-i_{k+1})$, where $i_{m+1}\equal 0$. 

For instance,
this product will be just $L_n$ for $\la=(n)$ (a row with $n$ boxes), 
$L_1^n$ for the $n$-column $1^n$,
$L_3^2$ for $\yng(3,3)$\,, 
$L_2 L_3$ for $\la=\yng(3,2)$\,, $L_1L_2^2$ for $\yng(2,2,1)$\,,
$L_1^2L_3$ for $\yng(3,1,1)$\,, and $L_1L_3L_4$ for $\yng(4,3,1)$.
\vskip 0.2cm

Next, we define the braid $\b_\la=
\tilde{L}_{u_1}^{v_i}\cdots \tilde{L}_{u_m}^{v_m}T_1 T_2\cdots T_{u_m}$
by replacing $L_k$  with
$\tilde{L}_k=T_kT_{k-1}\cdots T_2 T_1^2T_2\cdots 
T_{k-1} T_k$.
Finally, $\k_\la$ will be its canonical closure.
For instance, $\b_\la=T_1^2(T_3T_2T_1^2T_2T_3)T_1T_2T_3$ 
and $\k_\la=$ K12n725 for $\la=(3,1,1)$, 
and $\b_\la=T_1^2(T_2T_1^2T_2)^2 T_1T_2$ and $\k_\la=$K12n242
for $\la=(2,2,1)$. Also,  $(4,3,1)$, the last diagram provided above,
gives $\k=C\!ab(13,2)T(3,2)$.
\vskip 0.2cm

Paper \cite{GL2} interprets this procedure as a canonical passage
from {\sf\em elliptic 
monotone paths}, those in $T^2$ times a small interval, 
into $S^3$. The Young diagrams $\la$ are their natural
combinatorial
invariants (not topological).
See Proposition 7.5 there:\, the passage from 
monotone links to Coxeter links. Also, see Remark 7.10
concerning \cite{Neg} and his $D$-operators, and 
Section 7.2 there (based on  \cite{GL1})
on {\sf\em postroid links}.
\medskip

We need $W_\la$ for our 
$\mathfrak{H}_\la^{daha}$, explicit 
products of $X$
and $Y$ in terms $\la$.
The Coxeter knots $\k_\la$  are  used a lot in this work, but
only for finding the corresponding HOMFLY-PT polynomials.

Our diagrammatic  procedure is similar to that
from \cite{GL2}, but there are deviations. For instance,
our $X,Y$-interpretation of the ``moves"  $R$ and $U$
along the border of $\overline{\la}$ 
is different from their interpretation. We adjust it to
our DAHA-based (nonsymmetric) construction. 
\medskip

We note that the braids  $\tilde{L}_k$ above were important in the 
theory of representations of affine Hecke
algebras of type $A$ and their quotients
in Theorem 3.4 from \cite{ChT}. They generalized
and simplified the Jucys-Murphy (additive) elements. More
specifically, they provided a canonical lift from HA to AHA
of the irreducible representations of non-affine Hecke
algebras of type $A$ associated with Young diagrams.

The main result of \cite{ChT} is an explicit description of
semisimple irreducible representations of AHA of type $A$, 
which appeared those associated with {\sf\em skew} Young diagrams. 
The DAHA counterpart of this theory is in Section 3.7 of \cite{C101}.
This may be of interest in the context of the present paper, 
and, possibly,  related to the
postroid links from \cite{GL1}. There are some reasons 
to relax the construction of $W$ below by allowing (certain)
negative powers of $X,Y$; see Conjecture  \ref{conj:extW}.

\medskip
Different diagrams can result
in isotopic knots.  For
instance, $(3,2)$, $(2,2,1)$, $(4,1)$ and 
$(2,1,1,1)$
give the same knot up to isotopy. For instance,
let us take the last Young diagram. 
The corresponding
braid is $T_1^6T_2T_1^2T_2T_1T_2=
T_1^6T_2T_1^3T_2T_1$, which has the same closure as
$T_1^7T_2T_1^3T_2$, which is exactly the classical braid
presentation for K12n242. This knot will be one of the
main examples below. We will present a conceptual proof that 
transposed diagrams give the same $\mathfrak{H}^{daha}$-polynomials,
but their coincidence for $(3,2)$ and $(4,1)$ is by inspection.
\vskip 0.2cm

{\bf Elliptic paths.}
Let us comment on Theorem 1.3 from \cite{AlS}
and on the connection of  {\sf\em Coxeter knots} $\k_\la$
with elliptic paths following \cite{GL2}. 

This Theorem is a version of the equivalence of the two
definitions of DAHA. Topologically, $\HH$ is
the quotient of the group algebra of the
orbifold fundamental group $\pi_1^{orb}$
of the elliptic configuration space 
 by the quadratic relations for the half-turns. The 
definition 
in terms of $T_0$ is directly connected with
the equivariant
$K$-theory of affine flag varieties, closely related 
to Theorem 1.3. 

Concerning the relation to the {\sf\em $q,t$-Skein} of the elliptic
curve $E$ for $A_1$, the DAHA quadratic relation
becomes the Kauffman bracket relation, and taking
the $\S_2$-invariants is due to the 
definition of the elliptic configuration
space. The {\sf\em spherical} DAHA occurs here because the 
knots are elliptic braids up to conjugation. See \cite{AlS}.

Actually, the conjugation symmetry is stronger
in the construction of \cite{CJ}. The coinvariant
$\bigl\{ \be \bigr\}$  satisfies the
relations $\bigl\{X^a T^c \be Y^b T^c\bigr\}\sim\bigl\{\be\bigr\}$ 
up to $q,t$-factors for any
$a,b,c\in \Z$. This reflects the flexibility 
with fixing the initial points of the braids associated with
 $\be$.
\medskip

The relation to {\sf\em Coxeter knots} is essentially as follows.
We begin with a path $\be \in \pi_1^{orb} \Bigl(
\bigl(E\setminus \{0\}\bigr)/\S_2\Bigr)$, 
identifying $E$ with the horizontal
torus $T^2$.  Consider $\be$ as the path in $T^2\times [0,\ep]$
for small $\ep$, which we cut by  $S_{vert}^1\times [0,\ep]$ 
through $0$. We obtain  a braid 
in the annulus with $m$ strands,  assuming that each makes
exactly one horizontal turn (non-reentrant).
Then repeat this procedure (under the same assumption)
using the horizontal
inner circle $S^1_{hor}\subset T^2$ through $0$; let $n$ be the number of
the corresponding vertical turns of $\be$. 

Thus, $\be$ becomes a path from some
 sufficiently general point {\large$\star$}  near $0$ in a 
small neighborhood
of the rectangle 
$m\times n$, which goes from {\large $\star$} to 
$\pm${\large$\star$} shifted by $(m,n)$ and does
not contain integer points. The Young diagram $\la$ is a collection
of boxes in this rectangle beyond it, and $W_\la$
is the presentation of $\be$ in terms of $X$ and $Y$ for 
$+${\large$\star$}.

The conditions for $\be$ 
correspond to {\sf\em monotone elliptic paths}
in  \cite{GL2}. They calculate explicitly
the $m$-braids as above  associated with $\be$,  
which are
in terms of the standard $T_1,T_2,
\ldots, T_{m-1}$.  They are  $\b_\la$, where 
the Coxeter elements 
$T_1\cdots T_{m-1}$ make their (canonical) closures  $\k_\la$ 
knots, not links. We obtain the correspondence
$W_\la\rightsquigarrow \k_\la$.

{\bf From diagrams to DAHA superpolynomials.}
%We will need the following variant of
%encoding Young diagram in terms of formal $L_k$.
Given a 
Young diagrams $\la$,
we present it as a sequence of $R$-moves (one step right)
and $U$-moves (one step up), moving from the bottom to the top
along the border  of the complement
$\overline{\la}$ of $\la$. In our standard notation, we 
begin at the box 
$(i=\ell(\la), j=0)$ and end at box  
$(i=0, j=\la_1)$.
% the first move is always $R$ and the last is $U$,  
\smallskip

We write these $R,U$ 
right-to-left, which product will be denoted by $\w=\w_\la$.
The corresponding  $X,Y$-word $W=W_\la$ is obtained
by replacing 
$R$ by $X$ and $U$ by $Y$ in $\w_\la$. Practically, we omit $\w$
and produce $W$ by reading the diagram from top to bottom,
replacing the corresponding $D,L$ by $X,Y$, and writing them 
left-to-right.
%Finally, 
%$W_{\upsilon,\la}\equal (W_\la)^\upsilon X^\upsilon$ for $\upsilon\ge 1$. 

\medskip
For instance,

$\w=UR^2U^2R$ and $W=YX^2Y^2X$  for $\la=(3,1,1)\,=\,\yng(3,1,1)$, 

$\w=U^2RUR$ \ and  $W=Y^2XYX$ \ for $\la=(2,2,1)\ =\,\yng(2,2,1)$,

$\w=URUR^2UR$,\ \,  $W=YXYX^2YX$ \  for $(4,3,1)\ =\,\yng(4,3,1)$.

\begin{definition}\label{def:frakH}
(i) Consider a root systems $R$ as in Section \ref{sec:daha}.
Let $W_{\la}[m]\equal W_{\la}(X\mapsto X_m)$, where
$X_m=X_{\om_m}$ for any
roots systems $R$ and {\sf\em minuscule} $\om_m$. Then
the DAHA-Jones polynomial associated with a Young diagram $\la$ is
$\mathfrak{J}_{\la}^{daha}[m]\equal
\bigl\{ W_{\la}[m] X_m \bigr\}$ for the DAHA 
coinvariant $\bigl\{\ldots\bigr\}$. Accordingly, 
$\hat{\mathfrak{J}}_{\la}$ is 
$\mathfrak{J}_{\la}$ 
 under the hat-normalization. 

(ii) We now take $\mathfrak{J}_{\la}^{daha}[m]$ from
 $(i)$ for the coinvariant and $X_m$ 
for the root systems $gl_N$ or $A_{N-1}$ (then all $\om_m$ are
minuscule).  Recall that 
$X_m=\mathbb{X}_1\cdots \mathbb{X}_m$
for $\mathbb{X}_i=\mathbb{X}_{\mathbbm{e}_i}$, and so are $Y_m$.
Then, the following relations uniquely determine the
 polynomial 
$\mathfrak{H}_{\la}=\mathfrak{H}_{\la}^{daha}[m]$ in terms of 
$q,t^{\pm 1}$ and $\aa$:\ \ 
$
\mathfrak{H}_{\la}(q,t,\aa=-t^N)=
\hat{\mathfrak{J}}_{\la}
\text{ for } N\ge m.$ \sq

\comment{
(iii) Generalizing $(i,ii)$, let $\la$ be any Young diagram,
$\om_m$ is minuscule, and $\upsilon \ge 1$. We set
$\mathfrak{J}_{\upsilon, \la}[m]\equal
\Bigl\{\, \bigl(W_{\la}[m] X_m\bigr)^\upsilon\, \Bigr\}$ 
In type $A$, 
there exists a unique polynomial 
$\mathfrak{H}_{\upsilon,\la}$ in terms of 
$q,t^{\pm 1},\aa$, $\upsilon$-iterated $\mathfrak{H}_\la$,
such that\ \
$
\mathfrak{H}_{\upsilon,\la}(q,t,\aa\!= -t^N)\!=\!
\hat{\mathfrak{J}}_{\upsilon,\la}(q,t)
\text{ for } N\ge m.$
}
\end{definition}

This is similar to  Proposition 7.8 in \cite{GL2},
but there are deviations. First, our construction
is entirely in terms of DAHA. There is no 
need to consider any symmetrization and plethystic formulas.
Second, we replace $R,U$ directly
by $X,Y$ in our $W$,
which is different from the substitution in \cite{GL2},
and then we multiply $W_\la$ by $X$. Third, we
do $X_m$ for any minuscule $\om_m$. Also, this is just
a one-line formula: 
$\mathfrak{H}_\la[m]=\{ W_\la[m]X_m\}$ (followed by the 
stabilization).
Any properties of $\mathfrak{H}_\la$ are 
some DAHA facts for us. For instance, see 
Theorem \ref{thm:trans}, where we prove that $\la$ and its
transpose give the same $\mathfrak{H}$. 

\vskip 0.2cm

{\bf From $\mathbf{\h}^{daha}$ to $\mathfrak{H}^{daha}$.}
Let us discuss now the connections between $\mathfrak{H}$ and
the original (old) DAHA superpolynomials $\h^{daha}_\k$ from 
\cite{CJ} and
further papers. These two
constructions intersect as follows.

The theorem below is  direct from
the definition of  $\hga_{\rr,\ss}(X_m)$. Recall that
$\ga_{\rr,\ss}=$ {\footnotesize 
$\begin{pmatrix} \rr & *\\ \ss & *\end{pmatrix}$}. 
We assume that $gcd(\rr,\ss)=1$ and $\rr,\ss>0$.

\vskip 0.2cm
 
Then $\hga_{\rr,\ss}=\tau_+^{a_1}\tau_-^{a_1}\cdots \tau_-^{a_{2p}}$
for $\{a_i\}$ from the {\sf\em continued fraction}
$\rr/\ss=a_1+\frac{1}{a_2+\frac{1}{\cdots}}$, which must be of 
 even length. Recall that there are $2$ choices for continued fraction
of rational numbers. Note that the same formula can be used for 
$r<0<s$ and $|r|>s$;
then $a_1<0$ and  $a_i>0$ for $i>1$.
Note that  $a_1>0$ and $\hga$
begins with a power of $\tau_+$ for positive $\rr,\ss$
if and only if $\rr>\ss$. Thus, the corresponding
product of $\tau_{\pm}$  depends only on $\rr/\ss$
and is fully determined by its continued fraction of even length. 

For instance, $\frac{2}{3}=0+\frac{1}{1+\frac{1}{1+\frac{1}{1}}}$
and $\hga_{2,3}=\tau_-\tau_+\tau_-$. However, 
$\frac{3}{2}=1+\frac{1}{2}$
and $\hga_{3,2}=\tau_+\tau_-^2$.
Both continued fractions are of even length. These cases are 
$T(2,3)$ and $T(3,2)$. Accordingly, the products for $\hga$ will
be $\tau_+^{-1}\tau_-^3$ for $T(-2,3)$ or $T(2,-3)$,
and $\tau_+^{-2}\tau_-^2$ for $T(-3,2)$ or $T(3,-2)$, 
\smallskip

It is instructional to calculate the coinvariant
$\{\hga_{-1,3}(X)\}$ for $A_1$ (or any $A_m$). It must be
$\sim 1$ since $T(-1,3)=\unknot\,$, but it is not immediate
algebraically. The continued fraction for $-\frac{1}{3}$ is
$[-1;1,1,1]$ in the classical form (recall that it must be 
of even length). Thus,
$\hga=\tau_+^{-1}\tau_-\tau_+\tau_-$, 
$\{ \hga(X) \}\sim \tau_+^{-1}(Y X Y^2 X)\sim
\{X^{-1}Y X X^{-1}Y X^{-1}Y X\} \sim \{ Y^2 X^{-1} YX\}\sim
\{Y^2 X^{-1}X\}\sim \{Y^2\}\sim 1$, indeed.
\medskip

The Young diagram $\la=\la_{\rr,\ss}$ for $T(\rr,\ss)$ is defined as 
a collection of all boxes in the rectangle $\ss\times \rr$ 
above the diagonal from the lowest most-left corner (maximum $i$)
to the top  most-right one (maximum $j$). 
 We denote the corresponding $\mathfrak{H}$
by $\mathfrak{H}_{\rr,\ss}[m]$.  By $[m]$, we
mean the substitution  $X\mapsto X_m$ as above. Generally,
it is  $\mathfrak{H}_\la[m]$, 
where  $\c=I_\la$ is the conductor $\mathfrak{c}$
of $\r$ lifted to $\c\subset \F_q[[x,y]]$.
In the case of $T(\rr,\ss)$,
$\r=\F[[x=z^{\rr},y=z^{\ss}]]$ and $\c=I_\la$ 
for the diagram $\la_{\rr,\ss}$. See below.

\begin{theorem}\label{thm:Htor}
For the superpolynomial  $\mathfrak{H}_{\rr,\ss}[m]$ defined above
and  the (original) DAHA superpolynomial 
$\h^{daha}_{\rr,\ss}$ colored by $\om_m=1^m$:\,
$$\mathfrak{H}^{daha}_{\rr,\ss}[m](q,t,\aa)=
\h^{daha}_{\rr,\ss}(q,t,\aa; \om_m).
$$
For any root system $R$ from Section \ref{se:roots}
and minuscule $\om_m$:\,

\centerline{$\hat{\mathfrak{J}}^{daha}_{\rr,\ss}[m](q,t,\aa)=
\hat{\j}^{daha}_{\rr,\ss}(q,t,\aa; \om_m)$.}
\end{theorem}
{\it Proof.} Both claims are straightforward:\, continued
fractions result exactly in $W_\la$ and 
$\bigl\{\hga_{\rr,\ss}(X)\bigr\}\sim
\bigl\{WX\bigr\}$, where $\sim$ means  up to $q^\bullet t^\bullet$.
For the sake
of concreteness, we check it for  $T(7,4)$, a sufficiently 
large example 
that contains all steps of the verification. One has
$\frac{7}{4}=1+\frac{1}{1+\frac{1}{2+\frac{1}{1}}}$,
and $\hga_{7,4}=\tau_+\tau_-\tau_+^2\tau_-$.
Thus, $\hga_{7,4}(X)\sim \tau_+\tau_-(X^2YX)\sim
\tau_+(YXYXY^2X)\sim (XY)X(XY)X(XY)(XY)X$. Upon taking
the coinvariant, it can be replaced by $YX^2YX^2Y X^2$ due to
$Y(X)\sim X$ and  $\{X^k H\}\sim \{H\}$ for any $k$.
The corresponding $\la$ is $\yng(5,3,1)$. It generates
$\w=UR^2UR^2UR$, which results in  $W=YX^2YX^2YX$, and, finally,
$WX$ coincides with the product obtained above. 
Finally, $X\mapsto  X_m$ for 
minuscule $\om_m$. So $\lambda$ encodes taking the corresponding
coinvariant, which is for any $\rr,\ss$. 
 Everything here and in similar calculations with coinvariants
below is under the {\sf\em hat-normalization}.
  \sq
\vskip 0.2cm

According to this theorem and the discussion above,
the superpolynomials from \cite{GL2} {\sf\em for torus knots}
and the construction from \cite{Neg}
are equivalent 
to the original DAHA-superpolynomials for torus knots from \cite{CJ} 
in the uncolored case. We omit details.
%\vskip 0.2cm

\subsection{\bf Further properties}
Let us now follow the proof of the relation
$\j_{\rr,\ss}=\j_{\ss,\rr}$ 
from Theorem 1.2,(i) in \cite{CJJ} to verify that  $\la$
and its transpose $\la'$ result in the same $\j$-polynomials
up to the hat-normalization. Then this automatically holds for
$\mathfrak{H}_\la^{daha}$. 

\begin{theorem}\label{thm:trans}.
Let $\la'$ be the transpose of $\la$. Then the corresponding
$W'$ is the reverse of $W$ for $\la$ where we substitute
$X\leftrightarrow Y$.  Let $W_m=W(X\mapsto X_m, Y\mapsto Y_m)$ for
any root system $R$ and minuscule $\om_m$, and
$W'_m=W'(X\mapsto X_m, Y\mapsto Y_m)$. Then
$\bigl\{W_m X_m\bigr\}\sim \bigl\{W'_m X_m\bigr\}$ up to
$q^{\bullet}t^{\bullet}$ for any root system $R$, and
$\mathfrak{H}_\la[m]=\mathfrak{H}_{\la'}[m]$ in type $A$.
\end{theorem}
{\it Proof.}  The formula for $W'$ is immediate from the
definition. For instance,
consider $W_m=Y_mX_m^2Y_mX_m^2Y_mX_m$ above for $T(7,4)$,
 which is for $\la=(5,3,1)=\yng(5,3,1)$. Its transpose is
$\la'=(3,2,2,1,1)$ and $W_m'=
Y_mX_mY_m^2X_mY_m^2X_m$. We will need the following
generalization.
\medskip

Let $a=(a_1,\ldots, a_{k})$, $b=(b_1,\ldots, b_{k})$ be two
sequences of any integers,
positive or negative, and $W[a,b]=Y^{a_1}X^{b_1}\cdots Y^{a_k}X^{b_k}$.
Applying $\eta$ from (\ref{aut-eta}) and using 
the symmetry (\ref{eta-vph}):\,
 $\bigl\{W_m[a,b]\bigr\}^\star=
\bigl\{\eta (W_m[a,b])\bigr\}\sim\bigl\{(X_m^{-1}Y^{a_1}_m X_m)
(X_m^{-b_1})\cdots(X_m^{-1}Y^{a_k}_m X_m)
(X_m^{-b_k})\bigr\}=$

\noindent
$\bigl\{X_m^{-1}\bigl( Y_m^{a_1}X_m^{-b_1}\cdots 
Y^{a_k}_m X_m^{1-b_k}\bigr)\bigr\}\sim$
$\bigl\{Y_m^{a_1}X_m^{-b_1}\cdots 
Y^{a_k}_m X_m^{1-b_k}\bigr\}$, where
 $X_m^{-b_i}$ is replaced by $X_m^{-1}X_m^{-b_i}X_m$
and $\{X_b H\}\sim \{H\}$ is used.
Summarizing, 
$\bigl\{W_m[a,b]\bigr\}^\star\sim\bigl\{ W_m[a,1_k-b]\bigr\}$
for $1_k=(0,\ldots,0,1)$ of length $k$.

Now, let us apply this formula to $[-b^{inv},-a^{inv}]$,
where $inv$ means that the order is reversed. We use 
that $\bigl\{W_m[a,b]\bigr\}=
\bigl\{\vph(W_m[a,b])\bigr\}= \bigl\{W_m[-b^{inv},-a^{inv}]\bigr\}$
and arrive at 
$\bigl\{W_m[a,b]\bigr\}^\star\sim\bigl\{ W_m[-b^{inv},1_k+a^{inv}]\bigr\}$.

Finally, $\bigl\{ W_m[a,-b]X_m\bigr\}\sim 
\bigl\{ W_m[-b^{inv},a^{inv}]X_m\bigr\}$. Replacing $-b\mapsto b$,
we obtain a general identity $\bigl\{\, W_m[a,b]\,X_m\, \bigr\}\sim 
\bigl\{\, W_m[b^{inv},a^{inv}]\,X_m\, \bigr\}$, 
which concludes the proof.\sq
\vskip 0.2cm

This proof can be  an indication that $\mathfrak{H}$
can be extended to some $W$ with positive and negative 
powers of $X,Y$; cf. \cite{ObR2}. A simple example is 
torus knots $T(-\rr,\ss)$; we take  $W$ for $T(\ss,\rr)$ 
and then replace there every $Y$ by $X^{-1}Y$.
However, our experiments
produced by now only $1$ series of non-trivial {\sf\em superdual}  
DAHA superpolynomials $\mathfrak{H}$ 
with negative powers of  $X,Y$ in the corresponding $W$; 
see Conjecture \ref{conj:extW}.

%Diagrammatically,
%after every step up ($U$), there must be one step to the left 
%($R^{-1}$), i.e., $U\cdots U$ become the ``staircases".
%Generally, it makes sense to consider arbitrary $W[a,b]$; the
%$\aa$-stabilization is granted. We have already considered 
%$\{ \pi_m \hga_{\rr,\ss}E^\circ_b \}$ for fixed $m$ in type
%$A$ (unpublished). 

The superduality is a difficult matter for reduced
Khovanov-Rozansky polynomials, but it is believed to hold.
The positivity of uncolored ones is expected too,
the non-negativity of the $q,t,\aa$-coefficients.

Our positive
$\mathfrak{H}_{\la}[m]$ presumably  coincide with the reduced 
KhR polynomials 
colored by $1^m$ for the {\sf\em Coxeter knots}
$\k_\la$ associated with $\la$.  This  is 
up to $\sim$ and upon some 
transformation of $q,t,\aa$. Also, the relation of 
$\mathfrak{J}_{\la}[m]$ to
the corresponding QG-invariants is expected when $t=q$ 
($t_{\sht}=q, t_{\lng}=q^2$ for $B,C,D$). 
We note that Khovanov-Rozansky polynomials are
difficult to calculate for knots like K12n242 (though 
the Khovanov polynomial is known for the latter). We
mostly compare our ones against the corresponding
HOMFLY-PT polynomials.

\vskip 0.2cm

{\bf Rectangles \mathversion{bold}
{$\up\ss\times \up\rr \text{\, with \,} p=\pm1$}.} 
\vskip 0.2cm
There are two natural Young diagrams associated with
such rectangles (for $\up>1$). Consider the boxes 
above the diagonal. Then  $\la^{\boldsymbol{\natural}}$ 
for $p=+1$ is 
if we allow the
boxes that touch the diagonal, and
 $\la^{\boldsymbol{\flat}}$\, for $p=-1$  otherwise.
They coincide when $\up=1$.

\begin{theorem}\label{thm:sharp-flat}
The polynomials $\mathfrak{H}_{\upsilon,
\la^{\boldsymbol{\natural}}}[m]$ and 
$\mathfrak{H}_{\upsilon,\la^{\boldsymbol{\flat}}}[m]$ for $X=X_{\om_m}$
coincide with the DAHA superpolynomials
$\h_{\upsilon,\pm}(q,t,\aa;\om_m)$ defined correspondingly for
the cables $C\!ab(\upsilon \rr\ss\pm 1,\upsilon)T(\rr,\ss)$.
They are for any color $\om_m=1^m$,
$\ga_1=\ga_{\rr,\ss}$, and
$\ga_2=\ga_{\pm\upsilon,1}=${\scriptsize
$\begin{pmatrix} \pm\upsilon & *\\ 1 & *\end{pmatrix}$}.
\end{theorem}
{\it Proof.}
Note that we can allow $\upsilon=1$ 
because  $C\!ab(a,1)K=K$ for arbitrary $a\in \Z$ and
any knot $K$; then we arrive at
Theorem \ref{thm:Htor}.

The DAHA-Jones polynomials corresponding to these cables are
$\j_{\upsilon,\rr,\ss,\pm}(\om_m)=
\Bigl\{\hga_{\rr,\ss}
\Bigl(\hga_{\pm \upsilon,1}\bigl(X_m\bigr)(1)\Bigr)
\Bigr\}$ for any minuscule weights $\om_m$,  and
$\h_{\upsilon,\pm}(\om_m)$ are 
their $\aa$-stabilization of their hat-normalizations in type $A$.
As above,  equalities
up to $q^\bullet t^\bullet$ are denoted by $\sim$.

Given a torus knot $T(\rr,\ss)$ with $\rr>\ss>0$, let 
$\hat{W}_{\rr,\ss}\equal \hat{\ga}_{\rr,\ss}(X)$; it equals 
$X^{a_0}W_{\rr,\ss}YX$ for $a_0=Floor[\rr/\ss]$, where $W_{\rr,\ss}$
is constructed for the Young diagram $\la$ associated with $T(\rr,\ss)$.
Let $\tilde{W}_{\rr,\ss}\equal
\hat{W}_{\rr,\ss}X^{-1}Y^{-1}XY$. Then 
$\hat{W}^{[m]}_{\rr,\ss}$ and $\tilde{W}^{[m]}_{\rr,\ss}$ will be
upon  the substitution $X\mapsto X_{m}$ for a minuscule $\om_m$
(for any root system $R$).

\begin{lemma}\label{lem:W7-3}

(i) \(\h_{\upsilon,+}(\om_m)\sim
\bigl\{(\hat{W}^{[m]}_{\rr,\ss})^\upsilon\bigr\}, \ 
 \h_{\upsilon,-}(\om_m)\sim 
\bigl\{(\tilde{W}^{[m]}_{\rr,\ss})^\upsilon\bigr\}\),

\qquad \qquad \ \ \ (ii) \   
\(\mathfrak{H}_{\upsilon,\la^{\boldsymbol{\natural}}}[m]\sim
\bigl\{(\hat{W}^{[m]}_{\rr,\ss})^\upsilon\bigr\}, \ 
\ \ \mathfrak{H}_{\upsilon,\la^{\boldsymbol{\flat}}}[m]\sim
\bigl\{(\tilde{W}^{[m]}_{\rr,\ss})^\upsilon\bigr\}.\)
\end{lemma}
{\it Proof of $(i)$.} For $p=+1$:\ \,  $\hat{\ga}_2=\hat{\ga}_{\up,1}=
\tau_+^{\up-1}\tau_-$, $\hat{\ga}_2(X)\sim X^{\up-1}Y X$,
and $\bigl(\hat{\ga}_2(X)\bigr)(1)\sim X^{\up}$ due to
$Y(X)\sim X$. Upon the $\aa$-stabilization:\, 
$\h_{\upsilon,+}(\om_m)\sim$
% $\j_{\upsilon,\rr,\ss,+}\sim
$\bigl\{\hga_{\rr,\ss}(X^\up)\bigr\}=
\bigl\{\hat{W}_{\rr,\ss}^\up\bigl\}$, which
is under $X,Y\mapsto X_m, Y_m$.

When $p=-1$,\, let $F_m=X_m^{-1}Y_m^{-1} X_m Y_m$. It is
invariant with respect to the action of
$PSL_2^\wedge(\Z)$. Indeed, 
$\tau_+(F_m)\sim X_m^{-1}Y_m^{-1} X^{-1}_m X_m X_m Y_m=
F_m$ and, similarly, for $\tau_-$.
 Explicitly, this and similar 
group commutators can be
expressed  in terms of $T_1,\cdots,T_{n-1}$, so they are
invariant. See formulas (1.4.64)-(1.4.69) from \cite{C101}.
 Also, note that
$\bigl\{H X^{-1}Y^{-1}XY\bigr\}\sim \bigl\{H\bigr\}$
for any $H\in \HH$. 

Thus,
$\tilde{W}_{\rr,\ss}=\hga_{\rr,\ss}(X_m) F_m=
\hga_{\rr,\ss}(X_m F_m)=\hga_{\rr,\ss}(Y_m^{-1}X_m Y_m)$.
Here  $Y_m^{-1}X_m Y_m= 
T_{u_m}^{-1}\pi_m X_m \pi_m^{-1} T_{u_m}\sim 
T_{u_m}^{-1} X_{m^\iota}^{-1} T_{u_m}$, which is 
from Section \ref{sect:Aut}:\ for minuscule $\om_m$, $\iota=-w_0$ 
and the reduced decomposition $\om_m=\pi_m u_m$. 
We obtain that \(\tilde{W}^{[m]}_{\rr,\ss}\sim
T_{u_m}^{-1}\bigl(\tau_{\rr,\ss}(X_{m^\iota}^{-1})\bigr)T_{u_m}$,
and $\bigl\{(\tilde{W}^{[m]}_{\rr,\ss})^\up\bigr\}\sim
\bigl\{ T_{u_m}^{-1} 
\bigl(\tau_{\rr,\ss}(X_{m^\iota}^{-1})\bigr)T_{u_m}\bigr\}\sim
\bigl\{\tau_{\rr,\ss}(X_{m^\iota}^{-1})\bigr\}\sim
\bigl\{\tau_{\rr,\ss}(X_{m}^{-1})\bigr\}\), where we use
that $\iota$ fixes  $\tau_{\pm}$ and  
the coinvariant $\bigl\{\cdot\bigr\}$ due to $\iota(\rho_k)=\rho_k$.
 Finally,
$\hga_2=\tau_+^{-\up-1}\tau_-$,
$\hga_2(X)=X^{-\up-1} Y X$ and $\hga_1(1)=
X^{-\up-1} \bigl((Y (X)\bigr)\sim X^{-\up}$.
Recall, that the cable for $\hga_1=\hga_{a,b}$ and 
$\hga_2=\hga_{c,d}$ is $C\!ab(abc+d,c)T(a,b)$, and
$C\!ab(-a,-b)=C\!ab(a,b)$, 
which gives $C\!ab(\up\rr\ss-1,\up)T(\rr,\ss)$.

\vskip 0.1cm
{\it Proof of $(ii)$.}
The following example seems quite sufficient for 
$\la^{\boldsymbol{\natural}}$; all steps are shown.
Consider $T(7,3)$ corresponding to $\la=(4,2)=\yng(4,2)$, and 
take $\upsilon=2$. One has:\, $\w_{7,3}=U R^2UR^2$ and
$W_{7,3}=YX^2YX^2$,  
$\la^{\boldsymbol{\natural}}=(11,9,7,4,2)$, where 
$|\la^{\boldsymbol{\natural}}|=33=
\frac{(\upsilon \rr\ss)(\upsilon-1)}{2}+
\upsilon\frac{(\rr-1)(\ss-1)}{2}$. 
Then, $\hga_{7,3}=\tau_+^2\tau_-^3$ and 
$\hat{W}_{7,3}\sim \tau_+^2(Y^3X))\sim (X^2Y)^3X=X^2YX^2YX^2YX$.
Applying the coinvariant:

{\small \(
\{X^2YX^2YX^2YX\,\cdot\, X^2YX^2YX^2YX\}\sim
\{YX^2YX^2YX^3YX^2YX^3\}\)}, 

\noindent
which coincides with
$\bigl\{W_{\la^{\boldsymbol{\natural}}}X\bigr\}$. 
Indeed, $W_{\la^{\boldsymbol{\natural}}}$ is the result
of the substitution $R\mapsto X, U\mapsto Y$
in $\w_{\la^{\boldsymbol{\natural}}}=\w U R^3 
\w=UR^2UR^2\,UR^3\,UR^2UR^2$
for $\w$ for $T(7,3)$, which 
gives $W_{\la^{\boldsymbol{\natural}}}X=YX^2YX^2YX^3YX^2YX^3$, the word
inside the last $\{\cdots\}$ above. This proves the case of $p=1$.

Concerning the passage from $\hat{W}$ to $\tilde{W}$,
this example seems insufficient; we will do this formally.
First:\, $W_{\la^{\boldsymbol{\flat}}}=\bigl(W(XYX^{a_0})\bigr)^{\upsilon-1}W$,
where $W=X^{-a_0}\hat{W}X^{-1} Y^{-1}$ for $W=W_{\rr,\ss},
\hat{W}=\hat{W}_{\rr,\ss}=X^{a_0} W YX$ (up to $\sim$\,).
In the example above:\  
$W=(YX^2)^2,\, \hat{W}=X^2YX^2YX^2YX$, where
 $\la^{\boldsymbol{\flat}}=(11,9,6,4,2)$ for $T(7,3)$.
Second,  we obtain:\, $W_{\la^{\boldsymbol{\flat}}}=$

{\small  $ W(XYX^{a_0})W(XYX^{a_0})\cdots W(XYX^{a_0})W=
X^{-a_0}\hat{W}X^{-1} Y^{-1}(XYX^{a_0})$

$X^{-a_0}\hat{W}X^{-1} Y^{-1}
(XYX^{a_0})\cdots
 X^{-a_0}\hat{W}X^{-1} Y^{-1}(XYX^{a_0})X^{-a_0}
\hat{W}X^{-1} Y^{-1}.$
}
Finally, $\bigl\{ W_{\la^{\boldsymbol{\flat}}} X\bigr\}\sim 
\bigl\{X^{a_0}W_{\la^{\boldsymbol{\flat}}} XY\bigr\}\sim
 \bigl\{\bigl(\hat{W}X^{-1}Y^{-1}XY\bigr)^\upsilon\bigr\}.$ \sq

\medskip

Note that we can present   
$\bigl\{\bigl(\hat{W}_{\rr,\ss}\bigr)^\upsilon\bigr\}$ as
$\bigl\{W(YX^{a_0+1}W)^{\upsilon-1}X\bigr\}$ up to $\sim$, 
which means
that we insert $YX^{a_0+1}$ between $W$ in $W^\up$ 
inside $\bigl\{ W^\upsilon X\}$.
%They do any diagrams and this was new to me
%(and this is new), 
This presentation presumably gives a topological approach to
the justification of this theorem:\, using directly the definition 
of cabling. Similarly, one has:\, 
$\bigl\{\bigl(\tilde{W}_{\rr,\ss}\bigr)^\upsilon\bigr\}\sim
\bigl\{W(XYX^{a_0}W)^{\upsilon-1}X\bigr\}$, i.e. now we insert
$XYX^{a_0}$ for the minus-sign.

\vskip 0.2cm

{\bf Rank-two formula.} Let us provide  the DAHA superpolynomial
of rank $2$ for K12n242, which is for $\la=(3,2)$ and $\cn=2$;\,
 $\mathfrak{H}_\la[2]=$
{\small \(
1+q t+q^2 t+q t^2+2 q^2 t^2+q^3 t^2+q^2 t^3+3 q^3 t^3+q^4 t^3+q^2 t^4
+2 q^3 t^4+4 q^4 t^4+q^3 t^5+3 q^4 t^5+3 q^5 t^5+q^3 t^6+2 q^4 t^6
+4 q^5 t^6+q^6 t^6+q^4 t^7+3 q^5 t^7+2 q^6 t^7+q^4 t^8+2 q^5 t^8
+4 q^6 t^8+q^5 t^9+3 q^6 t^9+2 q^7 t^9+q^5 t^{10}+2 q^6 t^{10}
+3 q^7 t^{10}+q^6 t^{11}+3 q^7 t^{11}+q^6 t^{12}+2 q^7 t^{12}
+2 q^8 t^{12}+q^7 t^{13}+2 q^8 t^{13}+q^7 t^{14}+2 q^8 t^{14}
+q^8 t^{15}+q^9 t^{15}+q^8 t^{16}+q^9 t^{16}+q^9 t^{17}+q^9 t^{18}
+q^{10} t^{20}+\aa^4 \bigl(q^7+q^8+\frac{q^6}{t^2}+\frac{q^7}{t}
+q^8 t+q^8 t^2+q^9 t^3+q^9 t^4+q^{10} t^6\bigr)
+\aa^3 \bigl(2 q^5+4 q^6+q^7+\frac{q^4}{t^2}+\frac{q^5}{t^2}
+\frac{q^4}{t}+\frac{2 q^5}{t}+\frac{q^6}{t}+q^5 t+4 q^6 t+3 q^7 t
+2 q^6 t^2+5 q^7 t^2+q^8 t^2+q^6 t^3+4 q^7 t^3+3 q^8 t^3+2 q^7 t^4
+5 q^8 t^4+q^7 t^5+4 q^8 t^5+q^9 t^5+2 q^8 t^6+3 q^9 t^6+q^8 t^7
+3 q^9 t^7+2 q^9 t^8+q^{10} t^8+q^9 t^9+q^{10} t^9+q^{10} t^{10}
+q^{10} t^{11}\bigr)+\aa^2 \bigl(2 q^3+4 q^4+3 q^5+\frac{q^3}{t^2}
+\frac{q^2}{t}+\frac{2 q^3}{t}+\frac{2 q^4}{t}+q^3 t+5 q^4 t+7 q^5 t
+2 q^6 t+2 q^4 t^2+8 q^5 t^2+6 q^6 t^2+q^7 t^2+q^4 t^3+5 q^5 t^3
+11 q^6 t^3+3 q^7 t^3+2 q^5 t^4+8 q^6 t^4+8 q^7 t^4+q^5 t^5+5 q^6 t^5
+10 q^7 t^5+2 q^8 t^5+2 q^6 t^6+8 q^7 t^6+5 q^8 t^6+q^6 t^7+5 q^7 t^7
+9 q^8 t^7+2 q^7 t^8+7 q^8 t^8+2 q^9 t^8+q^7 t^9+5 q^8 t^9+4 q^9 t^9
+2 q^8 t^{10}+5 q^9 t^{10}+q^8 t^{11}+4 q^9 t^{11}+q^{10} t^{11}
+2 q^9 t^{12}+q^{10} t^{12}+q^9 t^{13}+2 q^{10} t^{13}+q^{10} t^{14}
+q^{10} t^{15}\bigr)+\aa \bigl(q+2 q^2+2 q^3+\frac{q}{t}
+\frac{q^2}{t}+2 q^2 t+5 q^3 t+3 q^4 t+q^2 t^2+4 q^3 t^2+7 q^4 t^2
+2 q^5 t^2+2 q^3 t^3+7 q^4 t^3+8 q^5 t^3+q^6 t^3+q^3 t^4+4 q^4 t^4
+10 q^5 t^4+4 q^6 t^4+2 q^4 t^5+7 q^5 t^5+9 q^6 t^5+q^7 t^5+q^4 t^6
+4 q^5 t^6+10 q^6 t^6+4 q^7 t^6+2 q^5 t^7+7 q^6 t^7+8 q^7 t^7+q^5 t^8
+4 q^6 t^8+9 q^7 t^8+q^8 t^8+2 q^6 t^9+7 q^7 t^9+5 q^8 t^9
+q^6 t^{10}+4 q^7 t^{10}+7 q^8 t^{10}+2 q^7 t^{11}+6 q^8 t^{11}
+q^9 t^{11}+q^7 t^{12}+4 q^8 t^{12}+3 q^9 t^{12}+2 q^8 t^{13}
+4 q^9 t^{13}+q^8 t^{14}+3 q^9 t^{14}+2 q^9 t^{15}+q^{10} t^{15}
+q^9 t^{16}+q^{10} t^{16}+q^{10} t^{17}+q^{10} t^{18}\bigr).
\)}

We can  define then {\sf\em formally\,}
the superpolynomial for  $2\om_1$ as 
its image under the 
{\sf\em superduality}, which is 
$\sim \mathfrak{H}_\la(1/t,1/q,\aa)[2]$. 
The one for $2\om_1$  is conjectured to match
 $\mathscr{H}^{inst}_\la[2]$.

The superpolynomials we construct satisfy all usual properties
of superpolynomials. For instance,
$\mathfrak{H}_\la[2](t=1)=
\bigl(\mathfrak{H}_\la[1](t=1)\bigr)^2=$ {\small 
\((1 + q + a q + 2 q^2 + 2 a q^2 + 2 q^3 + 3 a q^3 + a^2 q^3 + 2 q^4 + 
  3 a q^4 + a^2 q^4 + q^5 + 2 a q^5 + a^2 q^5)^2\)}.
\vskip 0.2cm

{\bf Nonpositive superpolynomials.}
The simplest hyperbolic Coxeter knot is for $\la=(2,2)$ 
and $W=Y^2X^2$; it is $10_{139}=$K10n27.
Its uncolored  $\mathfrak{H}^{daha}_{(2,2)}$ is superdual:\,
{\small 
\(1 + q t + q^2 t - q^3 t + q^2 t^2 + q^3 t^2 + q^3 t^3 + q^4 t^4 + 
 \aa^2 \bigl(q^3 - q^4 + q^4 t\bigr) + 
 \aa \bigl(q + q^2 - q^3 + q^2 t + 2 q^3 t - q^4 t + q^3 t^2 + q^4 t^2 + 
    q^4 t^3\bigr).
\)
}
%Recall that the construction is two-step:\, 
Recall that we take the coinvariant $\bigl\{Y^2 X^3\bigr\}$ for $X_1,Y_1$
and perform the $\aa$-stabilization. Notice negative terms here.
The DAHA superpolynomials were conjectured to be positive
for any algebraic knots colored by 
rectangles in \cite{CJ} and further papers. Nonpositivity
can be an issue when linking our superpolynomials to
reduced KhR superpolynomials.

We mention that the corresponding $\h^{inst}_{(2,2)}$ 
is positive but not
superdual. It is:\ {\small \(
1+q t+q^2 t^2+q^3 t^2+q^3 t^3+q^4 t^4+\aa^2 q^4 t
+\aa \bigl(q+q^2 t+q^3 t+q^3 t^2+q^4 t^2+q^4 t^3\bigr)
\)}. This is the simplest counterexample to the  unrestricted  
Conjecture \ref{conj:main}. Uncolored $\sH^{inst}$ are always
positive in all (quite a few) examples we calculated, including
non-monomial conductors.
\vskip 0.2cm

There are other sources of nonpositive superdual
superpolynomials. The following one is,
in a sense, a continuation of the case of 
$\la^{\boldsymbol{\flat}}$ from
Lemma \ref{lem:W7-3}.
We will use  $Y^{-1}X^{-1}YX$,  the reciprocal of $X^{-1}Y^{-1}XY$
considered there. Generally, a challenge is to allow
any negative powers of $X,Y$ in our words $W$, but the superduality
almost always fails.

\begin{conjecture} \label{conj:extW}
For any diagram $\la$ and the corresponding
$X,Y$-word $W_\la$, let $\overline{W}_{\la}=W Y^{-1}X^{-1}YX$,
and $\overline{W}_{\la}^{[m]}=\overline{W}_\la(X\mapsto 
X_m, Y\mapsto  Y_m)$.
Then the hat-normalization of 
coinvariants $\bigl\{\overline{W}_{\la}^{[m]}\bigr\}$ for
$A_{N-1}$ stabilizes to 
$\overline{\mathfrak{H}}^{daha}_\la(q,t,\aa; m)$ upon
$\aa=-t^{N}$. It is superdual for $m=1$, and its $deg_{\aa}$ 
coincides with that for $\la$. It can be linked to the reduced KhR
polynomial 
of a certain knot $\overline{K}_\la$ (presumably unique such).
 At least, we conjecture that 
$\overline{\mathfrak{H}}^{daha}_\la(q,t=q,\aa; m)$  
coincides with the 
(properly normalized) HOMFLY-PT polynomial of $\overline{K}_\la$
colored by $\omega_m$. \sq
\end{conjecture}
Finding $\overline{K}$ is, generally,
based on the identification of the formal HOMFLY-PT polynomials
$\overline{\mathfrak{H}}^{daha}(q=t,t,a=-\aa; m=1)$ 
with the ones in {\sf\em \, jhomflytable.txt\,} or so.
More directly, we 
expect that $\overline{K}_\la$ is the knot associated with 
$\be Y^{-1}X^{-1}YX$ (adding a loop at
{\large $\star$} around $0$) 
if $K_\la$ is associated
with an elliptic path  $\be$ as in Section \ref{sec:coxeter}.

\medskip
 Let us provide 
some examples (we have only uncolored ones).
\smallskip

(0) For $\la =\emptyset$ 
($K=\unknot$\,):\,  $\overline{\mathfrak{H}}^{daha}=1$. 
For instance, \ $\bigl\{Y^{-1}X^{-1}YX^2\bigr\}\sim
\bigl\{T^{-1}\pi X^{-1} YX^2\bigr\}
\sim\bigl\{\pi YX^2\bigr\}\sim\bigl\{X^2\bigr\}\sim 1$ for $A_1$.

\smallskip

{

\smallskip

(1) For $T(3,2)$:\,  $\la=\yng(1)\,$, $W=YX$,
 and $\overline{W}=YXG$ for $G=Y^{-1}X^{-1}YX$. Then:
$\overline{\mathfrak{H}}^{daha}_\la=$
{\small \(1-a^2 q-t+q t+a (-1+q-q t).\)}
\smallskip

(2) For $\la=\yng(1,1)$\,:\,  $\overline{W}=Y^2XG$,\,
$\overline{\mathfrak{H}}^{daha}_\la=$
{\small
\(1-t+q t-q t^2+q^2 t^2+\aa^2 (-q-q^2 t)+\aa (-1+q-2 q t+q^2 t-q^2 t^2).\)
}
\smallskip

(3) For $\la=\yng(2)$\,:\,  $\overline{W}=YX^2G$,\,
$\overline{\mathfrak{H}}^{daha}_\la=$
{\small
\(1-t+2 q t-q t^2+q^2 t^2+\aa^2 (-q+q^2-q^2 t)+
\aa (-1+2 q-2 q t+2 q^2 t-q^2 t^2).\)}
\smallskip

(4) For $\la=\yng(2,1)$\,:\,  $\overline{W}=YX^2YXG$,\,
$\overline{\mathfrak{H}}^{daha}_\la=$
{\small
\(1-t+q t+q^2 t-q t^2+q^2 t^2+q^3 t^2-q^2 t^3+q^3 t^3-q^3 t^4+q^4 t^4
+\aa^3 (-q^3-q^4 t)+\aa^2 (-q-q^2+q^3-q^2 t-2 q^3 t+q^4 t-q^3 t^2
-q^4 t^2-q^4 t^3)
+\aa (-1+q+q^2-2 q t+2 q^3 t-2 q^2 t^2+q^4 t^2-2 q^3 t^3+q^4 t^3-q^4 t^4).
\)
}
}

\smallskip
(5) For $\la=\yng(3,2)$\,:\,  $\overline{W}=YXYX^2G$,\, 
$\overline{\mathfrak{H}}^{daha}_\la=$
{\small
\(
1-t+q t+2 q^2 t-q t^2+2 q^3 t^2-q^2 t^3+2 q^4 t^3-q^3 t^4+q^4 t^4
-q^4 t^5+q^5 t^5+\aa^3 (-q^3+q^5-q^4 t-q^5 t^2)+\aa^2 (-q-q^2+2 q^3
+q^4-q^2 t-3 q^3 t+2 q^4 t+q^5 t-q^3 t^2-3 q^4 t^2+2 q^5 t^2
-q^4 t^3-q^5 t^3-q^5 t^4)+\aa (-1+q+2 q^2-2 q t-q^2 t+4 q^3 t
+q^4 t-2 q^2 t^2-2 q^3 t^2
+4 q^4 t^2-2 q^3 t^3-q^4 t^3+2 q^5 t^3-2 q^4 t^4+q^5 t^4-q^5 t^5).
\)
}
\vskip 0.2cm
\smallskip 
Notice that the knots $K$ 
in (2) an (3) are isotopic to $T(5,2)$. The corresponding
$\overline{K}$ are (obviously) not isotopic. 
%We do not have any diagrammatic constructions
%in terms of $\la$. 
In these examples, they and the corresponding
braids are:
\smallskip

\centerline{
$(1,2,3):\, 5.2,\, 7.3,\, 7.6\,$ (Rolfsen Table),\ $(4):\, 12n502,\, 
(5):\,14n21924.
$}
\vskip 0.2cm

\ \ $(1):\, T_1^{-2} (T_2 T_1^{-1})^2 T_2^2, 
\ \ (2):\, T_1^5T_2^{-1}T_1T_2^{-1}$
\ (both are rational),

\ \ $(3):\, T_1^{-4} T_2^{-1} T_1 T_2^{-2},\ \ \ \,
(4):\, T_1 T_2T_3 T_1^4T_3 T_2T_3^{-1}T_3T_1T_2$.

\ \ $(5):\, 
T_{1}^{-1}T_{2}T_{3}T_{4}^{-1}T_{3}T_{2}T_{1}T_{2}^{-1}
T_{3}T_{4}^{-1}T_{3}T_{2}^{-1}T_{3}T_{2}T_{4}^{-1}
T_{2}^{2}T_{3}^{-1}.$

{\bf On Volume Conjecture.}
Since $\cn$ is arbitrary in our construction, we can
now address the {\sf\em Volume Conjecture} for hyperbolic
knots $K_\la$, which states that the standard hyperbolic volume
$Vol(S^3\setminus K)$ is
$\lim_{N\to \infty} 2\pi \log |J_N(K; q=e^{2\pi \imath/N})|/N$,
for the Jones polynomial $J_N$ corresponding to our $(N-1)\om_1$.
This conjecture is checked only for very few hyperbolic knots.
At least,
we can offer an alternative approach, not based on
the usage of $R$-matrices and similar tools. See \cite{Mu}.

Employing the superduality and our coincidence 
conjectures,   the Jones polynomial
$J_N(K_\la;q)$ equals 
$\mathfrak{H}_\la(q^{-1},q^{-1},\aa=-q^2; N-1)$
up to $q^\bullet$. Since $q^N=1$,
$|J_N(K_\la;q)|=
|\mathfrak{H}_\la(q^{-1},q^{-1},\aa=-q^2; N-1)|=$
$|\mathfrak{H}_\la(q,q,\aa=-q^{2N-2}; N-1)|=
\bigl|\bigl\{W^{[m]}_\la X_m \bigr\}_{A_{2m-1}}\bigr|$,
where $m=N-1$. We did the conjugation $q\mapsto q^{-1}$
inside $\mathfrak{H}$.

Note that the coinvariant here is for the
root system $A_{2m-1}$:\, we need to add $2N$ to $-2$ instead of adding $N$.
This is because $\aa=-q^{m-1}$ is below the $1${\footnotesize st}
$\aa=-t^{m}$ we consider for finding the superpolynomials
for $\om_{m}$, which is $\mathfrak{H}_\la(q,t,\aa=-t^{m};m)=1$.
So we need to go to the root systems 
$A_{2m-1}$ (instead of $A_{m-2}$) 
to use directly the coinvariant.

Thus, we can restate the Volume Conjecture for knots
$K_\la$ as follows:
{\small
\begin{align}\label{voldaha}
Vol(S^3\setminus K_\la)=
\lim_{m\to \infty} 2\pi \frac
{\log |\{W^{[m]}_\la X_m \}_{A_{2m-1}}|}{m+1}.
\end{align}
}

We performed numerical tests for $\la=(3,2)$ (knot K12n242),
and $\la=(2,2)$ (knot $10_{139}=$K10n27) and $N$ up to $6$.
For such small $N$, the output was far from the conjectured one.
We deal with sums of very many powers of $q$, where even
the existence of the limit (upon taking $log$ and
the division by $N$) is questionable.
The corresponding coinvariants were $C_1(N)=\{Y^2XYX^2\}$ and 
$C_2(N)=\{Y^2X^3\}$. The first is actually for $\la=(2,2,1)$, but it is
the transpose of $(3,2)$. 

\comment{
%K12n242: N=5
C_1(5)=-11+10 e^{-\frac{1}{5} (2 i \pi )}-17 e^{\frac{2 i \pi }{5}}
+14 e^{-\frac{1}{5} (4 i \pi )}+5 e^{\frac{4 i \pi }{5}}
C_1(4)=-8 - 17 i
}

For instance, $C_1(5)=-11+10 e^{-\frac{1}{5} (2 i \pi )}-17 
e^{\frac{2 i \pi }{5}}
+14 e^{-\frac{1}{5} (4 i \pi )}+5 e^{\frac{4 i \pi }{5}}$,
$C_1(6)=-72-16 e^{-\frac{1}{3} (i \pi )}-19 e^{\frac{i \pi }{3}}
+16 e^{-\frac{1}{3} (2 i \pi )}+20 e^{\frac{2 i \pi }{3}}$,
and the approximation to the volume is $4.89829$ instead of
the actual $2.82812$. Also, $C_2(6)=
-136-128 e^{-\frac{1}{3} (i \pi )}+67 e^{\frac{i \pi }{3}}-73 
e^{-\frac{1}{3} (2 i \pi )}+145 e^{\frac{2 i \pi }{3}}$, and
the corresponding approximation is 
$6.30269$ instead of $4.85117$.

The best techniques here for relatively simple $\la$ 
seems the usage of nonsymmetric Pieri-type formulas from 
\cite{Ch4}. Namely, (5.16) and Proposition 5.5.
The notations there are $\sharp b=\pi_b, \ 
\ep_{\pi_b}=E_b/E_b(t^{-\rho})$,
and $\ep_{\hw}=0$ if $\hw\in \hW$ is not $\pi_b$ for some $b\in P$.

\section{\sc Motivic superpolynomials}\label{sec:motsup}
This section is devoted to {\sf\em motivic superpolynomials} 
$\h^{mot}_\r(q,t,\aa; \cn)$ of plane curve
singularities. We mostly follow 
\cite{ChP1,ChP2,ChW,ChQ} and prepare the reduction
formulas for the passage to 
$\mathscr{H}^{inst}_\mathscr{C}$ for the
{\sf\em instanton slices} $\mathfrak{S}_\mathscr{C}$
in the next section.
Here and further  $\mathbbm{n}$
is the rank of the corresponding modules.

The main conjecture
in \cite{ChQ} was that $\h^{daha}_\l(q,t,\aa)[\si]$ for 
any {\sf\em algebraic} link $\l$ with $\kappa$ components 
 coincides with  
$\h^{mot}_{\r}(q,t,\aa)[\si]$ for 
the ring $\r$ of the corresponding plane curve singularity.

The sequence  $\si=(c_1,c_2,\cdots, c_\kappa)$ 
gives the colors $c_i\om_1$ of the 
connected components of the corresponding link
for $\h^{daha}$. On the motivic
side, $c_i$ are the ranks of the modules over the
corresponding irreducible components of the 
singularity. This is Conjecture 9.1 from \cite{ChQ}.
\medskip

We mostly consider  unibranch singularities
and (algebraic) knots in this paper, which is the 
setting of  Section 6.1 in \cite{ChQ}. 
Though rank-one multibranch singularities will be considered
in Section \ref{sec:multi} below. The reduction formulas
work for them.
 
Generally,  $\h^{daha}=\h^{mot}$
is a very general version of the Shuffle Conjecture. The 
path to this conjecture is via the usage of nonsymmetric 
Macdonald polynomials, and the weights must be minuscule 
to make $\h^{daha}$ entirely combinatorial.
See the end of Section \ref{sec:polynom} above.
\medskip

There is
another connection conjecture, Conjecture 4.5 from
\cite{ChW}; see also \cite{KTr} and
Conjecture \ref{conj:groeb} below.
 It states that $\h^{mot}(qt,t,\aa)=L(q, t,\aa)$
for the $L$-function from \cite{Gal, Sto} extended by $\aa$.
Notice $q\mapsto qt$ in $\h$. It was extended
in \cite{ChQ} to rank-one multibranch plane curve singularities.
Galkin's $L$-functions are directly connected
with those from the {\sf\em ORS conjecture}
from \cite{ORS}:\, parallel constructions
for different {\sf\em motivic measures}. Let us mention
here {\sf\em Kapranov's zeta} \cite{Kap}.

The $p$-adic Galkin-St\"ohr approach is
for any Gorenstein curve singularities 
and the base fields $\F_q$ 
can be different for different components. The 
$p$-adic methods have important advantages and
are well-developed. See, for instance, 
\cite{KTs} on the usage of the
{\sf\em Waldspurger induction}.

One of the  main advantages
of the geometric approach ({\sf\em ORS Conjecture}, etc.) 
is that
the polynomial dependence on $q$ is manifest.
However, the corresponding algebraic-geometric tools (Serre-Deligne's
weight filtration) can be 
involved. Singular varieties are always challenging in
algebraic geometry, and we deal with complicated singularities.

The coincidence of $\h^{mot}$ and $L$ ($p$-adic or geometric)
is sometimes called now ``Hilb-vs-Quot"; see
\cite{KTr}. To be more exact, it is ``Hilb" versus Compactified
Jacobians. The substitution $q\mapsto qt$ remains mysterious here
(at least, to the author). 
\medskip

We extend the Coincidence Conjecture   $\h^{daha}=\h^{mot}$  to
 {\sf\em instanton slices} in this paper. However, the definition
of $L$-functions is unknown for the latter; so  we do not have
by now any $\h^{mot}=L$ or {\sf\em Hilb vs. Quot} for instanton
superpolynomials.
%\medskip

Concerning the connection conjectures  with
the Khovanov-Rozansky polynomials from \cite{Kh, KhR1, KhR2}
and many other papers, the motivic superpolynomials were conjectured
to be connected with {\sf\em reduced} KhR-polynomials, at least for
uncolored algebraic knots, and this is expected for
the superdual instanton superpolynomials $\sH^{inst}$ introduced
and studied 
in this paper. For algebraic uncolored links in the ``Hilb"
case,  this relation is essentially the {\sf\em ORS Conjecture}.

\subsection{\bf Unibranch singularities}\label{sec:basic}
For any base field $\F$, plane curve singularities are
for the {\sf\em singularity rings} 
$\r=\F[[x,y]]/(F(x,y))$, where
$F(x,y)\in \F[[x,y]],$ 
$F(0,0)=0$ and  $(F(x,y))\equal F(x,y)\F[[x,y]]$.
 The {\sf\em unibranch} ones
are for irreducible $F(x,y)$, assumed absolutely  
irreducible (over the algebraic closure  $\overline{\F}$ of $\F$)
in this work,

Equivalently, let $\o\equal \F[[z]]$, where $z$  is
the {\sf\em uniformizing parameter}. Then plane curve singularities
are for $\r=\F[[x,y]]\subset \o$, where the generators
$x,y\in \r$ are in the maximal ideal $\mathfrak{m}_\o\!=\!(z)\!=\!z\o$ 
of $\o$, and the localization of
$\r$ must be the whole $\F((z))$. The notation  $\r=\F[[x,y]]$ 
means that elements in $\r$
are {\sf\em series} in terms of the generators $x,y$. 
Also, let  $\r/\mathfrak{m}_\r=\F$ for
the maximal ideal $\mathfrak{m}_\r\!=\!\r \cap (z)$ of $\r$.

The approach via $\o$ is the main in this paper, 
necessary for the
theory of $\h^{mot}$ in contrast to the $Hilb$-theory, where the
ideals of $\r$ are sufficient. The equation $F(x,y)=0$ 
will not be really needed.

\vskip 0.2cm
The simplest invariants of a singularity are the
{\sf\em multiplicity}, which is  $mult=$
dim$\,\F[[z]]/\F[[z]]\mathfrak{m}_\r$, and
the {\sf\em Serre number} $\de=dim\,\F[[z]]/\r$,  the arithmetic
genus.

\vskip 0.2cm
{\bf Valuation semigroup.}
It is one of the key in the theory of unibranch singularities.
The definition of is as follows:
$\Ga\equal\bigl\{\,\upsilon_z(f), 0\neq f\in \r\subset \o=\F[[z]]
\,\bigr\}$, where
$\nu_z$ is the valuation, the order of $z$. Then
$\de=|G|$ for the {\sf\em set of gaps}
$G\equal\Z_+\setminus \Ga$, and 
 $mult=\min(\Ga\setminus \{0\})$.
\vskip 0.2cm

If $M\subset \o$ is $\r$-invariants, then $\De(M)\equal \{\nu_z(f),
f\in M\}$ is $\Ga$-module, which simply means that 
$\Ga+\De\subset \De$. One has: $deg(M)\equal dim_\F(\o/M)=
|Z_+\setminus \De(M)|$, the number of the corresponding gaps.

An $\r$-module  $M\subset \o$  is called {\sf\em standard} if it 
contains an element $1+z(\cdots)$, Equivalently, $0\in \De(M)$
and $\Ga+\De=\De$ for $\De=\De(M)$. For any $M$, the module
$z^{-v}M$ is standard for $v=\min(\De(M))$.

\medskip
The {\sf\em Compactified Jacobian} of $\r$ 
is $Jac_\r=\{M=M_{st}\subset \o\}$.
It can be supplied with the structure of projective (singular)
variety, which we do not need in this paper. Moreover, our
definitions can be potentially used in characteristic zero:\ 
orders in Dedekind domains in $p$-adic fields instead of  
$\r\subset \F_q[[z]]$. The usage of algebraic geometry is 
limited here.

Following \cite{Pi}, the {\sf\em Piontkowski cells}
$Jac_\De\equal\{M=M_{st}\mid \De(M)=\De\}$ are defined 
for any standard $\Ga$-modules $\De$. Some can be empty.
Actually, empty ones are always present beyond quasi-homogeneous
plane curve singularities, the {\sf\em Piontkowski phenomenon}. 
\medskip

{\bf Reciprocity map.}
The {\sf\em conductor} of $\r$ is the greatest
ideal $\mathfrak{c}=(z^\cc)=z^\cc\o$ in $\o$ that belongs to $\r$.
Then $\cc=2\de$ if and only if $\r$ is 
{\sf\em Gorenstein}, which includes plane curve singularities. 
Generally, $\de+1\le \cc\le 2\de$. Equivalently,
$|\Z_+ \setminus \Ga|=\de=
|\{\ga\in \Ga \mid 0\le \ga < 2\de-1\}$. 
\vskip 0.2cm

Assuming that $\r$ is Gorenstein, let\, $M\mapsto 
M^\hollowstar\equal\r:M=\{f\in \F((z)) \mid fM\subset \r\}$.
Generally, 
$N:M\equal \{f\in \F((z)) 
\mid fM\subset N\}\simeq M^\hollowstar:N^\hollowstar$.
For instance, $\r^\hollowstar=\r$ and $\o^\hollowstar=(z^{2\de})$.
See \cite{PS,St,Sto}.

Let us justify that  $\De(M^\hollowstar)=(\De(M))^\hollowstar$ for
\begin{align}\label{Dedual}
\De^\hollowstar\equal \{p\in \Z\mid p+\De\subset \Ga\}=
(2\de-1)-(\Z\setminus \De), 
\end{align}
where $v\pm X=\{v\pm x \mid x\in X\}.$  See Lemma 2.7 in \cite{GM}.

Any modules $M$ here can be represented in the form 
 $z^m M_{st}$ for standard $M_{st}$ and $m\in \Z$. 
For standard ones, $\Ga\subset\De\subset \Z_+$ and  
$\De^\hollowstar=(2\de-1-(\Z_+\setminus \De))\cup (2\de+\Z_+)$. 
%The ring $\r$ is assumed Gorenstein.
For instance, $\Ga^{\hollowstar}=\Ga$ and 
$(Z_+)^\hollowstar=2\de+\Z_+=\De(\o^\hollowstar)$.
Using (\ref{Dedual}), we obtain that 
 $\De(M^\hollowstar)\subset \De(M)^\hollowstar$ and 
$dim\, N/M=dim\, M^\hollowstar/N^\hollowstar$, which readily
gives the required coincidence: 
$\De(M^\hollowstar)=\De(M)^\hollowstar$.
See, e.g., (2.5),(2.6) from \cite{Sto}.

In particular,  
$\r\subset M\subset \o$ is equivalent to
$\r\supset M^\hollowstar \supset (z^{2\de})$,
which gives the passage 
from such $M$ to ideals $M^\hollowstar\subset \r$ containing 
$(z^{2\de})$.
%We note that $r\!k_q(M)=r\!k_q(M^\hollowstar)$ for any $M$. 
\vskip 0.2cm

This is for any Gorenstein $\r$, not only planar.
However, plane curve singularities are necessary
for the coincidence of the $\varrho$-ranks. Namely. 
$\varrho(M)=\varrho (M^\hollowstar)$ holds for any
plane curve singularities.  
 This is due to the  
{\sf\em Auslander-Buchsbaum theorem}; see Lemma 10 in \cite{ObS}.
The  Gorenstein ring  
$\mathbb{F}_q[[z^4,z^5,z^6]]$ is a counterexample \cite{ChQ}.
\vskip 0.2cm

{\bf Algebraic knots.} Let $\F=\C$. 
Recall that the link corresponding to
$\r$ is the 
intersection of $\{(x,y)\in \C^2 \mid  F(x,y)=0\}$
with the $3$-sphere 
$S_\ep^3\subset \C^2$ centered at $(0,0)$ of small
radius $\ep$.
They are {\sf\em knots} (connected) for unibranch 
singularities, which are for irreducible $F(x,y)$.

The invariants $\de$ and $mult$ are topological: they depend
only on the isotopy class of the corresponding link. 
Importantly, $\Ga$ totally determines
the topological type  
of the corresponding algebraic {\sf\em knot} 
(considered up to isotopy). Thus, 
topological invariants
of (irreducible) rings $\r$ over $\C$
are exactly those expressed in terms of $\Ga$.

Topological equivalence  
of algebraic {\sf\em links} is significantly
more ramified. Generally, {\sf\em splice diagrams} 
are needed here. See \cite{EN,ChD2}.
%\vskip 0.2cm
\vfil

{\bf Basic families.} Algebraic torus knots $T(\rr,\ss)$ 
are those for the
singularities $x^\ss\!=\!y^\rr$, where $\rr,\ss>0$ and
$gcd(\rr,\ss)\!=\!1$. The corresponding rings are
$\r=\F[[x\!=\!z^\rr,y\!=\!z^\ss]]$; they are called
{\sf\em unibranch
quasi-homogeneous} due to their invariance with respect to
the action $x\mapsto \ze^\ss, y\mapsto \ze^\rr$
for $\ze\neq 0$. For them, $mult=\min(\rr,\ss)$ and  
$\de =$ {\large$\frac{(\rr-1)(\ss-1)}{2}$.} 
The {\sf\em Piontkowski cells} are always 
nonempty $\mathbb{A}^N$ in this case. See  \cite{Pi}, which
establishes that cells are affine spaces, and count
their dimensions for some other families; also there are
conditions for empty cells there.
\vfil

The simplest ``non-torus" family 
is $\r\!=\!\C[[x=z^6\!+\!z^{7+2m}, y=z^4]]$ for $m\!\in\! \Z_+$,
when  $mult=4$ and  $\de_m\!=\!8\!+\!m$. The corresponding
cables are $C\!ab(13+2m,2)T(3,2)$; recall that 
$C\!ab(2,3)=T(3,2)$. One can replace here 
$T(3,2)$ by $T(2,3)$; the cabling parameters
are topological invariants. Note that $C\!ab(11,2)T(3,2)$
is non-algebraic. See Section \ref{sec:ITER-KNOTS} for
iterated torus knots, cabling parameters, and
Newton's pairs. 
The Piontkowski cells $Jac_\De$ are affine spaces
for this family, and there are empty ones (always $2$ for any 
$m\ge 0$).
\vskip 0.2cm

Generalizing, let us provide  $\Ga$ and $\de$ for 
the ring 
$\r=\F[[ x=z^{\upsilon \rr} +z^{\upsilon \rr+p}, y=z^{\upsilon \rr}]]$,
where  $\rr>\ss>0$, $p>0$, and we (always) assume that
$gcd(\rr,\ss)=1=gcd(\upsilon,p)$. The
corresponding cable for $\upsilon>1$ is 
$C\!ab(\upsilon \rr \ss\!+\!p,\,\upsilon)\,
T(\rr,\ss)$. One has:
\begin{align}\label{Ga-cab}
&\Ga=\lan \upsilon \rr,\,
\upsilon \ss,\,\upsilon \rr \ss\!+\!p\ran,\ 
2\de=\upsilon^2 \rr \ss\!-\!\upsilon (\rr\!+\!\ss)+
(\upsilon\!-\!1)p+1.
\end{align}
\vskip 0.2cm

For such $\r$, the {\sf method of syzygies} from \cite{Pi}
provides a machinery 
for proving that the cells are affine spaces (when they are)
and finding their dimensions, including the
count of the number of
empty cells. However, this can be combinatorially involved.
See \cite{ChP1}, where 
$\varrho$-ranks were added and we proved that
the corresponding cells remain affine spaces for
certain families.  Paper \cite{ChP2} was devoted to modules of  
any ranks $\cn$; it was proven there that
the cells are still affine spaces for (irreducible)
quasi-homogeneous singularities. See also \cite{GMO}, 
which developed Piontkowski
formulas in the case of $p=1$, including the count
of empty cells; the $\varrho$-ranks 
and $\cn>1$ were not discussed there.

Papers \cite{ChP1,ChP2} contained a lot of examples of non-affine
cells, which were conjectured there to be always configurations of
affine spaces. They can be of
various kinds (and even disconnected). Some of them are 
of Euler characteristic $0$; for instance,
$\mathbb{A}^N\setminus \mathbb{A}^{N-1}$ 
and empty ones. Such cells for $p=1$
were discussed in  \cite{GMO}.
\medskip

Let us mention that the Euler characteristic
of Compactified Jacobians was counted in Corollary 5.48
of \cite{KTs} in terms of {\sf\em nested Dyck paths} in the case of
any unibranch plane curve singularities. See Theorem 5.46, including
the identification of their 
($p$-adic) construction
with the DAHA superpolynomials  from \cite{ChD1} for $t=1$.
The extension to instanton slices $\mathfrak{S}_\sC$ 
with  monomial $\sC$
is our Theorem \ref{thm:echar1}. 
\medskip

This theorem in the case of $\cn=1$ is based on the claim that
every {\sf\em Gr\"obner cells} $Gr$ (defined below) 
contains exactly one monomial ideal. Provided that
$|Gr|$ are $q$-polynomials, which is a conjecture for 
arbitrary $\la$,  this proves 
that their Euler characteristics are $1$. This argument
works in any rank $\cn$.
If $\sC$ is not monomial, there can be empty
{\sf\em Gr\"obner cells}. For instance, this happens
(for $\cn=1$) when $\sC$ is obtained by adding
$x^2+y^3$ to $I_\la$ for $\la=(4,2,1)$. 

%We conjecture that Euler characteristics of $G_{\mu\subset \la}$
%are always $1$ if $\sC=I_\la$, which is  for any Young diagrams $\la$. 

\vskip 0.2cm

{\bf Good reduction.}
The plane curve singularities are mostly over finite 
fields $\F=\F_q$ in this paper:\, for $q=p^k$ and prime $p$.
The following is the passage from the base field $\C$ to $\F_q$.

We begin with 
$\r$ over $\C$,  define it over $\Z$, which is always 
doable {\sf\em within a given isotopy type}, and then
consider $\r\otimes _{\Z} \F_q$ for $q=p^k$.  
Prime $p$ is called a {\sf\em place (prime) of good
reduction} if
$F(x,y)$ remains irreducible over $\F_q$,  
and if $\Gamma$ remains unchanged upon this reduction. 
\vskip 0.2cm

All primes $p$ are good in this sense for the
rings $\C[[x\!=\!z^\rr,y\!=\!t^\ss]]$.
It is expected that there are no primes $p$  
of bad reduction for any algebraic knots.
This means that  given any $p$, there exists
 $\r$ within a given topological type
 where this $p$ is good.
\vskip 0.2cm

For instance,
$\Z[[x\!=\!t^4,y\!=\!t^6\!+\!t^7]]$ has bad reduction at $p=2$. 
Indeed, 
$\nu_z(y^2-x^3)=14$ over $\F_2$,  which is $13$ for $p\neq 2$. 
However,
this singularity is equivalent over $\C$
(analytically, not only topologically) to the one for
$\Z[[t^4+t^5,t^6]]$, where $p=3$ is the
only place of bad reduction. Thus, the corresponding cable
has no bad reduction. 
%\vskip 0.2cm

\subsection{\bf Modules of any ranks}\label{sec:ranks}
We switch from modules in $\o$ to modules in 
$\Om\equal \o^{\oplus \cn}=\oplus_{i=1}^\cn
\o\vep_i$, a free module  over $\o$ of
rank $\cn$.  The basis
$\{\vep_1,\ldots, \vep_\cn\}$ will be fixed in this section. 

The
{\sf\em valuation} in $\Om$ will be
$\nu(\sum_{i=1}^\cn f_i \vep_i)\equal\min_{i=1}^\cn(\nu_z(f_i)+\nu^i)$,
where we take $\nu^i=\nu(\vep_i)$  assuming that 
$0=\nu^1<\nu^2<\cdots<\nu^\cn<1$. Then one has:\ 
$\nu(\Om)=\cup_{i=1}^\cn 
\{\nu^i+\Z_+\}$, $\Ga^\nu\equal \nu(\r^{\oplus \cn})=
\cup_{i=1}^{\cn}\{\nu^i +\Ga\}$.

Actually, $\nu^i$ can be arbitrary here such
that  $\nu^i-\nu^j\not\in \Z$ for $i\neq j$. Different choices
of $(\nu_i)$ result in ``wall crossing". 
\medskip

For any $\r$-submodule
$M\subset \Om$, 
$dim_{\F}(\Om/M)=|G^\nu(M)|$, where $G^\nu(M)$, 
the {\sf\em set of gaps},
is defined as $\nu(\Om)\setminus\De^{\!\nu}(M)$ 
for $\De^{\!\nu}(M)\equal \nu(M)$. The latter set
belongs to $\nu(\Om)$ and is a  $\Ga$-modules:\,
$\Ga+\De^{\!\nu\!}\subset \De^{\!\nu\!}$.

\medskip
{\sf\em Standard modules}
are now defined as $\r$-invariant $\F$-subspaces $M\subset \Om$
such that $\o\,M=\Om$. 
Equivalently, the last condition means that 
$\De^{\!\nu}(M)$ must be  standard, where
a $\Ga$-module  $\De^{\!\nu}\subset \nu(\Om)$ is called standard
if $\{\nu^i\}\subset \De^{\!\nu}$. The number of 
standard $\De^{\nu}$
is $(std\,_\Ga)^\cn$,
where $std\,_\Ga$ is the number of standard $\Ga$-modules
$\De\subset \Z_+$ in rank one.

Let  $dev(M)\equal \de\, \cn-deg(M)$ for $deg(M\equal \,
dim_{\F}(\Om/M)=|D^\nu(M)|$. Here
$D^\nu(M)\equal\De^{\!\nu}(M)\setminus \Ga^\nu$, which is 
the set of {\sf\em added gaps} from $\Ga^\nu$ to $\De^\nu(M)$.   
In particular, $|G^\nu|=\de\, \cn=dev(\Om)=$\,
$dim_{\F}\Om/\r^{\oplus \cn}$
for $G^\nu\equal D^\nu(\Om)=\cup_{i=1}^{\cn}\{\nu^i +G\}$,
where $G=\Z_+\setminus \Ga$ as above. 
\medskip

{\bf Invertible modules.} They are, by definition,
 standard ones with $\cn$ generators
over $\r$. We will use symbol  $\Im$ for them (instead of $M$).

There are $3$ other (equivalent) definitions:\
(a)\, $\De^{\!\nu}(M)=\Ga^\nu$,
which is obviously the smallest standard $\Ga$-module;
(b)\, they  are standard of $dev=0$ (deviation from $\r^\cn$);
(c)\, they are  standard and $deg(M)=\de\,\cn.$ 
\vskip 0.2cm

It is easy to see directly that {\sf\em invertible modules} $\Im$ 
depend on
 $\de \cn^2$ free parameters. Indeed, we can make the generators
in the form $f^i=\vep_i+ \sum_{ijk} c_{ijk}z^k\vep_j$,
where $k\in D=\Z_+\setminus \Ga$, and $1\le j\le \cn$ are arbitrary.
This gives $\de \cn^2$ free parameters. 
See Lemma \ref{lem:inv} below for a generalization. 

Also, one can identify the
set  $\{\Im\}$ with $GL_\cn(\o)/GL_\cn(\r)$
under the action of $GL_\cn(\o)$ on elements of $\Om$ considered
as vectors:\ linear combinations of $\vep_i$ with coefficients 
in $\o$. Then $GL_\cn(\r)$ is the stabilizer of $\r^\cn$. It 
suffices  to consider  here elements in $GL_\cn(\o)$ 
(and in $GL_\cn(\r)$) that are $\bf 1$ modulo
$zM_{\cn}(\o)$, which gives the required $\de\cn^2$. 
\medskip

This dimension formula is a singularity counterpart of the
famous dimension formula $(g-1)\cn^2+1$ for the space of 
stable 
bundles of rank $\cn$ over a smooth projective 
curve of genus $g\ge 2$.
We do not have ``global sections"
and therefore do not divide  by the action of $GL_\cn(\F)$, 
which explains
missing $-1$ in our formula.
\medskip

We note that theory of torsion-free sheaves of 
ranks greater than $1$ over {\sf\em singular} 
(projective) curves is quite involved even for
simple singularities and small ranks $\cn$. The problem is
that the resulting varieties are highly singular, the definition
of (semi)stable sheaves is involved and so on. The restriction
to ``pure" singularities is much simpler, though the corresponding
Quot-schemes can be (still) quite involved. However, we can employ
here various additional tools. 
\medskip 

The following generalization of the dimension formula $\de\cn^2$ 
and 
Lemma 6.1 from \cite{ChQ} are
important for our construction.
\medskip

Let 
$\{e_g, g\in D^\nu \}$ be a system of any elements $e_g\in M$
such that $\nu(e_g)=g$, where $D^\nu=D^\nu(M)$ for a given standard $M$.
They generate $M$ as an $\r$-module.
For any {\em invertible} submodule
$\Im\subset M$, its  $\r$-generators can be taken as follows:
$\ep^i=e_{\nu^i}+\sum_{g\ge 1}a_g^i e_g$, 
where  $1\le i \le \cn$,\, $g\in D^\nu$ and  $a_g^i\in \F$. 
Recall that $0\le \nu^i<1$ for $1\le i\le \cn$. Here
$|D^\nu|=dev(M)$ and we employ the Nakayama Lemma. Claim $(ii)$
below readily follows from the identity
$\bigl|\{\ga\in \Ga \mid 1\le \ga < 2\de\}\bigr|=\de-1$, which holds
for any Gorenstein $\r$. 

\medskip
\begin{lemma}\label{lem:inv}
(i) The parameters $\{a_g^i\}$ make 
the set $\{\Im\subset M\}$ an 
affine space $\mathbb{A}^{dim}$ (over $\F$), where $dim=dev(M)\,\cn$.
This  identification  
depends on the choice of  $\{e_g, g\in D^c(M)\}$.  
In particular, all invertible modules, which are those for 
$M=\Om$ above, form the affine space  $\mathbb{A}^{\de \cn^2}$. 

(ii) Inside any given invertible $\Im$, the number of free
parameters (from $\F$) for the 
systems of generators  $\tilde{\ep}^{\,i}=\ep^i+z(\cdots)$ considered
modulo $z^{2\de}\Om$  equals $(\de-1)n^2$. Here  $1\le i\le \cn$
and this number is the same as that for $\Im =\r^\cn$; use
the natural action of $GL_{\cn}(\o)$. 
\sq
\end{lemma}

%\vskip 0.2cm
%Finally, the superpolynomial of rank $\cn$ of  $\r$ over $\F=\F_q$
%is as follows:

%\medskip
\subsection{\bf Unibranch superpolynomials}
We will now define the motivic superpolynomials for
any ranks $\cn\ge 1$:
\begin{align}\label{sup-rank}
&\h^{mot}_\r(q,t,\aa; \cn)\equal
\sum_{M=M_{st}} t^{deg(M)} (1+\aa q^{\cn})
\cdots (1+\aa q^{\varrho(M)-1}),\\
&\varrho(M)=dim_\F (M/\mathfrak{m}_\r M),\,
\text{ and the sum is over standard } M.\notag
\end{align}
%and the summation is over all standard $M\subset \Om=\o^{\cn}$.
\noindent
This sum is finite here because standard $M$ contain elements
$\vep_i+ z(\cdots)$ for any $1\le i\le \cn$, and
are trapped between  $\mathfrak{c}^\cn$ for the
conductor $\mathfrak{c}$ of $\r$ and $\Om=\o^\cn$. The space
$\Om/\mathfrak{c}^\cn$ is of dimension $2\de\cn$ over $\F_q$.

The dependence on $q$ is via the count of $M_{st}$. Formula
(\ref{sup-rank})
does not give that $\h^{mot}_\r(q,t,\aa;\cn)$ is a $q$-polynomial, 
but this is conjectured.
\vskip 0.2cm

We mention Corollary 1.11, Theorem 8.12  and Proposition 8.14 
in \cite{KTs}, 
where the number
of $\F_q$-points of Compactified Jacobians for (any) unibranch
plane curve
singularities was proven to be a polynomials in terms of $q$ 
and a topological invariant of a singularity. This is the case 
of  $\h^{mot}(q,t\!=\!1,\aa\!=\!0)$. See also Corollary
\ref{cor:3gen-top} below. 
%and Theorem 
%\ref{thm:general-semiroot-conductor} below.
\vskip 0.2cm

{\bf Reduction formula.}
We will use the natural action of $GL_\cn(\o)$ in $\Om=\o^\cn$,
and in the variety  $\mathbb{J}_\cn\equal \{M_{st}\}$. Also, let
$GL_\cn^{(k)}\equal \{\Phi=\mathbf{1}+ z^{k}Mat_{\cn}(\o)\}$ and
$\u\equal GL^{(1)}_\cn(\o)/GL_\cn^{(2\de)}(\o)$.
Consider the variety  
$\tilde{\mathbb{J}}^{ext}_\cn$ of
pairs $\bigl\{M, \tilde{\Psi}=\bigl(\tilde{\psi}^{i}\bigr)\bigr\}$ 
modulo 
the action of $GL_\cn^{(2\de)}(\o)$, 
where $M$ is standard,  
$\psi^{i}=\vep_i+z(\cdots)\in M$ for $1\le i\le \cn$, and
$\tilde{\psi}$ is $\psi$ modulo $z^{2\de}\Om$. These
elements generate an invertible module
$\Im\subset M$, and any invertible $\Im$ has such
a basis by definition. The group
 $\u$ acts in $\tilde{\mathbb{J}}^{ext}_\cn$.
\vskip 0.2cm

There are $2$ natural surjective 
projections 
from $\tilde{\mathbb{J}}_\cn^{ext}$.
The $1${\footnotesize st}  one is\, $\pi_1:\tilde{\mathbb{J}}^{ext}
\to \{M'\supset \r^\cn\}$,  
sending $(M,\tilde{\Psi})\mapsto M'=\Psi^{-1} M\supset \r^\cn$, 
where we consider $\tilde{\Psi}$ and its lift 
$\Psi$ as (invertible) matrices. The fibers are isomorphic to $\u$,
so they are $\mathbb{A}^{dim_1}$ for $dim_1=(2\de-1)\cn^2$.
This image is $\{M \mid \r^\cn\subset M\}$; these $M$ are 
automatically standard. 

The $2${\footnotesize nd} is the forgetting
map $\pi_2: (M,\Psi)\mapsto M$. It is fiber at 
$M\in \tilde{\mathbb{J}}^\cn$
is an affine space of dimension $dim_2(M)=dev(M)\cn+(\de-1)\cn^2=
(\cn\de-deg(M))\cn+(\de-1)\cn^2=\de \cn^2 -deg(M)\cn$, where we
use  Lemma \ref{lem:inv},(i,ii). For this projection:\

\begin{align*}
&\h^{ext}_\r(q,t,\aa;\cn)\equal
\sum_{\{M,\tilde{\Psi}\}\in \tilde{\mathbb{J}}^{ext}} t^{deg(M)} 
(1+\aa q^{\cn})\cdots (1+\aa q^{\varrho(M)-1})\\
&=q^{(2\de-1)\cn^2}\,
\sum_{M=M_{st}} \Bigl(\frac{t}{q^\cn}\Bigr)^{deg(M)} 
(1+\aa q^{\cn})\cdots (1+\aa q^{\varrho(M)-1}).
\end{align*}

Using the first projection  $\pi_1$:\, 
\begin{align*}
&\h^{ext}_\r(q,t,\aa;\cn)=q^{(2\de-1)\cn^2}\h^{red}_\r
(q,t,\aa;\cn),
 \text{ where we set }\\
&\h^{red}_\r(q,t,\aa;\cn)\equal
\sum_{\r^\cn\subset M} t^{deg(M)} 
(1+\aa q^{\cn})\cdots (1+\aa q^{\varrho(M)-1}).
\end{align*}

Combining these expressions for $\h^{ext}$, we obtain the
identities
$$ \h^{mot}(q, \frac{t}{q^\cn}, \aa)=\h^{red}(q,t,\aa),\ 
\h^{mot}(q, t , \aa)=\h^{red}(q,\frac{1}{t q^\cn},\aa),
$$
which hold for any (irreducible) Gorenstein $\r$.

The last step is employing the {\sf\em reciprocity map},
which is now :
\begin{align}\label{recipcn}
M^\hollowstar\equal \{\psi\in \F_q((z))^\cn 
\mid (\!(\psi,\phi)\!)\in \r 
\text{ for any } \phi\in M \},
\end{align}
 where 
$(\!( \sum_i a_i\vep_i, \sum_i b_i\vep_i)\!)=
\sum_ia_ib_i$. It satisfies all reciprocity properties for 
Gorenstein $\r$ and the relation 
$\varrho(M)=\varrho (M^\hollowstar)$ (only) for plane curve 
singularities 
 due to the  
{\sf\em Auslander-Buchsbaum theorem}. 

The embedding $\r^{\cn}\subset M\subset \Om=\o^\cn$ gives that 
$M^\hollowstar$ can be  arbitrary $\r$-invariant 
submodules $M'$  in $\R^\cn$ containing  $\mathfrak{c}^\cn$
for the conductor $\mathfrak{c}$ of $\r$. Setting
$deg(M')=dim_\F(\r^\cn/M')$, let us introduce:

$$ 
\h^{dual}(q, t, \aa)\equal\sum_{\mathfrak{c}^\cn\subset M'
\subset \r^\cn}
t^{deg(M')} (1+\aa q^{\cn})\cdots (1+\aa q^{\varrho(M')-1}).
$$

Using that $deg(M)=dim_\F(\Om/M)=
\de \cn -dim_\F(\R^\cn/M^\hollowstar)$, we obtain the 
final version of the {\sf\em reduction formula}:

\begin{align}\label{redforranks}
\h^{mot}(q,t,\aa)=t^{\de\cn} q^{\de \cn^2} 
\h^{dual}(q,1/(t q^{\cn}),\aa),
\end{align}
which holds as such for any unibranch plane curve singularities
and arbitrary ranks $\cn$, but must be considered at $\aa=0$
for Gorenstein $\r$, if they are not from plane curve singularities.  
Making $\aa=0$ in this case  can be avoided by 
adjusting the definition of $\varrho$ for $M'$; see 
Theorem \ref{thm:gen-sing}.
 We note that {\sf\em superduality}, generally,
does not hold for Gorenstein curve singularities (even in relatively
simple examples), though it 
holds for (uncolored) $L$-functions due to Galkin and St\"ohr.

\subsection{\bf Multibranch singularities} \label{sec:multi}
The generalization of our reduction formulas to 
{\em non-unibranch} plane singularities is quite doable,
but we will consider only the case of rank $1$ in this
paper. This is the case of {\sf\em Affine Springer fibers} of type $A$
for square-free characteristic polynomials, corresponding
to uncolored algebraic {\sf\em links} in topology.  
See \cite{ChQ} for, possibly, non-square-free
characteristic polynomials. 
\medskip

{\bf Adding idempotents.}
The ring will be now $\r\subset \o\equal 
\oplus_{i=1}^\kappa \o_i e_i$,
where $\o_i=\F[[z_i]]$ and $e_i e_j=\de_{ij} e_i$,
and $\kappa$ is the number of (absolutely)
irreducible components
of a singularity.
We set
$e\equal\sum_{i=1}^\kappa e_i$ (the unit in $\o$),
 $z\equal \sum_{i=1}^\kappa z_i e_i$ and identify $z_i$ with
$z e_i$. 
Generally, $f_i$ will be the projection $f e_i$ for any $f\in \o$. 
As above, $\r$ contains $1=e$ and have $2$ generators: 
$x=\sum_{i=1}^\kappa x_i$ and 
$y=\sum_{i=1}^\kappa y_i$ in $\mathfrak{m}_{\o}=z\o$. 
Also, the localizations of the projection 
$\r_i$ of $\r$ must be $\F((z_i))$, and $\r_i/\mathfrak{m}_i=\F$
for $\mathfrak{m}_i=\r_i\cap z_i\o_i$, the maximal ideals of $\r_i$.

By construction, 
$\prod_{i=1}^\kappa
F_i(x,y)=0$, where $F_i(u,v)$ are irreducible polynomials over $\F$,
assumed irreducible over its algebraic closure $\overline{\F}$, 
such that 
$F_i(x_i,y_i)=0$.  The polynomiality of $F_i(u,v)$ 
can be achieved by deforming $x_i$ and $y_i$ 
if necessary (the Weierstrass Preparation Theorem). 
The product $F(u,v)=
\prod_{i=1}^\kappa F_i(u,v)$ 
must be {\em square-free}, i.e. $F_i$ and $F_j$ are 
assumed non-proportional
for $i\neq j$. This is standard for curve singularities.
The general case (any ranks $\cn_1,\ldots, \cn_\kappa$) 
is when $\r$-modules $M$ are considered in $\Om=\o^{\cn_1}\oplus 
\o^{\cn_2}\oplus\cdots\oplus \o^{\cn_\kappa}$, for $\r$
and $\o$ as above, acting in $\Om$ diagonally.  See \cite{ChQ}. 

The connection conjecture
$\h^{daha}(q,t,\aa)=\h^{mot}(q,t,\aa)$ is stated
there in full generality, with comprehensive
discussion and many example. However the conjecture 
$\h^{mot}(qt,t,\aa)=L(q,t,\aa)$ in the case of
multibranch singularities is restricted to rank one.
\vskip 0.2cm

Let
$\Ga\equal\{\nu(f), \r\ni f\neq 0\}$, where 
$\nu(e_i)=\nu^i$ and $\nu(f)\equal\min_{i=1}^\kappa
\{\nu_i(f_i)+\up^i\}$ for any $f\in \o$.  We assume that 
$\up^1=0<\up^2<\cdots<\up^\kappa<1$. Actually, they can be any 
numbers from $\R_+$ such that $\up^i-\up^j\not\in \Z$ for $i\neq j$.
Let $\de\equal\text{dim}_{\F}(\o/\r)=|\nu(\o)\setminus \Ga|$. 
Note that
$|\nu(\mathfrak{m}_\o)\setminus \nu(\mathfrak{m}_\r)|=\de-\kappa+1$, 
where
$\mathfrak{m}_\o\equal\oplus_{i=1}^\kappa \mathfrak{m}_{\o_i}$ 
for $\mathfrak{m}_{\o_i}=z_i\o_i\subset \o_i$ and 
$\mathfrak{m}_\r\equal
\r\cap \mathfrak{m}_\o$. 
\vskip 0.2cm

As above, the 
{\sf\em conductor} $\mathfrak{c}$ is the greatest 
$\o$-ideal that belongs to $\r$. It is $\o$-generated by 
$z_i^{n_i}\in  \r$ for 
minimal $n_i>0$ such that $z_i^{n_i+j}\in  \r$ for any $j\ge 0$.
 Equivalently, $n_i$ are minimal such that  
$n_i+\up_i+\Z_+\subset \Ga$. As in the unibranch case, 
dim\,$_\F(\o/\mathfrak{c})=\sum_{i=1}^\kappa n_i=2\de$.
% which
%is for any plane curve
%singularities
%(the defining property of Gorenstein rings). 

\subsection{\bf Multibranch reduction}\label{sec:sup-links}
The  equation $F(u,v)=0$ that
gives the corresponding singularity can be now reducible
with pairwise non-proportional irreducible factors. Accordingly,
the corresponding link for $\F=\C$, which is 
$\{F(u,v)=0\}\cap S^3_\ep$ in $\C^2$
with the coordinates $u$ and $v$, has $\kappa$ components.
Its isotopy type gives the topological type of the singularity.
It is fully determined by the connected components and their
pairwise linking numbers (algebraically, intersection ones).

The passage from $\C$ to $\F_q$ is basically the same as in the 
unibranch case. Namely, we pick $x,y\in \Z[[z]]$ within
a given topological type and then switch to 
$\F_q$ for $q=p^m$ provided that $p$ is a prime of {\sf\em
good reduction}. By definition, good $p$ are such that 
the corresponding $F_i$ remain irreducible and 
pairwise non-proportional over $\F_q$. Also, the  semigroups 
$\Ga_i$ for $\r_i$
and the pairwise intersection numbers 
 must remain unchanged.
It suffices to assume that $\{\de_i\}$ (for the components $\r_i$)
and the total $\de$ remain unchanged. The latter is given in
terms of $\{\de_i\}$ and the pairwise intersection numbers.

Good $p$  are necessary to relate 
the motivic superpolynomials (in our approach) 
to the DAHA superpolynomials and the KhR-polynomials.
\vskip 0.2cm

{\bf Standard modules.} Let $\F=\F_q$ and $\r\subset \o$ are 
as above. The {\sf\em standard modules} $M\subset \o$ are 
$\r$-invariant
such that $M_i=M e_i$ are standard in $\o_i$. Equivalently,
$\o M=\o$. This means that for any 
$1\le i\le \kappa$, there exists some $\phi\in M$ such that 
$\phi_i-e_i\in z_i\o$, where $\phi_i$ is the projection of
$\phi$ onto $\o_i$.  Such $\phi$ can be different
for different $e_i$, but it is simple to check by induction 
that for any standard $M$, there 
exists one $\phi$ serving all $e_i$, i.e., such 
that $\phi\in \o^*\cap M$, 
where $\o^*=\prod_{i=1}^\kappa \o_i^*$ (the group of
invertibles in $\o$).

A special case is when  $\r_i$ are all isomorphic to $\r_\circ$.
Then $\r\cong \r_\circ$ and our $\o$ (in this case) is actually
$\Om_\circ=\o_{\circ}^{\oplus \kappa}$ from Section 
\ref{sec:ranks}. They are isomorphic as $\r$-modules. 
The  difference (and a significant one) is
that the standard modules for $\Om_\circ$ must satisfy
stronger relations. There must exist $\phi^{(i)}$
for every $1\le i\le \kappa$ such that
$\phi^{(i)}=e_i$ modulo $z\Om$; the basic vectors were
denoted $\vep_i$ for $\Om_\circ$. 
Such modules  are standard for $\r\subset \o$ (with
$\kappa$ branches), but not 
the other way round. Recall that one 
 $\phi$  can ``serve"
all $e_i$ for standard $M$ for $\r\subset \o$.

\vskip 0.2cm

The uncolored {\sf\em motivic
superpolynomials} (when $\nn=1$)  are as follows:
$$
\h^{mot}\equal\sum_{M} t^{\hbox{\tiny dim}(\o/M)}
\prod_{j=1}^{\varrho(M)-1}(1+\aa q^j)
\text{ summed over standard } M.
$$
Here $\varrho(M)=$\,dim\,$_{\F_q} M/\mathfrak{m}_\r M$ and
\,dim\,$(\o/M)=$\,dim\,$_{\F_q}(\o/M)$, which we generally
denote by $deg(M)$. One has:\, 
$deg(M)=
|\nu(\o)\setminus \De(M)|$, where $\De(M)\equal\{\nu(f) \mid
f\in M\}$ contains $\Ga$ for standard $M$.

Note that the action of $\Ga$ on  $\De(M)$ is given
by the ``twisted" formulas $(\nu^i+\ga_i)+(\nu^i+x_i)=
(\nu^i+\ga_i+x_i)$, which is due to our usage of $\{\nu^i\}$.
 One has: $\Ga+\De=\De$ in this (twisted) sense
for standard $\De$. As above, standard $M$ contain 
the conductor $\mathfrak{c}$. Indeed,
 $M=\r M\supset \mathfrak{c}M\supset
\mathfrak{c}\,(1+z(\cdot))=\mathfrak{c}.$ This gives
that  there are finitely many standard modules for $\F=\F_q$. 

\comment{
Extending the corresponding definition in the unibranch case, 
let $\j(\De)\equal \{ M \mid \De(M)=\De\}$ for standard 
$\Ga$-modules $\De$.  They are quasi-projective schemes over 
$\F_q$. We conjecture that they are {\sf\em affine configuration 
spaces}, which is based on extensive calculations.
}
%\vskip 0.2cm

Let us emphasize  that evaluations and gaps heavily depend on
$\{\nu^i,1\le i\le \kappa\}$. See Lemma 2.6 in \cite{ChP2} 
concerning the passage from one such set to another. It is
not always convenient to make $\nu_i<1$. For instance, 
let $n_1\ge n_2\ge \cdots\ge n_\kappa$ for the conductor 
$\mathfrak{c}=\oplus_{i=1}^\kappa z^{n_i}\o$, 
and $\nu^i=n_1-n_i-\ep_i$ for small sufficiently general
$\ep_i>0$, where $i\ge 2$ and $\nu^1=0$ as above. 
Then $\nu(\mathfrak{c})=\{v\in \nu(\o) \mid v>n_1-1\}$.
\vskip 0.2cm

{\sf\em Invertible modules} now
are standard of $\varrho$-rank $1$, equivalently,
those with one $\r$-generator, which is 
$\phi=1+\sum_{i=1}^{\kappa-1}\al_i e_i+\sum_{g} a_g g$
for $g\in G\equal\nu(\mathfrak{m}_\o)\setminus 
\nu(\mathfrak{m}_\r)$. The parameters $\al_i\in \F_q^*$
and $a_g\in \F_q$
are free and such modules constitute a group scheme
 $\mathbb{G}_m^{\kappa-1}\times
\mathbb{G}_a^{\de-\kappa+1}$ over $\F$. 
Their contribution to $\h^{mot}$ is
$(q-1)^{\kappa-1} q^{\de-\kappa+1} t^{\de}$. 

The greatest $\varrho$-rank is that 
for $\o$,
which is $m_1+\cdots+m_\kappa$ for 
$m_i=\min\bigl\{\Ga_i-\nu^i\setminus \{0\}\bigr\}$,
where $\Ga_i$ are the valuation
semigroups for $\r_i$. The $m_i$ are the greatest $\varrho$-ranks
of the corresponding components, which are the multiplicities
of the irreducible singularities plus $1$. 
%Note that, generally,
%$\o$ is not  a unique standard module of maximal $\varrho$-rank. 
%\vskip 0.2cm
\medskip

{\bf Reduction formula.}
The following formula reduces $\h^{mot}$ to considering 
$M$ containing $\r$; such $M$ are standard automatically.
Let $\overline{M}\equal (M+\mathfrak{m}_\o)/\mathfrak{m}_\o$
for standard $M$; it is a vector spaces over $\F_q$ of dimension
$\upsilon_M=|\De(M+\mathfrak{m}_\o)\setminus \De(\mathfrak{m}_\o)|$.
One has: $\upsilon_\o=\kappa, \upsilon_\r=1$. 
Also, we set $\overline{U}=\o^*/(\F_q^*+\mathfrak{m}_\o)\cong
(\F_q^*)^{\kappa-1}$. 
The {\sf\em Generalized Jacobian} of $\r$ is 
$U\equal \o^*/\r^*$, which is isomorphic
to $\overline{U}\times  \F_q^{\de -\kappa+1}$ as a group.

Let $I_M=\{\phi\in M\cap \o^*\} \mod \r^*$ and
$I_{\overline{M}}= \{\overline{M}\cap \overline{\o}^*\} \mod \F_q^*$.
The latter  is a set of the elements in the group $(\F^*)^\kappa$
modulo $\F_q^*$ 
that belong to $\overline{M}$, the image of $M$ in $\F^\kappa$. 
%Since $M$ are standard, $I_M$ and $I_{\overline{M}}$
%are never empty. 

The groups
$U$ and $\overline{U}$ act in $I_M$ and $I_{\overline{M}}$ respectively. 
One has $I_{\overline{M}}\cong (\F_q^*)^{\upsilon_M-1}$ for 
sufficiently general $M$, but the set $I_{\overline{M}}$ can be 
smaller than this. 
Generally, finding such sets 
leads to interesting linear algebra problems. 
This can be reformulated as finding subsets of elements in vector
subspaces $V\subset \F_q^\kappa$ with {\sf\em all nonzero} 
projections 
onto the basic vectors; see, e.g., \cite{Fr}.
We set  $i_{\overline{M}}\equal |I_{\overline{M}}|$, which is 
$(q-1)^{\upsilon_M-1}$ for sufficiently general  $M$.

As in the unibranch case, the key 
is the {\em\sf extended Compactified Jacobian},
the variety $\j^{ext}$ of
pairs $(M, \vph)$, where $M$ is standard and 
$\vph\in I_{M}$. 

There are $2$ natural surjective 
projections 
from $\j^{ext}$.
The $1${\footnotesize st}  is\, $\j^{ext}\to \{M'\supset \r\}$,  
sending $(M,\vph)\mapsto M'=M\vph^{-1}\supset \r$, which is
with the fibers of size 
$(q-1)^{\kappa-1}q^{\de-\kappa+1}$.
The $2${\footnotesize nd} one is the forgetting
map $(M,\vph)\mapsto M$ with the fibers of size 
$q^{dev(M)-\up_M+1}i_{\overline{M}}$.
Recall that $dev(M)=\de-deg(M)$; we use a counterpart of
 Lemma \ref{lem:inv} for links.
%$(q-1)^{\kappa-1-u_M}q^{\hat{d}}$. 
Combining these maps, we
reduce the calculation of $\h^{mot}$ to the following weighted
sum over the image of the
first projection:
%\begin{align}\label{M-one}
%\h^{mot}\!=\!\sum_{M\supset \r}(qt)^{deg(M)}
%\Bigl(\frac{(q-1)}{q}\Bigr)^
%{u_M}\,(1\!+\!aq)\cdots(1\!+\!aq^{\varrho(M)-1}).
%\end{align}
\begin{align}\label{M-one}
\h^{mot}\!=\!\sum_{M\supset \r}(qt)^{deg(M)}
\Bigl(\frac{\,(q-1)^{\kappa-1}}{q^{\kappa-\up_M}\,
i_{\overline{M}}}\Bigr)\,
(1\!+\!\aa q)\cdots(1\!+\!\aa q^{\varrho(M)-1}).
\end{align}
This formula is helpful practically, especially for $\kappa=1$,
when the weights $w_M=
\Bigl(\frac{\,(q-1)^{\kappa-1}}{q^{\kappa-\up_M}\,
i_{\overline{M}}}\Bigr)$ are trivial.
Theoretically, we will combine it with
the {\sf\em reciprocity map}, which  reduces 
the calculation of $\h^{mot}$ to the corresponding  
{\sf\em weighted instanton superpolynomials}.
%See \cite{ChG} and Section \ref{sec:unibr} around
%formula (\ref{Dedual}) (a comment on the reduction formula).
\vskip 0.2cm

Namely, the ideals $I=M^\hollowstar$ for $\r\subset \o$,
are arbitrary such that 
$\mathfrak{c}\subset I\subset \r$ for the conductor
$\mathfrak{c}$ of $\r$. The definition of the
{\sf\em reciprocity map} $M\mapsto M^\hollowstar$ remains the
same as in the unibranch (irreducible) 
case.% this is for any Gorenstein rings.

Setting $deg(I)\equal dim_{\F_q}(\r/I)$, we arrive at
the  formula:\,
{\small
\begin{align}\label{redmult}
&\h^{mot}\!=\!(qt)^\de \sum_{\mathfrak{c}\subset I\subset \r}
(qt)^{-deg(I)}
\Bigl(\frac{\,(q-1)^{\kappa-1}}{q^{\kappa-\up_M}\,
i_{\overline{M}}}\Bigr)\,
(1\!+\!\aa q)\cdots(1\!+\!\aa q^{\varrho(I)-1}).
\end{align}
}

The weights $w_M$ can be determined
in terms of $I=M^\hollowstar$. This is done in
Theorem \ref{thm:weight} in terms of the
total lifted conductor of $\r$ 
and some special ideals in
$\a$ for the pairs
of the components. 
%$\mathfrak{c}^\#\equal (\mathfrak{m}_\r)^\hollowstar=
%\{f\in loc(\r) \mid zf\in \mathfrak{c}\}$
%for the localization $loc(\r)$ of $\r$. For instance,
%$\up_M=dim_{\F_q}I^\#$ for
%$I^\#=\mathfrak{c}^\#/\bigl(I\cap \mathfrak{c}^\#\bigr)$.
%and
%$i_{\overline{M}}$ is the cardinality of the set 
%of vectors in $I^\#$ with some preimages in
%$\mathfrak{c}^\#$ with nonzero projections onto basic vectors. 
\vskip 0.2cm

In the next section, we will reformulate this formula and the
previous reduction formula, switching to the preimages of
$M^\hollowstar$ in $\F_q[[x,y]]$ and instanton sums. 
The preimages $\c$ of
the conductors $\mathfrak{c}$ in the multibranch case,
frequently  coincide with those for the corresponding
unibranch deformations. So the resulting instanton sums can be 
sometimes over the
same sets of modules upon this passage; the difference is the
usage of the {\sf\em weights} in the multibranch
case.  Let us give an example.
\vskip 0.2cm

The simplest uncolored link is the 
Hopf $2$-link $T(2,2)$,
which is for $\r=\F_q[[1=e_1+e_2,x=z_1,y=z_2]]\subset \o=
\F_q[[e_1,e_2,z_1, z_2]]$. Then, standard $M$ can be $\o$ of 
$\rho$-rank $2$ or
$(e_1+\om e_2)\r$ for $\om\in \F_q^*$. This gives
 $\mathcal{H}^{mot}_{2,2}=
(1+\aa q)+(q-1)t$. 

The conductor is $\mathfrak{c}=(z)=\o z$.
Its lift to $\a=\F_q[[x,y]$ is $\c=\mathfrak{m}_\a=x\a+y\a$,
which is the same as for $\F_q[[x=z^3,y=z^2]]$. There are
only $2$ modules $M$ such that $\r\subset \o$, which are $\r$ and
$\o$. Accordingly, $M^\hollowstar$ 
are $\r$ and $\mathfrak{m}_\a$,
which are the same as for  $\F_q[[x=z^3,y=z^2]]$. So the only
difference is the usage of weights: $qt$ is replaced by $(q-1)t$
in $\h^{mot}_{3,2}=(1+\aa q)+ qt$, since
 $\kappa=2$, $\upsilon=1$ and 
$i_{\overline{M}}=1$ for $M=\r$. 

\vskip 0.2cm
Paper \cite{ChQ} contains many examples of links, including
colored ones; see Section 7.3 there. The reduction formula can be
extended to the most general case, when 
the ranks are arbitrary for different irreducible
components (the case of the most general ASF in type $A$),
 but the corresponding
linear algebra becomes involved. The corresponding $GL$-groups
(of different ranks) must be used then
instead of $U$ and $\overline{U}$. We  omit this generalization
in this paper.

\section{\sc Instanton slices}
We will begin with the general definition
and restating the {\sf\em reduction formulas}
from (\ref{redforranks}) and (\ref{M-one}) via 
{\sf\em instanton slices} $\mathfrak{S}_\mathscr{C}$
for the {\sf\em lifted conductors} 
$\c\subset \a\equal \F[[x,y]]$ of plane curve singularities 
and their powers $\mathscr{C}=\c)\r^\cn\subset \a^\cn$ for
and ranks $\cn$.

\subsection{\bf Main results}\label{sec:mainresults}
Generally,  for any $\a$-invariant 
submodule  $\mathscr{C}\subset \a^\cn$, the corresponding
{\sf\em instanton
slice} is $\mathfrak{S}^\cn_{\mathscr{C}}\equal 
\{ \m \subset \a^\cn\, |\, \mathscr{C}\subset \m\}$, where 
$\a$-submodules $\m$ will be always assumed of finite
codimension:\ $deg(\m)\equal dim_\F(\a^\cn/\m)<\infty$.
We will call  $\sC$ {\sf\em conductors}. They 
will be mostly of finite codimension, but we will 
consider ``free theory", where  $\mathscr{C}=\{0\}$, and
some $\mathscr{C}$ of infinite
codimension. 

The rank will be denoted by $\cn$, and the 
notation 
$\mathfrak{m}_\a$ will be used for the 
maximal ideal\, $x\a+y\a$\, of $\a$.  Also,
$\varrho(\m)= dim_\F \m/\mathfrak{m}_\a \m$, the number
of generators of $\m$, and  the base field will be $\F=\F_q$. 

The {\sf\em instanton superpolynomials} are defined as follows:\,
\begin{align}\label{instH}
&\sH^{inst}_{\sC}(q,t,\aa; \cn)=\sum_{\m\in 
\mathfrak{S}^\cn_\sC} t^{deg(\m)}
(1+\aa q^\cn)\cdots(1+\aa q^{\varrho(\m)-1}).
\end{align}

Informally, the class of conductors  $\sC$  
is assumed ``perfect" if all properties (known and
expected) of superpolynomials hold for $\sH^{inst}_\sC$
(or are conjectured). The key examples so far 
are $\sC=\c^\cn$ for $\c$ corresponding to
plane curve singularities and  monomial $\c=I_\la$, provided
the superduality for the corresponding $\sH^{inst}_{\sC}$.
\vskip 0.2cm

We will  restate now the reduction theorems above 
using the corresponding instanton slices.
Let $\r\subset \o$ be the singularity ring of a unibranch plane curve
singularity, $\cn\ge 1$ any rank, and  $\c=\c_\r\subset \a$ will be 
the (full) preimage in $\a$ of the
conductor $\mathfrak{c}$ of $\r$.

\comment{
\begin{align}\label{redforranks-1}
\h^{mot}(q,t,\aa;\cn)=t^{\de\cn} q^{\de \cn^2} 
\h^{dual}(q,1/(t q^{\cn}),\aa),
\end{align}

\begin{align}\label{redmult-1}
&\h^{mot}\!=\!(qt)^\de \sum_{\mathfrak{c}\subset I\subset \r}
(qt)^{-deg(I)}
\Bigl(\frac{\,(q-1)^{\kappa-1}}{q^{\kappa-\up_M}\,
i_{\overline{M}}}\Bigr)\,
(1\!+\!\aa q)\cdots(1\!+\!\aa q^{\varrho(I)-1}),
\end{align}
}

\begin{theorem}\label{thm:mainred}
(i) Let $\r$ be an irreducible plane curve
singularity, $\c=\c_\r$ the lifted conductor of $\r$,
%$\mathfrak{c}=(z^{2\de})$ of $\r$ to $\a$, 
and $\sC=\c^\cn$ for $\cn\ge 1$.  Then 
\begin{align}\label{hmotinst}
\h_\r^{mot}(q,t,\aa;\cn)=
q^{\de \cn^2} t^{\de \cn} \sH^{inst}_\sC (q,1/(tq^\cn),\aa;\cn).
\end{align}

(ii) The same coincidence holds for $\cn=1$ and 
multibranch $\r$:\,
\centerline{
$
\h_\r^{mot}(q,t,\aa)=
\!(qt)^\de \sum_{\c\subset \i\subset \a}
(qt)^{-deg(\i)}\, w_{\i}\,
(1\!+\!\aa q)\cdots(1\!+\!\aa q^{\varrho(\i)-1}),
$
}
%and the lift of 
%$(\mathfrak{m}_\r)^\hollowstar=
%\{f\in loc(\r) \mid zf\in \mathfrak{c}\}$,
where the weights $w_{\i}$ must be taken from
(\ref{redmult}). The intrinsic construction of $w_{\i}$ 
is in Theorem \ref{thm:weight}, which is
in terms of the preimage $\i$  of $M^\hollowstar$,
and a system
of ideals in $\a$ for pairs of components. \sq
\end{theorem}

{\bf Using flags.}
 A more conceptual approach to
$\h_\r^{mot}(q,t,\aa)$
is via {\sf\em admissible} $\ell$-flags. 
As in \cite{ChP2,ChQ}:
 $\vec{M}=
\{M_0\!\subset\! M_1\!\subset \cdots \subset\! M_{\ell}
\!\subset\! \o^\cn\}$
for standard $M_i$,
where $deg(\vec{M})=deg(M_0)$. Then

\begin{align}\label{mot-flag}
&\h_\r^{mot}(q,t,\aa;\cn)=\sum_{\vec{M}_{st}}
t^{deg \vec{M}}\,\aa^{\ell(\vec{M})}.
\end{align}

%For reducible $\r$, this definition is getting
%involved, but we mostly
%consider irreducible 
%singularities in this paper.
Similarly to {\sf\em Nested
Hilbert Schemes}: $\mathfrak{m}_\r M_\ell\subset M_0$. The
definition of admissible (standard) flags is in terms
of special flags of $\De_i=\De(M_i)$, which significantly impacts
the definitions and formulas. 
%They were introduced in terms of
%special $\Delta$-flags in \cite{ChP1,ChP2}. 
See \cite{ChQ} for details and the
following application,  
the {\sf\em $2${\footnotesize nd} decomposition theorem}:

\begin{align}\label{2ndHfomula}
\h_\r^{mot}(q,t,\aa)\!=\! \sum_{M\subset \Om} t^{deg(M)}\,
\Bigl(1\!+\!\frac{\aa}{t}\Bigr) \Bigl(1\!+\!\frac{\aa}{t}q\Bigr)\cdots 
\Bigl(1\!+\!\frac{\aa}{t}q^{\varrho^\Diamond(M)-1}\Bigr),
\end{align}
where the summation is over standard $M\subset \Om=\o^\cn$,
and $\varrho^\Diamond(M)=dim_{\F_q}(M^\Diamond/M)$ for
$M^\Diamond\equal M:\mathfrak{m}_\r=\{f\in \Om \mid \mathfrak{m}_\r f
\subset M\}$.  

For instance, this formula readily
gives that $\h_\r^{mot}(q,t,\aa=-t;\cn)=1$ because
$\Om$ is the only module with $\varrho^\Diamond=0$.
Let us provide $\varrho^\Diamond$ for $\F_q[[z^4,z^3]]$.
Here $G=\{1,2,5\}$, $\de=3$. We provide 
$\Delta_i=\Delta(M_i)$ and
$(d_i,c_i)$, where $d_i=deg(M_i)$
and $\mathbb{A}^{c_i}$ are the corresponding cells $\{M_i\}$:
\[
\begin{array}{c|ccccc}
i&0&1&2&3&4\\ \hline
\De_i\setminus \Ga&\varnothing&\{5\}&\{1,5\}&\{2,5\}&\{1,2,5\}\\
(d_i,c_i)&(3,3)&(2,2)&(1,1)&(1,2)&(0,0).
\end{array}
\]

Then $\h^{mot}_{4,3}(\cn\!=\!1)\!=\!$
{
\( 1\, \lan\text{\tiny\it $\De_5$\,}\ran+
q^3t^3\,(1+\frac{\aa}{t})\,
\lan\text{\tiny\it $\De_0$\,}\ran+
\,(1+\frac{\aa}{t})\,(1+\frac{\aa q}{t})\,q^2t^2\, 
\lan\text{\tiny\it $\De_1$\,}\ran+
(q+q^2)t\,(1+\frac{\aa}{t})\,
 \lan\text{\tiny\it $\De_{2,3}$\,}\ran.
\)
}

The rank-two cells are indexed by $(i,j)$ with the degree
$d_i+d_j$. We provide $(c_{ij},r_{ij})$ 
for $r_{ij}=\varrho^\Diamond(M_{ij})$,
which are constant in the corresponding cells 
$\mathbb{A}^{c_{ij}}=\{M_{ij}\}$. They are always affine
for $\F_q[[z^\rr,z^\ss]]$ and any $\cn$, where
$gcd(\rr,\ss)=1$, which is due to \cite{ChP2}. One has:
%\[
%r_{ij}=dim_{\mathbb F_q}
% \bigl((M:\mathfrak m_R)/M\bigr).
%\]
%The degree of this cell is $d_i+d_j$.
\[
\begin{array}{c|ccccc}
i\backslash j&0&1&2&3&4\\ \hline
0&(12,2)&(10,3)&(10,2)&(9,2)&(8,1)\ \\
1&(11,3)&(8,4)&(8,3)&(6,3)&(5,2)\ \\
2&(11,2)&(9,3)&(8,2)&(6,2)&(4,1)\ \\
3&(10,2)&(7,3)&(6,2)&(4,2)&(2,1)\ \\
4&(9,1)&(7,2)&(5,1)&(3,1)&(0,0)\,.
\end{array}
\]

One can go from (\ref{2ndHfomula}) to
$\sH^{inst}_{\sC}$ as in (\ref{hmotinst}),
{\sf\em but 
$\varrho^\Diamond(M)$ must be used,
not those for the corresponding $\m\subset \a^\cn$.}
For $\m\subset \a^\cn$, 
$\varrho^\Diamond(\m)=\varrho(\m)-\cn$ if we set
$\varrho^\Diamond(\m)=dim_{\F_q}(\m^\Diamond/\m)$ for
$\m^\Diamond\equal\{f\in \a^\cn \mid \mathfrak{m}_\a f
\subset \m\}$, which is not the case for $M\subset \o^\cn$

\comment{
 $\vec{\m}=
\{\m_0\!\subset\! \m_1\!\subset \cdots \subset\! \m_{\ell}
\!\subset\! \a^\cn\}$
for $\a$-submodules  $\sC\subset \m_i\subset \a^\cn$,
where $deg(\vec{\m})=deg(\m_0)$, and $\m_0\supset 
\mathfrak{m}_\a\m_{\ell}.$ We obtain that:
\begin{align}\label{inst-flag}
&\sH_\sC^{inst}(q,t,\aa;\cn)=\sum_{\vec{\m}}
t^{deg \vec{\m}}\,\aa^{\ell(\vec{\m})}\\
\label{2nd-inst-fomula}
= \sum_{\m\subset \a^\cn}& t^{deg(\m)}\,
\Bigl(1\!+\!\frac{\aa}{t}\Bigr) \Bigl(1\!+\!\frac{\aa}{t}q\Bigr)\cdots 
\Bigl(1\!+\!\frac{\aa}{t}q^{\varrho^\Diamond(\m)-1}\Bigr),
\end{align}
where the last summation is over  $\sC\subset \m\subset \a^\cn$,
and $\varrho^\Diamond(\m)=dim_{\F_q}(\m^\Diamond/\m)$ for
$\m^\Diamond\equal\{f\in \a^\cn \mid \mathfrak{m}_\a f
\subset \m\}$. Importantly, $\varrho^\Diamond(\m)=\varrho(\m)-\cn$.
As above, the {\sf\em 2{\footnotesize nd} 
decomposition}  begins with $1$.   
}

\vskip 0.2cm
{\bf Main Conjecture.}
Part $(i)$ of the following conjecture
is based on Conjecture
9.1 from \cite{ChQ}, 
%restricted to irreducible $\r$ and,
%correspondingly, to 
%algebraic knots. 
where we will switch  from 
$\cn\om_1$ to $\om_\cn=1^\cn$ using the superduality
of the DAHA superpolynomials:\ 
\vskip 0.1cm

\centerline
{$\h^{daha}_\k(q,t,\aa; \om_\cn)=
q^{\de \cn^2} t^{\de\cn}\h^{daha}_\k(t^{-1},q^{-1},\aa; \cn \om_1).$
}

\vskip 0.2cm

{\sf\em Superduality} for $\sH^{inst}$ for $\cn=1$, needed in
$(ii)$ below, is defined as 
the invariance of
its $q,t,\aa$-expansion under the map:
\vskip 0.1cm

\centerline{
$\aa\mapsto \aa,\ 
q^i t^j\mapsto q^i t^{\de-j+i},\  \text{ where }\de=|\la|.
$
}
\noindent
Equivalently, $t^\de\sH^{inst}(qt,1/t,\aa)=\sH^{inst}(q,t,\aa)$. 
For $\mathfrak{H}_\la^{daha}$ and $\cn=1$, it becomes 
$(qt)^\de \mathfrak{H}^{daha}_\la(t^{-1},q^{-1},\aa)=
\mathfrak{H}^{daha}_\la(q,t,\aa)$, the same as for 
$\h^{daha}_\k$.  %This is only for $\cn=1$.
%See Section \ref{sec:2row} for examples
%and counterexamples.

\begin{conjecture}\label{conj:main}
(i) Let $\r$ be the ring of a unibranch plane curve singularity,
and $\k$ the corresponding algebraic knot.
Combining Conjecture 9.1 from \cite{ChQ} with
formula (\ref{hmotinst}) above:\ 
%\begin{align*}
%&\h^{daha}_\k(t^{-1},q^{-1},\aa)(\om_\cn)=
%\sH^{inst}_\sC (q,1/(tq^\cn,\aa;\cn).
%\end{align*}
\begin{align}\label{daha-sH}
&\h^{daha}_\k(q,t,\aa;\, \om_\cn)=
\sH^{inst}_\sC (t^{-1},q t^\cn,\aa;\cn), \\ 
&\sH^{inst}_\sC (q,t,\aa;\,\cn)=
\h^{daha}_\k(t q^\cn,1/q,\,\aa;\, \om_\cn).\label{daha-sH-inv}
\end{align}

(ii) Switching to $\la$ and $\sC=I_\la^\cn$,
%$\mathfrak{H}^{daha}_\la (q,t, \aa)[\cn]$, they are conjectured to be
%superdual for $\cn=1$ and any Young diagrams
%$\la$, and  $\sH^{inst}_\sC[\cn]=\sH^{inst}_\sC (q,t,\aa;\,\cn)$ 
%for $\sC=I_\la^\cn$. Then 
the superduality of  $\sH^{inst}_\sC[\cn=1]$ or, equivalently,
the $q$-positivity of $\mathfrak{H}^{daha}_\la[\cn=1]$ 
is necessary and sufficient for the following counterpart 
of (\ref{daha-sH}):
\begin{align}\label{daha-la}
&\sH^{inst}_\sC (q,t,\aa;\,\cn)=
\mathfrak{H}^{daha}_\la(t q^\cn,1/q,\,\aa)[\cn];
\text{ equivalently, for $\cn=1$:}\\ \label{daha-la-one}
&\mathfrak{H}_\la^{daha}(q,t,\aa)[1]\,=\,
\h^{inst}_\sC(q,t,\aa; 1)\,\equal\, 
(q t)^\de \,\mathscr{H}_\sC^{inst}(q, 1/(qt),\aa; 1).
\end{align}
%\vskip -0.7cm

(iii) Concerning the
equivalence in (ii),
$\sH^{inst}_\la(q,t,\aa;\cn)$ are 
conjectured polynomials in terms of $q,t,\aa$ 
with non-negative coefficients for any $\la$ and $\cn\ge 1$.  
This is manifest for $t,\aa$, but the $q$-polynomiality and
positivity are challenging.  
Accordingly, $\mathfrak{H}^{daha}_\la(q,t,\aa)[\cn=1]$
are conjectured DAHA-superdual for any Young diagram $\la$.
%the claim is that (\ref{daha-la}) and (\ref{daha-la-one})
%hold  if and only if all coefficients of
%$\mathfrak{H}^{daha}_\la[\cn=1]$ are non-negative.
\sq
\end{conjecture}

{\bf Comments.} Due to $\cn=1$,
we do not employ the superduality in 
(\ref{daha-la-one}); it matches 
(\ref{hmotinst}).
Conjecture \ref{conj:supd}
concerns  the combinatorial conditions for $\la$ 
that ensure the superduality of $\sH^{inst}_\la$ or, 
equivalently, the positivity of
$\mathfrak{H}^{daha}_\la$ for $\cn=1$, and imply the
corresponding coincidence.

The recalculation from $\sH^{inst}$ to
$\h^{inst}$ in (\ref{daha-la-one}) makes 
the instanton superduality as in DAHA:\, 
under  $q\leftrightarrow  t^{-1}$. The (old) DAHA superpolynomials
$\h^{daha}$ are known to be superdual (for iterated torus links).
\smallskip

It makes sense to provide a complete deduction for $(i)$:
\begin{align*}
&q^{\de \cn^2} t^{\de\cn}\h^{daha}_\k(t^{-1},q^{-1},\aa; \om_\cn)=
\h^{daha}_\k(q, t,\aa; \cn \om_1)\\
&=\h^{mot}_\r(q,t,\aa; \cn)
=q^{\de \cn^2} t^{\de \cn} \sH^{inst}_\sC (q,1/(tq^\cn),\aa;\cn),
\end{align*}
where the first equality is a DAHA theorem, and the last one
is Theorem \ref{thm:mainred},(i). Then we cancel coinciding 
$q^\bullet t^\bullet$ and change the variables.

%Assuming $(ii)$, the superduality for $\sH^{inst}$ is
%a simple recalculation of that for $]mathfrak{H}^{daha}$. 
\medskip

When $\la$ are associated with $T(\rr,\ss)$, then
(\ref{daha-la}) is equivalent to (\ref{daha-sH}), as well as 
for $C\!ab(\upsilon \rr\ss\pm 1, \upsilon)T(\rr,\ss)$
when $\rr>\ss>0$, $gcd(\rr,\ss)=1$.
\medskip

In this conjecture, $(i)$ is as close as possible to the 
{\sf\em Shuffle Conjecture}. Indeed,
the calculation of $\h^{daha}_\k(q,t,\aa;\, \om_\cn)$ upon the
usage of nonsymmetric Macdonald polynomials requires only
$X_\cn=X_{\om_\cn}$, the action of $PSL_2^{\wedge}(\Z)$ in DAHA, which
is very explicit, and the {\sf\em coinvariant}. 
%The latter
%is direct from the definition of the polynomial representation.
%\medskip

\vskip 0.2cm
{\bf Instanton Volume Conjecture.}
Let us state the instanton counterpart of
(\ref{voldaha}), which appeared simpler and 
more efficient.  For  $\a=F_q[[x,y]]$ and
$\varrho(\vec{\m})$ as above, 
let $\Sigma_\la(n,r)=\{\a\text{-modules } \m \mid
I_\lambda^\cn\subset M\subset \a^\cn, \varrho(\m)=r\}$,
 $\Sigma^{(\cn)}_\lambda=
1+\sum_{r>\cn}(1-q)\ldots(1-q^{r-\cn})\,|\,\Sigma(\cn,r)\,|$,
and, finally,
$\Sigma_\lambda^\intercal\equal\Sigma^{(n)}_\lambda
(q\mapsto \exp(2\pi 
\imath/N)),$ where $N\equal \cn+1$. 

\begin{theorem}\label{thm:volconj}
We assume that $\mathfrak{H}_\la^{daha}(q,q,\aa=-q^2)$ 
coincides up to $q^\bullet$ with the Jones polynomial
$J_N(K;q)$, for $K=K_\la$ colored by $\cn\om_1$. Also, 
conjectural formula
(\ref{daha-sH}) will be used, provided that $\la$ is admissible.
Then $\Sigma_\lambda^\intercal= q^\bullet J_N(K;q\!=\!
\exp(\frac{2\pi \imath}{N}))$ and the volume conjecture becomes
\begin{align}\label{instvol}
&|\Sigma_\lambda^\intercal|
\stackrel{\text{conj}}{\simeq}
 |\delta_K|^{-1/2}\, N^{3/2} \,
\exp\bigl( V(K) 
\frac{N}{2\pi}\bigr)\,\bigl|\exp\bigl(\sum_{m\ge 2}^\infty S_{K,m}
\bigr(\frac{2\pi \imath}{N}\bigl)^{m-1}\bigr)\bigr|,
\end{align}
where $\delta_K$ is the adjoint twisted Reidemeister torsion,
$V(K)$ is the hyperbolic volume of $S^3\!\setminus\!K$, 
$S_{K,m}$ is the $m$-loop
(topological) invariant, and  $K=K_\lambda$, which is for 
admissible $\lambda$.  The same identity holds without
absolute values $|\cdot|$ up to some normalization factor 
 if $V(K)$ is replaced by $V(K)- \text{\large\em\it i\,}C\!S(K)$
for the Chern-Simons invariant.
\end{theorem}
{\it Proof.} Conjecturally,
$\mathfrak{H}^{daha}_\la(q,t,\aa;\, \om_\cn)=
\sH^{inst}_\la (t^{-1},q t^\cn,\aa;\cn)$, and
$J_N(q)\sim\mathfrak{H}^{daha}(q^{-1},q^{-1},\aa=-q^2)[\cn]$
up to a power of $q$.
Combining, $J_N(q)\sim
\sH^{inst}_\la (q,q^{-\cn-1},\aa=-q^2;\cn)$,
which becomes $\sH^{inst}_\la (q,1,\aa=-q^2;\cn)$
if $q^N=1$.
Thus, \(J_N(q)=\sum_{\vec{\m}}
(1-q^\cn\aa)\cdots (1-q^{\varrho(\vec{\m})-1}\aa)=
\sum_{\vec{\m}}(1-q)\cdots(1-q^{\varrho(\vec{\m})-N-1+2})=
1+\sum_{r>\cn}(1-q)\cdots(1-q^{r-\cn})\,|\Sigma(\cn,r)|\)
when $q^N=1$. Notice that the same sum will be
obtained for $\aa=-t^m$ and $N=\cn+m-1$, i.e. for the
root system $A_{m-1}$.
\sq

%$\stackrel{\text{sometext}}{=}$

{\bf Brief comments.} We checked that 
$\Sigma_{(3,2)}^\intercal$ coincide 
with $\overline{C}_1(N)$ provided after (\ref{voldaha}) for $N=5,6$. 
The classical Volume Conjecture  is known only for very
few hyperbolic knots (mostly twisted ones).
% Dimofte and Garoufalidis constructed the 
The $S$-coefficients were calculated in \cite{DiG}  
from level-$N$ Neumann-Zagier data and finite sums of Feynman 
diagrams. See there concerning the
asymptotic equality in (\ref{instvol}), and 
Zagier's Quantum Modularity.
We use their 
normalization. For instance, $\de_K$ can be
alternatively defined via  
 the adjoint monodromy representation 
of $\pi_1(S^3\setminus K)$. It, indeed, coincides with the 
Reidemeister torsion due to Ming and Wu (2026).

\comment{
For K12n242, we successfully approach
$Vol(K)=2.828122,
|\delta_K|^{-1/2}=0.534054$, 
and also ${\rm Im}(S_{K,2})$. The convergence becomes much
better if $N^{3/2}$ and
$\delta_K=-2(3\xi-2)\xi^{-2}$ are assumed in this conjecture, where 
$\xi^3-\xi^2+1=0,\ 
 {\rm Im}(\xi)<0.$ The same holds for 
$K12n725$.
}

%\comment{
For K12n242, the Bloch--Zagier (exact) formula gives 
\begin{equation*}
 Vol(K12n242)=9\frac{23^{3/2}}{4\pi^4}\zeta_E(2)
 =2.828122088330783163\ldots\,,
\end{equation*}
where the trace field $E$ is 
\(
 \mathbb{Q}(\xi),\ 
 \xi^3-\xi^2+1=0,\ 
 {\rm Im}(\xi)<0,
\)
and $|\delta_K|^{-1/2}=0.5340539980588458\ldots$, where
$\delta_K=-2(3\xi-2)\xi^{-2}$.
We successfully approach these values, though the convergence
to $V(K)$ is slow due to $|\delta_K|^{-1/2}N^{3/2}$. 
We capture very well (numerically) $|\de_K|, \Im(S_{K,2})$. 

Assuming that the leading $N$-term is 
$\text{Const}N^{3/2}$,
the convergence to $V(K)$ can be made much faster. Generally, 
reasonable approximations are expected for $2$-row diagrams
and for hooks, including
 $K=$K12n725, where we predict that 
$|\Im S_{K,2}|=0.19585255\ldots$\,; it is $S_{2,1}$ in \cite{DiG}.

We note that our approach works for big $N$ only for these $2$
sequences of $\la$;
there are mathematical reasons for this. 
We begin with a general Burnside Lemma and some related
homological algebra, which reduces the calculation of $|\Si(\cn,r)|$
to formulas only in terms of the algebra 
$\a/\i_\la$. However, then we heavily use the simplicity
of these algebras. This results in formulas that work for 
$N\sim 1000$ and beyond, which are especially efficient for
hooks. As a byproduct, we prove that 
$|\,\Sigma_\la(\cn,r)|\in \Q[q]$ (notice $\Q$) for 
hooks, which can be presumably extended to   
arbitrary  $2$-row diagrams.  See also Appendix 
\ref{sec:inst-count}.

It appeared convenient to use $(4,1)$ for  K12n242,
and we calculated a lot with $(5,1)$, which corresponds to
 K12n6022. Both are pretzel knots:  P(-2,3,7) and  P(-2,3,9).
Generally, $(m,1)$ corresponds to P(-2,3, 2m-1).

Following Sec.~6.4 of \cite{DiG}, the trace field for
 K12n6022 is
$E=\Q(\alpha)$ for 
$\alpha^5-\alpha^4+\alpha^3+2\alpha^2-2\alpha+1=0$, and 
$\de_K=-4\alpha^4+10\alpha^3-10\alpha^2+2\alpha+14$ for the
root with $\Im(\al)<0$. 
Numerically,
$|\de_K|^{-1/2}=0.268906323397755142816081854844969\ldots $\,, and
$V(K)=4.124903251807676155505899848869378\ldots$\, (via SnapPy). 
We match them reasonably well. Our prediction is 
$|\Im S_{K,2}|
 =0.105526149690541.$

%with geometric embedding
%$\alpha=0.4425377456\ldots-0.4544788919\ldots\imath$, and the exact
%one-loop element
%\begin{equation}\label{eq:delta51}
%\end{equation}

\comment{
We will provide in the Attachment , a complete recurrence procedure
for counting $|\Si(\cn,r)|$ for K12n242, which we will represent
as $(4,1)$ and K14n6022 corresponding to $(5,1)$. Both are
pretzel knots:\, P(-2,3,7), and  P(-2,3,9). The latter is {\sf\em
partner} of the mirror of the knot $8_20$; partners share the
same volume.
}

%known, but $S_{K,2}$ was not posted).

\subsection{\bf Gr\"obner cells}
\comment{
In the right-hand side expressed via  $\h^{mot}$
for torus knots and when the coincidence is known,
the approach and explicit formulas 
from \cite{Pi} can be used.  This is only  when 
the {\sf\em Piontkowski cells} $Jac_\De$ are affine spaces;
otherwise the considerations become significantly more involved.
The instanton formulas and {\sf\em Gr\"obner} cells
 can be certainly used too, and this is the necessity when
$\sC=I_\la$ for arbitrary ``admissible" Young diagrams $\la$. 
}
Let us begin with the following example. We take $\sC=I_\la$ for
$\la=(3,3,2,1)$ from Figure 
\ref{9-diag}, which is from \cite{ChG}.  Then 
$\sH^{inst}_{\sC}(q,t,\aa)=$
{\small \(
1+t+t^2+q t^2+t^3+q t^3+q^2 t^3+t^4+q t^4+2 q^2 t^4+t^5+q t^5
+2 q^2 t^5+t^6+q t^6+2 q^2 t^6+q^3 t^6+t^7+q t^7+2 q^2 t^7+t^8+q t^8
+q^2 t^8+t^9
+\aa \bigl(q t+q t^2+q^2 t^2+q t^3+2 q^2 t^3+q^3 t^3+q t^4
+2 q^2 t^4+3 q^3 t^4+q t^5+2 q^2 t^5+4 q^3 t^5+q^4 t^5+q t^6+2 q^2 t^6
+4 q^3 t^6+2 q^4 t^6+q t^7+2 q^2 t^7+4 q^3 t^7+2 q^4 t^7+q t^8+2 q^2 t^8
+3 q^3 t^8+q^4 t^8+q t^9+q^2 t^9+q^3 t^9\bigr)
+\aa^2 \bigl(q^3 t^3+q^3 t^4+q^4 t^4
+q^3 t^5+2 q^4 t^5+q^5 t^5+q^3 t^6+2 q^4 t^6+2 q^5 t^6+q^3 t^7+2 q^4 t^7
+3 q^5 t^7+q^3 t^8+2 q^4 t^8+2 q^5 t^8+q^3 t^9+q^4 t^9+q^5 t^9\bigr)
+\aa^3 \bigl(q^6 t^6+q^6 t^7+q^6 t^8+q^6 t^9\bigr).
\) }

Note that
the coefficient of $\aa^3$ is $q^6$ times the contribution of $I_\mu$
for $\mu=\la$, and $3$ diagrams obtained by removing one of
$\square\in (3,3,2,1)\setminus (3,2,1)$. All $4$
have $4$ outer corners, which gives $\varrho=4$. 

\comment{
This superpolynomial is {\sf\em superdual}\,:\, invariant under
$\aa\mapsto \aa$, and 
$q^it^j\mapsto q^i t^{i-j+\de}$, where $\de=|\la|=9$
(in the instanton parameters $q,t$). 
\vskip 0.2cm
}

Our calculations are based on {\sf\em Gr\"obner cells}. 
Their definition depends on the {\sf Gr\"obner ordering} of
the variables $x,y$. Let $1\prec x\prec y$, which
means that $x^N$ is smaller than $y$ for any $N>0$, 
and any positive power of $x$ is greater than $1$.
Let $f^0$ be the leading (smallest) monomial in
$f$, and    $\i^0\equal \{ f^0\, \mid\, f\in \i\}$ for
any ideal $\i\subset \a$ of finite degree. Then
$\i^0$ is a monomial ideal,  and
$deg(\i^0)=deg(\i)$. 

Monomial ideals are very simple to describe.
They are $I_\la$ for
Young diagrams $\la$:\, the ideals linearly generated
by $x^i y^j$, where  $(i,j)$ are the coordinates (indices)
of \ $\yng(1)\,\in \overline{\la}=\Z_+\setminus \la$.
Recall that $i,j\ge 0$, where $i$ is the row number, and $j$ is
the column number. The monomials for the outer corners of $\la$,
the corners  of $\overline{\la}$,
are {\sf\em Gr\"obner generators} of $I_\la$. Their number is
$\varrho(I_\la)$, but this is that simple only for monomial ideals.

We set $dgrm(\i)\equal\la$ for $\la$ obtained from $\i^0=I_\la$. Then 
$\i$ is generated as an ideal by 
$f_{ij}=x^i y^j+\sum_{kl} C_{ij}^{lk} x^k y^l$,
where $(i,j)$ are outer corners of $\la$, $(i,j)\in \la$,
and $k>i$. However, the actual minimal number of generators can 
be smaller.
\vskip 0.2cm

Finally, {\sf\em Gr\"obner cells} are defined for
Young diagrams $\mu$ as follows:

\centerline{
$G_{\mu}(\sC)\equal\{ \i \,|\, \sC\subset \i, \i^0=I_\mu\}$,
\  $G^{(r)}_{\mu}(\sC)=\{ \i\in G_{\mu}(\sC) \mid
\varrho(\i)=r\}$. }

For monomial  $\sC=I_\la$, the notation will be  
$G_{\mu\subset \la}$ and 
$G^{(r)}_{\mu\subset \la}$.

\begin{figure}
{\noindent
\thicklines
%{\makebox(20,20){}}
{\makebox(20,20){}}
{\makebox(20,20){}}
{\makebox(20,20){}}
{\makebox(20,20){}}
{\makebox(20,20){j=0}}
{\makebox(20,20){j=1}}
{\makebox(20,20){j=2}}
{\makebox(20,20){j=3}}
{\makebox(20,20){}}
{\makebox(20,20){}}
{\makebox(20,20){}}
{\makebox(20,20){}}\\
\thicklines
{\makebox(20,20){}}
{\makebox(20,20){}}
{\makebox(20,20){i=0}}
{\framebox(20,20){$1$}}
{\framebox(20,20){$y$}}
{\framebox(20,20){$y^2$}}
{\makebox(20,20){$y^3$}}
{\makebox(20,20){}}
{\makebox(20,20){}}
{\makebox(20,20){}}\\
\thicklines
{\makebox(20,20){}}
{\makebox(20,20){}}
{\makebox(20,20){i=1}}
{\framebox(20,20){$x$}}
{\framebox(20,20){$xy$}}
{\framebox(20,20){$xy^2$}}
{\makebox(20,20){}}
{\makebox(20,20){}}
{\makebox(20,20){}}
{\makebox(20,20){}}\\
\thicklines
{\makebox(20,20){}}
{\makebox(20,20){i=2}}
{\framebox(20,20){$x^2$}}
{\framebox(20,20){$x^2y$}}
{\makebox(20,20){$x^2y^2$}}
{\makebox(20,20){}}
{\makebox(20,20){}}
{\makebox(20,20){}}\\
\thicklines
{\makebox(20,20){i=3}}
{\framebox(20,20){$x^3$}}
{\makebox(20,20){$x^3y$}}
{\makebox(20,20){}}
{\makebox(20,20){}}
{\makebox(20,20){}}\\
{\makebox(20,20){i=4}}
{\makebox(20,20){$x^4$}}
{\makebox(20,20){}}
{\makebox(20,20){}}
{\makebox(20,20){}}
{\makebox(20,20){}}
\vskip -0.4cm
\caption{$\la=\{3,3,2,1\}$}
\label{9-diag}
}
\vskip -0.3cm
\end{figure}

\vskip 0.2cm
For $\la=(3,3,2,1)$  as above, the cell
$G_{\mu\subset \la}$ is not an affine space only for one
 $\mu=(3,2,1)$, when $|G_{\mu\subset \la}|=2q^2-q$. It is 
a union of two $\mathbb{A}^2$ intersected at  $\mathbb{A}^1$. 
Thus, if $\aa=0$ (disregarding $\varrho$-ranks), 
we need only to know the dimensions of the remaining cells. 

Nonaffine $G_{\mu\subset \la}$ may occur for $\la$ associated
with torus knots. For instance,
there are several of them for  $\la=(4,3,2,1)$, which diagram 
is associated with $T(6,5)$. They are
of the same type as above in this case:\, 
$2\mathbb{A}^N\setminus \mathbb{A}^{N-1}.$
This is different from the {\sf\em Piontkowski cells} $Jac_\De$
for quasi-homogeneous (irreducible) plane curve singularities,
which are always affine spaces. However, {\sf\em Gr\"obner
cells} have various advantages.

\vskip 0.2cm
{\bf Free rank-one theory.}
This is when the conductor is $\{0\}$.
The corresponding {\sf\em Gr\"obner cells} are then $G_\la\equal 
\{\i \mid \i^0=I_\la\}$. They are always affine spaces of dimension
$|\la|-\ell(\la)$, where $\ell(\la)$ is the {\sf\em length} of $\la$,
the size of the $1${\footnotesize st} column, which is well-known.
This dimension formula can be obtained directly as the following 
corollary of  
Theorem 2.7 in \cite{ChG}, 
essentially due to A.~Conca and G.~Walla.

\begin{corollary}\label{cor:grob-gen}
Given an outer corner $(i,j)$, the
corresponding generator $f_{i,j}$ of $\i$ is 
$x^iy^j+\sum_{k,l}C_{ij}^{kl} x^ky^l$, where $(k,l)\in \la$ and $l>j$.
The following $C$-coefficients are independent
free parameters of $\,G_\la$; there are $|\la|-\ell(\la)$
of them. Let $(i',j')$ be the next outer
corner after $(i,j)$ (with $j'>j$). Then we keep $C_{ij}^{kl}$ only
when $(k,l+j'-j)\not\in\la$, and such $C$  give the required 
parametrization of this cell. The other coefficients
are expressed polynomially in terms of these ones. \sq
\end{corollary}

The $\varrho$-rank is not constant in $G_\la$.
Generally, the cells $G_\la^{(r)}$ and $G^{(r)}_{\mu\subset \la}$ 
can be 
complicated. At the very least, they are expected to be always
blowups of unions of affine spaces.

\vskip 0.2cm

It is of importance to calculate
$mc(\la)$, the Young diagram for the
{\sf\em monomial conductor} of $G_\la$, 
the greatest monomial ideal inside {\sf\em all}
$\i\in G_\la$. For instance, let $\la=(3,3,2,1)$, the
one from Figure \ref{9-diag}.  It is from \cite{ChG}. Then 
$mc(\la)=(3,3,2,2,2,1,1,1)$. 
\vskip 0.2cm

{\sf\em Free theory} above is of course very different from
that in the presence of conductors. The dimensions
of {\sf free cells} grow when $|\mu|$ grows, and no superduality
can be expected.  However, the free formulas do occur
in $\sH^{inst}_\sC$ from (\ref{instH}) when
$\i^0$ are sufficiently small. The condition $mc(\mu)\subset \la$
for $\sC=I_\la$ is necessary and sufficient for
$G_{\mu}\subset G_{\mu\subset \la}$.
\medskip

{\bf Adding one box.}
As an example, let $\la=(3,3,2,1)$ from 
Figure \ref{9-diag}. 
Then  $G_\la=\mathbb{A}^5$, and 
$\varrho$-ranks are  $2,3,4$. 

We set 
$gr_r=|G^{(r)}_{\la}|$. Then $gr_1=0$, $gr_2=q^5-2q^4+q^3$,
$gr_3=2q^4-q^3-q^2$, $gr_4=q^2$. The sum of all $gr_r$ is 
(of course) $q^5$. Note that $I_\la$ is not the
only module of $\varrho=4$; they form a $2$-parametric family.
%Also, the cell $G^{(2)}_{\la}$ is when we
%remove from $\mathbb{A}^5$ two  $\mathbb{A}^4$ intersected
%at $\mathbb{A}^3$.

%\medskip
Using Nakayama lemma, the total number of submodules
in  $I_{\la}$ of codimension $1$ is (generally)
$\sum_{i}\frac{q^i-1}{q-1}\, gr_i $\,. This sum is surprisingly
simple in this case:  $q^5+3 q^6$. Similar to the considerations
in \cite{NY2}, this is because there are $4=1+3$ 
extensions of $\la$ by adding $1$ box at one of the outer corners
with the following outcomes.
Adding a horizonal box does not change $G_\la$, and
addling one at the bottom multiplies its size by $q$. This is
general, and can be extended to any ranks $\cn$.
\vskip 0.2cm

\subsection{\bf Any ranks}
The key objects will be now $\a$-invariant submodules
$\m\subset \a^\cn$, where $\a=\F[[x,y]]$.
 From the view point of $\F[x,y]$, we
consider torsion free modules over $\mathbb{A}^2$ 
supported at $(0,0)$. In what follows, we will fixed the basis
$(e_1,\ldots, e_\cn)$ in $\a^\cn$.

{\bf Gr\"obner cells.} They will be for the
following ordering: 
$1\prec x\prec e_2/e_1\prec e_3/e_2\prec \cdots\prec y$.
For example, $e_1\prec x^m e_1\prec e_2, \ 
x^m e_\cn\prec y e_1\prec ye_2$ for any $m\ge 0$, and so on. 
Informally,
we put $e_i$ between $x$ and $y$. 

Monomials now are $x^iy^j e_k$, and $f^0$ is the lowest monomial
of $f\in \a^\cn$ for the order above.
 Accordingly, $\m^0$ is a monomial ideal, which
can be canonically represented as $I_{\vec{\la}}$, where
$\vec{\la}=\{\la^1,\cdots, \la^\cn\}$ (ordered), and $\la^i$ is 
the Young
diagram for an ideal $I[i]\subset \a$ of the coefficients of $e_i$.
Then, $deg(\m)\equal dim_\F( \a^\cn/\m)=|\vec{\la}|\equal
\sum_{i=1}^\cn |\la_i|$. The multi-diagram $\vec{\la}$ 
depends on the ordering, but the latter formula always holds. 
For instance, $(\la^i)$ will be different
if we put $\{e_i\}$ before $x$ or after $y$. 

We will continue
using the notation $dgrm(\m)=\vec{\la}$ if $\m^0=I_{\vec{\la}}$.
Then,  {\sf\em Gr\"obner cells} are 
$G_{\vec{\la}}=\{\m\subset \a^\cn \mid dgrm(\m)=\vec{\la}\}.$
They are affine spaces. See here and below \cite{Nek,Nak1,NY1,NY2},
 and \cite{CGM} for the parabolic Hilbert schemes.
Let us calculate the dimension of $G_{\vec{\la}}$ using
Theorem 3.2 from \cite{NY1}.

Consider the dual action of the torus 
$U=\{(u_1,u_2,\cdots, u_\cn)\}$ with  $u_i\in \F^*$ on
$(e_i)$, and the action of the torus $T=\{(t_1,t_2)\}$ 
on $x,y$: $t_1(x)=t_1x,\, t_2(y)=t_2 y$ ( trivial otherwise). 
They act in 
(smooth) $M(\cn,m)=\{\m\subset \F[x,y]^\cn\, | \,deg(\m)=m\}$,
and in its tangent spaces 
at monomial $I_{\vec{\la}}$ such that  $|\vec{\la}|=m$. 
Let us fix $\vec{\la}$. Then 
the tangent space is decomposed with respect to $1$-dimensional
 $U\times T$-modules corresponding to the terms in the
following sum: 

\Yboxdim4pt
\begin{align}\label{instsum}
\sum_{\al,\be=1}^\cn u_\al u_\be^{-1}\Bigl(\sum_{\yng(1)\in \la^\al}
t_1^{-l_\be(\yng(1))} t_2^{a_\al(\yng(1))+1}+
\sum_{\yng(1)\in \la^\be}
t_1^{l_\al(\yng(1))+1} t_2^{-a_\be(\yng(1))}\Bigr).
\end{align}
\Yboxdim5pt
Here $a(\yng(1))$, the {\sf\em arm-length},  is the number of 
boxes {\sf\em after} $\,\yng(1)\,$ until the
end of its  row in $\la$. We extend it 
to boxes outside $\la$; the general formula will be $\la_i-j-1$ if 
$(i,j)$ are the indices of $\yng(1)$\,. Recall that
$i,j\ge 0$ and the first box has indices $(0,0)$.
Thus,  it is zero if $\yng(1)$ is the last in its row 
of $\la$, and negative 
if $\yng(1)\not\in \la$. We will set $\la_k=0$ if
$k>\ell(\la)$. The
{\sf\em leg-length} is defined for the columns
instead of the rows.
\vskip 0.2cm

The terms in this sums are the corresponding characters.
Their total number is $2m\cn$ which equals 
the dimension of  $M_\cn^m$ ($\al$ and $\be$ are not
ordered and can coincide). 
Let $|\la|^\dag \equal |\la|-\ell(\la)$. 
\vskip 0.2cm

Matching  our Gr\"obner ordering, we make 

\centerline{
$t_2\gg t_1$ and $t_2\gg  u_{\cn}/u_{\cn-1}\gg\cdots 
\gg u_2/u_1\gg t_1$,
}

\noindent
and then tend  them  to $\infty$. The dimension of $G_{\vec{\la}}$
will be then the number of terms  in (\ref{instsum})
that approach $0$  in this limit.
\vskip 0.2cm

\Yboxdim4pt
Due to $t_2^{a_\al(\yng(1))+1}$,
\Yboxdim5pt
 only the $2${\footnotesize nd}
sum with $\yng(1)\in \la^\be$ there
can contribute, and only in the following cases:\,
either (1)  when  $a_\be(\yng(1))>0$, or 
(2) $a_\be(\yng(1))=0$ provided that $u_\al/u_\be\to 0$ (which
means that $\al<\be$). The $2${\footnotesize nd}  case
is due to $t_1^m u_\al/u_\be\to 0$ for $\al<\be$ and any $m>0$.

We obtain that given $\be$,
%the pair $(\be,\be)$ will contribute $|\la^\be|^\dag$, 
the pairs $(\al <\be)$ contribute $|\la^\be|$,
and the pairs $(\al\ge \be)$ contribute $|\la^\be|^\dag$.
The total contribution of $\la^\be$ will be then 
$(\be-1)|\la^\be|+(\cn -\be+1)|\la^\be|^\dag$.
The absolute total, which gives the required dimension is:
\begin{align}\label{dim-f}
dim\, G_{\vec{\la}}=\bigl(\cn &|\la^1|^\dag+(\cn-1)|\la^2|^\dag
+\cdots
+|\la^\cn|^\dag\bigr)\\
+& \bigl(|\la^{2}|+2|\la^{3}|+\cdots+
(\cn-1)|\la^\cn|\bigr)=\notag\\
\cn\bigl(|\la^1|^\dag+\cdots+|\la^\cn|^\dag\bigr)&+
\bigl(\ell(\la^{2})+2\ell(\la^{3})+\cdots+
(\cn-1)\ell(\la^\cn)\bigr).\notag
\end{align}

\vskip 0.2cm
\Yboxdim7pt

It readily gives the following product formula, which 
is basically Corollary 3.10 from \cite{NY1} (we add $\ell$).
Consider free instanton sums 
from (\ref{instH}), those  for $\sC=\{0\}$, where 
$ \mathscr{H}^{inst}_{\text{\tiny $\le 
\ell$}}(q,t,\aa; \cn)$ means that the summation is
over $\vec{\la}$ where  the length of all diagrams $\la^i$
are  $\le\ell$. Then:
\begin{align}\label{prod-free-0}
&\mathscr{H}^{inst}_{\text{\scalebox{.7}{$\le$\,}} \ell} 
(q,t,\aa=0; \cn)
=\prod_{i=1}^\ell\,\prod_{j=1}^{\mathbbm{n}}
\frac{1}{1-q^{\mathbbm{n}i-j} t^i}.
\end{align}

{\bf Explicit calculations.} The following examples are instructional.
Generally, 
$dim\, G_{\la^1,\la^2}=2|\la^1|^\dag+|\la^2|^\dag+|\la^2|$
for $\cn=2$. 
Let $\la^1=(3)=\yng(3)\,,
\la^2=(2,1)=\yng(2,1)$. Then, $dim\, G= 2\cdot 2+1+3=8$. Let
us obtain this dimension directly.

The generators of any $\m$ such that $\m^0=I_{\{\la^1,\la^2\}}$
are:\, 

$f_1=x e_1+a_1 ye_1+ a_2y^2e_1+(\al_1+\al_2 x+\al_3 y)e_2,
\ f_2=y^3 e_1,$ 

$g_1=x^2e_2+b_1ye_2+(\be_1y+\be_2y^2)e_1,\,
g_2=x y e_2+ \ga_1 y^2 e_1,\, g_3=y^2 e_2.
$

\comment{

f1=x e1+a1 y e1+ a2 y^2 e1+(al1+al2 x+al3 y) e2;
f2=y^3 e1; 
g1=x^2 e2+b1 y e2+(be1 y+be2 y^2)e1;
g2=xy e2+ ga1 y^2 e1; g3=y^2 e2;

a1 = 1; a2 = 2; al1 = 1; al2 = 2; al3 = 1; b1 = 2; 
be1 = 0; be2 = 0; ga1 = 2;
a1 = 1; a2 = 2; al1 = 0; al2 = 1; al3 = 1; b1 = 2; 
be1 = 0; be2 = 0; ga1 = 3;

(*be1 must be 0*)

xx = GroebnerBasis[{f1, f2, g1, g2, g3}, {x, y, e1, e2}];
yy=xx /. {e2^2 -> 0, e2^3 -> 0, e2^4 -> 0, e2^5 -> 0, 
e1^2 -> 0,  e1 e2 -> 0};

{e2 y^2, 0, e1 y^3, 
 e1 x + e2+ 2 e2 x + e1 y + e2 y + 2 e1 y^2, e2 x y + 2 e1 y^2, 
 e2 x^2 + 2 e2 y}

}

There are $9$ parameters here, and one relation as follows.
Considering $y g_2$ and using $f_2$, we obtain that $xy^2e_1$
belongs to $\m$. Next, 
$yg_1-xg_2-b_1g_3=(\be_1 y^2+\be_2 y^3-\ga_1 xy^2)e_1\in \m$, where
$ y^3e_1, x y^2e_1\in \m$. Hence, $\be_1 y^2\in \m$, which would
change $\la^1$ unless $\be_1=0$. This gives $8$, which was
double-checked 
by the usage of the Gr\"obner basis software. 
\vskip 0.2cm

Let us switch the order of these diagrams. Now
$\la^1=(2,1)=\yng(2,1)$, and $\la^2=(3)=\yng(3)$\,,
and the dimension must be  $7$. One has:

$f_1=x^2 e_1+a_1ye_1+(\al_1 +\al_2 y+\al_3y^2)e_2,\,
f_2=xye_1+ (\be_1y+ \be_2y^2) e_2,$

$ f_3=y^2 e_1+ \ga_1 y^2 e_2,\,  
g_1=xe_2+c_1 ye_2+c_2 y^2 e_2+  \de_1 y e_1,\, g_2=y^3e_2.$

There are $10$ parameters here. Considering $y^2g_1$, we obtain
that $xy^2e_2\in \m$. Next, $yf_2-xf_3=\be_1y^2e_2-\ga_1xy^2e_2$,
which gives that $\be_1=0$. Now, $yf_1-xf_2=a_1y^2e_1+
(\al_1y+\al_2y^2)e_2$ modulo $\m$. Using that  
$y^2e_1=f_3-\ga_1y^2e_2$, we obtain that 
$(\al_1y+(\al_2-\ga_1a_1)y^2)e_2\in \m$, which gives that $\al_1=0,\,
\al_2=\ga_1 a_1$. Thus, we have $10$ initial parameters and
$3$ relations in this case. 
\vskip 0.2cm

%Formula (\ref{dim-f}) will give the denominator of
%$\sH^{inst}$ in the
%following theorem. 

%%% e<x<y NEW PART BY CHATGPT

\vskip 0.2cm

{\bf Changing the ordering.}
Let us consider now the Gr\"obner ordering
 $ e\prec x\prec y,$ which means that
$1\prec e_2/e_1\prec\ldots,\prec e_\cn/e_{\cn-1}\prec x\prec y$.  

In terms of the
action of torus $U=\{(u_1,\cdots,u_\cn)\}$, $t_1$ and $t_2$, 
\[
t_2\gg t_1\gg u_\cn/u_{\cn-1}\gg\cdots\gg u_2/u_1.
\]
Then we tend all these parameters to $\infty$.  This is different
from the ordering $x\prec e\prec y$ used in (\ref{dim-f}); the
multi-diagram $\vec{\la}$ attached to a given module will generally
change. % when we pass from one ordering to the other. 
The cells remains affine spaces $\mathbb{A}^N$; their dimensions
are as follows.

We continue using the tangent character (\ref{instsum}).  Let
$s=(i,j)$ be a square 
\Yboxdim5pt
$\yng(1)\,$
\Yboxdim7pt
in $\la^\be$, where $i,j\geq 0$, and denote by
$l_\al(s)$ its {\sf\em extended leg-length} with 
respect to $\la^\al$.  Let
$Ends(\la^\be)$ be the set of the last squares in the rows of
$\la^\be$.  For $1\leq \al,\be\leq\cn$, set
\begin{align}\label{Nab-second}
N_{\al\be}(\vec{\la})\equal
\#\Bigl\{s\in Ends(\la^\be)\ \Bigm|\ &l_\al(s)+1<0,
\ \hbox{or}\notag\\[-2mm]
&l_\al(s)+1=0\ \hbox{ and }\al<\be\Bigr\}.
\end{align}
For $s\in\la^\be$ which is not the last square in its row, the
factor $t_2^{-a_\be(s)}$ in the second sum of (\ref{instsum}) tends
to zero for every $\al$.  This gives\, 
$\cn\!\sum_\be|\la^\be|^\dag$.  Recall that 
$|\la|^\dag \equal |\la|-\ell(\la)$. 

For $s\in Ends(\la^\be)$, its
$t_2$-power is zero.  The sign of $l_\al(s)+1$ then decides the
limit; in the equality case, it is decided by $u_\al/u_\be$, which
tends to zero exactly when $\al<\be$. Thus,
\begin{align}\label{dim-f-second}
dim\, G^{\,e\prec x\prec y}_{\vec{\la}}
=\cn\bigl(&|\la^1|^\dag+\cdots+|\la^\cn|^\dag\bigr)
+\sum_{\al,\be=1}^{\cn}N_{\al\be}(\vec{\la}).
\end{align}

For $\cn=2$, this formula can be written without 
{\sf\em leg-lengths}.  We
set missing rows equal to zero and index rows from $0$.  Then
\begin{align}\label{dim-f-second-rank2}
dim\, G^{\,e\prec x\prec y}_{\la^1,\la^2}
=2\bigl(|\la^1|^\dag+|\la^2|^\dag\bigr)+C_{12}+C_{21}, 
\end{align}
where $C_{12}=\#\{i\geq0:\ \la^2_i>\la^1_i\},\ 
C_{21}=\#\{i\geq1:\ \la^1_i>\la^2_{i-1}\}.$
\smallskip

For instance, let $\la^1=(3)=\yng(3)$ and
$\la^2=(2,1)=\yng(2,1)$, considered above as an example. Here
$|\la^1|^\dag=2$, $|\la^2|^\dag=1$.  The last square in the
second row of $\la^2$ satisfies
$l_1(s)+1=0$, and $1<2$; it is the only additional contribution.
Thus $C_{12}=1$, $C_{21}=0$, and
\[
dim\, G^{\,e\prec x\prec y}_{(3),(2,1)}=2(2+1)+1=7.
\]
Formula (\ref{dim-f}) for  $x\prec e\prec y$ gave
$8$ for the same ordered pair. 

%If we switch the diagrams, $\la^1=(2,1)$ and $\la^2=(3)$, then
%again $C_{12}=1$, $C_{21}=0$, and
%\[
%dim G^{\,e\prec x\prec y}_{(2,1),(3)}=2(1+2)+1=7.
%\]
%In contrast, formula (\ref{dim-f}) gives $7$ for this ordered pair.
%Thus the second ordering is not obtained merely by switching the
%diagrams in (\ref{dim-f}).

Since, the instanton sums do not depend on the Gr\"obner ordering,
we arrive at the following non-trivial combinatorial identity:
\begin{align}\label{sum-second}
&\mathscr H^{e\prec x\prec y}_{\infty}(q,t,0;\cn)
=\sum_{\vec{\la}}
q^{\,dim\, G^{\,e\prec x\prec y}_{\vec{\la}}}
\,t^{|\vec{\la}|}\notag\\
&=\sum_{\vec{\la}}
q^{\,\cn\sum_\be|\la^\be|^\dag+
\sum_{\al,\be}N_{\al\be}(\vec{\la})}\,
t^{\sum_\be|\la^\be|}
=\prod_{i=1}^\infty\,\prod_{j=1}^{\mathbbm{n}}
\frac{1}{1-q^{\mathbbm{n}i-j} t^i}.
\end{align}
No bound is imposed on the lengths of $\la^i$.  The finite
level $\ell$ requires one additional care.  The same module can have
different leading diagrams for the two Gr\"obner orders.  Therefore,
after changing the order, the condition $\ell(\vec\nu)\leq\ell$ must
be transported with the change of diagrams. We construct in Appendix
\ref{app:sec-Gr}
the map\, $\mathcal U$\, transforming $\vec{\la}$ that 
sends the dimension formula for
$e\prec x<y$ to that for the main $x\prec e\prec y$. 
Then the restriction of the 
 length of $\,\mathcal U(\vec{\la})$ by $\ell$ gives the 
$\ell$-product.

 We note that the ordering $x\prec y\prec e$ is quite
doable too, but it is less challenging combinatorially.

\subsection{\bf Generating functions}
We will employ
the  limits  of motivic rank-$\cn$ superpolynomials
for quasi-homogeneous singularities
$\r_{\rr,\ss}=\F_q[[x=z^\rr,y=z^\ss]]$
as  $\rr\to \infty$ and
 $gcd(\rr,\ss)=1$. Concerning
their general (motivic) theory in any ranks, see  Theorems 3.1 
and 3.2 from \cite{ChP2}.
 The 1{\footnotesize st} theorem is for $\aa=0$, and 
the 2{\footnotesize nd} is its extension to any $\aa$. 
There are many explicit formulas available, mostly
for $\h^{daha}$.

In the limit $\rr\to \infty$,
they coincide  with  {\sf\em free instanton sums}
from (\ref{instH}) for $\sC=\{0\}$
and  $\ell(\vec{\la})\le \ss-1$.
 The notation
$\mathscr{H}^{inst}_{\text{\scalebox{.7}{$\le$\,}} \ell}$ will
be used for $ \mathscr{H}^{inst}_{\text{\tiny $\ell(\vec{\la})\le 
\ell$}}(q,t,\aa; \cn)$, where 
$\ell(\vec{\la})=\max_i\{\ell(\la^i)\}$ for 
$\vec{\la}=\{\la^1,\cdots, \la^\cn\}$,
i.e. we bound  $\ell(\la^i)$ by  $\ell$ for all $i$. 
By definition,  
$\mathscr{H}^{inst}_{\text{\scalebox{.7}{$\le$\,}} \ell}$
is $Lim^{\ (\le \ell)}_{M\to\infty}\,  \sH^{inst}_{\vec{\la}_M}$, 
where this limit is over any sequences $\vec{\la}_M$ such
that $\ell(\vec{\la}_M)\le \ell$ and, given any $L$,
 almost all diagrams 
$\la_M^i$ contain the
rectangle $\ell\times L$.

Recall that we consider the sums over all
$\a$-submodules $\m\subset \a^\mathbbm{n}$, $dgrm(\m)=
\vec{\la}=\{\la^{1},\ldots,
\la^{\mathbbm{n}}\}$  for monomial module 
$\m^0=I_{\vec{\la}}$ formed by  
the leading terms of $\m$, and
$deg(\m)=|\vec{\la}|=\sum_{i=1}^\mathbbm{n} |\la^{i}|$. 
%and $\ell(\vec{\la})=\max_i \{\ell(\la^{i})\}.$

\begin{conjecture}\label{conj:stbtotor}
(i) In the notation above,
\begin{align}\label{prod-free}
&\mathscr{H}^{inst}_{\text{\scalebox{.7}{$\le$\,}} \ell} 
(q,t,\aa; \cn)
=\prod_{i=1}^ \ell\,\prod_{j=1}^ \mathbbm{n}
\frac{1+\aa q^{\mathbbm{n}(i+1)-j} t^i}{1-q^{\mathbbm{n}i-j} t^i}.
\end{align}

(ii) Let  $\ss=\ell+1$. We set
$\h^\circledast_{\rr,\ss}(q,t, \aa; \cn)
\equal 
t^{\mathbbm{n}\de}\h^{mot}_{\rr,\ss}
(q,1/(q^\mathbbm{n}t),\aa;\cn )$
for the motivic superpolynomial for 
$\r_{\rr,\ss}=\F_q[[x=z^\rr,y=z^\ss]]$, where $\rr>\ss$ and
$gcd(\rr,\ss)=1$. Then 
$\mathscr{H}^{inst}_{\text{\scalebox{.7}{$\le$\,}} \ell} =
\lim _{\rr\to \infty}\h^\circledast_{\rr,\ss}(q,t,\aa;\cn)$, and
\begin{align}\label{prod-lim}
&\lim _{\rr\to \infty}\h^\circledast_{\rr,\ss}(q,t,\aa;\cn)
=\prod_{i=1}^ \ell\,\prod_{j=1}^ \mathbbm{n}
\frac{1+\aa q^{\mathbbm{n}(i+1)-j} t^i}{1-q^{\mathbbm{n}i-j} t^i},
\end{align} 
which formula is equivalent to  (\ref{prod-free})
due to Theorem \ref{thm:mainred},(i). \sq
\end{conjecture}
Part $(i)$ was carefully checked numerically by us, the
author and ChatGPT. Conveniently, the right-hand side is a rational
function.
\vskip 0.2cm

The case of $\aa=0$ in (\ref{prod-free})
is formula (\ref{prod-free-0}). To incorporate $\aa$, 
the inductive procedure
from \cite{NY2} can be used. Namely, we
add $1$ box  to $\vec{\la}$ ``horizontally" (no change of $\ell$)
and  ``vertically" (then $\ell$ increases by $1$). The corresponding
(free) Gr\"obner
cells, including the $\varrho$-ranks, are connected, which is
outlined in the 
example right after Corollary
\ref{cor:grob-gen}. This can be generalized; we omit the details.
\vskip 0.2cm

{\bf Comment on zeros.} Let us
discuss the DAHA-meaning of the $\aa$-zeros in
the right-hand side of (\ref{prod-lim}).
We set:\, $F(q,t)^\ddag\equal
F(t q^\cn,\frac{1}{q}).$

First of all, the relation 
$\sH^{inst}(\aa=-q^{-\cn})=1$, which is direct from
the definition, obviously matches that in the product in the 
right-hand side. This a particular case of the relation
$\h_\k^{daha}(q,t,\aa=-t^{\cn};\om_n)=1$, which is 
for any cables $\k$. One has:
$\h_\k^{daha}(q,t,\aa;\om_\cn)^\ddag (\aa\!=\!-q^{-\cn})=
\h^{daha}_\k(q,t,\aa\!=\!-t^\cn;\om_\cn)^\ddag=1$. Let us employ 
the DAHA-superduality.
%which holds for any iterated torus knots. 
\vskip 0.2cm

Recall that
$\h^{daha}_\k(q,t,\aa;\cn \om_1)=(q^{\cn}t)^{\cn\de}
\h^{daha}_\k(1/t,1/q,\aa;\om_\cn)$, and
$\h^{daha}_\k(q,t,\aa;\om_n)=(t^{\cn}q)^{\cn\de}
\h^{daha}_\k(1/t,1/q,\aa;\cn\om_1)$. 
Applying $\ddag$ to the latter formula:\,
$\h^{daha}_\k(q,t,\aa;\om_\cn)^\ddag =$
$t^{\cn\de}\h^{daha}_\k(q,1/(q^\cn t),\aa;\cn \om_1).$

%In particular,
%$t^{\cn\de}\Bigl(q^\cn \frac{1}{q^{\cn} t}\Bigr)^{\cn\de}
%\h^{daha}_\k(1/t,1/q,\aa;\om_\cn)$,
%\noindent
%where the right-hand side is, indeed, $1$  at $\aa=-(1/q)^\cn$.
%\medskip

\medskip
Using the latter  in the formula 
$\h^{daha}_\k(q,t,\aa=-t;\cn\om_1)=1$,
we obtain that 
$\h^{daha}_\k(q,t,\aa;\cn\om_1)^\ddag(\aa\mapsto
-t_{old}=-q^{-\cn}t^{-1})=t^{\cn\de}$.
Since $\de\to \infty$ if $\rr\to \infty$ in $(ii)$, this gives that
$\sH^{inst}(\aa=-q^{-\cn}t^{-1})=0$ in this limit,  which 
is obvious in the
product formula too.
\vskip 0.2cm

Generalizing, 
$t^{\cn\de}\h^{daha}_\k(q,\frac{1}{q^\cn t},
\aa\!\!=\!\!-\bigl(\frac{1}{q^\cn t}\bigr)^m;
\cn\om_1)\, =\, $
 $t^{\cn\de}\hat{\j}_\k^{A_{m-1}}(q,\frac{1}{q^\cn t})$
for the hat-normalized DAHA-Jones polynomial $\hat{\j}_\k^{A_{m-1}}(q,t)$ 
for the root system $A_{m-1}$ (or $gl_m$). 
We obtain that $\aa\!\!=\!\!-\bigl(\frac{1}{q^\cn t}\bigr)^m$ are  zeros
of $\sH^{inst}$ for  $m=1,\ldots, \ss-1$,
 which matches the product formula. 

Conceptually,
one can use that $\hat{\j}_\k^{A_{m-1}}(q,t)$ are significantly
reduced in this range of  $m$ versus the corresponding 
superpolynomials, which reduction grows as $\rr$ increases. 
We need only to monitor the smallest powers of
$\,t\,$ in these polynomials, which is not too difficult.
%\smallskip

For instance, $\h^{daha}_\k(\aa=-t^m)^\ddag=t^k(\cdots)$ for $m=\ss-1$
where $k>\ep\de$ for some $\ep>0$, which depends  only on $\ss$,
which gives the required $0$ as $\rr\to \infty$. 
We use here directly  
$\{ \hga_+^k\tga_-^\ss (X_{\cn \om_1}) \}$ for $T(k\ss+1,\ss)$.
This argument proves the conjecture
for $\cn=1$; compare with \cite{HS}. Generally, a more systematic
analysis of  $\hat{\j}_\k^{A_{m-1}}(q,t)$ is needed here.

We use  the connection conjectures
$\h^{mot}(q,t,\aa; \cn)=\h^{daha}(q,t,\aa; \cn\om_1)$
in this discussion
and the DAHA-superduality. These connections and the
relation with the Galkin-St\"ohr $L$-function (by now only
for $\cn=1$) were used
in \cite{ChQ} to obtain several (up to $5$) different
decompositions of $\h$ parallel to that in terms 
of $\prod_{i}(1+q^i \aa)$. For arbitrary $\cn$,  
there is at least one more:\, in terms of $\prod_{i}(1+q^i \aa/t)$.
For $\cn=1$,, there is another one:\, 
in terms of $\prod_i (1+(q/t)^i a)$, which
goes via the $L$-functions.

Generally, we think that the direct analysis of $\h^{mot}$
similar to that in \cite{Pi} (where $\aa=0, \cn=1$) will be
possible in the limit $\rr\to \infty$. There are some formulas
of this kind in his paper, and we can use 
the formulas from \cite{ChP2} for torus
knots in any ranks $\cn$.  Also,  
the recurrence relations for families $T(k\ss+1,\ss)$ are 
doable when $k\to \infty$.

%$q\leftrightarrow 1/t$.
%They all can be used here (at least conditionally). 
%\medskip

\comment{
Thus, the direct approach to the justification of
(\ref{prod-lim}) is possible without any 
relation to DAHA superpolynomials. The latter is
known (with minor reservations) for uncolored torus knots,
but we need any ranks $\mathbbm{n}$. Also,  
the recurrence relations for families $T(k\ss+1,\ss)$ are 
doable when $k\to \infty$, which are for both, $\h^{daha}$
and $\h^{mot}$.
}
\medskip

The following theorem  ``almost" proves this conjecture.
Namely, it gives $(ii)$ modulo the coincidence of $\h^{mot}$ 
with $\h^{daha}$ for  families of torus knots with 
the lifted conductors  $\c=I_\la$  approaching
$\ell\times \infty$. 
This can be considered  proven for uncolored torus knots, but
we need any ranks $\cn$.
Equivalently, (\ref{daha-la}) %for such families $\{\la\}$
reduces it to the following theorem.

\begin{theorem}\label{thm:daha-free}
Setting $\mathfrak{H}^{daha}_\la(q,t,\aa)[\cn]^\ddag\equal
\mathfrak{H}_\la^{daha}(t q^\cn, 1/q,\, \aa)[\cn]$:
\begin{align}\label{prod-lim-daha}
&lim_{M\to\infty}^{\ (\le \ell)}\, 
\mathfrak{H}^{daha}_\la(q,t,\aa)[\cn]^\ddag
=\prod_{i=1}^ \ell\,\prod_{j=1}^ \mathbbm{n}
\frac{1+\aa q^{\mathbbm{n}(i+1)-j} t^i}{1-q^{\mathbbm{n}i-j} t^i},
\end{align}
where $lim_{M\to\infty}^{\ (\le \ell)}$ means that 
$\ell(\la)\!\le\! \ell$
and $\min\{\la_i\}\!>\! M$ for almost all such $\la$.
Assuming conjectural (\ref{daha-la}),
this formula is equivalent to  (\ref{prod-free}).
\end{theorem}
{\it Proof.} 
First, it suffices to consider in (\ref{prod-lim-daha})
$\la=\{\nu,\ldots,\nu)$  for
the rectangle diagram $\nu=\nu_M=\ell\times M$.  Using
Theorem \ref{thm:trans},  $\mathfrak{H}^{daha}_\nu=
\mathfrak{H}^{daha}_{\nu'}$ for $\nu'=\nu^{tr}=M\times \ell$.
The latter is the result of the $\aa$-stabilization 
of the hat-normalized 
$\bigl\{ Y_\cn^M (X_\cn^{\ell+1})\bigr\}$. We present:\,
$X_\cn^{\ell+1}=E_{(\ell+1)\om_\cn}+\{\,\text{\em lower terms}\,\}$,
where $E_\la$ are {\sf\em nonsymmetric Macdonald polynomials} 
and $\{\,${\em lower terms}$\,\}$ for $\la\in P_+$ 
means a sum of $E_\mu$ with some coefficients
for  $\mu_+<\la_+$ in the usual sense for dominant weights,
 where $\mu_+=W(\mu)\cap  P_+$ for the Weil group $W$ (it is unique
such).

The spectrum of $Y$-operators  at $E_\la$ is well-known; 
see \cite{C101} and formula (1.29) in \cite{ChD2}.
Generally, 
$Y_a(E_b)\,=\,q^{-(a,b+w_b^{-1}(\rho_k))}E_b$
for $a,b\in P$,  where
$w_b$ for $b\in P_+$ is the element of maximal 
length  in the 
centralizer of $b\,$ in $W$, and $q^{\rho_k}=t^\rho$.

Here $a=M\om_\cn$ and 
$P_+\ni b_+<\ell \om_\cn$. Remember that $q\mapsto q^{\cn}t$
and $t\mapsto 1/q$ under $\ddag$.  
We obtain that the eigenvalues of the lower terms vs. 
the $1${\footnotesize st} one 
have extra positive $t$-factors that tend to $0$ as $M\to \infty$.
Check that $\ell \om_\cn- w(b_+)$ must contain $p\al_\cn$ with
$p>0$ for any $w\in W$ in its decomposition in terms of simple roots. 
Thus, only $E_{(\ell+1)\om_\cn}$ contributes to the limit
(upon the $\ddag$-transformation). 

Recall that the hat-normalization, 
necessary in the definition of $\mathfrak{H}^{daha}_\la$,
is up to powers of $q,t$. In our case, only one $E_a$ contributes 
to the limit, which gives that all $M$ will give the same.
 What matters is the coinvariant 
$\bigl\{E_{(\ell+1)\om_\cn}\bigr\}$, given by the
evaluation formula for {\sf\em nonsymmetric} Macdonald polynomials
from \cite{Ch4} in type $A$. 
\smallskip

Let us provide this formula for the Young diagram $\la=L\times M=
(M,\ldots,M)$ ($L$ times), corresponding to $L\,\om_M$,
and the root system $A_{N-1}$ for $N\ge M\ge 0,\, L\ge 1 $
(this $M$ is not from the limit above):
\begin{align}\label{ev-rec}
\bigl\{ E_\la \bigr\}= E_\la(t^{-\rho})=
\frac{1}{t^{L(N-M)M/2}}\,\prod_{i=1}^{L-1}\,\prod_{j=1}^M
\frac{(1-q^i t^{j+N-M})}{(1-q^i t^j)},
\end{align} 
where empty products are $1$.
It can be immediately presented in terms of $\aa=-t^N$, but we
will do this only after the $\ddag$-transformation. 

Making here $L=\ell+1, M=\cn$, we obtain that 
\begin{align*}
&\bigl\{ E_\la\bigr\}^\ddag= 
q^{(\ell+1)(N-\cn)\cn/2}\,\prod_{i=1}^{\ell}\,\prod_{j=1}^\cn
\frac{(1-(q^\cn t)^i (1/q)^{j+N-M})}{(1-(q^\cn t)^i (1/q)^j)}\\
&=q^{(\ell+1)(N-\cn)\cn/2}\,
\prod_{i=1}^ \ell\,\prod_{j=1}^ \mathbbm{n}
\frac{1+\aa q^{\mathbbm{n}(i+1)-j} t^i}{1-q^{\mathbbm{n}i-j} t^i}
\text{ for } \aa=-1/q^N.
\end{align*}
This is (\ref{prod-free}) up to $q^{\bullet}$
(up to the hat-normalization). \sq

\vskip 0.2cm

{\bf The case of ${\bf T(2p+1,2)}$}.
Let us discuss the simplest $\ell=1$, when
$\ss=2, \rr=2\pp+1$ and $\de=\pp$. The general formula for superpolynomials
of $T(2\pp+1,2)$ was justified in \cite{CJJ} using the counterpart
of DAHA for the root system $C^\vee C_1$. It was predicted in
physics papers, and 
let us mention Habiro's
formula (around 2000)
 for $\pp\!=\!1, q=t, \aa=-t^2$. 
We  set $\mathbbm{n}=n$
for the sake of readability.
\medskip

Using
$(x;q)_n\!=\!(1\!-\!x)\cdots (1\!-\!x q^{n-1}),$
the formula reads:
\begin{align}\label{2p+1form}
\h_{2\pp+1,2}(\,&q,t,\aa; n\om_1) \,=
\frac{(q;q)_n}{(-\aa;q)_n\,(1-t)}\sum_{k=0}^n (-1)^{n-k}
(qt)^{\frac{n-k}{2}}\notag\\
\times\!\Bigl(
(q^{\frac{n(n\!+\!1)}{2}-\frac{k(k\!+\!1)}{2}})
&\bigl(\!\frac{\hbox{\small\em t}}{\hbox{\small\em q}}\bigr)
^{\!\!\frac{n\!-\!k}{2}}\Bigr)^{2\pp\!+1}
\,\frac{(t;q)_k\,(-\aa;q)_{n\!+\!k}\,
(-\aa/t;q)_{n\!-\!k}\,(1\!-\!q^{2k}t) }
{(q;q)_k(qt;q)_{n\!+\!k}\,(q;q)_{n\!-\!k}}\notag\\
&=\frac{(q;q)_n}{(-\aa;q)_n\,(1-t)}\sum_{k=0}^n (-1)^{n-k}
q^{\pp(n^2-k^2)}t^{(n-k)(\pp+1)}
\\
\times\!\Bigl(
q^{\frac{n(n\!+\!1)}{2}-\frac{k(k\!+\!1)}{2}}
&
\Bigr)
\,\frac{(t;q)_k\,(-\aa;q)_{n\!+\!k}\,
(-\aa/t;q)_{n\!-\!k}\,(1\!-\!q^{2k}t) }
{(q;q)_k(qt;q)_{n\!+\!k}\,(q;q)_{n\!-\!k}}.\notag
%\h_{3,2}(n\om_1)\! =\!\!
%\sum_{k=0}^n q^{nk}&t^k \frac{(q;q)_n(-\aa/t;q)_k}
%{(q;q)_k(q;q)_{n-k}},\ (x;q)_n\!=\!(1\!-\!x)\cdots 
%(1\!-\!x q^{n-1}).\notag
\end{align}
Here  $\bigl(t^{(n-k)(\pp+1)}\bigr)^\dag=
t^{n\pp}(q^n t)^{-(n-k)(\pp+1)}=q^{-n(n-k)(\pp+1)}t^{k\pp-n+k}$.
Therefore, only $k=0$ contributes to the limit, and this
sum reduces to:
\begin{align*}
&\frac{(q;q)_n(-\aa;q)_{n}\,
(-\aa/t;q)_{n}\,(1\!-\!t)}{(-\aa;q)_n\,(1-t)(qt;q)_{n}\,(q;q)_{n}}(-1)^{n}
q^{\pp\,n^2}t^{n(\pp+1)} 
 \times\!
q^{\frac{n(n\!+\!1)}{2}}\\
&\,\,\\
&=\frac{\,
(-\aa/t;q)_{n} }
{(qt;q)_{n} }(-1)^{n} q^{\pp\,n^2}t^{n(\pp+1)} 
 \times\!
q^{\frac{n(n\!+\!1)}{2}}\equal \Pi\,.
\end{align*}

Applying $\dag$ to the latter, we obtain (\ref{prod-lim}) 
(for $\ell=\ss-1=1)$:
\begin{align}\label{lim2p+1}
&\Pi^\dag=t^{\pp n} \frac{\,
(-\aa(tq^n);q)_{n} }
{(\,q/(tq^n);q\,)_{n} }(-1)^{n} q^{\pp\,n^2}(tq^n)^{-n(\pp+1)} 
 \times\!
q^{\frac{n(n\!+\!1)}{2}}\\
= &\frac{\,
(-\aa(tq^n);q)_{n} }
{(t^{-1}q^{1-n};q)_{n} }(-1)^{n} t^{-n} 
 \times\!
q^{-\frac{n(n\!-\!1)}{2}}=\prod_{i=1}^1\,\prod_{j=1}^ n
\frac{1+\aa q^{n(i+1)-j} t^i}{1-q^{ni-j} t^i},
\end{align}
where\ \   
$ q^{\frac{n(n\!-\!1)}{2}} (-1)^n\, t^n\, (\,t^{-1}q^{1-n};q\,)_{n}\ =\ 
(-1)^n
q^{\frac{n(n\!-\!1)}{2}}\, t^n\, (1-t^{-1}q^{1-n})$ $\times
(1-t^{-1}q^{2-n})\cdots
(1-t^{-1})=(1-t)(1-tq)\cdots (1-t q^{n-1}).
$
%\eject

\comment{ %%% T(3,2)
\vskip 0.2cm
{\bf The case of T(3,2).}
Let us show how formulas above can be obtained using
the passage to instanton slices. We will consider only 
$\r=\F_q[[x=z^3,y=z^2]]$ for any rank $n$. The case  of
any $T(2\pp+1,2)$ is quite doable as well. The 
calculations are especially
simple for $T(3,2)$ because $\c=\mathfrak{m}_\a=x\a+y\a$ and the
modules $\m$ such that $\mathfrak{m}_\a^n\subset \m\subset \a^n$
are one-to-one with $\F_q$-subspaces $V\subset \F_q^n$, the
images of $\m$. Let $codim(V)=dim_{\F_q}(\F_q^n/V)$. Then
$\varrho(\m)=n+codim(V)$, and the number of  $V$ for a fixed
$k=codim(V)$ is {\tiny $\qbin{n}{n-k}$,} which gives 
formula $\sH^{inst}_n=\sum_{k=0}^n t^k\,\text{\tiny$\qbin{n}{n-k}$}
\,(-\aa q^n;q)_k$. Finally, 
$\h^{mot}_n(q,t,\aa)=(q^n t)^n\sH^{inst}_n(q,1/(q^n t),\aa)=
\sum_{k=0}^n (q^n t)^k\,\text{\tiny$\qbin{n}{k}$}
\,(-\aa q^n;q)_k$ in the case of $T(3,2)$,
 which is formula right above (6.49) in \cite{ChQ}. 
Formula (6.49)
is the super-Habiro one (see there).
} %%%T(3,2)

\subsection{\bf Gorenstein rings and beyond}
A natural setting 
for Theorem \ref{thm:mainred} and its justification (the
reduction formulas) is as follows. Let $\b$ be a local
(complete) ring of {\sf\em any} isolated singularity (no restriction
here), and $\r\subset \o=\F[[z]]$ be
a Gorenstein ring  presented as a quotient of $\b$,
provided, as always, that the localization $loc(\r)$ is $\F((z))$, 

\begin{theorem}\label{thm:gen-sing}
We define $\h^{mot}_\r$ using
the same formula, but now 
\begin{align}\label{inst-flag-tilde}
&{\sH}^{\checkmark inst}_{\c^\cn}(q,t,\aa;\cn)
\!=\!\!\! \sum_{\m}\! t^{deg(\m)}
\Bigl(1\!+\!q^{\cn}\aa\Bigr)\cdots 
\Bigl(1\!+\!q^{\varrho^\checkmark(\m)-1}\aa\Bigr),
\end{align}
where $\c^\cn\subset \m\subset\b^\cn$, and now
$\varrho^\checkmark(\m)=dim_{\F_q}(\m^\heartsuit/\m)+\cn$ for
$\m^\heartsuit\equal\{f\in \b^\cn \mid \mathfrak{m}_\b f
\subset \m\}$, and $\c$ is the lift of 
$\mathfrak{c}$ to $\b$.
Then :\,
\begin{align}\label{hmotinst-gen}
\h_\r^{mot}(q,t,\aa;\cn)=
q^{\de \cn^2} t^{\de \cn} \sH^{\checkmark inst}_{\c^\cn} (q,1/(tq^\cn),\aa;\cn).
\end{align}
\vskip -0.5cm \sq
\end{theorem}

Compare $\m^\heartsuit=Soc(\m)$ with
$M^\Diamond$ for $M\subset \o^\cn$
in  formula (\ref{2ndHfomula}). The justification of
the usage of $\varrho^{\checkmark}(\m)$ 
is as follows.

%The last step is employing the {\sf\em reciprocity map},
%which is now 
%$M^\hollowstar\equal \{\psi\in \F_q((z))^\cn \mid (\!(\psi,\phi)\!)\in \r 
%\text{ for any } \phi\in M \}$, where we use the standard 
%{\sf\em dot product} $(\!( \sum_i a_i\vep_, \sum_i b_i\vep_i)\!)=
%\sum_ia_ib_i$.

\begin{lemma}\label{lem:socle-rank}
For Gorenstein $\r$, 
let $\r^{\cn}\subset M\subset \o^\cn$ be an $\r$-submodule, and 
\(
 J=M^{\hollowstar}
\)
be defined as in (\ref{recipcn}).
%where we use the  
%{\sf\em dot product} $(\!( \sum_i a_i\vep_, \sum_i b_i\vep_i)\!)=
%\sum_ia_ib_i$.
Then $\mathfrak{c}^\cn\subset J\subset\r^\cn$ and
\begin{equation}\label{rank-socle-identity}
 \varrho(M)=\cn+dim_{\F_q} Soc(\r/J).
\end{equation}
If $\j$ is the lift of $J$ to $\b$, then 
$\varrho(M)=\varrho^\checkmark(\j)$.
\end{lemma}

{\it Proof.} Let $\cn=1$ for the sake of simplicity.
Since $\r$ is $1$-dimensional and Gorenstein, 
its canonical module is
$\r$ and every torsion-free rank-one module is reflexive.
Applying Gorenstein duality to
$0\to J\to\r\to\r/J\to0$, the quotient of
$J^{\hollowstar}=M$ by its distinguished generator $1_M$ can be identified
with the {\sf\em Matlis dual}
of $\r/J$. This gives \eqref{rank-socle-identity}:\,
 the minimal generators of this quotient
modulo $\mathfrak{m}_{\r}$
are dual to $Soc(\r/J)$, and the generator $1_M$ contributes 
one. Finally, the isomorphism 
\(
 \b/\j\simeq\r/J
\) is 
compatible with the actions of the respective maximal ideals.  Their
socles are therefore identical. \sq
\vskip 0.2cm

%Concerning $\varrho^\Diamond(\m)$, cf.
Suitable rings $\b$ here can be
surface cyclic quotient singularities; see \cite{Reid}. Generally,
the monomial structure is needed for Gr\"obner theory.
The simplest is $\b=\F[[u=x^2,w=xy, v=y^2]]$, where $uv=w^2$ (the case
of the rational double point), which is the ring of invariants of 
$\a=\F[[x,y]]$
under the involution  $x\mapsto -x, y\mapsto -y$. If  $\cn=1$,
the monomial ideals 
are $\b\cap I_\la$ for {\sf\em even}
Young diagrams,\, those with {\sf\em even} outer corners $(i,j)$.
The latter means that $i+j$ are even; recall that $i,j\ge 0$. 
In this presentation,
the degree of $deg(\i)$ is the number of {\sf\em even}
 boxes in $\la$ such that 
$\i^0=\b\cap I_\la$ for (even) $\la$, where  $\i^0$ are generated by
the leading terms of $f\in \i\subset \b$ (with respect to $x,y$). 
\medskip

Presenting instanton sums for isolated  surface singularities
via curve singularities 
 can be a valuable tool. There is a significant progress here
for the $ADE$-singularities. 
but only for their Euler characteristic (when $q=1$);
See \cite{To,GNS,Nak}. Adding $q$ remains quite a challenge.
\smallskip

Let $\r=\F[[u=z^4,w=z^5,v=z^6]]$. Then $\Ga=\{0,4,5,6,8,9,10,\ldots\}$,
$\de=4, \mathfrak{c}=(z^8)$, and $\c=I_\la$ for the even diagram
$\la=(4,3,2,1)$.
The number of even boxes $(i,j)\in \la$
equals  $\de$. All 
{\sf\em Gr\"obner cells} are affine spaces $\mathbb{A}^{dim}$
in this case. 
\smallskip

The corresponding  even(!) $\mu\subset \la$,  $deg$ and 
$dim$ are as follows:  

(1) $\mu=\emptyset,\ \ \ \ \ \  deg=0, dim=0$;

 (2) $(4,3,2,1), \ deg=4, dim=0$; \  (3)  $(4,1),\  deg=2, dim=2$; 

(4) $(2,1),\  \ \ \ \ \ \, deg=1, dim=0$; \   (5) $(2,2),\  deg=2, dim=1$; 

(6) $(2,1,1,1),\  deg=2, dim=0$; \ (7) $(4,3),\  deg=3, dim=2$; 

(8) $(4,1,1,1),\  deg=3, dim=1$; \ (9) $(2,2,2,1),\  deg=3, dim=0$. 

\smallskip
Accordingly, $\sH^{inst}(\aa=0)= 1+t^4+q^2t^2+t+ qt^2+t^2+q^2t^3+qt^3+t^3$.
Then $\h^{inst}$, $\sH$ recalculated to $\h$,
is obtained by $q^i t^j\mapsto 
q^{\de+i-j} t^{\de-j}$, which gives:\  $q^4 t^4+ 1+q^4 t^2+q^3t^3+q^3t^2+
q^2t^2+q^3t+q^2t+qt$, where we follow the list above.
  Finally,  $\h^{inst}=
1+qt+q^2t+q^2t^2+q^3t+ q^3t^2+q^3t^3+q^4t^2+ q^4t^4$.
It coincides with the corresponding $\h_\r^{mot}(\aa=0)$, but
$(qt)^4\, \h(1/t,1/q,\aa\!=\!0)\neq \h(q,t,\aa\!=\!0).$
It is supposed to be close to the DAHA superpolynomial $\h^{daha}_\k$
of $\k=C\!ab(5,2)T(3,2)$, which is superdual. The following two terms
from $\h$ are beyond $\h^{daha}_\k$ : $q^3t+q^4t^2$. The term $q^4 t^2$
destroys the superduality. The corresponding $\mu$-diagrams are
$\yng(4,3)\,$ and $\yng(4,1)\,$. One can redefine 
$\sH^{inst}$ to ``eliminate" them for some families of diagrams, but this
is unknown generally. Unwanted cells in basic families
have some geometric meaning.

%\medskip
The whole uncolored $\h^{daha}_\k$ for this $\k$ is 
{\small
\(1 - \aa^3 q^5/t + q t + q^2 t + q^2 t^2 + q^3 t^2 + q^3 t^3 
+ q^4 t^4 + 
 \aa^2 \bigl(q^3 - q^5 - q^4/t + q^4 t\bigr) + 
 \aa \bigl(q + q^2 - q^4 + q^2 t + 2 q^3 t + q^3 t^2 
+ q^4 t^2 + q^4 t^3\bigr).
\)
}
See Propositions 4.2, 4.3 from \cite{ChW} and formula
(\ref{H-m-full}) in this paper. The borderline case
is $\k=C\!ab(11,2)T(3,2)$; the corresponding $\h^{daha}$
is interpreted as $\sH^{inst}_\la$ in this paper for
 $\la=(3,2,1,1)$, upon the recalculation to $\h^{inst}$.
It coincides with  $\h_\k^{daha}=$
{\small 
\(
1 + q t + q^2 t + q^3 t + q^2 t^2 + q^3 t^2 + 2 q^4 t^2 + q^3 t^3 + q^4 t^3 + 
 2 q^5 t^3 + q^4 t^4 + q^5 t^4 + q^6 t^4 + q^5 t^5 + q^6 t^5 + q^6 t^6 + 
 q^7 t^7 + \aa^3 \bigl(q^6 + q^7 t\bigr) + 
\aa^2 \bigl(q^3 + q^4 + q^5 + q^4 t + 2 q^5 t + 2 q^6 t + q^5 t^2 
+ 2 q^6 t^2 + q^7 t^2 + q^6 t^3 + q^7 t^3 + q^7 t^4\bigr) + 
 \aa \bigl(q + q^2 + q^3 + q^2 t + 2 q^3 t + 3 q^4 t + q^5 t + q^3 t^2 + 2 q^4 t^2 + 
   4 q^5 t^2 + q^6 t^2 + q^4 t^3 + 2 q^5 t^3 + 3 q^6 t^3 + q^5 t^4 + 
   2 q^6 t^4 + q^7 t^4 + q^6 t^5 + q^7 t^5 + q^7 t^6\bigr).
\)
}

\vskip 0.2cm
However, $\h^{mot}$  for (Gorenstein)
$\r=\F[[z^4,z^6,z^{11}]]$ contains extra unwanted terms
and is not superdual. Interestingly such terms are
 automatically eliminated upon the
Floer-type substitution $\aa=-t/q$.
%Thus, the problem of ``natural" superdual reductions of
% $\h^{mot}$  remains unsolved for Gorenstein $\r.$

\vskip 0.2cm
{\bf Substitution \mathversion{bold}{$\aa=-t/q$}.}
Conjecture 4.7 from \cite{ChW} states that
$\h^{mot}(q,t,\aa)=L(\frac{q}{t},t,\aa)$ in rank $\cn=1$
for the extended {\sf\em Galkin-St\"ohr} $L$-function
$L(q,t,\aa)=(1-t)Z(q,t,\aa)$ for the corresponding zeta.
Following \cite{ChQ}, Section 8.1:
\begin{align}\label{Z-exp}
&Z(q,t,\aa)=\!\!\sum_{I\subset \r} 
t^{\hbox{\tiny \,dim}(\r/I)}(1\!+\!\aa q)\cdots
(1\!+\!\aa q^{\varrho(I)-1}),
\end{align} 
summed over all ideals $ I\subset \r$,
where $\varrho(I)=dim_{\F_q}I/\mathfrak{m}_\r I.$ 
Moreover, this connection was expected to hold for 
any Gorenstein $\r$ upon $\aa=-t/q$, when $Z$ becomes
the {\sf\em Z\'u{\oldt{n}}iga 
zeta function}:\, for the summation in $Z$  over only
{\em principle} 
ideals $f\, \r\subset\r$; see \cite{Sto,Zu,ChS,ChQ}.
Then $dim_{\F_q}\bigl(\r/(f\,\r)\bigr)=
\nu_z(f)$, and the corresponding cells are
explicit.

When $q\to 1$ (for the ``field" with $1$ element):\,
$Z_{\hbox{\tiny prncpl}}=\sum_{\nu\in \Ga}t^\nu$, and
$L_{\hbox{\tiny prncpl}}=(1-t)Z_{\hbox{\tiny prncpl}}$
is essentially the {\em  Alexander polynomial}. 

 Following \cite{Sto}
and Section 6.3 from \cite{ChS}, let $G=\Z_+\!\setminus\! \Ga$,
$v(x)=|\{\nu\!\in\! \Ga \mid 0\!\le\! \nu\!\le\! x\}|$,
$g(x)=|\{g\!\in\! G \mid 0\!<\! g\! <\! x\}|$, and 
\begin{align}\label{rho-zun}
&L(\frac{q}{t},t,\aa=-\frac{t}{q})=q^{\de}+(1\!-\!q)(1\!-\!t)\rho(q,t) 
\text{ \, \,for }
\\ \notag
&\rho(q,t)\equal\sum_{x\in G}q^{g(x)}\frac{1\!-\!t^{v(x)}}{1-t}=
\sum_{G\ni x>y\in \Ga}q^{g(x)}t^{v(y)-1}.
\end{align}

We note the superduality $\rho(q,t)=(q t)^{\de-1}
\rho(t^{-1},q^{-1})$. To justify it, let 
 $\La_\Ga=\bigl(g_{i}\!-\!i\!+\!1\mid i=\de, \de\!-\!1,\ldots, 1\bigr)$,
where 
$G=\{1\le g_1<g_2<\cdots<g_{\de}=2\de-1\}$. It is self-conjugate due
to the defining symmetry of Gorenstein $\Ga$. Then
$\La_\Ga=\La_\Ga^{tr}$  results in the superduality.

Generally, the substitution $\aa=-t/q$ (in our parameters)
is related to the {\sf\em Floer knot homology}. For instance,
$\rho(q,t)$ is a refinement of $\rho(1,1)=\la(S^3_{+1}(\k))$,
the {\sf\em Casson invariant}
of the $+1$-surgery on  the corresponding knot $K$, with a direct
link to instanton Floer homology.

\begin{conjecture}\label{conj:groeb}
For any Gorenstein ring $\r\subset \o$, let 
$\h^{mot}_\r$ and $\sH^{\checkmark inst}_{\c}$ be as above for  $\cn=1$
and the lifted conductor $\c$ of $\r$. For the knot $\k$ 
associated with $\Gamma_\r$ (below) and the
DAHA superpolynomial $\h_\k^{daha}$:\, 
\begin{align}
&\ \h_\k^{daha}(q,t,\aa=-\frac{t}{q})=
\h^{mot}_\r(q,t,\aa=-\frac{t}{q})\notag\\
=(qt)^\de &\sH^{\checkmark inst}_\c(q,\frac{1}{qt},\,\aa\!=\!-\frac{t}{q})
=q^{\de}+(1\!-\!q)(1\!-\!t)\rho(q,t).\label{inst-zu}
\end{align}
All of them are superdual:\, satisfy\, 
$(qt)^\de H(t^{-1},q^{-1})=H(q,t)$. \sq
\end{conjecture}

Here we use the generalization  of the passage from $\Ga_\r$ 
for plane curve singularities to
the corresponding algebraic cables. We need it 
for arbitrary Gorenstein $\r$.
Let us provide this construction.
\vskip 0.2cm

{\bf Formal cables of a valuation semigroups.}
Given Gorenstein $\r$, we present the generators of
its valuation semigroup  $\Ga$ as follows:
\begin{align}\label{ordered-semigroup}
&\Ga=\lan0,\beta_0,\beta_1,\ldots,\beta_g\ran,\ \
 e_i=\gcd(\beta_0,\ldots,\beta_i),\\
\label{gcd-cabling-data}
&\rr_i=e_{i-1}/e_i, 
 \ \aa_i=\beta_i/e_i, \text{ where  } 1 \leq i\leq g.
\end{align}
We assume $e_g=1$; factors $\rr_i=1$ are allowed.  Then, equivalently,
\begin{equation}\label{formal-semigroup-shape}
 \Ga=\lan0,\rr_1\cdots\rr_g,
 \aa_1\rr_2\cdots\rr_g,
 \aa_2\rr_3\cdots\rr_g,\ldots,\aa_{g-1}\rr_g,\aa_g\ran.
\end{equation}
Define the formal Newton pairs by
\begin{equation}\label{formal-Newton-pairs}
 (\rr_1,\ss_1)=(\rr_1,\aa_1),
 \ 
 \ss_i=\aa_i-\rr_{i-1}\rr_i\aa_{i-1}\  (i\geq2).
\end{equation}
One has $\gcd(\rr_i,\ss_i)=1$.  Then the associated iterated cable is
\begin{equation}\label{formal-cable}
 \k_{\Ga}\equal 
 C\!ab(\aa_g,\rr_g)\cdots C\!ab(\aa_2,\rr_2)T(\aa_1,\rr_1).
\end{equation}

This construction is called formal because some $\ss_i$ may be
zero or negative, which does not happen for
the Newton pairs of plane curve singularities.
Note that the {\sf\em cabling order} matters.

Then we lift the columns $(\rr_i,\ss_i)^{tr}$ to $\ga_i\in PSL_2(\Z)$
and to $\hga_{\rr,\ss}$, the products of the corresponding 
$\tau_{\pm}$ possibly with negative exponents. Finally, $\h_\k^{daha}$
is defined following  (\ref{jones-ditx}) and 
(\ref{jones-sup-hat}).

\vskip 0.2cm

{\bf An example with 5 generators.}
\comment{
For
\(
 \r=\F_q[[z^6,z^8,z^{11}]],\ \de=11,
\)
the cabling order $(6,8,11)$ gives
\[
 (e_0,e_1,e_2)=(6,2,1),\qquad
 (\rr_1,\rr_2)=(3,2),\qquad
 (\aa_1,\aa_2)=(4,11),
\]
and therefore
\[
 (\rr_1,\ss_1)=(3,4),\qquad
 (\rr_2,\ss_2)=(2,-13),\qquad
 \k_{\Ga}=\Cab(11,2)T(4,3).
\]
The common specialization calculated independently on the motivic and
DAHA sides is
\begin{align}\label{6811-L-polynomial}
 L_{6,8,11}(q,t)={}&1-t+q^5t-q^5t^2+q^6t^2-q^6t^3
 +q^8t^3-q^8t^5\notag\\
 &{}+q^9t^5-q^9t^6+q^{10}t^6-q^{10}t^{11}+q^{11}t^{11}.
\end{align}
It satisfies \eqref{specialized-superduality}.  In these calculations,
$\aa=-t/q$ was the unique tested specialization producing both the
coincidence and the duality.  The comparison at $\aa=0$ failed, and so
did the further ruling specialization $q=t$ of the $\aa=0$ polynomial.
Thus Conjecture \ref{conj:Gorenstein-cable} is deliberately a statement
at $\aa=-t/q$, not a conjecture that the complete triply graded
polynomials coincide.
}
For \(
 \Ga=\lan0,8,12,10,11\ran
\)\,:\,
\[
 (e_0,e_1,e_2,e_3)=(8,4,2,1),\ 
 (\rr_1,\rr_2,\rr_3)=(2,2,2),\ 
 (\aa_1,\aa_2,\aa_3)=(3,5,11),
\]
\begin{equation*}%\label{8121011-pairs}
 (\rr_i,\ss_i)= (2,3),(2,-7),(2,-9),
 \ 
 \k_{\Ga}=C\!ab(11,2)C\!ab(5,2)T(3,2).
\end{equation*}
Then, $\hga_1=\tau_-\tau_+\tau_-$ (or $\tau_+\tau_-^2$
can be taken), $\hga_2=\tau_-^{-4}\tau_+\tau_-$, 
$\hga_3= \hga_2=\tau_-^{-5}\tau_+\tau_-$, and we (fully)  
calculated 
$(qt)^\de \sH^{\checkmark inst}_\c(q,\frac{1}{qt},\aa)$ 
in this case. Its specialization at $\aa=-t/q$ does
coincide with $q^{\de}+(1\!-\!q)(1\!-\!t)\rho(q,t)=$

{\small
\(
1 - t + q^7 t - q^7 t^2 + q^8 t^2 - q^8 t^5 + q^{11} t^5 - 
q^{11} t^6 + 
 q^{12} t^6 - q^{12} t^{13} + q^{13} t^{13}.
\)
}

This is an involved example. With $4$ and $3$ generators,
Conjecture \ref{conj:groeb} was checked for many
$\r$. Also, we checked it for quite a few $\la$, 
standard Young diagrams
and even ones.
%for  $\sH^{inst}_\la$ and $\mathfrak{H}_\la$.
Let us discuss some examples.
\vskip 0.2cm

{\bf Alexander polynomials.} We always try to go beyond cables 
in this paper. Let $\b=\F_q[[u,v,w]]$ for $uv=w^2$, the case
of rational double point. First of all, the slices $\mathfrak{S}_\la$
for  $2$-row partition $\la=(b,a)$ can be naturally 
identified with the corresponding {\sf\em even}
slices for $\b$ for 
(even) $\la^{ev}=(2b, 2a+1)$. This is simply because the
first generator of the latter is $u=x^2$ and 
$\b/(x^2=u, x y^{2a+1}=wv^a, y^{2b}=v^b)$ is isomorphic to 
$\a/(x,xy^a,y^b)$,
For example, $\la=(3,2)$ gives $\la^{ev}=(6,5)$.

Thus, no new example can be expected for even $2$-row diagrams.
Let us consider the even diagram $\la=(6,3,2,1)$.
Only the specializations 
$\aa\mapsto -\frac{t}{q}$ of the superpolynomial 
$\h^{\checkmark inst}(q,t,\aa)=
( q t)^\de \sH^{\checkmark inst}(1, 1/(qt),\aa)$ 
can be expected superdual. We follow Conjecture 
\ref{conj:groeb}.
One has:
 \begin{align*}
 \h^{\checkmark inst}_{(6,3,2,1)}(q,t,\aa\!=\!-t/q)=
1-t+q^3t-q^3t^2+q^4t^2-q^4t^5+q^5t^5,
\end{align*}
which is superdual. By the way, the corresponding $\rho(q,t)$
has the form (\ref{rho-zun}) for {\sf\em sets}
$\Ga=\{0,4,6,7,8\}$ and $G=\{1,2,3,5,9\}$, not associated
with any $\r$ (just sets!). The corresponding 
{\sf\em Alexander polynomial}, is given by the general 
formula $ \Delta(t)=t^{-\de}\h(t,t,\aa=-1)$. In this case,
it is formal (of course): 
$\Delta_{(6,3,2,1)}(t)=t^5-t^4+t-1+t^{-1}-t^{-4}+t^{-5}.$

No knot with this Alexander polynomial was identified in our
searches, including the database of all knots with
up to $16$ crossings. 
%The comparison was made in the symmetric 
%Laurent format, not merely by comparing 
%coefficient lists in an unspecified normalization.
The same searches readily found the corresponding knots
for the following $3$ (standard) $\la$ of the
same degree $5$, and fail with the last one.
\vskip 0.2cm

{\small\mathversion{bold}$\la=(5)$} (found for $T(11,2)$):

$\h_{(5)}=1+qt+q^2t^2+q^3t^3+q^4t^4+q^5t^5
+\aa\bigl(q+q^2t+q^3t^2+q^4t^3+q^5t^4\bigr)$,

$\Delta_{(5)}(t)={}t^5-t^4+t^3-t^2+t-1+t^{-1}-t^{-2}+t^{-3}
-t^{-4}+t^{-5}.$
\vskip 0.2cm

{\small\mathversion{bold}$\la=(3,2)$} (found for $K12n242$):

\(
 \h_{(3,2)}=1+qt+q^2t+q^2t^2+q^3t^2+q^3t^3+q^4t^3+q^4t^4+q^5t^5\\
 +\aa^2\bigl(q^3+q^4t+q^5t^2\bigr)
 +\aa\bigl(q+q^2+q^2t+2q^3t+q^3t^2+2q^4t^2
       +q^4t^3+q^5t^3+q^5t^4\bigr).\)

$\Delta_{(3,2)}(t)=t^5-t^4+t^2-t+1-t^{-1}+t^{-2}-t^{-4}+t^{-5}.$

\vskip 0.2cm

{\small\mathversion{bold}$\la=(3,1,1)$} (found for $K12n725$):

\( \h_{(3,1,1)}=1+qt+q^2t+q^2t^2+2q^3t^2+q^3t^3+q^4t^3+q^4t^4+q^5t^5
 +\aa^2\bigl(q^3+2q^4t+q^5t^2\bigr)
 +\aa\bigl(q+q^2+q^2t+3q^3t+q^3t^2+3q^4t^2
       +q^4t^3+q^5t^3+q^5t^4\bigr),\)

$ \Delta_{(3,1,1)}(t)=t^5-t^4+t^2-2t+3-2t^{-1}
+t^{-2}-t^{-4}+t^{-5}.$
\vskip 0.1cm

Only in this example (from 5), $\rho(q,t)$  is non-positive.

\vskip 0.2cm
Also, the {\sf\em Challenging example} from (\ref{chall}) was 
not found:

\( \h_{chlng}=1+qt+q^2t+q^2t^2+q^3t^3+q^4t^3+q^4t^4+q^5t^5
 +\aa^2\bigl(q^3+q^5t^2\bigr)
 +\aa\bigl(q+q^2+q^2t+q^3t+q^3t^2+q^4t^2
       +q^4t^3+q^5t^3+q^5t^4\bigr),\)

$ \Delta_{chlng}(t)=t^5-t^4+t^2-1+t^{-2}-t^{-4}+t^{-5}.$
\vskip 0.1cm

The last one is for $\h^{inst}_\sC$\, with
{\sf\em non-monomial} $\sC\subset \a$; presumably, it is
beyond knots.
\vskip 0.2cm

Generally, we found not a single example when
(usual, for $\a$) non-superdual superpolynomials
$\h^{mot}(q,t,\aa)$ or $\h^{inst}_\la$ 
become superdual upon the Heegaard-Floer
substitution $\aa=-t/q$, where the superduality is much
easier to check. In many example, $\rho(q,t)$ 
has the form (\ref{rho-zun}) for some {\sf\em sets\,} $\Ga$ and
$G$ (not associated with any $\r$ for hyperbolic knots). Only $\la=(3,1,1)$ is
not of this kind among the $5$ cases above. For instance,
 $\la =(4,3,2), (5,3), (5,3,2,1)$ have such $\Ga,G$ (just sets).

\comment{
\paragraph{Search record and scope.}
The comparison was made in the symmetric Laurent format,
%\eqref{eq:symmetric-normalization}, 
not merely by comparing coefficient
lists in an unspecified normalization.  The first three queries led
directly to $K11a367=T(11,2)$, $K12n242$, and $K12n725$, respectively, and
their displayed database Alexander polynomials agree term by term with
the three formulas above.  Exact-polynomial searches for
$\Delta_{\mathrm{str}}$ and $\Delta_{(6,3,2,1)}$ gave no identification.
This is a report of the searches performed, not an exhaustive theorem that
no knot can have either polynomial; in particular, no exhaustive search
through all knots up to a prescribed crossing number is being asserted.
}

\comment{
Similarly, the slice for $\la=(b,a)$ for $b>a>0$ can be identified with 
the even one for $\la''=(2b,2b-1,2a,2a-1)$. Indeed, 
We identify
$\a$ with the regular subring $\F[[u,v]]\subset\b$, so that
\begin{equation}\label{A1-free-module}
 \b=\a\oplus w\a, w^2=uv.
\end{equation}
Accordingly, for any ideal 
 $I\subset\a$, put
\begin{equation}\label{socle-lift-ideal}
 I'=I\oplus w(I:\mathfrak{m}_\a)\subset\b.
\end{equation}
This is a $\b$-ideal: one has $I\subset(I:\mathfrak{m}_\a)$ and
$uv(I:\mathfrak{m}_\a)\subset I$, which are exactly the two conditions obtained
by multiplying \eqref{socle-lift-ideal} by $w$.  Moreover,
$I'\cap\a=I$, so the construction is injective.  It respects
inclusions and therefore embeds every constant-socle stratum of 
an old
instanton slice into the corresponding $\b$-slice.

For a strict two-row diagram $b>a>0$, the conductor in
\eqref{socle-lift-ideal} is
\begin{equation}\label{four-row-conductor}
 I_\lambda'=
 (u^2,uv^a,v^b,uw v^{a-1},wv^{b-1}).
\end{equation}
Its even diagram is
\begin{equation}\label{four-row-even-image}
 \la''=(2b,2b-1,2a,2a-1).
\end{equation}
For $a=0$, the last two rows are omitted.  Thus $(3)$ gives $(6,5)$,
while $(3,1)$ gives $(6,5,2,1)$.  Formula
\eqref{socle-lift-ideal} is also the reason that the socle, 
rather than
the Nakayama rank in the outer singularity, is natural here.
}

%The full-check DAHA calculation for this example is currently in
%progress; no conclusion from that run is included here.

%Compare with the formula for 
%$\varrho(q,t)\equal \frac{\h(q,t,a=-t/q)-q^{\de}}{(1-t)(1-q)}=
%\rho(\frac{q}{t}, t)$.

\setcounter{equation}{0}
\section{\sc Two-row diagrams, etc}\label{sec:2row}
\subsection{\bf Two basic embeddings}
\Yboxdim7pt
We begin with two basic operations and the corresponding
relations for  the cells in the instanton slices.
In this section, and always when  {\sf\em superduality}
is discussed,  only the modules of rank $\mathbbm{n}=1$
are considered.

As above, $\a=\F[[x,y]]$ and  $deg(\i)=dim_{\F}\a/\i$
for ideals $\i\subset \a$ (always of finite codimension).
Generally, the {\sf\em conductor}
$\sC$ can be any ideal  of finite degree, but 
we will consider here only $\c=I_\la$
for monomial ideals $I_\la$
associated with Young diagrams $\la$.
The notation is $\la=(\la_1\ge \la_2\ge \cdots
\ge\la_i\ge \cdots
\la_\ell>0)$, where $\ell=\ell(\la)$ is the 
length of $\la$.  Recall that $I_\la$ is the ideal in $\a$  generated
by the monomials $x^i y^j$ for the {\sf\em outer corners}
 $(i,j)$ of the complement $\overline{\la}$ of $\la$ in $Z_+^2$.
We will denote this set by  $corn(\overline{\la})$.
The indices $(i,j)$ are for the rows and columns, with the 
pair $(0,0)$ representing  (the initial)  $\yng(1)$\,.
\Yboxdim5pt
Also, $I_\emptyset =\a$,  
$|corn(\overline{\sf\emptyset})|=1$, 
$I_{\yng(1)}=m_{\a}=
\lan x,y\ran$,
the maximal ideal of $\a$, and $|corn(\overline{\yng(1)})|=2$.
\medskip

As above, $\i^0$ will be the  $\a$-ideal generated by the leading (lowest)
terms of $f\in \i$. It is monomial:\, 
$\i^0=I_\la$ for a certain $\la$. We set $dgrm(\i)=\la$.
Importantly,  $deg(\i)=deg(\i^0)=|\la|$.
We continue to use the standard Gr\"obner ordering:
 $y\succ x^m\succ 1$ for
any $m\ge 1$.

%The corresponding ring is $\a/I_{\la}$, and
%its {\sf\em superpolynomial} will be denoted by $\mathscr{H}_\la$.
The {\sf\em instanton slices} for $\la$ are  $\mathfrak{S}_\la=
\{\i\, \mid\,  I_\la\subset \i
\subset \a \}$, where $\i$ are ideals. The corresponding 
{\sf\em instanton superpolynomial} is defined as
$\mathscr{H}_\la=\sum_{\i\in \mathfrak{S}_\la} t^{deg(\i)}(1+\aa q)\cdots 
(1+\aa q^{\varrho(\i)-1})$ for 
$\varrho(\i)\equal dim_{\F} \i/(m_\a \i)$; recall that $\F=\F_q$.  
The {\sf\em Gr\"obner cells} for the diagrams $\mu\subset \la$ are 
$G_\mu \equal \{\i \in \mathfrak{S}_\la \mid dgrm(\i)=\mu\}$ and
$G_\mu^{(r)}\equal\{\i \in G_{\mu\subset\la} \mid \varrho(\i)=r\}$. 
\vskip 0.2cm

Note that
$\varrho(\i)$ here is no greater than 
$|corn(\overline{\mu})|$, the number of outer corners of $\mu$.
 The equality is for $\i=I_\mu$, but
there can be other  ideals $\i\in G_{\mu\subset \la}$ of such top 
$\varrho$-rank. The latter
is in the range from $1$ (only 
for empty $\mu$) to $\rho+1$ for the greatest possible
triangle subdiagram $(\rho,\rho-1,\ldots,2,1)\subset \la$, so it 
can be greater than $\varrho(I_\la)=|corn(\overline{\la})|$.

We obtain that $\mathfrak{S}_\la$ is the disjoint union of 
$G_{\mu\subset \la}^{(r)}$
for $\mu\subset \la$ and $1\le r \le \rho+1$ for $\rho$ as above.
It will be technically convenient below to use the notation 
$G^{(r)}[\mu\subset \la]$ instead of $G^{(r)}_{\mu \subset \la}$. 
Accordingly, we will use 
$\sH[\la]$ instead of $\sH_\la$, and set 
$\sH^{(r)}[\mu\subset \la]\equal\bigl|G^{(r)}[\mu\subset \la]\bigr|$
(or in the absence of $(r)$). Recall that
$\cn=1$; $\i$  will be used instead of $\i$.

Assuming that 
$\mathscr{H}^{(r)}[\mu\subset\la]$ is a $q$-polynomial, 
its top degree
will be denoted by  $dim_q(G_{\mu\subset \la}^{(r)})$, or  
by $dim_q(G_\mu)$ if the $\varrho$-rank 
$r$ is disregarded. 
This will be mostly used
when $G$-cells are pure affine spaces 
$\mathbb{A}^N$; then $|G|=q^N$, and $dim_q=N$. 
We set $dim_q(\emptyset)=-1$.

Generally, $G_{\mu\subset\la}^{(r)}$ 
are conjectured to be blowups of configurations of
affine spaces. Then,
$|G_{\mu\subset \la}^{(r)}|$ will be, indeed, a
$q$-polynomials with integral coefficients.
\vskip 0.2cm

There are several natural relations between $\mathscr{H}$
for embedded Young diagrams. For instance, let us define
the union $\la\, \vec{\cup}\, \nu $ by considering the
Young diagrams as 
collections of columns and placing the diagram 
$\nu$ after $\la$. We assume that
$\ell(\nu)$ (the size of its 1{\footnotesize st} column)
is no greater than the size of the last column of $\la$.

The following proposition is essentially sufficient
to analyze the {\sf\em superduality} of $\sH^{inst}$ in the case 
of $2$-row diagrams.

\begin{definition}\label{def-superdual}
For $\de\in \Z_+$, a polynomial $\mathscr{H}=
\sum_{k,l,m}c_{k,l,m}\, \aa^k q^l t^m$ is called 
{\sf\em $\de$-superdual} if it is invariant  under 
$\Xi_\de:\, \aa^k\, t^l q^m \mapsto \aa^k\,t^{m+\de-l} q^m$, where 
$c_{k,l,m}\mapsto c_{k,l,m}$ (they will be integral coefficients). 
Equivalently, the {\sf\em superduality} means: 
\begin{align}\label{supdua}
&\mathscr{H}(q,t,\aa)= \Xi_\de\bigl(\mathscr{H}(q,t,\aa)\bigr)\equal
t^\de\, \mathscr{H}(q\mapsto qt, t\mapsto 1/t, \aa\mapsto \aa).
\end{align} 
\end{definition}

\begin{proposition}\label{prop:shift-emb}
(i) Consider $\mu\subset \la$ 
as subdiagrams of  $\la'=\la\, \vec{\cup}\, \nu $, provided that
the last column of $\mu$ is no smaller than $\ell(\nu)$.
Then the degrees $deg(\i)=|\mu|$ for the ideals 
$\i\in G_{\mu\subset \la}$ and their $\varrho$-ranks $\varrho(\i)$ 
do not change
upon the natural embedding $G_{\mu\subset \la}\, \to \,
G_{\mu\subset \la'}$, which is due to $I_{\la'}\subset I_\la$,
 and   
$\bigl|\,G^{(r)}_{\mu\subset \la}\,\bigr|=
\bigl|\,G^{(r)}_{\mu\subset \la'}\,\bigr|$
for any $\varrho$-ranks $r$.

(ii) Let $\de=|\la|$,
$\la^\#=\nu\,\vec{\cup}\,\la$, and $\de^\#=\de+|\nu|$.
Accordingly, let 
$\mu^\#=\nu\, \vec{\cup}\, \mu$ for $\mu\subset \la$, and
 $\kappa=
\bigl|\,corn(\overline{\mu}^{\,\#})
\setminus corn(\overline{\mu})\,\bigr|$ (the number of
extra outer corners of
$\mu^\#$).
%We set $G^{(r),\,\#}_{\mu\subset\la}=G^{(r)}_{\mu^\#\subset\la^\#}$.
Then $\bigl|\,G^{(r+\kappa)}_{\mu^\#\subset \la^\#}\,\bigr|=
\bigl|\,G^{(r)}_{\mu\subset \la}\,\bigr|$.  In particular,
\begin{align}\label{lamda-sharp}
&\Xi^\#\,\bigl(\sum_{\mu} \mathscr{H}
[\mu^\#\subset \la^\#](\aa=0)\bigr)=
\Xi\,\bigl(\sum_{\mu} \mathscr{H}[\mu\subset\la](\aa=0)\bigr),
\end{align}
where
$\Xi,\Xi^\#$ are those for $\de,\de^\#$.
\end{proposition}
{\it Proof.} The first claim is simple.
In $(ii)$, the diagram  $\mu^\#$
considered inside $\la^\#$ results in the $G$-set with the
same number of $\F_q$-points  as for 
$\mu$ considered inside $\la$. Indeed, the additional generators are 
the monomial generators of $I_\nu$, where 
the last one (a power of $y$) is omitted.
Adding them
does not influence the relations between  the coefficients
of the prior generators, which are for $\mu$ considered inside $\la$,
and increases the rank by $\kappa$.
However, now $deg(\mu^\#)=deg(\mu)+|\nu|$, which results
in $\Xi^\#:\, t^{l+|\nu|} q^m \mapsto 
t^{m+\de^\#-l-|\nu|} q^m=
t^{m+\de-l} q^m= 
\Xi( t^l q^m)$ for the products $t^l q^m$ from 
$\mathscr{H}[\la]$, and we must increase the $\varrho$-ranks of
the ideals $\i\subset \a$ by the corresponding $\kappa$ when
going from $\la$ to $\la^\#$.
Here $\kappa=|corn(\overline{\nu})|-1$ if the
last column of $\nu$ is greater than the 1{\footnotesize st} 
one in $\mu$,  
and with $-2$ instead of $-1$ if they are of the same size. \sq
\Yboxdim7pt
\vskip 0.2cm

\subsection{\bf Classification results}
We switch now to the $2$-row case; let  $\la=(b\, \ge\, a\ge 0)$. 
Accordingly,
its subdiagrams will be $\mu=(v\, \ge\, u\ge 0)$, where
$v\le b, u\le a$. We will include $1$-row diagrams (when $a=0=u$),
but the diagram $\mu=(0,0)$ corresponding to $\i= \a$ will be
excluded, the only ideal $\i$  of $\varrho$-rank $1$. 
Below, $Floor[x]$ is the integer part of $x$ and 
$Ceiling[x]=\min(n\ge x, n\in \Z).$ 
%Also, $I$  is used instead of $\i$.

\begin{theorem} \label{thm:2-row}
(i) For $\la, \mu$ as above, $G_{\mu\subset \la}$
is an affine space and  
$$
dim_q(G_{\mu\subset \la})=\min \bigl(Floor[\frac{v-u}{2}],a-u\bigr),
$$
In full detail,  this $dim_q$ equals
$Floor[\frac{v-u}{2}]$ when $a\ge Floor[\frac{v+u}{2}]$, and
$a-u$ when $a\le Floor[\frac{v+u}{2}]$.

(ii) The $\varrho$-ranks of any ideals 
$I_{\la}\subset \i \varsubsetneqq \a$ can be
$2$ or $3$. For $\varrho=3$,  $G^{(3)}_{\mu\subset\la}$ is
an affine space with
$dim_q=\min\bigl(Ceiling[\frac{v-u}{2}]-1, a-u\bigr)$ provided that $u>0$.
This space is empty if $dim_q=-1$ or $u=0$. 
%The whole $G_{\mu\subset\la}$ is of $\varrho$-rank $3$ 
%for $u>0$ if $v\!-\!u$ is odd  and 
%$a\ge Floor[\frac{u+v}{2}]$.

%Ceiling[\frac{v-u}{2}]-1$

(iii) The corresponding superpolynomial
$\mathscr{H}_{\la}(q,t,\aa)$ is {\sf\em superdual} in
the sense of Definition \ref{def-superdual} if and only 
if $b\ge 2a-1$, which is for 
$\de=|\la|=a+b$ in Definition \ref{def-superdual}.
\end{theorem}
{\it Proof} of {\bf (i)}. The generators of  $\i$  with 
$dgrm(\i) =\mu$ are
$f_1=x^2$ (only) if $u>0$, 
 $f_2=x y^u\!+\!\al_p y^{u+p}\!+\!\cdots+\al_{v-u-1} y^{v-1}$, where $\al_p$
is the first potentially nonzero coefficient when 
$1\le p \le v-u-1$, and $f_3=y^v$. If $\al_p\neq 0$, then
the leading term of $x f_2$ is $\al_p xy^{u+p}$. By ``potentially",
we mean that $\al_0\neq 0$
 does not contradict the condition $\i^0=I_\mu$.
However,
$x^2\in \i$ , and this  gives that $x y^{u+p}\in \i$ . We use here 
and below that
any  series from $1+\mathfrak{m}_{\a}$ is invertible in $\a$,
where $\mathfrak{m}_{\a}= x\a+y\a$ is the maximal ideal of $\a$. 

Generally,  $x y^{u+p}$ must belong to $\i$  modulo the ideal 
$(y^{u+p+1})$ due to $\i^0=I_{\mu}$,  but this argument gives that 
it is {\sf\em pure}, i.e. sits inside $\i$  without any 
higher additional terms. 
Also, $x y^a$ is {\sf\em pure} since $I_{\la}\subset \i$ .
\medskip

Next, $y^p f_2= x y^{u+p}+\al_p y^{u+2p}$ modulo $(y^{u+2p+1})$,
which gives that $y^{u+2p}$ is pure in $\i$ . Similarly,
$y^{a-u} f_2= x y^a+\al_p y^{p+a}$ modulo $(y^{p+a+1})$, which
gives that $y^{p+a}$ is pure too. These observations result in
the inequalities $u+2p\ge v$ and $p+a\ge v$; otherwise
$f_3=y^v\in \i$  will not be the smallest power of $y$ in $\i$ . 
There are no other conditions for $\{\al_i\}$.
Finally, $1\le p \le v-u-1$,
$p\ge Ceiling[\frac{v-u}{2}]$, and $p\ge v-a$. 

Intersecting
these $2$ ranges of indices,  we obtain
that the coefficients $\al_p,\ldots, \al_{v-u-1}$ are free
parameters of ideals $\i$ such that $I_\la\subset \i$  and
 $dgrm(\i) =\mu$, where 
$p=\max\bigl(Ceiling[\frac{v-u}{2}], v-a\bigr)$ as above. 

Hence:\, 
$dim_q(G_{\mu\subset \la})= 
\min\bigl(v-u-Ceiling[\frac{v-u}{2}], v-u-(v-a)\bigr)$, which equals
$\min\bigl(Floor[\frac{v-u}{2}], a-u\bigr)$
due to the general relation $Floor[\frac{n}{2}]+Ceiling[\frac{n}{2}]=n$ 
for $n\in \Z$. 
This concludes $(i)$.

Note that $dim_q(G_{\mu\subset \la})$ does not depend on
$b$, which matches Claim $(i)$ in Proposition  \ref{prop:shift-emb}.
Its dependence on $a$ is only via $a-u$, which matches 
Claim $(ii)$ there. 
\medskip

{\bf (ii).} 
%The ideals $\i$  from 
%$G^{(2)}_{\mu\subset\la}$ (of $\varrho$-rank $2$) are exactly
%those with $u=0$ or when $p$ obtained above satisfies
%$u+2p=v$ or $p+a=v$. 
The $\varrho$-rank of $\i$  is always $2$ when  $u=v$ or $u=0$. 
Indeed, the generators of $\i$  are  $f_3$ and either 
$f_1$ or $f_2$ in these cases.
If $v\!>\!u\!>\!0$, the calculation above
demonstrates that $f_3\!=\!y^v$ belongs
to $\a f_1 +\a f_2$ in $\i\in G_{\mu\subset\la}$
and $\varrho(\i)=2$ if and only if $u+2p'=v$ for $p'\ge p$,
where $p$ is that above. Accordingly, the $dim_q$ 
of $G^{(3)}_{\mu\subset \la}$, which is an affine space,
 is $\min\,\bigl(Ceiling[\frac{v-u}{2}],a-u\bigr)$
for $u>0$; this cell becomes empty if $u=0$. Note that
we used in $(i)$ that $y^v$ is
an $\a$-linear combination of $f_1,f_2$ and $xy^a$,
but this does not give that 
 $\varrho=2$ because $x y^a$ is involved.
This completes $(ii)$.
\medskip

{\bf (iii).} The proof will be by induction. Let us begin
with $\mathscr{H}_0=\mathscr{H}(\aa=0)$, i.e. we will disregard
the $\varrho$-ranks. Assume, first, that $b\ge 2a-1$ and that
the superduality
holds for $\la=(b\ge a\ge 0)$.
We consider
$\la^\#=(b^\#=b+1, a^\#=a+1)$ as in Proposition 
\ref{prop:shift-emb},(ii). 
Accordingly, $\mu^\#=(v^\#=v+1,u^\#=u+1)$
for $\mu=(v\ge u\ge 0)$. Then, 
$\Xi^\#\,\bigl(\mathscr{H}_0[\la^\#]\bigr)=
\Xi\,\bigl(\mathscr{H}_0[\la]\bigl) +\Xi^\#\,\bigl( \S_0^\#\bigr)$
for $\S_0^\#=1+\Si_{i=1}^{b+1} |G_{(i,0)\subset \la^\#}|\,t^i$,
which is the contribution of $1$-line diagrams to $\mathscr{H}_0[\la^\#]$.
Recall from this proposition that $\Xi=\Xi_\de$ for $\de=a+b$,
and $\Xi^\#=\Xi_{\de^\#}$ for $\de^\#=\de+2$. 

\medskip
Using Part $(i)$ of this theorem, 
 $\S_0^\#=1+\Si_{i=1}^{b+1} q^{m(i)}t^i$ for
$m(i)= \min\, \bigl(Floor[\frac{i}{2}],a+1\bigr)$. Explicitly,
 $\S_0^\#=(1+t)+q(t^2+ t^3)+\cdots+q^a(t^{2a}+t^{2a+1})+
q^{a+1}(t^{2a+2}+\cdots+t^{b+1}),$
where the last sum is $0$ if $b<2a+1$, and 
 $\Xi^\#\,\bigl(\S_0^\#\bigr)= 
(t^{\de+2}+t^{\de+1})+q(t^{\de+1}+t^{\de})+\cdots+
+q^a(t^{\de+2-a}+ t^{\de+1-a})+
q^{a+1}(t^{1+\de-a}+\cdots+t^{2+a+\de-b}).$ For instance,
$\Xi^\#(q^{a+1}t^{b+1})=q^{a+1}t^{a+1+\de+2-b-1}=q^{a+1}t^{2+a+\de-b}=
q^{a+1}t^{2a+2}$, though the last sum is present only when $b\ge 
2a+1$, equivalently, $b^{\,\#}\ge 2a^\#-1$.
\vskip 0.2cm

Let us assume that $b^{\,\#}\ge 2a^\#-1$ and deduce
the {\sf\em superduality} for $\la^\#$ from that for $\la=(b,a)$.
Notice that there are $b^\#+1$ terms in 
$\Xi^\#\,\bigl(\S_0^\#\bigr).$

We will consider now $\mu$ as subdiagrams of $\la^\#$. These
subdiagrams are those in $\la^\#$ without the last boxes in its
$2$ rows.  This will not alter $deg$ but may increase $dim_q$
by $1$ from $\min\, \bigl(Floor[\frac{v-u}{2}],a-u\bigr)$ to 
$\min\, \bigl(Floor[\frac{v-u}{2}],a+1-u\bigr)$, 
which happens if and only if
$a<Floor[\frac{v+u}{2}]$. In  particular, $v\ge u+2$ must hold
for this due to $u\le a$. 

To avoid this change of $dim_q$, we define
$\mu^\flat$ as $\mu$ if $a\ge Floor[\frac{v+u}{2}]$,
and as $(v-1,u+1)$ if $a<  Floor[\frac{v+u}{2}]$.
Then $deg(\mu^\flat)=deg(\mu)$, $dim_q(G_{\mu^\flat\subset \la^\#})=
dim_q(G_{\mu\subset\la})$,
and $\mathscr{H}_0[\mu\subset \la]$ is a sum of all terms
of $\mathscr{H}_0[\mu'\subset \la^\#]$ for $\mu'$ represented as
$\mu^\flat$ for $\mu\subset \la$. 

The map inverse to $\mu\mapsto \mu^\flat$ is
as follows. The inequalities above become 
$a^\#=a+1 > Floor[\frac{v'+u'}{2}]$ and 
$a^\#=a+1 \le Floor[\frac{v'+u'}{2}]$. Accordingly,
$\mu=(v,u)=(v',u')$ if the $1${\footnotesize st} inequality 
holds subject to $u'\le a=a^\#-1, v'\le b=b^\#-1$, 
and $\mu=(v,u)=(v'+1,u'-1)$ subject to $v'\le b-1=b^\#-2$
if the $2${\footnotesize nd} holds. 

The following $b^\#+1=b+2$ diagrams $\mu'=(v',u')$ cannot be
obtained by the first or the second map from the definition
of $\mu^\flat$:
$$
\{ (b+1,i), (b,i) \mid 0\le i \le a+1 \} \text{ and }
\{ (b-i,0) \mid 1\le i \le b-2a-2 \}.
$$

It is direct to see that 
their contribution to $\mathscr{H}_0[\la^\#]$ is {\sf\em exactly}
$\Xi^\#\,\bigl(\S_0^\#\bigr)= 
%\sum_{i=0}^{a^\#-1} q^i\,(t^{\de^\#-i}+t^{\de^\#-i-1})+
%q^{a^\#}\bigl(\sum_{i=\de^\#+a^\#-b^\#}^{\de^\#-a^\#}
%t^i\bigr)$
\sum_{i=0}^{a} q^i\,(t^{\de+2-i}+t^{\de+1-i})+
q^{a+1}\bigl(\sum_{i=\de+2+a-b}^{\de+1-a}
t^i\bigr)$ obtained above. Then we use Part $(ii)$ of
Proposition 
\ref{prop:shift-emb} and  obtain that the condition
$b\ge 2a-1$ is sufficient for {\sf\em superduality}
for $a=0$. 
The same calculation gives that it is necessary.
We will omit this. Proposition \ref{prop:psi} below gives
more than this:\, a potential way to calculate all  $\mu$
that destroy the superduality when this inequality
does not hold. 

The case of $\aa=0$ and the superduality for the 
ideals $\i$  with $\varrho(\i) =3$ 
are formally sufficient to justify it for the whole 
$\mathscr{H}[\la](q,t,\aa)$, i.e. for any $\varrho$-ranks.
The sum of the corresponding $t^i q^j$-terms times $q^3$ 
 will be the top $\aa$-coefficient of the latter superpolynomial, 
that for $\aa^2$. We use now the formulas from Part $(ii)$ of Theorem 
\ref{thm:2-row}. A technical simplifications is that
only diagrams $\mu$ of maximal possible $\ell(\mu)=2$ are 
sufficient to consider (we disregard pure rows). \sq

\begin{proposition}\label{prop:psi}
(i) In the case of $2$-row diagrams $\la=(b\ge a\ge 0)$ and its
subdiagrams $\mu=(v,u)$,  the map
$\phi: \mu \mapsto q^i t^j$ for $j=|\mu|$ and
 $i=dim_q(G_{\mu\subset \la})$ is injective. The corresponding
algebraic formula  can be
naturally extended to an injective map for any $u,v\in \Z$.
Accordingly, $q^it^j$ in
$\mathscr{H}[\la](a=0)$ are with coefficients $0$ or $1$. The same
holds for $G_{\mu\subset \la}^{(r)}$ and $\mathscr{H}^{(r)}[\la]$, where
(by definition) the $\varrho$-rank $r$ is fixed.

(ii) Using explicit
formulas from  Theorem \ref{thm:2-row}, the superduality map
can be naturally extended to an involution $\psi: (v,u)\mapsto
(v',u')$ well-defined for any $u,v\in \Z$. Let
$B_\la=\{(v,u) \mid \psi(v,u)\not\subset \la$\}, which is non-empty
if $b<2a-1$. Namely, let $\mu=(b,0)$ for
$\la=(b,\,Ceiling[b/2]+i)$, where 
 $2\le i \le Floor[b/2]$ (to ensure that $b<2a-1$).
 Then $\mu'=\psi(b,0)=(b+Floor[i/2],\,Ceiling[i/2])$, 
and it is obviously not in $\la$.
 If $i=1$ here, then
$\psi(b-1,0)=(b+1,1)$  for odd $b$ and $\psi(b,0)=(b+1,0)$ for
even $b$. This consideration covers all $\la$ such that $b<2a-1$.
\end{proposition}
{\it Proof.} If two diagrams $\mu$ and $\mu'$ result in
the same $q^i t^j$ then $|\mu|=u+v=|\mu'|=u'+v'$, and the inequality
$a\ge Floor[\frac{u+v}{2}]$ or $a\le Floor[\frac{u+v}{2}]$ holds for
$u',v'$ too. This inequality determines which formula gives
$dim_q$:\, $Floor[{v-u}{2}]$ or $(a-u)$. If the latter is
applicable, then obviously $u=u'$ and, therefore, $v=v'$. If the former
works, then the required coincidence is due to  identity
$Floor[\frac{v-u}{2}]=Floor[\frac{v+u}{2}]-u$ and the relation
$v+u=v'+u'$. Using Part $(ii)$ of Theorem \ref{thm:2-row}, we
can extend this reasoning to the coefficients of
any powers of $\aa$. 
\medskip

{\bf (ii).} The extension of $\psi$ to any integers is
straightforward, as well as the justification that such 
an extension 
is an involution.  There are some formulas in terms of
integer parts, which we will omit, that can be checked 
to define an involution.
\medskip

The representatives of $B_\la$ 
for $b<2a-1$ provided above are the simplest ones; they are sufficient
to complete Part $(iii)$ of Theorem  \ref{thm:2-row}.
There can be (many) other subdiagrams $\mu$ in  $B_\la$,
especially for $a=b$; the algebraic formulas are not too difficult
to find using computers. The smallest example of nonempty $B_\la$
is for $\la=(2,2)$, when $\psi(2,0)=(3,0)$. See right before
Section \ref{sec:motsup} for the discussion of the corresponding
superpolynomial; vs. the DAHA one $\mathfrak{H}_\la$, which
is selfdual.

For the sake
of readability, let us consider $\la=(b,Ceiling[b/2]+2)$,
which is for $i=2$; we take $\mu=(b,0)$ and must
check that $\mu'=\psi(b,0)=(b+1,1)$. The corresponding dimension 
for $\mu'$ is $dim_q=\min\bigr( Floor[b/2],Ceiling[b/2]+1 \bigr)=
Floor[b/2]$, and $deg=|\mu'|=b+2$. Thus, $\mu'$ formally results
in $q^{l'} t^{m'}$ 
with $l'=Floor[b/2], m'=b+2$. Similarly, $q^l t^m$ from $\mathscr{H}_0$
for $\mu$ is
with $l=Floor[b/2]$ and $m=b$. We obtain, indeed, that
$\Xi_\de(q^l t^m)=q^l t^{\de+l-m}=q^l t^{b+2} = q^{l'} t^{m'}$, where
$\de=b+Ceiling[b/2]+2$
and $\de+l-m=b+Ceiling[b/2]+2+Floor[b/2]-b=b+2$. This proves
that $B_\la \neq \emptyset$, and in a very explicit way. The case
of any $i$ (any $\la$ from Part $(ii)$) is no different.
\sq
\vskip 0.2cm

\subsection{\bf Combinatorial conjecture}
The approach  from Proposition \ref{prop:psi},  when we identify 
the subdiagrams that destroy the superduality, can help with
the following conjecture.
Given a Young diagram $\la=(\la_1\ge \la_2\ge \cdots \ge \la_\ell>0$,
let $k_i\equal\la_i-\la_{i+1}$ for $1\le i<\ell$ and $k_\ell=\la_\ell$.

\begin{conjecture}\label{conj:supd}
We call $\la$ {\sf\em generic} if $k_i$ are all positive and 
$\sum_{i=a}^{a+m} k_{i}\ge \sum_{i=b}^{b+m} k_i-1$ for any 
$a<b$ and any $m\ge 1$.
The superduality of $\sH_\la^{inst}$ holds
if and only if  $\la$ or its transpose $\la^{tr}$ satisfy
one of the following:\ 
(i)\,  $\la$ is generic,\,  or\ (ii)\, the 
 $1${\footnotesize st} column is the tallest
($k_\ell=1=\la_\ell$) and
the diagram obtained by its removal is generic.

The same conditions are necessary and sufficient for
the positivity of  $\mathfrak{H}_\la^{daha}$\,:\, 
the non-negativity of  all its $q,t,\aa$-coefficients.
\sq
\end{conjecture}

This was checked numerically for $\la$  with up to $4$ corners
on the DAHA side, together with  ChatGPT, 
and for $\sH^{inst}_\la$ (a lot). 
For $5$ corners, we confirm the positivity of $\mathfrak{H}^{daha}$
for $k=\{1,1,2,1,2\}$, corresponding to $\la=(5,5,4,3,3,2,1)$, just to
give an idea of the range of our calculations. Checking real big
diagrams is technically involved, so corrections to
this conjecture will be not unexpected. 

\vskip 0.2cm

{\bf The case of $\la=(3,3,2).$}
The 
corresponding  $k=\{0,1,2\}$
satisfies the inequalities above, but $0$ is not allowed by the
conjecture. 

One has:\, $\mathscr{H}^{inst}_\la(\aa=0)=$
{\small $1 + t + t^2 + q t^2 + t^3 + q t^3 + 
q^2 t^3 + t^4 + q t^4 + 
 q^2 t^4 + t^5 + q t^5 + q^2 t^5 + t^6 + q t^6 + q^2 t^6 + t^7 + 
 q t^7 + t^8$}. Its instanton-superdual is 
$\Xi_\de\bigl(\mathscr{H}_0[\la]\bigr)=$
{\small $1 + t + t^2 + q t^2 + t^3 + q t^3 + t^4 + q t^4 + q^2 t^4 + t^5 + 
 q t^5 + q^2 t^5 + t^6 + q t^6 + q^2 t^6 + t^7 + q t^7 + q^2 t^7 + t^8.
$}

To see the failure of superduality,
the following  $G$-cells are sufficient:\,
those for $\mu=(3)$, its potential dual
$\mu'=(3,3,1)$ and   $\mu''=(3,2,2)$.

For $\mu$, the generators of $\i\subset I_\la$ with $dgrm(\i) =\mu$ 
are $x^2+\al y+\be y^2$ and $y^3$, where the parameters $\al,\be$
are free; i.e. $I_\la$ belongs to the {\sf\em monomial conductor}
of any such $\i$ . The corresponding term in $\mathscr{H}$ is 
$q^i t^j=q^2 t^3$; its superdual is $q^i t^{\de-j+i}=q^2 t^7$.
Thus, we need to check the $dim_q$ of the cells of degree $7$,
which are for $\mu'$ and $\mu''$. The generators of $I'$ for $\mu'$
are $f_1=x^3, f_2=x^2 y+ \al xy^2+\be y^2, f_3=y^3$. However, 
$x f_2= \be x y^2$ modulo $I_{\la}$, which results in $\be=0$.
Thus, its contribution is $q t^7$. There are no free parameters
for $\mu''$ and the corresponding contribution is $ t^7$. Neither
contribution coincides with $q^2 t^7$, 
which concludes this ``counterexample". 

The corresponding $\mathfrak{H}^{daha}_\la(\aa=0)$ equals
{\small
\(1+q t+q^2 t+q^3 t-q^4 t-q^5 t+q^2 t^2+q^3 t^2+2 q^4 t^2-q^6 t^2
-q^7 t^2+q^3 t^3+q^4 t^3+2 q^5 t^3-q^7 t^3+q^4 t^4+q^5 t^4+2 q^6 t^4
-q^8 t^4+q^5 t^5 +q^6 t^5+2 q^7 t^5-q^8 t^5+q^6 t^6+q^7 t^6+q^8 t^6
+q^7 t^7+q^8 t^7+q^8 t^8
\)}, which is superdual for $q\leftrightarrow t^{-1}$, but with
negative terms.
\smallskip
 
Let us provide $\h^{inst}_\la(\aa=0)=
(q t)^\delta \sH^{inst}_\la(,q,1/(qt),\aa=0)=$
{\small
\(1 + q t + q^2 t + q^2 t^2 + q^3 t^2 + q^4 t^2 + q^3 t^3 + q^4 t^3 + 
 q^5 t^3 + q^4 t^4 + q^5 t^4 + q^6 t^4 + q^5 t^5 + q^6 t^5 + q^7 t^5 +
  q^6 t^6 + q^7 t^6 + q^7 t^7 + q^8 t^8\)}. They obviously cannot
coincide, which match Conjecture 
\ref{conj:supd}.

\medskip
{\bf Comment on paper \cite{ObR2}.}
Let us mention that 
the inequalities  $k_i\ge k_{i+1}-1$ (for Coxeter links)
occurred in  Conjecture 6.5.1 there
concerning vanishing certain cohomology. It was stated
that such  vanishing holds for torus knots due to 
 progress with the
{\sf\em ORS conjecture} for such knots.  No evidence or 
rationale for these inequalities was provided 
beyond torus knots. Comparing with our inequalities\:\,
considering only pairs of neighboring ones are insufficient
(for us), the
case $k_\ell=1$ did not occur in \cite{ObR2},
 and we need to consider the
transposition of Young diagrams. So we have  different 
situations (beyond torus knots). 

% it is not clear from the paper
%what is the rationale for their conjecture for such an
%extension. 

\vskip 0.2cm
The approach in their paper
 is quite different from our one. It requires
a compactification of punctual Hilbert schemes of $\C^2$ and
cohomology of certain sheaves there. Explicit 
calculations in \cite{ObR2}
are limited to uncolored $T(2k+1,2)$. These are very basic, 
but sufficient to
see the differences. See Section 5.3 there, especially its end.
The passage to our parameters is $t=t_{\text{\scalebox{0.7}{OR}}}^2, 
\, q=q_{\text{\scalebox{0.7}{OR}}}^{-2}$; recall
that we normalize our superpolynomials to make the first term $1$.

The uncolored superpolynomials for
$K=T(2k+1,2)$ are immediate for $\h^{mot}$ and
quite simple  for $\h_K^{daha}$, which are
coinvariants $\{ Y_1(X_1^{k+1}) \}$ upon the 
stabilization of $A_{N-1}$.
%We will discuss below practical calculations of this kind.
%Adding colors $\omega_m$ is when $X_1\mapsto X_m$.

 %and simple in almost any
%approaches to superpolynomials, beginning with the original
%Khovanov's definition. 

{\sf\em Coxeter links} are mostly due to
\cite{ObR2} and some  prior papers, but 
\cite{GL2}, including their usage of Young diagrams,
is better compatible with our work. 
Similarly to the superduality ($\cn=1$) of
{\sf\em EHA-superpolynomials} from \cite{GL2}
(partially known), we 
conjecture that our  $\mathfrak{H}^{daha}_\la$ are
superdual. In examples, non-symmetric 
{\sf\em Pieri formulas}
for minuscule weights from \cite{Ch4} and similar tools
are helpful.
%As we see, the {\sf\em superduality}
%of $\sH^{inst}_{\la}$ requires  non-trivial conditions for $\la$.

\subsection{\bf The case of  K12n242} 
This is one of the most famous hyperbolic knots. 
It is the simplest 
hyperbolic one with superdual $\sH^{inst}_\la$
among those with 2-row diagrams. Recall that the knot for
$\la=(2,2)$ is hyperbolic but with non-superdual  $\sH^{inst}_\la$.
 The diagrams $(k, 2k)$
and $(k, 2k+1)$ correspond to  $\F[[x=z^{3k+1}, y=z^3]]$
and $\F[[x=z^{3k+2}, y=z^3]]$, which are those for torus
knots $T(3k+1,3)$ and $T(3k+2,3)$.
Also, $\la=(k)$ serve $T(2k+1,2)$. Any other $2$-row diagrams give
hyperbolic knots, i.e. they are beyond our prior constructions
of DAHA and motivic superpolynomials. For $\cn=1$, 
one has:\ 
%Let us provide $\mathscr{H}^{inst}_\la$ for $\la=(3,2)$; \ 
$\mathscr{H}^{inst}[(3,2)]=$
{\small
\begin{align}\label{K12n242}
1+ &(1\!+\!\aa q)(t + t^2 + q t^2 + q t^3 + q t^4)+
(1\!+\!\aa q)(1\!+\!\aa q^2)(t^3 + t^4 + t^5)\notag\\
=1&+t+t^2+q t^2+t^3+q t^3+t^4+q t^4+t^5
+\ \aa^2 (q^3 t^3+q^3 t^4+q^3 t^5)\notag\\
&\ +\ \aa (q t+q t^2+q^2 t^2+q t^3+2 q^2 t^3+q t^4+2 q^2 t^4
+q t^5+q^2 t^5).
\end{align} 
}

The corresponding   HOMFLY-PT polynomial
in the normalization compatible with 
the standard databases is $H\!OM[(3,2)]=\, $
\(
z^4/a^{14}+(4 z^2)/a^{14}+3/a^{14}-z^8/a^{12}
-(9 z^6)/a^{12}-(27 z^4)/a^{12}-(31 z^2)/a^{12}-11/a^{12}+z^{10}/a^{10}
+(10 z^8)/a^{10}+(36 z^6)/a^{10}+(57 z^4)/a^{10}+(39 z^2)/a^{10}+9/a^{10}.
\)
For instance, the substitution $z \mapsto \imath M, a \mapsto 
\imath L^{-1}$ can be used to find this polynomial (up to a mirror)
 in the file
{\sf\em \, jhomflytable.txt\,} available online, which contains 
HOMFLY-PT polynomials for prime links with up to 16 crossings, about
2.5M.
K12 means that the number of crossing is $12$.  Actually, it will 
find {\sf\em m12.242n, m12.725n} etc.
The number of crossing cannot be smaller than the degree of $z$,
but this is just the lowest estimate. The passage to 
HOMFLY-PT polynomials from the  DAHA ones is generally
$q^{-|\la|}\h^{daha}(q,t\mapsto q,\aa\mapsto -a^2)$, where
$z=q^{1/2}-q^{-1/2}$.
We use Tchebyshev polynomials to switch from $q$ to $z$.
\vskip 0.2cm

{\bf Relation to DAHA.}
%Let us check it for K12n242 for $\cn=1$
Recall that Theorem \ref{hmotinst-gen}
combined with the Coincidence Conjectures 
$\h^{mot}=\h^{daha}=\mathfrak{H}^{daha}$ motivated 
the (conjectural) formula (\ref{daha-la-one}), which is (for $\cn=1$):
\begin{align}\label{sH=frak}
\mathfrak{H}_\la^{daha}(q,t,\aa)=
\h^{inst}_\la(q,t,\aa)\equal 
(qt)^\de \mathscr{H}_\la^{inst}(q, 1/(qt),\aa),
\end{align}
where $\de=|\la|=5$. Since $\h^{mot}$ is not applicable
to non-algebraic knots, we have only this connection conjecture
for K12n242 (for $[3,2]$).
  It holds (checked in many cases).  On the DAHA end, we have:
\vskip 0.2cm

$\mathfrak{H}_\la^{daha}(q,t,\aa)=$
{\small
\( 1+q t+q^2 t+q^2 t^2+q^3 t^2+q^3 t^3+q^4 t^3+q^4 t^4+q^5 t^5
+\aa^2 \bigl(q^3+q^4 t+q^5 t^2\bigr)+\aa 
\bigl(q+q^2+q^2 t+2 q^3 t+q^3 t^2+2 q^4 t^2
+q^4 t^3+q^5 t^3+q^5 t^4\bigr).
\)
}
\smallskip

Actually, there are $4$ different Young diagrams for the knot
K12n242:
% and they (really) give coinciding $\mathfrak{H}_\la^{daha}$:
{\small
$$
\mathfrak{H}_1=\mathfrak{H}_{(3,2)},\ 
\mathfrak{H}_2=\mathfrak{H}_{(2,2,1)}, \ 
\mathfrak{H}_3=\mathfrak{H}_{(4,1)}, \ 
\mathfrak{H}_4=\mathfrak{H}_{(2,1,1,1)}.
$$
}
Recall that we present $\la$ as the 
sequence of consecutive moves $R, U$ (right, up)
in the border of $\overline{\la}$ from the bottom 
to the top, writing them right-to-left, and then
replace $R$ by $X$ and $U$ by $Y$. More directly, $W$ is 
when we read the diagram top-to-bottom, and write $X,Y$ for
$L,D$ naturally (left-to-right).
Note that the procedure from \cite{GL2} is somewhat different
diagrammatically and with 
$U\mapsto Y, R\mapsto Y X Y^{-1}$. 
\vskip 0.2cm

Then,  $W^{[m]}$ is defined upon 
$X\mapsto X_m$ and $Y\mapsto Y_m$ in $W$,
where $X_{m}=X_{\om_m},\  Y_{m}=Y_{\om_m}$
for minuscule $\om_m$, and we calculate $\bigl\{W^{[m]}X_m\bigr\}$,
subject to the 
$\aa$-stabilization for $A_{N-1}$ (or $gl_N$),  where $\aa=-t^{N}$.

% with $X_{k}=X_1X_2\cdots X_k$, 
%$Y_{(k)}=Y_1Y_2\cdots Y_k$ for the standard
%generators of DAHA of type $gl_N$ corresponding
%to the standard $\{e_i\}\subset \R^N$.

Thus, the corresponding
$X,Y$-words are:
{\small
$$ W_1=YXYX^2,\, W_2=Y^2XYX ,\, W_3=YX^3YX,\, 
W_4=YXY^3X,$$
}

\noindent
where $m\ge 1$ must be the same in
$W,X$ and $Y$. We will consider $m=1$ below.
It is instructional to specialize (\ref{sH=frak}) to the case of
DAHA-Jones polynomials and to check that all four $X,Y$-words
above give coinciding coinvariants (up to proportionality).

Let us discuss $\mathscr{H}^{inst}_\la$.
Its reduction to $A_{1}$ 
 recalculated to $\h^{inst}_\la$ is:\,
{\small
$$
\h^{inst}_{\la}(\aa\!=\!-t^2)=
1 + q t + q^2 t - q t^2 + q^3 t^2 - q^2 t^3 - q^3 t^3 + q^4 t^3 - 
 q^4 t^4.
$$
}

It is straightforward to see its coincidence  
with the following $4$ coinvariants $\bigl\{W_{i}X\bigr\}$ 
corrected by 
the corresponding  $q,t$-factors:
{\footnotesize
\begin{align*}
&(1):\, q^{7/2} t^3\, Y\bigl(X Y(X^3)\bigr)(X\!\mapsto\! t^{-1/2}) 
,\ &(2):\, q^4 t^3 \,Y^2\bigl(X Y(X^2)\bigr)(X\!\mapsto\! t^{-1/2}) ,\\ 
&(3):\,q^{7/2} t^{7/2}\, Y\bigl(X^3 Y(X^2)\bigr)(X\!\mapsto\! t^{-1/2})
,\ &(4):\,  q^{9/2} t^{7/2}\, Y\bigl(X Y^3(X^2)\bigr)(X\!\mapsto\! t^{-1/2}). 
\end{align*}
}

The same coincidence holds for any $A_{N-1}$ and for 
arbitrary $X_{m},Y_{m}$. An instructional
exercise is to change $YXYX^3$ to $Y^{-1}XYX^3$. Then
$\bigl\{ Y^{-1}XYX^3\bigr\}=\bigl\{ T\pi XYX^3\bigr\}\sim
\bigl\{ \pi XYX^3\bigr\}\sim \bigl\{ \pi YX^3\bigr\}=
\bigl\{ T X^3\bigr\}\sim \bigl\{ X^3\bigr\}\sim 1$. This coinvariant
remains $\sim 1$ for any $A_{N-1}$.
%\vskip 0.2cm

\subsection{\bf The case of hooks} 
This is the simplest general 
series of hyperbolic knots for our construction. Let $\la$ be
a union of an $m$-column and an $n$-row intersected at the 
box with $i=0=j$. Thus, this diagram is $\la=(n,1,\ldots, 1)$
where $1$ occurs $(m-1)$ times. It is 
of length $\ell(\la)= m$ and $\de=|\la|=m+n-1$. Its subdiagrams
are $\mu=(b,1,\ldots, 1)$ of $deg=a+b-1$, i.e $1$ occurs $(a-1)$
times. Here $1\le b\le n$ and $1\le a\le m$. Let us calculate
its $\sH^{inst}$, which is 
\begin{align}\label{hook-exp}
\mathscr{H}^{inst}[\la]&=1+ (1+\aa q)\,\Si_{i=1}^m \, t^i\,  +\, 
(1+\aa q)(1+\aa q^2)\,\Si_{j=1}^{n-1}\ t^{j+m}\,+\\
(1\!-\!\de_{m,1})&\,q(1+\aa q)\,\Si_{j=2}^{n}\ t^j\ +(1+\aa q)(q+\aa q^2)\, 
\Si_{i=2}^{m-1}\,\Si_{j=2}^{n}\ t^{i+j-1}.\notag
\end{align}
First, $1$ is for $\mu=\emptyset$, the $1${\footnotesize st} $\Si$ 
is for $b=1$, i.e.
$\mu$ is a $j$-column, and the $2${\footnotesize nd} $\Si$ is for 
$a=m,b>1$. There are no parameters and  the $\varrho$-ranks are
$1$, $2$ and $3$ respectively. Second, the $3${\footnotesize rd}
 $\Si$ is only when $m>1$;
it counts 
$\mu$  inside the row when $b>1$. There is
one free parameter in this case and the $\varrho$-rank is $2$, which
gives the weight  $q(1+\aa q)$. 

Third, the last double summation is
for the remaining cases. As in the previous $\Si$,  
there is always one free parameter. 
Indeed, the generators
are $f_1=x^a+\al_1 y+\al_2 y^2+\cdots+\al_{b-1} y^{b-1}$, 
$f_2=xy, f_3=y^b$ for $a<m, b>1$ in these cases. However, 
$yf_1$ must have all $0$ coefficients modulo $(y^b)$, which gives
that only $\al_{b-1}$ is a free parameter. The $\varrho$-rank is 
$2$ in these cases if and only if $\al_{b-1}\neq 0$,
since $y^b=\al_{b-1}^{-1} \, yf_1$ modulo $(x y)$. Otherwise it is $3$.
This gives the
weight  
$(q-1)(1+\aa q)+(1+\aa q)(1+\aa q^2)=
(1+\aa q)(q+\aa q^2)$.
\vskip 0.2cm

Finally, counting the sums of $t$-powers,
\begin{align}\label{hook-rat}
&\mathscr{H}^{inst}[\la]\!=\!1\!+\! (1+\aa q)t \frac{1-t^m}{1-t}\! +\!
(1\!+\!\aa q)(1\!+\!\aa q^2)t^{m+1}\frac{1-t^{n-1}}{1-t}\\
&+(1\!-\!\de_{m,1})\,q t^2 (1\!+\!\aa q)\frac{(1-t^{n-1})}{1-t}\,
\frac{\bigl(1\!+\!\aa q t-t^{m-1}(1\!+\!\aa q)\bigr)}{1-t}.\notag
\end{align}

When $m>1$,  $(1-t)^2\mathscr{H}[\la]$ becomes
{\small
\(
1 - t + q t^2 - q t^{1 + m} - q t^{1 + n} - t^{m + n} + 
 q t^{m + n} + t^{1 + m + n} 
+  \aa^2 \bigl(q^3 t^3 - q^3 t^{2 + m} - q^3 t^{2 + n} 
+ q^3 t^{1 + m + n}\bigr)
+ \aa \bigl(q t - q t^2 + q^2 t^2 + q^2 t^3 - q^2 t^{1 + m} 
- q^2 t^{2 + m} - 
    q^2 t^{1 + n} - q^2 t^{2 + n} - q t^{m + n} + q^2 t^{m + n} + 
    q t^{1 + m + n} + q^2 t^{1 + m + n}\bigr)\,,
\)
}
which is obviously symmetric under $m\leftrightarrow n$.
\vskip 0.2cm

Formula (\ref{hook-rat})  is
{\sf\em superdual} (for $\de=m+n-1$), 
which is sufficient to check for 
$$
\mathscr{H}^{inst}_0[\la]= \mathscr{H}_{\aa=0}^{inst}[\la]=
1+t\frac{1-t^{m+n-1}}{1-t}+q t^2 
\frac{(1-t^{n-1})(1-t^{m-1})}{(1-t)^2},
$$
and for the top $\aa$-coefficient of $\mathscr{H}^{inst}[\la]$, which is
$q^2t^3\frac{(1-t^{n-1})(1-t^{m-1})}{(1-t)^2}$. We use that
the top degree of $\aa$ is $2$ (similarly to the
case of $2$ rows). % This is immediate.
For instance,  $\Xi_\de\bigl(t\frac{1-t^{m+n-1}}{1-t}\bigr)=
t^\de t^{-1}\frac{1-t^{-m-n+1}}{1-t^{-1}}=t\frac{1-t^{m+n-1}}{1-t}$.
\vskip 0.2cm

{\bf Knot K12n725.} This well-known knot is the simplest 
{\sf\em hyperbolic} in the hook family, which is for $m=3=n$.

The one for $m=1, n=4$ is hyperbolic too, but it is isotopic
to K12n242 considered above, and $\mathscr{H}^{inst}[(4,1)]=
\mathscr{H}^{inst}[(3,2)]$. We conjecture that instanton superpolynomials
 $\mathscr{H}$ are isotopic invariants of the
corresponding knots, which is not proven by now
(for the
DAHA-EHA approaches too). At least, it is obvious that
the instanton superpolynomials are the same 
for $\la$ by its transpose $\la^{tr}$, which we proved
for $\mathfrak{H}^{daha}$ (not immediate for the
EHA-superpolynomials). 

For K12n725, the corresponding  $\mathscr{H}^{inst}[(3,1,1)]$ is
as follows:
{\small
\begin{align}\label{K12n725}
1 + (1\! +\! \aa q)& (t + t^2 + q t^2 + 2 q t^3 {\bf - t^4} + q t^4) + 
(1\! +\!  \aa q) (1\! +\! \aa q^2) (t^3 + 2 t^4 + t^5)\notag\\
=\ 1 + t& + t^2 + q t^2 + t^3 + 2 q t^3 + t^4 + q t^4 + t^5 +
\aa^2 (q^3 t^3 + 2 q^3 t^4 + q^3 t^5)\notag\\
&+\ \aa (q t + q t^2 + q^2 t^2 + q t^3 + 3 q^2 t^3 + q t^4 + 3 q^2 t^4 + 
    q t^5 + q^2 t^5).
\end{align}
}
Notice $-t^4$, which is because some $G^{(2)}_{\mu \subset\la}$
are not affine spaces. There are no negative terms 
in (\ref{hook-exp}) because it was presented
using $(q+ q^2 \aa)$. 
%\vskip 0.2cm

\setcounter{equation}{0}
\section{\sc Further aspects}
\vskip 0.2cm
\subsection{\bf Deformations}
The conductors $\sC$ in {\sf\em instanton slices} 
can be, generally, arbitrary, not only monomial.
The lifted conductors $\C$ are monomial
for  $\F[[z^\rr,z^\ss]]$ and
$\F[[x=z^{\upsilon \rr}+z^{\upsilon \rr+p}, y=z^{\upsilon \ss}]]$, 
for $p=1$, where $gcd(\rr,\ss)=1$ due to Theorem 3.4 in \cite{ChG}.
Presumably, they are non-monomial for other plane curve 
singularities.

Considering {\sf\em instanton slices} for non-monomial
$\sC$ beyond algebraic knots is of clear interest. 
For instance, we can take $I_\la$ and add extra {\sf\em non-monomial}
generators. Quite a few superpolynomials of this kind
coincide with some  monomial conductors, i.e. we have some non-trivial 
{\sf\em flat} deformations for the latter. 
However, this is not always the case.
There is a couple of relatively small examples, where the
$q=t$ reduction of the corresponding superpolynomials
cannot be identified with any HOMFLY-PT polynomials 
among knots with up to $16$ crossing. 
\medskip

Let us discuss various examples of deformations. We 
will begin with $\mathfrak{S}_\la$ for $\la=(4,3,1)$,
which is for the ring  $\F[[x=z^6+z^7, y=z^4]]$ and 
the algebraic cable $C\!ab(13,2)T(3,2)$.
\vskip 0.2cm

 The corresponding
motivic superpolynomial from \cite{ChP1} is $\h^{mot}=$
{\small \(1+q t+q^2 t+q^3 t+q^2 t^2+q^3 t^2+2 q^4 t^2+q^3 t^3+q^4 t^3
+2 q^5 t^3+q^4 t^4+q^5 t^4+2 q^6 t^4+q^5 t^5+q^6 t^5+q^7 t^5+q^6 t^6
+q^7 t^6+q^7 t^7+q^8 t^8+\aa^3 (q^6+q^7 t+q^8 t^2)+\aa^2 (q^3+q^4+q^5
+q^4 t+2 q^5 t+2 q^6 t+q^5 t^2+2 q^6 t^2+2 q^7 t^2+q^6 t^3+2 q^7 t^3
+q^8 t^3+q^7 t^4+q^8 t^4+q^8 t^5)+\aa (q+q^2+q^3+q^2 t+2 q^3 t+3 q^4 t
+q^5 t+q^3 t^2+2 q^4 t^2+4 q^5 t^2
+q^6 t^2+q^4 t^3+2 q^5 t^3+4 q^6 t^3+q^7 t^3+q^5 t^4+2 q^6 t^4
+3 q^7 t^4+q^6 t^5+2 q^7 t^5+q^8 t^5+q^7 t^6+q^8 t^6+q^8 t^7)\,,\)} 
and $\mathscr{H}^{inst}=$
{\small \(1+t+t^2+q t^2+t^3+q t^3+q^2 t^3+t^4+q t^4+2 q^2 t^4
+t^5+q t^5+2 q^2 t^5+t^6+q t^6+2 q^2 t^6+t^7+q t^7+q^2 t^7+t^8
+\aa (q t+q t^2+q^2 t^2+q t^3+2 q^2 t^3+q^3 t^3+q t^4+2 q^2 t^4
+3 q^3 t^4+q t^5+2 q^2 t^5+4 q^3 t^5+q^4 t^5+q t^6+2 q^2 t^6
+4 q^3 t^6+q^4 t^6+q t^7+2 q^2 t^7+3 q^3 t^7+q^4 t^7+q t^8+q^2 t^8
+q^3 t^8)+\aa^2 (q^3 t^3+q^3 t^4+q^4 t^4+q^3 t^5+2 q^4 t^5+q^5 t^5
+q^3 t^6+2 q^4 t^6+2 q^5 t^6+q^3 t^7+2 q^4 t^7+2 q^5 t^7
+q^3 t^8+q^4 t^8+q^5 t^8)+\aa^3 (q^6 t^6+q^6 t^7+q^6 t^8)\,.\) }
\vskip 0.2cm

The (minimal) generators of $I_\la$ in this case are
 $f_1\!=\!x^4,f_2\!=\!x^3y,
f_3\!=\!x y^3,f_4\!=\!y^4$. Let us change them as follows:
\medskip

$(i)$\, replacing $f_3$ by  $xy^2$, we arrive at
$C\!ab(11,2)T(3,2)$ for $(4,2,1)$;

$(ii)$\, adding $x^2+y^2$, we obtain  $\mathscr{H}[\la]$ 
for the hook $\la^\circ\,=\,(4,1,1,1)$;

$(iii)$\, adding $xy+y^3$ results in $\mathscr{H}[\la^\circ]$
for the hook $\la^\circ=\, (3,1,1,1)$;

$(iv)$\, adding $x^2+ xy^2+y^3$ gives the same $\mathscr{H}$ as for
\ $(2,1,1,1,1,1)$; 

%this knot is {\sf\em m16.184868n} from {\sf\em jnonflytable.txt};

$(v)$\, replacing $x^4$ by  $x^2+xy+y^2$ gives $\mathscr{H}$ that
coincides with $\mathscr{H}'$ 

\ \ \ \ \ \ obtained  for $\la'=(4,3)$ with the 
$x^2$ in $I_{\la'}$ 
 replaced by  $x^2+ y^2$.

\medskip
Case $(v)$ is interesting.
The corresponding $\mathscr{H}$ is superdual and satisfies all ``usual"
properties of superpolynomials, but we failed to
find the corresponding (formal) HOMFLY-PT polynomial
in {\sf\em jhomflytable.txt\,}.
%formally obtained via $\h^{inst}$. 

The latter is {\small
\(1+q t+q^2 t+q^2 t^2+q^3 t^3+q^4 t^3+q^4 t^4+q^5 t^5+q^6 t^5
+q^6 t^6+q^7 t^7+\aa^2 (q^3+q^5 t^2+q^7 t^4)+\aa (q+q^2+q^2 t+q^3 t+q^3 t^2
+q^4 t^2+q^4 t^3+q^5 t^3+q^5 t^4+q^6 t^4+q^6 t^5+q^7 t^5+q^7 t^6)\) },
and the corresponding HOMFLY-PT polynomial is {\small \(
z^8/a^18+(8 z^6)/a^{18}+(20 z^4)/a^{18}+(16 z^2)/a^{18}+3/a^{18}
-z^12/a^{16}-(13 z^{10})/a^{16}-(65 z^8)/a^{16}-(156 z^6)/a^{16}
-(182 z^4)/a^{16}-(91 z^2)/a^{16}-13/a^{16}+z^{14}/a^{14}
+(14 z^{12})/a^{14}+(78 z^10)/a^{14}+(221 z^8)/a^{14}
+(338 z^6)/a^{14}+(272 z^4)/a^{14}+(100 z^2)/a^{14}+11/a^{14}\,.\)}
Though the $z$-degree here may be too big for $16$ crossings.
\vskip 0.2cm

%We considered quite a few similar examples. 
%For instance, adding
%$xy+y^2$ to the generators of $I_\la$ for $\la=(3,2,1)$ results in
%K12n725, which is for $(3,1,1)$.

Let us now
add $x^2+y^3$ to $I_\la$ for the Young diagram $\la=(4,2,1)$, 
corresponding
to (non-algebraic) $C\!ab(11,2)T(3,2)$ above. The resulting
knot will be  K14n6022, the pretzel
$P(-2,3,9)$ (up to a mirror) from the same pretzel/Montesinos
family as K12n242, which is $P(-2,3,7)$. 
The conductor $\sC$ is 
not monomial
in this example. However, the same 
superpolynomial can be also obtained from $\sC=I_\la$ for
$\la=(2,1,1,1,1)$ or from its transpose $(5,1)$.

%Similarly, replacing in $I_\la$
%for $\la=(4,2)$ the generator $x^2$ by
%$x^2+xy+y^3$ results in K14n21324, the one for $\la=(3,1,1,1)$.;
\medskip

Let us provide the corresponding
$\mathscr{H}^{inst}$,
recalculated the DAHA-form $\h^{inst}$, where the
superduality is under
$q\leftrightarrow t^{-1}$. One has:\, $\h_{(5,1)}^{inst}=$
{\small \(
1+q t+q^2 t+q^2 t^2+q^3 t^2+q^3 t^3+q^4 t^3+q^4 t^4
+q^5 t^4+q^5 t^5+q^6 t^6
+\aa^2 (q^3+q^4 t+q^5 t^2+q^6 t^3)+ \aa (q+q^2+q^2 t+2 q^3 t+q^3 t^2
+2 q^4 t^2+q^4 t^3+2 q^5 t^3+q^5 t^4+q^6 t^4+q^6 t^5)\)\,.}

In its non-monomial presentation above,  the {\sf\em Gr\"obner
cell} $G_{\mu\subset \la}$ for $\mu=(4)$ is {\sf\em empty}. 
Indeed, the generators of any $\i\in G_{\mu}$ are
$x+ay+by^2+cy^3$ and $y^4$
$\i\in G_{\mu}$ for this $\mu$.  And it must contain $x^2y, x y^2$
due to $\mu\subset \la$, which gives that $a=0$ and $x^2\in \i$. 
Since $x^2+y^3$ was added,  $y^3\in \i$ and 
$dgrm(\i)$ is smaller than $\mu$; a contradiction.

%The only one empty  cell in this case. 
All other cells are non-empty
affine spaces $\mathbb{A}^{0,1}$. Let us provide the original
instanton superpolynomial:\ $\sH^{inst}_{(5,1)}=$
{\small
\(1+t+t^2+q t^2+t^3+q t^3+t^4+q t^4+t^5+q t^5+t^6
+\aa \bigl(q t+q t^2+q^2 t^2+q t^3+2 q^2 t^3+q t^4
+2 q^2 t^4+q t^5+2 q^2 t^5
+q t^6+q^2 t^6\bigr)
+\aa^2 \bigl(q^3 t^3+q^3 t^4+q^3 t^5+q^3 t^6\bigr).\)
}
\vskip 0.2cm

\begin{theorem}\label{thm:echar1}
For $\cn\ge 1$ and monomial $\sC=I_{\vec{\la}}$,
assume that $|G^{(r)}_{\vec{\mu}\subset \vec{\la}}|$ are
$q$-polynomials. Then the
Euler characteristic of such  Gr\"obner cells  
is $1$ for $r=r_{\vec{\mu}}^{max}$, the number of outer corners of
$\vec{\mu}$, and $0$ otherwise. 
\end{theorem}
{\it Proof.} As it was outlined in the Introduction,
it suffices to use the action of $(\F_q^*)^2$ times $(\F_q^*)^{\cn-1}$
in $\a^{\cn}$, which preserves the conductor $\sC$ and $\varrho$-ranks.
There is exactly one invariant module in 
$G^{(r)}_{\vec{\mu}\subset \vec{\la}}$ for $r=r_{\vec{\mu}}^{max}$,
which is $I_{\vec{\mu}}$, and none otherwise. \sq
%and  $\sH^{inst}_\la$ is superdual. The 
%{\sf\em superduality} is the invariance  under
%$q\leftrightarrow t^{-1}$ up to $q^\bullet t^\bullet$
%for the corresponding $\h^{inst}_\la$.

Recall that we conjecture that 
all $q,t,\aa$-coefficients of $\sH^{inst}_\la$ are in $\Z_+$
for, at least, monomial conductors.
In our calculations, the positivity holds 
even for non-monomial conductors. 
%A conditional justification of $(i)$ was provided in the
%Introduction. 
\vskip 0.2cm

{\bf Challenge: beyond KhR?} The following example 
can be potentially a superdual superpolynomial 
beyond knots. We begin with 
$\la=(3,2)$ and replace the generator $x^2$ in $I_\la$
by $x^2+xy+ y^2$. Then $\h^{inst}_\sC(q,t,\aa)\equal 
(q t)^\de \mathscr{H}_\sC^{inst}(q, 1/(qt),\aa)=$
{\small
\begin{align}\label{chall}
&1+q t+q^2 t+q^2 t^2+q^3 
t^3+q^4 t^3+q^4 t^4+q^5 t^5 +\aa^2 (q^3+q^5 t^2) \\
&+\aa (q+q^2+q^2 
t+q^3 t+q^3 t^2+q^4 t^2+q^4 t^3+q^5 t^3+q^5 t^4)\,.\notag
\end{align} 
}%

\noindent
Interestingly, using here 
$x^2+y^2$ and $x^2+xy+y+y^2$ give $\sH^{inst}_{\la'}$
for $\la'=(3,1,1)$ (K12n725) and $\la'=(5)$ ($T(11,2)$)
correspondingly.  

The HOMFLY-PT polynomial (formally generated) for $\h$ in 
(\ref{chall}) is
{\small \(z^4/a^{14}+(4 z^2)/a^{14}+2/a^{14}-z^8/a^{12}-(9 
z^6)/a^{12} -(27 z^4)/a^{12}-(30 
z^2)/a^{12}-9/a^{12}+z^{10}/a^{10}+(10 z^8)/a^{10} +(36 
z^6)/a^{10}+(57 z^4)/a^{10}+(39 z^2)/a^{10}+8/a^{10}\)\,.} The 
$z$-degree here indicates that if this one is for any knot,
then $16$ crossings may be sufficient to capture it, but we
did not find it.  

The corresponding formal Jones polynomial is 
$1 + t^2 - t^5 + t^7 - t^9$ up to $t^\bullet$, and, possibly, for
$t\mapsto t^{-1}$. 
This one is not very likely to be a Jones polynomial. However,
as to our understanding, it can be such.  For instance, 
it satisfies
the following  test: \ Jones polynomials
 must become $\pm e^{2\pi \imath/3}$
upon $t\mapsto e^{2\pi \imath/3}.$ We also examined 
other invariants, including the 
Alexander and Conway polynomials. 
The latter (formally generated)
is $1+13z^2+31z^4+27z^6+9z^8+z^{10}$.
%\vskip 0.2cm

\subsection{\bf Examples and  discussion}
Let us consider $\F[[x=z^9+z^{10}, y=z^6]]$, which is
when  $\upsilon=3, \rr=3, \ss=2, p=1$ in the family
$\r=\F_q[[x=z^{\upsilon \rr}+z^{\upsilon \rr+p}, y=z^{\upsilon \ss}]]$,
for $gcd(\rr,\ss)=1=gcd(\upsilon,p)$ and positive $\rr,\ss,\upsilon$.
The corresponding cable is 
$C\!ab(\upsilon \rr \ss+p, \upsilon) T(\rr,\ss) =C\!ab(19,3)T(3,2)$ and 
$\de=\frac{(d-1)(\upsilon-1)}{2}+\frac{(\rr-1)(\ss-1)}{2}\upsilon$ for 
$d=\upsilon \rr \ss +p=19$, which is $\de=21$. Its conductor lifted 
to $\a=\F_q[[x,y]]$
is $\c=I_\la$ for $\la=(7,6,4,3,1)$.% with $|\la|=\de=21$.
\medskip

This diagram is formed by boxes above the diagonal in the
rectangle $6\times 8$, including those touching the diagonal.
We will also consider  $\la'=(7,5,4,2,1)$, denoted
by $\la^{\boldsymbol{\flat}}$ in Theorem
\ref{thm:sharp-flat},
where we remove the boxes touching the
diagonal. It corresponds to 
(non-algebraic) 
$C\!ab(17,3)T(3,2)$ with $\de=19$ boxes.
The diagrams are:
\vskip 0.2cm

\centerline{
$\la =\ \yng(7,6,4,3,1) (|\la|=21),\ \ \  
\la'=\ \yng(7,5,4,2,1) (|\la'|=19).$
}
\medskip

Collecting all constructions above, there are
$5$ different ways to calculate the corresponding uncolored 
superpolynomial $\h$ for such cables with $p=1$,
and $4$ ways for $p=-1$; $\h^{mot}$ cannot
be used for $p=-1$ to be discuss below.
We begin with the DAHA superpolynomials.
%\medskip

Generally, they are defined for any 
nonsymmetric polynomials $E_b$, where $b$ are
dominant weights of type $A_{N-1}$ or for $gl_N$ (which
gives the same $\h^{daha}$).  The corresponding {\sf\em coinvariants}  
are subject to the hat-normalization and the $N$-stabilization.
Only $b=\om_1$ is needed in this example (the uncolored case).
%in this section 

%\medskip
We will recall now general constructions, but 
the DAHA formulas below will be provided only for 
 $\j_{A_1}=\h^{daha}(\aa=-t^2)$, i.e. in the case of 
DAHA-Jones polynomials of type $A_1$. Calculations
are especially simple in this case:\, $\HH$ is generated by  
$X^{\pm 1},T,Y^{\pm 1}$,
$\bigl\{X^a\bigr\}=X^a(t^{-\rho})=t^{-a/2}$, and the
action of $Y$ in $\mathscr{X}$ is given by the 
formula $Y=\pi (t^{1/2}s+\frac{t^{1/2}-t^{-1/2}}{X^2-1}(s-1)$,
where $s(X)=X^{-1}, \pi(X)=q^{1/2}X^{-1}.$

\medskip
The passage to $A_{N-1}$ and  $\cn>1$
is straightforward, but this requires a computer program. 
We note that
the degree of $\aa$ in $\h^{daha}$ and $\mathfrak{H}^{daha}$
is the total number of $Y$ in $W$, which is helpful
when writing a code.
\medskip

{\bf Five approaches:\,{\mathversion{bold}$Cab(19,3)T(3,2)$}.} 

\smallskip

{\bf (i)\, } Following \cite{ChD1}, 
 $\h^{daha}(q,t,\aa)$ is, generally, the $N$-stabilization
of $\j^{A_{N-1}}=\Bigl\{\, \hat{\ga}_1\Bigl(\hat{\ga}_2\bigl(
E_b/E_b(t^{-\rho})\bigr)(1) \Bigr)\,\Bigr\}$ for $H(1)$ in the
polynomial representation of $\HH$ of type
$A_{N-1}$ and
$\{H\}=\bigl( H(1) \bigr)(X\mapsto t^{-\rho})$.

For $C\!ab(19,3)T(3,2)$, we obtain:\,   $\hat{\ga}_1=\tau_+\tau_-^2 
\rightsquigarrow
\begin{pmatrix} 3 & *\\ 2 & * \end{pmatrix}$ 
encodes $T(3,2)=C\!ab(2,3)$,
and $\hat{\ga}_2=\tau_+^2\tau_- \rightsquigarrow
\begin{pmatrix} 3 & *\\ 1 & *\end{pmatrix}$ encodes
its cabling with $\upsilon=3$.
Since $E_b=X_1$, it is simpler to switch to $(ii)$.

\medskip

{\bf (ii)\, } Thanks to Theorem \ref{thm:sharp-flat} and that we
have a $2$-cables with $p=1$,
$\h^{daha}(q,t,\aa)=\mathbb{H}^{daha}(q,t,\aa)$, where the
latter is 
$\Bigl\{\, \hat{\ga}_{1}(X_1^\upsilon)\, \Bigr\}$ upon 
the 
hat-normalization and stabilization. As above, $\hat{\ga}_1
=\tau_+\tau_-^2$ as above, but $\hat{\ga}_2$ is no longer
needed.
%but now we apply $\hat{\ga}_1$
%to $(X_1^\upsilon)$ instead of using   
For $A_1$: 

 $\mathbb{H}^{daha}(\aa\!=\!-t^2)
\!=\!q^{13} t^{13/2}\Bigl\{\, Y\Biggl(XY\biggl(X^2Y\Bigl(X
Y\bigl(X^2Y(X^2)\bigr)\Bigr)\biggr)\Biggr)\,\Bigr\}=
$
{\small
\(1 + q t + q^2 t + q^3 t + q^4 t + q^5 t - q t^2 + q^4 t^2 + q^5 t^2 + 
 3 q^6 t^2 + 2 q^7 t^2 + q^8 t^2 - q^2 t^3 - q^3 t^3 - 2 q^4 t^3 - 
 2 q^5 t^3 - 2 q^6 t^3 + 2 q^8 t^3 + 3 q^9 t^3 + q^{10} t^3 - q^5 t^4 - 
 2 q^6 t^4 - 4 q^7 t^4 - 4 q^8 t^4 - 3 q^9 t^4 + q^{10} t^4 + 
 2 q^{11} t^4 + q^{12} t^4 + q^5 t^5 + q^6 t^5 + q^7 t^5 - 2 q^9 t^5 - 
 4 q^{10} t^5 - 3 q^{11} t^5 + q^{12} t^5 + q^{13} t^5 + q^7 t^6 
+ q^8 t^6 + 
 2 q^9 t^6 + 2 q^{10} t^6 - 3 q^{12} t^6 - q^{13} t^6 + q^{11} t^7 + 
 2 q^{12} t^7 - q^{13} t^7 - q^{12} t^8 + q^{13} t^8\,.\) }
This calculation is quite doable. Use the formula for $Y$ 
provided above and that $Y(X)=(qt^{1/2})^{-1}X$.
The polynomial above coincides 
with the corresponding hat-normalized
$\hat{\j}^{A_1}$ obtained as in ($i$).

\medskip

{\bf (iii)\, } The method in $(ii)$ can be applied
only to $2$-cables with $p=\pm 1$.  One
can
use $\mathfrak{H}^{daha}_\la$ for arbitrary $\la$.
This will give the same result due to 
Lemma \ref{lem:W7-3}, where 
a model example was  $C\!ab(43,2)T(7,3)$ for $\rr=7,\ss=3$ and
$\up=2$. 

\comment{
 is similar to 
the EHA-polynomials from \cite{KTs, GL2}, where 
we employ the full (nonsymmetric)
DAHA instead of  EHA (which is its stable spherical part) and
replace their {\sf\em plethystic} formulas by a direct usage of the
{\sf\em DAHA coinvariant}. Also, our diagrammatic setting and
the usage of $X,Y$ is somewhat different.
}
\vskip 0.2cm

Generally, 
we present $\la$ as a sequence of $R$ (right)
and $U$ (up) moving from the bottom to the top
along the border of $\overline{\la}=\Z_+\setminus \la$.
In this example, this path is
 from 
$(i=\ell(\la)=5, j=0)$ to $(i=\la_1=7, j=0)$. We write these $R,U$ 
right-to-left:\, $URURRURURRUR$. Then we replace
$R$ by $X$ and $U$ by $Y$, and obtain $W_\la=YXYXXYXYXXYX$.
 Finally, 
$\mathfrak{H}^{daha}$ is $\bigl\{ W_\la X \bigr\}\sim
\{YXYX^2YXYX^2\}$ 
for the DAHA 
coinvariant $\bigl\{\ldots\bigr\}$  under the hat-normalization
and $\aa$-stabilization.
\vskip 0.2cm

Lemma \ref{lem:W7-3} states that
 $\bigl\{ W_\la X \bigr\}$ 
coincides with $\bigl\{ \hat{W}^3\}$ up to a $q,t$-factor for 
$\hat{W}=\tau_+\tau_-^2(X)=XYXYX$. Indeed:

\centerline{
$\bigl\{ XYXYX\, XYXYX\ XYXYX \}\sim \{YXYX^2YXYX^YX^2\bigr\}$,}

\noindent
which coincides with  $\bigl\{ W_\la X \bigr\}$ above.
Recall that 
$\mathfrak{H}_\la^{daha}$ for $\om_1$ conjecturally
coincide with the EHA-superpolynomials from \cite{GL2}.

\smallskip

{\bf (iv)\, }  Let us employ motivic superpolynomials.
They are,  generally,
$\h^{mot}_{\r}(q,t,\aa)=\sum_{M_{st}} t^{deg(M)} (1+q \aa)\cdots 
(1+q^{\varrho(M)-1}\aa)$,
where the sum is over $\r$-invariant {\sf\em standard}
modules $M=M_{st} \subset \o=\F_{q}[[z]]$, and $\rho(M)$
is the $\rho$-rank. See above and \cite{ChP1,ChP2,ChQ}. 
However, can skip them and go to $(v)$ using
Theorem \ref{thm:mainred}, which provides a 
reduction to the {\sf\em instanton
superpolynomials}. It is for any plane curve singularities.

{\bf (v)\, } One has:
{$\mathscr{H}_\c^{inst}=
\sum_{\c\subset \i\subset \a} t^{deg(\i)}(1+\aa q)\cdots 
(1+\aa q^{\varrho(\i)-1}).$}
where $\c=I_\la$ is the conductor of $\r$ lifted to 
$\a=\F_q[[x,y]$.
Then  $\h^{mot}_\r(q,t,\aa)= (qt)^\de 
\mathscr{H}_{\c}^{inst}(q, t\mapsto 1/(qt),\aa).$
In this example, only the reduction to $A_1$ is considered,
which  deserves a special comment. This can be of course done for any
$A_{N-1}$, but $A_1$ is very transparent.
%and we need only to perform the reduction to $A_1$. 

%The {\sf\em superduality} generally reads\, 
%$\h(q,t,\aa)= (qt)^\de \h(t^{-1}, q^{-1},\aa)$,
%which is known for DAHA, but is still a conjecture for motivic
%superpolynomials.

\comment{
The key conjecture in our prior papers is that
$\h^{mot}(q,t,\aa; \cn)$ coincides
with $\h^{daha}(q,t,\aa)$ defined for $b=\cn\om_1$,
where $\cn$ is the rank. This is for any plane curve
singularities; the ranks and the corresponding
colors $\cn\om_1$ in $\h^{daha}$  may depend on the 
components in the multibranch case. This is Conjecture 9.1 from
\cite{ChQ}.
It can be considered as a generalized version of the {\sf\em Shuffle
Conjecture}; see \cite{CaM,ChS,ChQ} and 
Conjecture \ref{conj:main},(i) above.
%In our case, $\cn=1$ and the singularity is unibranch. 
}

\vskip 0.2cm
{\bf Reduction to {\mathversion{bold}$A_1$.}}
We follow Corollary 5.2 from \cite{ChQ}, which is
for general (uncolored) motivic superpolynomials
$\h_\r^{mot}(q,t,\aa)$. The conjectural relation  
$\h^{mot}=\h^{daha}$ will be used 
and the superduality of the latter (a theorem).
This reduction holds for $\sH^{inst}_\c(q,t,\aa)$,
and even becomes somewhat simple. One has:
\vskip 0.2cm

\centerline{
$\h^{mot}(q,t,\aa\!=\!-q^{-2})=\h^{daha}(q,t,\aa\!=\!-q^{-2})$
}

\centerline{
$=(q t)^\de \h^{daha}(1/t,1/q,\aa\!=\!-t^2)
=(q t)^\de \hat{\j}^{A_1}(1/t,1/q)$.
}
\vskip 0.2cm

 Using definitions and
the passage to $\sH^{inst}$:\, 
$\h^{mot}(q,t,\aa=-q^{-2})=(q t)^\de+
\sum_{M} t^{deg(M)}(1-q^{-1})=(qt)^\de 
(1+\sum_\i (qt)^{-deg(\i)}(1-q^{-1})),$
summed over standard  $M\subset \o$ 
and over ideals $\c\subset \i\subset \a$ of  $\varrho$-rank $2$. Thus:

%$$\hat{\j}^{A_1}(1/t,1/q)=
%1+\sum_{\c\subset \i\subset \a}^{\varrho(\i) =2} 
%(qt)^{-deg(\i) }(1-q^{-1}).
%$$

\begin{align}\label{red-to-A1}
\hat{\j}^{A_1}(q,t)=
1+\sum_{\c\subset \i\subset \a}^{\varrho(\i) =2} 
(qt)^{deg(\i)}(1-t).
\end{align}

%\medskip

{\bf Cables next to algebraic.} Let us briefly 
discuss the case of  $\la'$, denoted by 
$\la^{\boldsymbol{\flat}}$ in Theorem
\ref{thm:sharp-flat}.
For any rectangle $\upsilon \ss \times \upsilon \rr$ provided
$gcd(\rr,\ss)=1$ and $\up>1$, 
consider the Young diagram obtained by the boxes {\sf\em strictly}
 above the
diagonal. 
We can 
use $\h^{daha}$ and $\mathfrak{H}^{daha}$
from $(i-iii)$ and $\mathscr{H}^{inst}$ from $(v)$
for $C\!ab(17,3)T(3,2)$ and $\la'$, but 
the superpolynomial $\h^{mot}$ is  not defined. 

Now $\hat{\ga}_2=\tau_-^{-1}\tau_+\tau_-^2 \rightsquigarrow
\begin{pmatrix} 3 & *\\ -1 & *\end{pmatrix}$ in $(i)$ with the same
$\hat{\ga}_1$.  Concerning $(iii)$, the diagram $\la'$ can be
encoded as  $URRURURRURUR$, and then we replace 
$R$ by $X$ and $U$ by $Y$:\, $W'=YXXYXYXXYXYX$. 
In this example, $\h^{daha}$ coincides with
$\h^{inst}$ and $\mathfrak{H}^{daha}$, where the latter
is  $q^{12}t^{13/2}\bigl\{ W'X\,\bigr\}$. 
In particular,  $\h^{daha}(\aa\!=\!-t^2)=
\mathfrak{H}^{daha}(\aa\!=\!-t^2),$ 
and they are equal to:\,
{\small
\(1 + q t + q^2 t + q^3 t + q^4 t + q^5 t - q t^2 + q^4 t^2 + q^5 t^2 + 
 3 q^6 t^2 + 2 q^7 t^2 + q^8 t^2 - q^2 t^3 - q^3 t^3 - 2 q^4 t^3 - 
 2 q^5 t^3 - 2 q^6 t^3 + 2 q^8 t^3 + 3 q^9 t^3 + q^{10} t^3 - q^5 t^4 - 
 2 q^6 t^4 - 4 q^7 t^4 - 4 q^8 t^4 - 3 q^9 t^4 + q^{10} t^4 + 
 2 q^{11} t^4 + q^5 t^5 + q^6 t^5 + q^7 t^5 - 2 q^9 t^5 - 4 q^{10} t^5 - 
 3 q^{11} t^5 + q^{12} t^5 + q^7 t^6 + q^8 t^6 + 2 q^9 t^6 + 2 q^{10} t^6 - 
 2 q^{12} t^6 + q^{11} t^7 + 2 q^{12} t^7 - q^{12} t^8\,.
\) }
\vskip 0.2cm

{\bf Miscellaneous remarks.} 
(a) Let us mention that the conductor $\c=I_\la$  for the ring $\r$ in 
the example of $C\!ab(19,3)T(3,2)$ 
above was calculated in \cite{ChG}, as well as that 
for $p=2$ (which is not monomial). The 
formula connecting  $\h^{mot}$ and $\mathscr{H}^{inst}$ is
a combination of Proposition 3.3 there and the 
reduction formula from Section 7.2 in \cite{ChQ}, which was for 
arbitrary plane curve singularities (in rank $1$). Though
instanton slices added a new dimension.
\vskip 0.2cm

(b) Calculations with instanton slices are sufficiently fast.
% See Section  \ref{sec:condform} below. 
The main piece of
software is a program for finding {\sf\em Gr\"obner bases}. 
Importantly,
this program and a related one for the $\varrho$-ranks 
 must be in the version for
the ring $\F[[x,y]]]$, i.e. for series, not polynomials.
This requires {\sf\em Singular} 
or similar, ... or your own code. 

Many thanks to ChatGPT for
moving the corresponding author's programs from Mathematica to
SAGE-Singular, and various improvements made in process of 
this transition. 
%\medskip

\vskip 0.2cm
(c) The DAHA superpolynomials require, generally, the corresponding
coinvariants for sufficiently many
$A_{N-1}$. In the case of the minuscule
weight $\om_m$, we begin with $A_{m-1}$ where this coinvariant is $1$.
It is convenient to switch to DAHA of type $gl_N$, which is with the
generators  $\mathbb{X}_i\!=\!X_{\mathbbm{e}_i}$,
 $\mathbb{Y}_i\!=\!Y_{\mathbbm{e}_i}$ 
 $(1\le i\le N)$ for basic $\mathbbm{e}_i$,
and $T_i (i\le i\le N-1)$. Importantly,  
$X,Y$-words $W$ do not depend on $N$; we calculate them for $A_1$
and then replace $X,Y$ by $X_m,Y_m$. For $m=1$,  
$Y_1=\pi T_{n-1}\cdots T_1   $, where 
$T_i=t^{1/2}s_i+\frac{t^{1/2}-t^{-1/2}}{\mathbb{X}_i 
\mathbb{X}_{i+1}^{-1}-1}(s_i-1)$ for $s_i: \mathbb{X}_{i}
\leftrightarrow \mathbb{X}_{i+1}$ 
(trivial otherwise), and $\pi: \mathbb{X}_1\!\mapsto\! \mathbb X_2,\ldots, 
\mathbb X_{N-1}\!\mapsto\! \mathbb X_{N}, \mathbb X_N\!
\mapsto\! q^{-1} \mathbb X_N$. 

%The passage 
%to  $A_1$ is as follows:\, $\pi:\,X=X_1^{1/2} X_2^{-1/2} \mapsto 
%X_2^{1/2} (q^{1/2} X_1^{-1/2})= q^{1/2} X^{-1}.$ 

\medskip
(d) Reduced DAHA formulas for the diagrams 
$\la^{\boldsymbol{\natural}}$ and $\la^{\boldsymbol{\flat}}$
from  Theorem \ref{thm:sharp-flat}, denoted by $\la$ and $\la'$
in $(ii)$ above, are new. Recall that $\la^{\boldsymbol{\natural}}$ and
$\la^{\boldsymbol{\flat}}$ are formed respectively by the boxes 
above the 
diagonal in ``$\upsilon \ss\times \upsilon \rr$", including and 
excluding the boxes that touch it. They provide an example
for the following general problem. What is the meaning of 
the (original) DAHA construction when 
$\tilde{\ga}$ and the iterations 
are applied to products of {\sf\em nonsymmetric}
Macdonald polynomials $E_b$? If they are symmetric, this is
the construction of DAHA superpolynomials for iterated torus
{\sf\em links} from \cite{ChD2}. Algebraically, considering
products of $E_b$ makes perfect sense.
%but no topological interpretation is known. 
We have $X_m^\up=E_m^\up$ for 
$\la^{\boldsymbol{\natural}}$ (the case of $p=1$).

%The connection of 
%$E_{\om_m}^\up=X_{\om_m}^\up$ 
%to $2$-cables of torus knots with $p=\pm 1$ is of interest.

\comment{
Adding colors is relatively direct, but only for columns
in $\mathfrak{H}^{daha}$. Namely,  we use {\sf\em the same} 
words $W$ as above, replacing $X$ and $Y$  by   
$X_{\om_k}=X_1\cdots X_k$ and $Y_{\om_k}=Y_1\cdots Y_k$. Then we go 
from the resulting $\j_{gl_{N}}$-polynomials  
colored by $\om_k$ to the superpolynomials upon the 
hat-normalization and stabilization. This is for arbitrary Young 
diagram $\la$ in the construction of $\mathfrak{H}^{daha}$. For 
$\la$ associated with torus knots, the latter coincide with the 
corresponding  $\h^{daha}$, colored by $k$-columns. This is 
expected for their $\upsilon$-iterations. 
}

\medskip
(e) Our main conjecture, which connects $\mathfrak{H}^{daha}_\la$
and $\sH^{inst}_\la$, and prior $\h^{daha}_\k=\h^{mot}_\r$
are directly related to the 
{\sf\em Shuffle Conjecture} in the case of uncolored torus knots 
$\k=T(\rr,\ss)$ and for  $\la^{\boldsymbol{\natural}}$. We think that 
this paper can help to clarify \cite{CaM,BHM} and other related papers.
Also, generalizations can be expected. For instance, 
the case of
$\la^{\boldsymbol{\flat}}$ seems new. 
The usage of 
(full) DAHA is the key in our approach. 
For torus knots,
our $\mathfrak{H}$ coincide with the EHA
superpolynomials. 
\vskip 0.2cm 

 We are grateful to Pavel Galashin for 
supplying  us with formulas for their {\sf\em 
EHA-superpolynomials} beyond those in \cite{GL2}, 
coincided with our $\mathfrak{H}^{daha}_\la$ (all we 
compared). 
%This helped us to switch to DAHA-based 
%$\mathfrak{H}^{daha}$ from their EHA-based ones. 
Formula-wise, 
their $a,q$ must be replaced by $\imath \aa^{-1/2}, q^{-1}$, and 
then our hat-normalization must be performed, which is 
multiplication by a proper $a^{\bullet} q^{\bullet}$:\, making all 
$q^i t^j a^k$ with non-negative exponents and with the constant 
term $1$. Also, \cite{KTs} contains the definition
of EHA superpolynomials for torus links, and those in the
case of arbitrary  algebraic knots upon  $t=1$, including the
connection of their construction to that from \cite{ChD1}.
%MY[f_] := 
% Collect[Expand[
%   q^(Exponent[f, q])*(a^(-Exponent[f, a])*f /. {a -> I/a^(1/2), 
%       q -> 1/q})], a]

\smallskip
(f) The main purpose of the restriction to $\j^{A_1}$ above, which is for 
$\aa=-t^2$, is to make the verifications simple and direct. Though
Jones polynomials are important from many viewpoints, including
their role in the Volume Conjecture. 
Conceptually, the passage from $A_1$ to  $A_{N-1}$ and from $\om_1$
to arbitrary $\om_m$ is straightforward, but the corresponding
superpolynomials quickly become very big. 

Interestingly,
$\mathfrak{H}_\la^{daha}$ can be non-positive for some $\la$, 
which is conjectured
for DAHA-superpolynomials colored by rectangles.
However, they are superdual in all examples we 
considered; 
superduality of $\h^{daha}$ for algebraic links is a general DAHA 
theorem. Remarkably,
the  superduality of  $\sH^{inst}$ 
fails exactly when the positivity of $\mathfrak{H}_\la^{daha}$
fails, which is our Main Conjecture. 
\vskip 0.2cm

(g) Let us provide the superpolynomial
 for $C\!ab(17,3)T(3,2)$ 
 corresponding  to 
$\la'$ considered above ($m=1$). It is superdual and positive:\ 

\noindent
\smallskip 
 $\h^{daha}=\mathfrak{H}^{daha}=\h^{inst}=$ 
{\small \( 1+q t+q^2 t+q^3 t+q^4 t+q^5 t+q^2 t^2+q^3 t^2+2 q^4 
   t^2+2 q^5 t^2+3 q^6 t^2+2 q^7 t^2+q^8 t^2+q^3
   t^3+q^4 t^3+2 q^5 t^3+3 q^6 t^3+4 q^7 t^3+4 q^8
   t^3+4 q^9 t^3+q^{10} t^3+q^4 t^4+q^5 t^4+2 q^6
   t^4+3 q^7 t^4+5 q^8 t^4+5 q^9 t^4+6 q^{10} t^4+4
   q^{11} t^4+q^5 t^5+q^6 t^5+2 q^7 t^5+3 q^8 t^5+5
   q^9 t^5+6 q^{10} t^5+7 q^{11} t^5+5 q^{12}
   t^5+q^{13} t^5+q^6 t^6+q^7 t^6+2 q^8 t^6+3 q^9
   t^6+5 q^{10} t^6+6 q^{11} t^6+8 q^{12} t^6+5
   q^{13} t^6+q^{14} t^6+q^7 t^7+q^8 t^7+2 q^9
   t^7+3 q^{10} t^7+5 q^{11} t^7+6 q^{12} t^7+8
   q^{13} t^7+5 q^{14} t^7+q^8 t^8+q^9 t^8+2 q^{10}
   t^8+3 q^{11} t^8+5 q^{12} t^8+6 q^{13} t^8+7
   q^{14} t^8+4 q^{15} t^8+q^9 t^9+q^{10} t^9+2
   q^{11} t^9+3 q^{12} t^9+5 q^{13} t^9+6 q^{14}
   t^9+6 q^{15} t^9+q^{16} t^9+q^{10} t^{10}+q^{11}
   t^{10}+2 q^{12} t^{10}+3 q^{13} t^{10}+5 q^{14}
   t^{10}+5 q^{15} t^{10}+4 q^{16} t^{10}+q^{11}
   t^{11}+q^{12} t^{11}+2 q^{13} t^{11}+3 q^{14}
   t^{11}+5 q^{15} t^{11}+4 q^{16} t^{11}+q^{17}
   t^{11}+q^{12} t^{12}+q^{13} t^{12}+2 q^{14}
   t^{12}+3 q^{15} t^{12}+4 q^{16} t^{12}+2 q^{17}
   t^{12}+q^{13} t^{13}+q^{14} t^{13}+2 q^{15}
   t^{13}+3 q^{16} t^{13}+3 q^{17} t^{13}+q^{14}
   t^{14}+q^{15} t^{14}+2 q^{16} t^{14}+2 q^{17}
   t^{14}+q^{18} t^{14}+q^{15} t^{15}+q^{16}
   t^{15}+2 q^{17} t^{15}+q^{18} t^{15}+q^{16}
   t^{16}+q^{17} t^{16}+q^{18} t^{16}+q^{17}
   t^{17}+q^{18} t^{17}+q^{18} t^{18}+q^{19}
   t^{19}+\aa^5 \Bigl(q^{15}+q^{16} t+q^{17} t+q^{17}
   t^2+q^{18} t^2+q^{18} t^3+q^{19} t^4\Bigr)+\aa^4
   \Bigl(q^{10}+q^{11}+q^{12}+q^{13}+q^{14}+q^{11}
   t+2 q^{12} t+3 q^{13} t+3 q^{14} t+3 q^{15}
   t+q^{16} t+q^{12} t^2+2 q^{13} t^2+4 q^{14}
   t^2+5 q^{15} t^2+4 q^{16} t^2+2 q^{17}
   t^2+q^{13} t^3+2 q^{14} t^3+4 q^{15} t^3+5
   q^{16} t^3+4 q^{17} t^3+q^{18} t^3+q^{14} t^4+2
   q^{15} t^4+4 q^{16} t^4+5 q^{17} t^4+3 q^{18}
   t^4+q^{15} t^5+2 q^{16} t^5+4 q^{17} t^5+3
   q^{18} t^5+q^{19} t^5+q^{16} t^6+2 q^{17} t^6+3
   q^{18} t^6+q^{19} t^6+q^{17} t^7+2 q^{18}
   t^7+q^{19} t^7+q^{18} t^8+q^{19} t^8+q^{19}
   t^9\Bigr)+\aa^3 \Bigl(q^6+q^7+2 q^8+2 q^9+2
   q^{10}+q^{11}+q^{12}+q^7 t+2 q^8 t+4 q^9 t+6
   q^{10} t+7 q^{11} t+6 q^{12} t+4 q^{13} t+2
   q^{14} t+q^8 t^2+2 q^9 t^2+5 q^{10} t^2+8 q^{11}
   t^2+12 q^{12} t^2+12 q^{13} t^2+9 q^{14} t^2+4
   q^{15} t^2+q^{16} t^2+q^9 t^3+2 q^{10} t^3+5
   q^{11} t^3+9 q^{12} t^3+14 q^{13} t^3+16 q^{14}
   t^3+12 q^{15} t^3+5 q^{16} t^3+q^{17} t^3+q^{10}
   t^4+2 q^{11} t^4+5 q^{12} t^4+9 q^{13} t^4+15
   q^{14} t^4+17 q^{15} t^4+12 q^{16} t^4+4 q^{17}
   t^4+q^{11} t^5+2 q^{12} t^5+5 q^{13} t^5+9
   q^{14} t^5+15 q^{15} t^5+16 q^{16} t^5+9 q^{17}
   t^5+2 q^{18} t^5+q^{12} t^6+2 q^{13} t^6+5
   q^{14} t^6+9 q^{15} t^6+14 q^{16} t^6+12 q^{17}
   t^6+4 q^{18} t^6+q^{13} t^7+2 q^{14} t^7+5
   q^{15} t^7+9 q^{16} t^7+12 q^{17} t^7+6 q^{18}
   t^7+q^{19} t^7+q^{14} t^8+2 q^{15} t^8+5 q^{16}
   t^8+8 q^{17} t^8+7 q^{18} t^8+q^{19} t^8+q^{15}
   t^9+2 q^{16} t^9+5 q^{17} t^9+6 q^{18} t^9+2
   q^{19} t^9+q^{16} t^{10}+2 q^{17} t^{10}+4
   q^{18} t^{10}+2 q^{19} t^{10}+q^{17} t^{11}+2
   q^{18} t^{11}+2 q^{19} t^{11}+q^{18}
   t^{12}+q^{19} t^{12}+q^{19} t^{13}\Bigr)+\aa^2
   \Bigl(q^3+q^4+2 q^5+2 q^6+2 q^7+q^8+q^9+q^4 t+2
   q^5 t+4 q^6 t+6 q^7 t+8 q^8 t+7 q^9 t+6 q^{10}
   t+3 q^{11} t+q^{12} t+q^5 t^2+2 q^6 t^2+5 q^7
   t^2+8 q^8 t^2+13 q^9 t^2+15 q^{10} t^2+15 q^{11}
   t^2+10 q^{12} t^2+4 q^{13} t^2+q^{14} t^2+q^6
   t^3+2 q^7 t^3+5 q^8 t^3+9 q^9 t^3+15 q^{10}
   t^3+20 q^{11} t^3+24 q^{12} t^3+17 q^{13} t^3+8
   q^{14} t^3+2 q^{15} t^3+q^7 t^4+2 q^8 t^4+5 q^9
   t^4+9 q^{10} t^4+16 q^{11} t^4+22 q^{12} t^4+28
   q^{13} t^4+22 q^{14} t^4+9 q^{15} t^4+2 q^{16}
   t^4+q^8 t^5+2 q^9 t^5+5 q^{10} t^5+9 q^{11}
   t^5+16 q^{12} t^5+23 q^{13} t^5+29 q^{14} t^5+22
   q^{15} t^5+8 q^{16} t^5+q^{17} t^5+q^9 t^6+2
   q^{10} t^6+5 q^{11} t^6+9 q^{12} t^6+16 q^{13}
   t^6+23 q^{14} t^6+28 q^{15} t^6+17 q^{16} t^6+4
   q^{17} t^6+q^{10} t^7+2 q^{11} t^7+5 q^{12}
   t^7+9 q^{13} t^7+16 q^{14} t^7+22 q^{15} t^7+24
   q^{16} t^7+10 q^{17} t^7+q^{18} t^7+q^{11} t^8+2
   q^{12} t^8+5 q^{13} t^8+9 q^{14} t^8+16 q^{15}
   t^8+20 q^{16} t^8+15 q^{17} t^8+3 q^{18}
   t^8+q^{12} t^9+2 q^{13} t^9+5 q^{14} t^9+9
   q^{15} t^9+15 q^{16} t^9+15 q^{17} t^9+6 q^{18}
   t^9+q^{13} t^{10}+2 q^{14} t^{10}+5 q^{15}
   t^{10}+9 q^{16} t^{10}+13 q^{17} t^{10}+7 q^{18}
   t^{10}+q^{19} t^{10}+q^{14} t^{11}+2 q^{15}
   t^{11}+5 q^{16} t^{11}+8 q^{17} t^{11}+8 q^{18}
   t^{11}+q^{19} t^{11}+q^{15} t^{12}+2 q^{16}
   t^{12}+5 q^{17} t^{12}+6 q^{18} t^{12}+2 q^{19}
   t^{12}+q^{16} t^{13}+2 q^{17} t^{13}+4 q^{18}
   t^{13}+2 q^{19} t^{13}+q^{17} t^{14}+2 q^{18}
   t^{14}+2 q^{19} t^{14}+q^{18} t^{15}+q^{19}
   t^{15}+q^{19} t^{16}\Bigr)+\aa
   \Bigl(q+q^2+q^3+q^4+q^5+q^2 t+2 q^3 t+3 q^4 t+4
   q^5 t+5 q^6 t+4 q^7 t+2 q^8 t+q^9 t+q^3 t^2+2
   q^4 t^2+4 q^5 t^2+6 q^6 t^2+9 q^7 t^2+10 q^8
   t^2+9 q^9 t^2+5 q^{10} t^2+2 q^{11} t^2+q^4
   t^3+2 q^5 t^3+4 q^6 t^3+7 q^7 t^3+11 q^8 t^3+14
   q^9 t^3+16 q^{10} t^3+13 q^{11} t^3+5 q^{12}
   t^3+q^{13} t^3+q^5 t^4+2 q^6 t^4+4 q^7 t^4+7 q^8
   t^4+12 q^9 t^4+16 q^{10} t^4+20 q^{11} t^4+19
   q^{12} t^4+9 q^{13} t^4+2 q^{14} t^4+q^6 t^5+2
   q^7 t^5+4 q^8 t^5+7 q^9 t^5+12 q^{10} t^5+17
   q^{11} t^5+22 q^{12} t^5+21 q^{13} t^5+11 q^{14}
   t^5+2 q^{15} t^5+q^7 t^6+2 q^8 t^6+4 q^9 t^6+7
   q^{10} t^6+12 q^{11} t^6+17 q^{12} t^6+23 q^{13}
   t^6+21 q^{14} t^6+9 q^{15} t^6+q^{16} t^6+q^8
   t^7+2 q^9 t^7+4 q^{10} t^7+7 q^{11} t^7+12
   q^{12} t^7+17 q^{13} t^7+22 q^{14} t^7+19 q^{15}
   t^7+5 q^{16} t^7+q^9 t^8+2 q^{10} t^8+4 q^{11}
   t^8+7 q^{12} t^8+12 q^{13} t^8+17 q^{14} t^8+20
   q^{15} t^8+13 q^{16} t^8+2 q^{17} t^8+q^{10}
   t^9+2 q^{11} t^9+4 q^{12} t^9+7 q^{13} t^9+12
   q^{14} t^9+16 q^{15} t^9+16 q^{16} t^9+5 q^{17}
   t^9+q^{11} t^{10}+2 q^{12} t^{10}+4 q^{13}
   t^{10}+7 q^{14} t^{10}+12 q^{15} t^{10}+14
   q^{16} t^{10}+9 q^{17} t^{10}+q^{18}
   t^{10}+q^{12} t^{11}+2 q^{13} t^{11}+4 q^{14}
   t^{11}+7 q^{15} t^{11}+11 q^{16} t^{11}+10
   q^{17} t^{11}+2 q^{18} t^{11}+q^{13} t^{12}+2
   q^{14} t^{12}+4 q^{15} t^{12}+7 q^{16} t^{12}+9
   q^{17} t^{12}+4 q^{18} t^{12}+q^{14} t^{13}+2
   q^{15} t^{13}+4 q^{16} t^{13}+6 q^{17} t^{13}+5
   q^{18} t^{13}+q^{15} t^{14}+2 q^{16} t^{14}+4
   q^{17} t^{14}+4 q^{18} t^{14}+q^{19}
   t^{14}+q^{16} t^{15}+2 q^{17} t^{15}+3 q^{18}
   t^{15}+q^{19} t^{15}+q^{17} t^{16}+2 q^{18}
   t^{16}+q^{19} t^{16}+q^{18} t^{17}+q^{19}
   t^{17}+q^{19} t^{18}\Bigr).
\)
}
%\medskip
%{ }

%\vskip 0.2cm 
\subsection{\bf Conductors for 1-2-cables} \label{sec:condform}
Let us calculate the lifts  $\c$ of the  
conductors $(z^{2\de})$ of  the rings 
$\r=\F_q[[x=z^{\upsilon \rr}+z^{\upsilon \rr+p}, y=z^{\upsilon \ss}]]$. 
Here $\rr>\ss\ge 1, \upsilon\ge 1$, 
$gcd(\rr,\ss)=1=gcd(\upsilon,p)$, and
$\de=\frac{(\mathbbm{m}-1)(\upsilon-1)}{2}+\frac{(\rr-1)(\ss-1)}{2}\upsilon$ for 
$\mathbbm{m}=\upsilon \rr \ss +p$. The corresponding cable is 
$C\!ab(\mathbbm{m}, \upsilon) T(\rr,\ss)$. The answers is relatively
simple if  $\up=1$, or if $p=1$, or when $p\ge 1$ in the limit. 
We use Theorem \ref{thm:general-semiroot-conductor} for
general $p$. We assume that $gcd(q,\ss)=1$ for $\up>1$,
or take any field $\F$
of characteristic zero.
%We assume that $\chi=0$ or   $gcd(s,\chi)=1$ for the 

\begin{theorem}\label{thm:cond}
(i) Due to Theorem 3.4,(ii) in \cite{ChG}, the conductor $\c$ is 
monomial for $p=1$ or when $\upsilon=1$.
 It is then $I_\la$, where $\la$ is 
formed by the boxes above the diagonal from $(i=\upsilon \ss, j=0)$ to 
$(i=0, j=\upsilon \rr)$ in the rectangle $\upsilon \ss\times \upsilon \rr$, 
including 
the boxes that touch the diagonal; the total number 
of such boxes is $\de$, the first generator of $I_\la$ is $x^{\upsilon 
\ss-1}$ and the last one is $y^{\upsilon \rr-1}$. The $\varrho$-rank of 
$I_\la$, the number of its outer corner, equals $\upsilon s$, 
the multiplicity of the corresponding singularity. 

(ii) The following outer corners $(i,j)$
of $\la=dgrm(\c)$ are present  for any 
$p\!\ge\! 1$: $ \bigl(\,i_b= \upsilon \ss -1 -b, 
j_b=Floor[\frac{\rr b}{\ss}]\,\bigr), 
\text{ where } 0\le b\le \ss-1$, The corresponding generators 
are $f_{b+1}= 
x^{\ss-b-1} y^{j_b} (x^\ss-y^\rr)^{\upsilon-1}$\, for\, $0\le b\le \ss-1$ 
modulo higher generators of $\c$. The $j$-ends of the other
rows of $\la$, those for $\,0\le i\le (\upsilon-1)\ss -1$, tend to infinity 
as $p\to \infty$. 
\end{theorem}
{\it Proof.} {\bf (i).} Generally, we represent $z^{2\de+k}$ 
$0\le k\le \upsilon \ss -1$ in terms of $x,y$ and $\frac{1}{\ss}(x^\ss-y^\rr)
= z^{\upsilon \rr \ss +p}\,\mod \bigl(z^{\upsilon \rr \ss +p+1}\bigr)$ for
$\up>1$, and their powers. The reduction to monomial generators 
is straightforward; see \cite{ChG} and below for examples. Related combinatorics
from \cite{GM, GMV1,GMV2} can be used, though lifted conductors were
not considered there.   

{\bf (ii).} Only $y^{a-1}x^{\ss-b-1} (x^\ss-y^\rr)^{\upsilon -1}$ 
for $ 0\le b\le \ss-1$  can 
be with the corresponding leading $x,y$-terms independent of $p\ge 1$.
Any smaller powers of $(x^\ss-y^\rr)$ result in the degrees of
$x$ or $y$ tending to $\infty$ as $p\to \infty$. We obtain the
inequality:
$$ 
2\de\leq  \ss (a-1)+\rr (\ss-b-1)+(\upsilon \rr\ss+p)(s-1)< 2\de +\upsilon \ss.
$$
Using that $2\de=(\upsilon \rr \ss +p -1)(\upsilon-1)+
\upsilon (\rr-1)(\ss-1)$, we reduce it to $1\le \ss a-\rr b\le \ss$ and
arrive at $a=Floor[\frac{\rr b}{\ss}]+1$. This is $j_b$, and
the leading monomial of $f_{b+1}$ is $x^{i_b}y^{j_b}$ for
  $i_b=\upsilon \ss-b-1$. \sq
%includes the top $x$-power  of $(x^\ss-y^\rr)^{\upsilon-1}$. \sq
\vskip 0.2cm

Let us provide the Young diagrams 
$\la=dgrm(\c_m)$ for $\rr=3,\ss=2, \upsilon=2$ and
$p=2m+1=1,3,5,7$:\ \ 
$
\yng(4,3,1),\ \ \yng(5,3,1),\ \ \yng(5,4,1),\ \  \yng(6,4,1)\,,
$
where only $\c_0$ is monomial.
The diagrammatic pattern is clear. Following the
theorem, the generators of $\c_1$ are  
$f_1=x (x^2-y^3), f_2= y(x^2-y^3), f_3=x y^3, f_4=y^5$. In this case,
$x^3=z^{18}(1+z(\cdot))\in \c_1$; it has the same leading
$z$-power as $x y^3$. Thus, $\c_1$ is generated by
$x^3, y(x^2-y^3), xy^3, y^5$. Starting with $\c=\c_2$, 
 $x^3\not\in \c$, and we need to take $f_1=x(x^2-y^3)$.
For $m\ge 2$, the generators are:\ 
$ x^3-xy^2, x^2y-y^4, x y^{a}, y^{b},$
 where $b=Ceiling[\de/2]=
4+Ceiling[m/2]$, $a=3+Floor[m/2]$.
\vskip 0.2cm

This theorem gives that torus knots and the double-cables with
$p=1$ are the only ones with monomial conductors $\c$ among
algebraic $1,2$-cables. This completes Theorem 3.4,(ii) 
in \cite{ChG}.
We think  that the  conductors $\c$ are non-monomial
for any other algebraic knots.

For instance,
let us consider the simplest algebraic ``triple" cable,
which is for $\r=\F[[x=z^{12}+z^{14}+z^{15},y=z^8]]$ with
the knot $C\!ab(53,2)C\!ab(13,2)T(3,2)$. See Section 4.1 in \cite{ChQ}.
In this case, 
$\Ga=\lan 8,12,27,53\ran$,
$\de=42$, and the Newton's pairs are $\{(2,3),(2,1),(2,1)\}$.
The corresponding $\c$ is relatively simple to calculate using
standard Gr\"obner software. Its generators are {\small
\(
y^{11}, x y^9, x^2 y^8, x^3 y^6, x^4 y^4 - 2 x^2 y^7 + y^{10}, x^5 y^3, 
 x^6 y - 3 x^2 y^7 + 2 y^{10}, x^7\)}.
We use that $\nu_z(u)=26$ for $u=x^2-y^3$, and
$ \nu_z(u^2-y^5 x)=\nu_z(u^2-y^2x^3)=53$. It is not monomial;
two of its generators are non-monomial.
 The corresponding Young diagram
$dgrm(\c)$ 
is $\la=(11,9,8,6,4,3,1)$; it is, indeed,  % $=\yng(11,9,8,6,4,3,1)$
with $\de=42$ boxes.
% but $\c$ is not monomial. 
\vskip 0.2cm

It is of interest to calculate the limits of lifted conductors 
for the {\sf\em families} of triple-cables, 
as in Theorem \ref{thm:cond}(ii).
\vskip 0.2cm

%Using the formulas for $\h^{daha}$ from \cite{ChW} for the
%family $r=3, s=2$ and  odd $p$, let us tend $p$ to $\infty$.
%We obtain the following nontrivial formula for {\sf\em free
%instanton sums} (without any $\c$).

\begin{corollary} \label{cor:lim-up2}
Consider relatively prime
$\rr>\ss\ge 1$, $\upsilon=2$, and $p=2m+1$. Then\,
$\de_m=(\rr\!-\!1)(\ss\!-\!1)\!+\!\rr\ss\!+\!m$\, 
is the corresponding $\de$.

(i) We set  $\h^\dag_m(q,t,\aa)\equal
t^{\de_m}\h(q,1/(qt),\aa)$  (uncolored) for the corresponding
DAHA or motivic superpolynomial. Then
\begin{align}\label{H-m-lim}
\lim_{m\to \infty} \h^\dag_m(q,t,\aa)=\h^\dag_0+
\frac{1+\aa q}{1-t}t^{\de_0+1}\h^\#_{\rr,\ss}(2\om_1),
\end{align}
where $ \h^\#_{\rr,\ss}(2\om_1)$ is the superpolynomial
for $T(\rr,\ss)$ colored by $2\om_1=\yng(2)$ upon the
substitution $t\mapsto 1/(qt)$.  

(ii) For $\rr=3,\ss=2,\upsilon=2$, consider 
$\sH^{inst}_{3,2;2}(\aa=0)\equal \sum_{I'\subset \a}
t^{deg(I')}$ subject to the following restriction.
The summation must be  over the ideals $I'$
containing $f_1=x(x^2-y^3)$ and $f_2=y(x^2-y^3)$. Then
$\sH^{inst}_{3,2;2}(\aa=0)=\frac{1+q t^2+q^2 t^3+q^2 t^4}{1-t}$.
When $\upsilon=3$, we have $f_1=x(x^2\!-\!y^3)^2, f_2=y(x^2\!-\!y^3)^2,$ 
and $\sH^{inst}_{3,2;3}(\aa=0)=\frac{(1+q^2t^3)}{(1-t)(1-qt^2)}
(1\!+\!q^3t^4\!+\!q^4t^5\!+\!q^4t^6).$ 
\end{corollary}
{\it Proof.}  The key is our formula
$\sH^{inst}_{\c}(q,t,\aa)=t^{\de}
\h^{mot}(q,1/(qt),\aa)$,  where $\c$ is the corresponding
(lifted) conductor. In this case,\, $\c=\c_m$ and we can use
the DAHA formula (4.23) from \cite{ChW} and its motivic counterpart,
formula (7.60) from \cite{ChQ}. Adjusting them  to $\dag$ and $\#$:
\begin{align}\label{H-m-full}
\h^\dag_m(q,t,\aa)=\h^\dag_0+
(1-t^m)\frac{1+\aa q}{1-t}t^{\de_0+1}\h^\#_{\rr,\ss}(2\om_1).
\end{align}
Recall that  $t^m$ vanishes when making  $m\to \infty$.
\vskip 0.2cm

{\bf (ii).} Let us provide explicit formulas for
$\h^\dag_0$ and $\h^\#_{3,2}(2\om_1)$ in this case, 
borrowed from our prior papers:\, $\h_{3,2}(2\om_1)=$
{\small \(
1 + a^2 q^5 + q^2 t + q^3 t + q^4 t^2 + \aa (q^2 + q^3 + q^4 t + q^5 t)\)},
$\h_0=$
{\small \(1 + q t + q^2 t + q^3 t + q^2 t^2 + q^3 t^2 + 2 q^4 t^2 + q^3 t^3 + 
 q^4 t^3 + 2 q^5 t^3 + q^4 t^4 + q^5 t^4 + 2 q^6 t^4 + q^5 t^5 + 
 q^6 t^5 + q^7 t^5 + q^6 t^6 + q^7 t^6 + q^7 t^7 + q^8 t^8 + 
 \aa^3 \bigl(q^6 + q^7 t + q^8 t^2\bigr) + 
 \aa^2 \bigl(q^3 + q^4 + q^5 + q^4 t + 2 q^5 t + 2 q^6 t + q^5 t^2 + 
    2 q^6 t^2 + 2 q^7 t^2 + q^6 t^3 + 2 q^7 t^3 + q^8 t^3 + q^7 t^4 + 
    q^8 t^4 + q^8 t^5\bigr) + 
 \aa \bigl(q + q^2 + q^3 + q^2 t + 2 q^3 t + 3 q^4 t + q^5 t + q^3 t^2 + 
    2 q^4 t^2 + 4 q^5 t^2 + q^6 t^2 + q^4 t^3 + 2 q^5 t^3 + 
    4 q^6 t^3 + q^7 t^3 + q^5 t^4 + 2 q^6 t^4 + 3 q^7 t^4 + q^6 t^5 + 
    2 q^7 t^5 + q^8 t^5 + q^7 t^6 + q^8 t^6 + q^8 t^7\bigr), \)}
and $\de_0=8$.

\vskip 0.2cm
Set here $\aa=0$. The description of the conductors and the formulas 
for $f_1,f_2$ are from Theorem \ref{thm:cond} and the
 example after it. We omit the justification
in the case of $\upsilon=3$; it is more involved.
 \sq
\vskip 0.2cm

The following corollary partially addresses  the fundamental problem
of topological invariance of motivic superpolynomials. 
 Torus knots (quasi-homogeneous
plane curve singularities) and certain  $2$-cables
were managed in \cite{ChG}. Now we have a proof for
arbitrary algebraic $2$-cables.
%This is for $2$-cables, but the method can be
%applied to any algebraic knots (coupled with the estimates
%in \cite{ChG}).

\begin{corollary}\label{cor:3gen-top}
Assume that the semigroup of irreducible $\r$ is with
$3$ generators (including $0$):\,  
$\Ga=\lan \up\ss, \up\rr,
\up\rr\ss +p\ran$ for $\rr>\ss>0, \up,p>0$ such that  
$gcd(\rr,\ss)=1=gcd(\up,p)$, which is for
$\r=\F_q[[x=z^{\up\rr}+
z^{\up\rr+p}\,g,\, y\!=\!z^{\up\ss}]]$, where $g\in 1+(z)$
and $\ss \neq 0$ in $\F_q$. Then
the corresponding instanton slice $\mathfrak{S}_\c$,
and the superpolynomials $\sH^{inst}_\c$ and $\h^{mot}_\r(q,t,\aa)$
depend only on $\Ga$, i.e. they are topological
invariants. 
\end{corollary}
{\it Proof}. The presentation in such a form is standard.
Following the proof of Theorem 
\ref{thm:cond},(ii) and Appendix \ref{app:mot-top},
let $f_k\in \a$ be {\sf\em any} monomials in terms
of $x,y$ and $h\equal \frac{1}{\ss}(x^{\ss}-y^{\rr})$ such that
$\nu_z(f_k)=k$ for  $k=2\de,\ldots,2\de+\nu_z(y)-1$. 
Then consider the ideal $J$ of formal series in
terms of such $f_k$. 
Theorem 
\ref{thm:explicit-semiroot-conductor} states that
$F\in \a$ such that $F(x,y)=0$ in $\r$ (the equation of
the singularity) belongs to $J$.
Thus, $J=\c$ and is fully determined by
the semigroup $\Ga$, which means that it is a topological
invariant. So are $\mathfrak{S}_\c$,\, $\sH^{inst}_\c$ and $\h^{mot}_\r$. 
\sq

Theorem \ref{thm:general-semiroot-conductor} proves that $F\in J$
for any unibranch plane curve singularities, 
but this is insufficient for the justification
of  topological invariance of
$\h_\r^{mot}$. We note that it 
was proven in \cite{KTs} for $\aa=0,t=1$.
Generally, $\a/\c$  may be non-isomorphic as abstract
Artin algebras for rings $\r$ with coinciding $\Ga$; 
an example is (\ref{two-parametrizations}) 
in Appendix \ref{app:mot-top}.

\subsection{\bf Weak Riemann Hypothesis} \label{sec:rh-hyper}
Concerning number theory aspects, let us extend the 
 {\sf\em Weak Riemann Hypothesis} from \cite{ChW} to the case of
$\la$-generated hyperbolic knots. Conjecture 4.9 there for $\aa=0$ 
in the
uncolored case and for algebraic knots $K$ was as follows.
For $\H_K(q,t)=\h^{daha}_K(q t, t,\aa\!=\!0)$, consider
its $t$-zeros. Riemann Hypothesis is
when all of them satisfy the relations $|t|=q^{-1/2}$. Conjecturally:
\vskip 0.2cm

$(A)$\, RH holds for sufficiently small $q>0$ (called Weak RH);

$(B)$\, RH holds for any (uncolored) algebraic knots if  $q\le  1/2$. 
\smallskip

Part $(A)$ was justified in \cite{ChQ} for the corresponding
$\h^{mot}$ (conjecturally coinciding with  DAHA
ones).  Consider now  
$\mathfrak{H}^{daha}_\la$ and  
put $\H^0_\la(q,t)\equal 
\mathfrak{H}^{daha}_\la(qt,t,\aa\!=\!0)$,
 $\H^\ast_\la(q,t)\equal 
\mathfrak{H}^{daha}_\la(qt,t,\aa\!=\!-t/q)$.
\smallskip

First of all, $\H_\la^0$  appeared
{\sf\em cyclotomic} for any $2$-row $\la$ in vast experiments.
See Section 5.4.3 in \cite{ChW} 
concerning this fact for torus knots and its connection to
the {\sf\em Shuffle Conjecture} (now a theorem). They are 
not cyclotomic for algebraic non-torus cables. 

Then, ChatGPT proved $(A)$ for superdual  $\sH_\la^{inst}$
for $2$-row $\la$   
 using Section \ref{sec:2row}
and the  proof of ``Weak RH" in \cite{ChQ}.
The evidence is solid for the following conjecture.
We consider any $\la$, not only admissible.

\begin{conjecture}\label{conj:rh-hyper}
Given $\lambda$, there exists $q_0$ such that
$|t|=q^{-1/2}$ provided $q<q_0$ for
all $t$-roots of $\,\H^0_\la(q,t)$ and  $\H^\ast_\la(q,t)$ 
(i.e. RH holds).
 \sq
\end{conjecture}

Part $(B)$ for $\H^0_\la$ fails for some hyperbolic $\la$-knots. 
See the table below,
where the relevant threshold $q_0(\la)$ is the upper endpoint
of the connected RH-interval issuing from $q=0$; possible accidental
recovery near $q=1$ is disregarded.   All three $\H^0_\la$ satisfy
superduality (conjecturally, for any $\la$).
 The bound $0.5$ fails for the $1${\footnotesize st} one:

\begin{table}[ht]
\centering
\begin{tabular}{c|c|c|c}
\hline
$\lambda$ & positivity & last verified RH & verified failure \\
\hline
$(6,3,2)$   & yes & $q=0.475$ & $q=0.476$ \\
$(6,4,3,3)$ & no  & $q=0.635$ & $q=0.640$ \\
$(6,5,3,1)$ & yes & $q=0.810$ & $q=0.820$ \\
\hline 
\end{tabular}   
\caption{Experimental bounds for RH.}
\label{tab:daha-rh-experimental-bounds}
\end{table}

More exactly, $\H$ at $\aa=0,-t/q$ conjecturally becomes  
$P^0=1+u^2+u^4+\cdots u^{2\de}$ and
$P^\ast=1+u^{2(\de-1)}$ for $\de>1$
as $u=\sqrt{q}t$ and  $q\to 0$. 
 The first claim simply means
that %(conjecturally)  
$t^\de \mathfrak{H}^{daha}(q,1/(qt),\aa)=
1+t+\cdots+t^\de \mod (q)$ for any $\la$. By the way, 
the coefficient of $q$ here (the next one) 
is conjecturally $\sum_{i=2}^{\de-1} t^i +\aa\sum_{i=1}^{\de} t^i$ 
for any $\la$ unless $\la=1^m$ or $(m)$.

\comment{
Let $\sH^{daha}_\la(q,t,\aa)\equal
t^\de \mathfrak{H}^{daha}_\la(q, 1/(qt),\aa)=P_0(t)+P_1(t)q+\ldots$\ .
We conjecture that $P_0=1+t+\cdots+t^\de$ and 
$P_1=t+\cdots +t^{\de-1}$ for any $\la$ and $\de=|\la|$,  
which gives Weak RH for $\aa\!=\!0$ and $\aa\!=\!-t/q$.
This matches Theorem 5.4 from \cite{ChQ} (for any plane curve
singularities).
}
\subsection{\bf Mixed characteristic}\label{sec:mixed}
Let us conclude this paper with a (direct) counterpart of
our constructions in the case of zero characteristic, and the
corresponding connection conjecture. In this case,
$\r$ is a subring (an order) in a $p$-adic Dedekind
domain with $2$ generators, to match plane curve
singularities. We state  
our conjecture 
only for  $y=p\in \Z_p$, 
 $\mathscr R=\Z_p[[x]]$ in the corresponding $\mathbb O=\Z_p[[p^{1/m}]]$,
where $gcd(m,p)=1$. The confirmations are solid. The basic examples
are as follows.

Let  $\r_{\rr,\ss}=\F_q[[x=z^{\rr},
y=z^{\ss}]]$ for $gcd(\rr,\ss)=1$ and  
$\r_{\rr,\ss,\upsilon, \rho}=\F_q[[x=z^{\up\rr}+z^{\up\rr+\rho}, y=z^{\up\ss}]]$
for $gcd(\rho,\up)=1$. Then their counterparts are:\,
 $\mathscr{R}_{\rr,\ss}=\Z_p[[x=p^{\rr/\ss}]]\subset \mathbb O=\Z_p[[p^{1/\ss}]]$ and 
$\mathscr{R}_{\rr,\ss,\upsilon,\rho}=
\Z_p[[x=p^{\rr/\ss}+\rho^{\rr/\ss+1+1/(\up\ss)}]]
\subset \mathbb O=\Z_p[[p^{1/(\up\ss)}]]$.

As in $char\!=\!p$, the passage to instanton slices is via the lift $\mathfrak{C}$ of
the conductor of $\mathscr{R}$ to $\mathbb A=\Z_p[[x]].$ The monomial ideals 
$I_\lambda\subset \F_q[[x,y]]$ become now $\mathbbm{I}_\la$:\,  
we replace  $x^i y^j\mapsto x^i p^j$.
 
The definitions of mixed
$\mathbbm{H}^{mot}_\mathscr{R}$ and $\mathbbm{H}^{inst}_\mathfrak{C}$ 
are the same as in
$char\!=\!p$, but now  $|\mathbb O/M|=p^{deg(M)} $ for
(standard) $\mathscr{R}$-modules $M\subset  \mathbb O$ 
and $|\mathbb A/\i|=p^{deg(\i) }$
for ideals $\mathfrak{C}\subset \i\subset \mathbb A$. As in 
$char\!=\!p$, $\mathbbm{H}^{mot}_\mathscr{R}(q,t,\aa)=$  
$(q t)^\de \mathbbm{H}_\mathfrak{C}^{inst} (q,t=1/(qt), \aa)$. 
We focus on  $\mathbbm{H}^{inst}.$
Replacing 
the lifted conductor
by an arbitrary $\mathbbm{I}_\la\subset \Z_p[[x]]$, we obtain 
mixed  $\mathbbm{H}^{inst}_\la$. 

The simplest ``hyperbolic" example is the following 
counterpart of K12n242
in $char\!=\!0$. The ideal $I_\la=\lan x^2, xy^2,y^3\ran\subset
 \F_q[[x,y]]$ 
for $\la=(3,2)$
becomes $\mathbb{I}_\la=\lan x^2, p^2x, p^3\ran\subset \Z_p[[x]]$.
Its superpolynomial $\mathbbm{H}_\la^{inst}$ coincides with
$\sH^{inst}_\la$. For instance, the hyperbolic counterpart
of the total number of standard 
$\mathscr{R}$-modules $M$ in $\mathbb O$ is given by 
$q^\de \sH^{inst}(q,1/q,\aa=0)=
1+q +2q^2 +2q^3 + 2q^4  +q^5$, where $\de=5$.
 The total number of ideals
$\mathbb{I}_\la\subset \i\subset \Z_p[[x]]$ is $6+3q$,
different of course.

\begin{conjecture} \label{conj:mixed}
As above, 
%go  from $\r\subset \o$ to $\mathscr{R}\subset \mathbb O$,
go from the ideals $I_\la\subset \i\subset \F_q[[x,y]]$ 
for free variables $x,y$ to
the ideals  $\mathbbm{I}_\la \subset \i\subset \Z_p[[x]]$.
%$\r$ is with monomial $\c$,  
 Then:\,

%$ \h_\r^{mot}(q,t,\aa)=\mathbbm{H}_\mathscr{R}^{mot}(q,t,\aa) \text{\, and\, }
\centerline{
$\mathscr{H}_\la^{inst}(q,t,\aa)=\mathbbm{H}_\la^{inst}(q,t,\aa),$}

\noindent
where $p$ in the mixed $\mathbbm{H}_\la^{inst}$
is assumed sufficiently large, and
$\la$ can be arbitrary (we do not require the superduality). \sq
\end{conjecture}

\comment{
We focus in this work on the $2${\footnotesize nd} equality;
the $1${\footnotesize st} conjecturally follows  from it. 
We note that the lifted conductor of
 $\mathscr{R}$ was assumed  {\sf\em  monomial} mostly because this is where
we experimented 
(joint with {\sf\em ChatGPT}). However, any lift of formulas
or equations for $x,y$ from $char\!=\!p$ to those in
$char\!=\!0$ may potentially result in deviations due
to the usage of Witt vectors. Let us discuss this a bit.
}
\smallskip

The $2$-row case, 
including the $\aa$-refinement, follows from
the  cell-dimension formula proved in
Section~\ref{sec:2row}. In contrast to the $2$-row case,
Conjecture~\ref{conj:mixed} does not hold at the level of individual
cells for three-row diagrams, but only after summing the contributions
of all cells.

The case of  all Young diagrams with at most
$3$ rows when $\aa=0$ is a theorem; see Appendix \ref{app:mixed}.
There are also preliminary confirmations for $4$-row
diagrams. Generally, the conjecture is well-checked. 
\medskip

Let us list some Young diagrams $\la$ beyond the $3$-row case
where this conjecture (with $\aa$) was confirmed:\ 
%$(3,2)$, the simplest hyperbolic $2$-row diagram corresponding
%to K12n242,   $(4,3,1)$, corresponding to the simplest algebraic cable 
%$C\!ab(13,2)T(3,2)$, 
$\la=(5,3,2,1)$, $\la=(5,4,3,1)$, and
$(3,3,2,2,1,1,1)$, corresponding to (non-algebraic)
 $C\!ab(19,2)T(5,2)$.

\medskip
Here one can consider an irreducible Eisenstein polynomial
$P$ over $\Z_p$, and the ring $\mathbb O=\Z_p[[z]]$, where $z$ is its root.
Then pick $x,y$ in the maximal ideal of $\mathbb O$ such that they generate
the whole $\Q_p[z]$ and consider a 
ring $\mathscr{R}=Z_p[[x,y]]\subset \mathbb O$. The 
superpolynomials will be our usual sums of
$t^{deg(M)}(1+\aa q^{\cn})\cdots(1+\aa q^{\rho(M)-1})$ over
{\sf\em standard} $\r$-invariant $M\subset \o^\cn$; by standard,
we mean that $M\mathbb O=\mathbb O^\cn$. Finally, lift the conductor
of $\mathscr{R}$ in $\mathbb{O}$ to $\mathfrak{C}$
in $\mathbb A=\Z_p[[y]][[x]]$, where $x$ is now considered 
a free variable, and
defined the instanton sums over $\mathbb A$-modules 
$\mathfrak{C}^\cn\subset \m\subset \mathbb A^\cn$. In contrast to $char\!=\!p$, 
the 
rings $\Z_p[[y]]\subset \mathbb O$ are non-trivial, and
such a theory can become significantly more involved. 

We experimented a lot with conductors $\mathbb{I}_\la$, but
did not consider so far more general  $\Z_p[[y]]$. The procedure
becomes less canonical, and the
lift of formulas
or equations from $char\!=\!p$ to those in
$char\!=\!0$ may potentially result in alternations due
to the impact of  Witt vectors.
%\smallskip

\comment{
Here ``RH'' means that every zero in $t$ has absolute value
$q^{-1/2}$.  The relevant threshold $q_0(\lambda)$ is the upper endpoint
of the connected RH interval issuing from $q=0$; possible accidental
recovery near $q=1$ is disregarded.  Thus the computations give
\[
\begin{array}{c}
  0.475\leq q_0(6,3,2)<0.476, \\
  0.635\leq q_0(6,4,3,3)<0.640, \\
  0.810\leq q_0(6,5,3,1)<0.820.
\end{array}
\]
The numerical root tests were performed at high precision.  For the
larger examples they were additionally stabilized by setting
$t=q^{-1/2}u$ and rewriting the reciprocal polynomial in
$x=u+u^{-1}$; the RH condition then says that all its $x$-roots are real
and belong to $[-2,2]$.  Notice especially that the example
$\lambda=(6,4,3,3)$ is not positive (its smallest coefficient is $-8$),
although it satisfies superduality and the RH property throughout the
initial interval displayed above.
}

\medskip
Generally speaking, {\sf\em Witt vectors} are almost
inevitable here, but they are not employed 
in the code. The argument based on DVR normal forms and passage
to the associated graded was developed instead
with substantial mathematical
and computational assistance from OpenAI's ChatGPT/Codex.
 However, the corrections  due to
Witt-addition and Witt-multiplication, especially for
non-trivial $\Z_p[[y]]\subset \mathbb O$, are actually close
to the levels when they can interfere with counting the
parameters of the Gr\"obner cells, at least for sufficiently
big diagrams. 
\medskip

%\vskip 0.2cm
{\bf Work with ChatGPT.}
Generally, we began by translating to {\sf\em SageMath} and 
{\sf\em Singular} and adapting some of the author's programs on 
instanton slices. They were written mostly in {\sf\em Mathematica}. 
 This was a substantial challenge because of the 
complexity of the original code and the differences among the 
underlying computational systems. ChatGPT then substantially optimized 
and accelerated these programs, frequently introducing alternative 
mathematical and computational approaches. The author remained 
actively involved throughout, in particular by running the programs, 
analyzing their outputs, supplying independent calculations using
his own programs, and with constant 
theoretical checks. In several important cases, ChatGPT 
designed and implemented entirely new programs, 
sometimes amazingly efficient.  
\medskip

Almost all significant findings
were cross-checked using 
independent implementations and extensive numerical testing. 
Agreement among genuinely different codes was regarded 
as the principal standard of computational validation. Working 
with ChatGPT made it possible to 
maintain such a high level of verification, including 
theoretical aspects.

\medskip
{\bf Acknowledgements.}
I are very grateful to Pavel Galashin 
and Oscar Kivinen for various discussions
and providing material beyond their published/posted
papers. Many thanks to Giovanni Felder, Vassily Gorbunov,
David Kazhdan, Nikita Nekrasov and  Andras Szenes. I thank
 MPIM (Bonn) for the kind invitation (a significant part of
this paper was written there), SKKU and KIAS (Korea), and
Einstein Institute of Mathematics (Jerusalem) for hospitality. 
\medskip

The author thanks OpenAI's ChatGPT/Codex for substantial 
contributions, including its participation in 
%the three-row
%comparison between the instanton superpolynomials over $\F_q$ and
%their mixed-characteristic counterparts, and  
Appendix \ref{app:mot-top}, the key
for Corollary \ref{cor:3gen-top}
on topological invariance of certain motivic
superpolynomials, and the Instanton Volume Conjecture. The author 
assumes full
responsibility for the corresponding statements. The underlying
arguments and computations were checked theoretically
and against various different programs, including 
the author's own code.

%\eject

%\includepdf[pages=-,lastpage=77]{rh-reduced.pdf} %%ONLY PDFLATEX

\vskip -2cm
%\medskip
\bibliographystyle{unsrt}

\begin{appendices}
%\appendix
\section{\bf Second Gr\"obner order}\label{app:sec-Gr}
\begin{theorem}\label{thm:second-order-sum}
There exists a map $\mathcal U$ on 
multi-diagrams such that 
$dim\, G^{\,e\prec x\prec y}_{\vec{\la}}$ from
(\ref{dim-f-second}) coincides with 
$dim\, G^{\,x\prec e\prec y}_{\mathcal{U}(\vec{\la})}$
from (\ref{dim-f}). It is  defined in
(\ref{U-second}) below. In the notation there, let
\begin{align}\label{admissible-second}
\Lambda^\cn_\ell
=\{\vec\la:\ \ell\,(\mathcal U(\vec\la))\leq\ell\},\ \,
\vec{\la}=(\la^1,\ldots,\la^\cn),
\end{align}
where $\ell(\vec\nu)=\max_\ga\ell(\nu^\ga)$ is the maximum number
of rows.  For arbitrary $\cn\geq1$ and $\ell\geq1$, the free
instanton sum at $\aa=0$ for the ordering $e\prec x\prec y$
summed  over $\Lambda^\cn_\ell$ results in:
\begin{align}\label{prod-free-second}
&\mathscr H^{e\prec x\prec y}_{\La^\cn_\la}
(q,t,\aa=0;\cn)
\equal\sum_{\vec\la\in\Lambda^\cn_\ell}
q^{\,dim\, G^{\,e\prec x\prec y}_{\vec\la}}\, t^{|\vec\la|}\\
\label{identity-second}
&=\sum_{\vec{\la}\in\Lambda^\cn_\ell}
q^{\,\cn\sum_\be|\la^\be|^\dag+
\sum_{\al,\be}N_{\al\be}(\vec{\la})}\
t^{|\vec{\la}|}
=\prod_{i=1}^{\ell}\prod_{j=1}^{\cn}
\frac{1}{1-q^{\cn i-j}t^i}.
\end{align}
\end{theorem}

{\it Direct proof.}
We construct a weight-preserving bijection which transforms the
dimension statistic in (\ref{dim-f-second}) into that in
(\ref{dim-f}).  Index the rows of every diagram from $0$.  For a
positive integer $p$, put
\begin{align}\label{height-second}
h_\al(p)=\#\{k:\ \la^\al_k\geq p\};
\end{align}
this is the height of the $p${\scriptsize th} column of $\la^\al$.

Consider row $i$ of $\la^\be$, and let $p=\la^\be_i$ be its
length.  We assign to this row the color
\begin{align}\label{color-second}
c(\be,i)=\#\Bigl\{\al:\ &h_\al(p)<i,
\ \hbox{or}\notag\\[-2mm]
&h_\al(p)=i\ \hbox{ and }\al<\be\Bigr\}.
\end{align}
Since $h_\be(p)>i$,  $\be$ itself is not counted, and
$0\leq c(\be,i)\leq\cn-1$.

For $0\leq c\leq\cn-1$, collect all rows of color $c$, arrange
their lengths in decreasing order, and use them as the rows of a
Young diagram $\nu^{c+1}$.  This defines
\begin{align}\label{U-second}
\mathcal U(\la^1,\ldots,\la^\cn)
=(\nu^1,\ldots,\nu^\cn).
\end{align}
The map $\mathcal U$ only recolors rows.  Therefore,
\begin{align}\label{preserve-second}
&\sum_\be|\la^\be|=\sum_\ga|\nu^\ga|,\,
\sum_\be\ell(\la^\be)=\sum_\ga\ell(\nu^\ga),\,
\sum_\be|\la^\be|^\dag
=\sum_\ga|\nu^\ga|^\dag.
\end{align}
It also preserves the multiset of row lengths.  It need not preserve
$\max_\be\ell(\la^\be)$, which is why the transported
condition (\ref{admissible-second}) is necessary.

The terminal square of row $i$ in $\la^\be$ is
$s=(i,p-1)$.  By the definition of the extended leg-length,
\[
l_\al(s)+1=h_\al(p)-i.
\]
Comparing (\ref{Nab-second}) and (\ref{color-second}), we obtain
\begin{align}\label{N-color-second}
\sum_{\al,\be}N_{\al\be}(\vec{\la})
=\sum_{\be,i}c(\be,i)
=\sum_{\ga=1}^{\cn}(\ga-1)\ell(\nu^\ga).
\end{align}

We recall the elementary stable-ranking lemma used here.  Suppose
that the rows longer than $p$ have already been fixed, and let
$L_\al$ be their number in component $\al$.  If $m_\al$ rows of
length $p$ are added, then $h_\al(p)=L_\al+m_\al$.  Give a new row
in position $i$, $L_\be\leq i<L_\be+m_\be$, the stable rank is
\[
\#\{\al:\ h_\al(p)<i,\ \hbox{or }h_\al(p)=i
\ \hbox{ and }\al<\be\}.
\]
For fixed $(L_1,\ldots,L_\cn)$, the multiplicities of these stable
ranks determine $(m_1,\ldots,m_\cn)$ uniquely.  Indeed, this is
the inverse of stable sorting of the column heights, with equal
heights resolved by the component index.  Applying this observation
successively to the row lengths, from the greatest to the smallest,
recovers $\vec{\la}$ from $\vec{\nu}$.  Thus $\mathcal U$ is a
bijection.

Using (\ref{preserve-second}) and (\ref{N-color-second}), we obtain
\begin{align}\label{dimensions-U-second}
dim\, G^{\,e\prec x\prec y}_{\vec{\la}}
&=\cn\sum_\be|\la^\be|^\dag
+\sum_{\al,\be}N_{\al\be}(\vec{\la})\notag\\
&=\cn\sum_\ga|\nu^\ga|^\dag
+\sum_\ga(\ga-1)\ell(\nu^\ga)
=dim\, G^{\,x\prec e\prec y}_{\vec{\nu}}.
\end{align}
Since $\mathcal U$ is bijective, preserves $|\vec{\la}|$, and maps
$\Lambda^\cn_\ell$ onto the set
$\{\vec\nu:\ell(\vec\nu)\leq\ell\}$, the left-hand side of
(\ref{identity-second}) equals
\[
\sum_{\ell(\vec\nu)\leq\ell}
q^{\cn\sum_\ga|\nu^\ga|^\dag+
\sum_\ga(\ga-1)\ell(\nu^\ga)}t^{|\vec{\nu}|}.
\]
This is the sum for the ordering $x\prec e\prec y$.  Formula
(\ref{dim-f}) gives its product expansion
\[
\prod_{i=1}^{\ell}\prod_{j=1}^{\cn}
\frac{1}{1-q^{\cn i-j}t^i},
\]
which proves (\ref{prod-free-second}). \qed

%\appendix
\section{\bf Mixed conjecture for 3 rows}\label{app:mixed}
We prove here the equal--mixed-characteristic comparison when the
Young diagram has at most three rows and $\aa=0$.  It is convenient to use
one notation for both cases.  Let $\d$ be a complete discrete valuation
ring, {\sf\em DVR},
 with uniformizer $\pi$ and residue field $\F_q$.  For
\[
\la=(A,B,C),\  A\geq B\geq C\geq0,
\]
the ring of the corresponding instanton slice is defined as
\begin{align}\label{three-DVR-ring}
\mathcal A_\d(\la)=\d[[x]]/
 (x^3,\,\pi^C x^2,\,\pi^B x,\,\pi^A).
\end{align}
If $C=0$, the summand generated by $x^2$ is absent; thus the same
notation includes the one- and two-row cases.

For $\d=\F_q[[y]]$ and $\pi=y$, this definition becomes
\begin{align}\label{three-equal-ring}
\F_q[[x,y]]/I_{(A,B,C)}=
\F_q[[x,y]]/(x^3,x^2y^C,xy^B,y^A).
\end{align}
The corresponding 
mixed-characteristic quotient
is for $\d=W(\F_q)$, the {\sf\em Witt ring over $\F_q$},
  and $\pi=p$.  The case $q=p$ is of particular interest: 
\begin{align}\label{three-mixed-ring}
\Z_p[[x]]/\mathbbm I_{(A,B,C)}=
\Z_p[[x]]/(x^3,p^C x^2,p^B x,p^A).
\end{align}
Finally, the degree of an ideal $L$,
$deg(L)$, is uniformly determined from
the relation $|\mathcal A_\d(\la)/L|=q^{\deg(L)}$.

\begin{theorem}\label{thm:three-mixed}
Let $\la=(A,B,C)$ have at most three rows.  Then for every finite field
$\F_q$ and $W(\F_q)$ taken as the (complete) DVR,
\begin{align}\label{three-inst-equality}
\mathscr H^{inst}_\la(q,t,0)
=\mathbbm H^{inst}_\la(q,t,0).
\end{align}
Explicitly:
\begin{align}\label{three-count-equality}
&\sum_{I_\la\subset \i\subset\F_q[[x,y]]}
t^{dim_{\F_q}(\F_q[[x,y]]/\i) }=
\sum_{\mathbbm I_\la\subset\i\subset W(\F_q)[[x]]}
t^{\log_q|W(\F_q)[[x]]/\mathbbm \i|}.
\end{align}
Thus the $2${\footnotesize nd} equality in Conjecture~\ref{conj:mixed} holds at
$\aa=0$ for all Young diagrams with at most three rows.
\end{theorem}

{\it Proof.}
We prove a slightly stronger statement.  Work first
over an arbitrary complete $(\d,\pi)$, DVR with residue field $\F_q$
and uniformizer $\pi$.
An ideal in $\mathcal A_\d(\la)$ is a $\d$-lattice
$L\subset \d^3$, containing
\begin{align}\label{three-conductor-lattice}
L_\la=\pi^A \d e_0\oplus\pi^B \d e_1\oplus\pi^C \d e_2,
\end{align}
and invariant under the following action of $x$:
\begin{align}\label{three-x-action}
xe_0=e_1,\  xe_1=e_2,\  xe_2=0.
\end{align}
Generally, $e_i\equal x^i$.  The multiplication by $y$ 
in $char\!=\!p$ and multiplication by $p$ in $char\!=\!0$
are part of  the $\d$-module structure.

Every such lattice has a unique upper-triangular DVR normal form
\begin{align}\label{three-HNF}
&f_0=\pi^a e_0+u e_1+v e_2,\ \,
f_1=\pi^b e_1+w e_2,\ \, 
f_2=\pi^c e_2, \text{ where }\\
\label{three-HNF-ranges}
&0\leq a\leq A,\,  0\leq b\leq B,\,  0\leq c\leq C,
\  u\ mod\, {\pi^b},\text{ and }  v,w\ mod\, {\pi^c}.
\end{align}
The index of $L$ in $\d^3$ is $q^{a+b+c}$, and it contributes
$t^{a+b+c}$ to $\sH^{inst}(\aa=0)$. 

We now write all the conditions on (\ref{three-HNF}).  Invariance
under $x$ gives
\begin{align}\label{three-x-conditions}
a\geq b\geq c,\ 
u\equiv\pi^{a-b}w\pmod {\pi^c}.
\end{align}
Indeed, reduction of $xf_0$ by $f_1,f_2$ gives the last congruence,
and reduction of $xf_1$ by $f_2$ gives $b\geq c$.

Containment of the last two generators in
(\ref{three-conductor-lattice}) gives
\begin{align}\label{three-conductor-12}
\pi^{B-b}w\equiv0\pmod {\pi^c},\ 
c\leq C.
\end{align}
For the first generator, put
\begin{align}\label{three-r0}
r=A-a.
\end{align}
Reduction of $\pi^Ae_0$ by $f_0,f_1,f_2$ gives precisely
\begin{align}\label{three-conductor-0}
\pi^r u&\equiv0\pmod {\pi^b},\notag\\
\pi^r v-\left(\frac{\pi^r u}{\pi^b}\right)w
&\equiv0\pmod {\pi^c}.
\end{align}
The quotient in the second line belongs to $\d$ because of the first
line.  Conversely, (\ref{three-x-conditions}),
(\ref{three-conductor-12}), and (\ref{three-conductor-0}) imply
$xL\subset L$ and $L_\la\subset L$.  Thus they are necessary and
sufficient.

Fix $(a,b,c)$.  We claim that the number of triples $(u,v,w)$ in
(\ref{three-HNF-ranges}) satisfying these congruences depends only
on $q$ and on the six integers $A,B,C,a,b,c$.  Namely, choose $w$
modulo $\pi^c$ subject to the first congruence in
(\ref{three-conductor-12}), and then choose $u$ modulo $\pi^b$
subject to
\begin{align}\label{three-u-system}
u\equiv\pi^{a-b}w\pmod {\pi^c},\ 
\pi^r u\equiv0\pmod {\pi^b}.
\end{align}
%Both steps count elements in ideals and cosets of the finite chain
%rings $\d/\pi^m \d $; their cardinalities are powers of $q$ determined
%only by the indicated valuations.

Also such $(u,w)$, must satisfy the following congruence:
\begin{align}\label{three-v-equation}
\pi^r v\equiv
h(u,w)\pmod {\pi^c},\ 
h(u,w)=\left(\frac{\pi^r u}{\pi^b}\right)w.
\end{align}
It is solvable exactly when
\begin{align}\label{three-v-solvability}
h(u,w)\in \pi^{\min(r,c)}\d/\pi^c \d;
\end{align}
if solvable, it has exactly $q^{\min(r,c)}$ solutions modulo
$\pi^c$.  The solvability condition is again only a valuation
condition on the product of $u$ and $w$.  The numbers $\#$
of elements
of each valuation in $\d/\pi^m \d$ is:\, 
\begin{align*}%\label{three-valuation-count}
\#=1\, \hbox{ for the zero class},\text{ and \,}
\#=(q-1)q^{m-d-1}\, \hbox{ for valuation }d<m.
\end{align*}
Therefore the number
\begin{align}\label{three-profile-count}
N^{a,b,c}_{A,B,C}(q)
=\#\{(u,v,w):\text{(\ref{three-x-conditions})--
(\ref{three-conductor-0}) hold}\}
\end{align}
is a polynomial in $q$ independent of the choice of $\d$, and 
\begin{align}\label{three-universal-polynomial}
\sum_{L_\la\subset L,\,xL\subset L}t^{\deg(L)}
=\sum_{\substack{0\leq c\leq C,\,c\leq b\leq B\\
b\leq a\leq A}}
N^{a,b,c}_{A,B,C}(q)t^{a+b+c}.
\end{align}
The right-hand side is the same for $\d=\F_q[[y]]$ and for
$\d=W(\F_q)$.  Using (\ref{three-equal-ring}) and
(\ref{three-mixed-ring}), we obtain (\ref{three-count-equality}).
The definition of the instanton polynomial at $\aa=0$ now gives
(\ref{three-inst-equality}). \sq

\smallskip
{\bf Comments.}
Generally, the strata in (\ref{three-HNF}) not always
coincide with the standard Gr\"obner cells used earlier in the
paper. The equality holds only after summing all strata of a 
fixed degree. Also, adding $\aa$ to our theorem is not straightforward.
The Nakayama number
$dim_{\F_q}L/(\pi,x)L$ is not
constant on  Gr\"obner cells.

\section{\bf Calculating conductors}\label{app:mot-top}

% Final insert.  The notation \r, \F, \rr, \ss and \upsilon is assumed
% to have been introduced in the surrounding text.

\begin{theorem}\label{thm:explicit-semiroot-conductor}
Assume that
\[
 y=z^{\upsilon\ss},\ 
 x=z^{\upsilon\rr}\bigl(1+z^p+\epsilon(z)\bigr),\ 
 \operatorname{ord}_z\epsilon>p,
\]
where
\(
 \rr>\ss,\ \ \gcd(\rr,\ss)=1,\ 
 \gcd(p,\upsilon)=1.
\)
In positive characteristic $\chi$, assume
$\gcd(\chi,\ss)=1$; there is no restriction in characteristic zero.
Put
\[
 h=x^{\ss}-y^{\rr}
\]
and let $J\subset\F[[x,y]]$ be generated by the generalized monomials
$x^ay^bh^c$ whose $z$-valuations are at least $2\delta$, where
\begin{equation*}%\label{explicit-semiroot-delta}
 2\delta=\upsilon^2\rr\ss-\upsilon(\rr+\ss)
 +(\upsilon-1)p+1.
\end{equation*}
If $F(x,y)$ is the reduced equation of $\r$, then $F\in J$, which 
conclusion is independent of the higher series $\epsilon(z)$.
Accordingly, it suffices to pick one monomial $f_k$ such that
$\nu(f_k)=k$ for $k=2\de,\ldots,2\de+\nu(y)-1$; 
formal series in terms of them will give the whole $J$. 
\end{theorem}

{\it Proof.}
We have
\begin{align}
 h(z)
 &=z^{\upsilon\rr\ss}
 \left(\bigl(1+z^p+\epsilon(z)\bigr)^{\ss}-1\right) \notag\\
 &=\ss z^{\upsilon\rr\ss+p}
 +\text{terms of higher $z$-order}.
 \label{explicit-semiroot-h-value}
\end{align}
Thus:\,
\begin{equation}\label{explicit-semiroot-basic-values}
 \nu(x)=\upsilon\rr,\ 
 \nu(y)=\upsilon\ss,\ 
 \nu(h)=\upsilon\rr\ss+p,
\end{equation}
where $\nu=\operatorname{ord}_z$.  The assumption on $\chi$ guarantees
that the leading coefficient $\ss$ in
\eqref{explicit-semiroot-h-value} is nonzero.

Normalize $F$ as a monic Weierstrass equation in $x$.  Its $x$-degree
is $\upsilon\ss$, while $h$ is monic of $x$-degree $\ss$.  Repeated
monic division by $h$ gives the unique formal expansion
\begin{equation}\label{explicit-semiroot-expansion}
 F=h^{\upsilon}+
 \sum_{\substack{0\leq a<\ss,\ 0\leq c<\upsilon\\ b\geq0}}
 C_{abc}x^ay^bh^c.
\end{equation}
When $F$ is a polynomial (which can be arranged),
only finitely many $C_{abc}$ are nonzero.  For
an arbitrary formal $\epsilon(z)$, the sum can be infinite, but it is
$(x,y)$-adically convergent and only finitely many terms have valuation
below any prescribed bound.

The valuation of a monomial $x^ay^bh^c$ is
\begin{equation}\label{explicit-semiroot-standard-value}
 w(a,b,c)=a\upsilon\rr+b\upsilon\ss
 +c(\upsilon\rr\ss+p).
\end{equation}
We claim that these valuations are pairwise distinct in the
summation above, which is for 
\(
 0\leq a<\ss,\ 
 0\leq c<\upsilon,\ 
 b\geq 0.
\)
Indeed, suppose that
\[
 w(a,b,c)=w(a',b',c').
\]
Reduction modulo $\upsilon$ gives
\[
 (c-c')p\equiv0\pmod{\upsilon}.
\]
Since $\gcd(p,\upsilon)=1$ and $0\leq c,c'<\upsilon$, we obtain
$c=c'$.  Canceling the $c$-terms gives the exact integer equality
\[
 \upsilon\bigl((a-a')\rr+(b-b')\ss\bigr)=0.
\]
Since $\upsilon>0$, this implies
\[
 (a-a')\rr+(b-b')\ss=0.
\]
Reduction modulo $\ss$, together with $\gcd(\rr,\ss)=1$ and
$0\leq a,a'<\ss$, gives $a=a'$, and then $b=b'$.  Notice that 
$2\delta$ does not occur here.

Let us use now that $F(x,y)=0$. We set
\begin{equation}\label{explicit-semiroot-D}
 D=\nu(h^{\upsilon})
 =\upsilon(\upsilon\rr\ss+p).
\end{equation}
 If
some term with $C_{abc}\neq0$ in
\eqref{explicit-semiroot-expansion} had valuation smaller than $D$, 
then the term with the smallest valuation is unique such and
could not be canceled, contradicting $F(x,y)=0$.

Therefore
\begin{equation}\label{explicit-semiroot-vanishing-below-D}
 C_{abc}\neq0\ 
 \Longrightarrow\ 
 w(a,b,c)\geq D.
\end{equation}
At value $D$, exactly one  monomial occurs in $\Si$; there is
at most one by distinctness; it triggers the
cancelations in $\F[[z]]$. Finally,
\begin{align*}
 D-2\delta
 &=\upsilon(\upsilon\rr\ss+p)
 -\bigl(\upsilon^2\rr\ss-\upsilon(\rr+\ss)
 +(\upsilon-1)p+1\bigr) \notag\\
 &=\upsilon(\rr+\ss)+p-1>0.
% \label{explicit-semiroot-D-conductor}
\end{align*}
Every generalized monomial actually occurring in $F$, including
$h^{\upsilon}$, therefore has value at least $D>2\delta$,
which proves $F\in J$.

Note that $\epsilon(z)$ can make $\Si$ infinite. Then all monomials
there will have distinct valuations. The argument involving
$D$ requires only $z^{\upsilon\rr+p}$.  \sq

%\begin{example}\label{ex:explicit-semiroot-C2}
{\bf Examples.}
The simplest family here is 
\[
 x=z^6+z^{7+2k},\ 
 y=z^4,\ 
 h=x^2-y^3,
\]
where  $\upsilon=2$, $\rr=3$, $\ss=2$, $p=2k+1$, and
\begin{align*}
 F(x,y)
 &=x^4-2x^2y^3-4xy^{k+5}+y^6-y^{2k+7} \notag
% &=(x^2-y^3)^2-4xy^{k+5}-y^{2k+7} \notag\\
=h^2-4xy^{k+5}-y^{2k+7}.
% \label{explicit-semiroot-C2-equation}
\end{align*}
The valuations are
\[
 \nu(h^2)=\nu(xy^{k+5})=26+4k,\ \ 
 \nu(y^{2k+7})=28+8k,
\]
where $2\delta=16+2k$.  This first two valuations coincide
(the must due to $F(x,y)=0$), but all others (only one here)
are greater indeed.

% Higher-tail example for the semiroot discussion.
% The notation is compatible with the surrounding text: \F is the base field,
% \r is the branch ring, and \nu denotes the z-adic valuation.

\vskip 0.2cm
A more instructive variation of the case $k=0$  
is as follows. 
Assume that $\operatorname{char}(\F)\ne 2$, and consider
\[
 y=z^4,\ \ x=z^6+z^7+z^8=z^6(1+z+z^2),
 \ \ \r=\F[[x,y]]\subset\F[[z]].
\]
Elimination of $z$ gives 
\(
 F(x,y)={}x^4-4x^3y^2+6x^2y^4-2x^2y^3-4xy^6
             +y^8+y^7+y^6=0.                % \label{higher-tail-F}
\)
This is %\eqref{higher-tail-F} is 
the resultant
\(
 \operatorname{Res}_z\bigl(z^4-y,\ z^6+z^7+z^8-x\bigr).
\)
As above, 
\(
 h=x^2-y^3.
\)

Division by the monic polynomial $h=x^2-y^3$ results in:\,
\begin{align}
 F={}&h^2-4xy^2h+6y^4h-4xy^6-4xy^5+y^8+7y^7=0.
\label{higher-tail-h}
\end{align}
The generalized monomials occurring in \eqref{higher-tail-h}, their
coefficients, and their $z$-valuations are as follows:
\[
\begin{array}{c|rrrrrrr}
 M&h^2&xy^2h&y^4h&xy^6&xy^5&y^8&y^7\\ \hline
 C_M&1&-4&6&-4&-4&1&7\\
 \nu(M)&26&27&29&30&26&32&28   
\end{array}  \, .                      %\tag{C.\!HT}
\]

\comment{
Here
\[
 h=2z^{13}+3z^{14}+2z^{15}+z^{16},\ \ \nu(h)=13.
\]
For completeness, direct substitution gives
\begin{align*}
 h^2&=4z^{26}+12z^{27}+17z^{28}+16z^{29}
       +10z^{30}+4z^{31}+z^{32},\\
 -4xy^2h&=-8z^{27}-20z^{28}-28z^{29}-24z^{30}
       -12z^{31}-4z^{32},\\
 6y^4h&=12z^{29}+18z^{30}+12z^{31}+6z^{32},\\
 -4xy^6&=-4z^{30}-4z^{31}-4z^{32},\\
 -4xy^5&=-4z^{26}-4z^{27}-4z^{28},\\
 y^8&=z^{32},\ \ 7y^7=7z^{28}.
\end{align*}
Consequently their sum is identically zero.  The first cancelation is
between $h^2$ and $-4xy^5$, both of valuation $26$; after this cancelation,
the coefficients in degrees $27,\ldots,32$ cancel successively as well.
Thus the higher term $z^8$ in $x$ produces a genuinely nontrivial 
cancelation
chain.  
}
The remainder monomials
\(
 xy^2h,\ y^4h,\ xy^6,\ xy^5,\ y^8,\ y^7
\)
have pairwise distinct valuations (as it is supposed to be),
and valuation $26$ of $xy^5$ coincides with that 
of the leading term $h^2$ (must occur too).

\vskip 0.2cm

{\bf Any plane curve singularities}.
We now come to the general statement.  
Characteristic values $\beta_i$, 
standard integers $n_i$, and (a complete
system of) semiroots $f_i$ will be used. See
%\cite{AbhyankarMoh73,AbhyankarExpansion}.  
\cite{AM},Part I, especially the construction by generalized 
Tschirnhausen transformations, and Part II 
for the successive approximate-root expansion.
Also, see in \cite{Cam} the chapters devoted to characteristic 
sequences, approximate roots,
standard expansions, and the value semigroup.

The steps are quite similar to the case we began with.
Now all  successive semiroots are used, not only $y,x,h$, and 
%Theorem~\ref{thm:explicit-semiroot-conductor} 
%is that of the single nontrivial semiroot $h$.  
the direct justification of
the distinctness of valuations of monomials in the
corresponding  $\Si$
is replaced by the standard
Ap\'ery uniqueness for the telescopic value semigroup.  The
minimum-valuation argument remains unchanged.

\begin{theorem}\label{thm:general-semiroot-conductor}
Let $\r\subset \o=\F[[z]]$ be a ring of
 irreducible plane curve singularity over $\F$ with
reduced equation $F(x,y)$.  In positive characteristic $\chi$, we 
assume that 
\(
 \gcd(\chi,n_1n_2\cdots n_g)=1;
\);
there is no restriction in characteristic zero.  
Let $J\subset \a=\F[[x,y]]$ be
generated by the generalized monomials in terms of  complete system
of semiroots (including $y,x$)
whose $z$-valuations are at least $2\delta_\r$.  Then 
$F\in J.$ Thus,
it suffices to pick one monomial $f_k$ such that
$\nu(f_k)=k$ for $k=2\de,\ldots,2\de+\nu(y)-1$;
formal series in terms of them will give $J$. 
%\begin{equation}\label{general-semiroot-membership}
\end{theorem}

{\it Proof.}
The standard semiroot expansion has the form
\begin{equation}\label{general-semiroot-expansion}
 F=f_g^{n_g}+\sum_M C_M M,
\end{equation}
where $M$ runs through the corresponding remainder monomials.  Semiroot
theory gives two facts we are going to use:
\begin{equation}\label{general-semiroot-two-facts}
 \nu(f_g^{n_g})=n_g\beta_g,\ \ 
 M\neq M'\ \Longrightarrow\ \nu(M)\neq\nu(M').
\end{equation}
The second assertion is equivalent to the uniqueness of the bounded
Ap\'ery representatives for the telescopic value semigroup.

Put
\(D\equal n_g\beta_g.\) Some deviation from the
case above is that the monomials in $\Si$
may have valuations below
$D$.  However, we claim that none occurs in $F$ with a nonzero 
coefficient.  Indeed, if
such terms occurred, choose one term of the 
smallest value $m<D$.  By
\eqref{general-semiroot-two-facts}, it would be then the
 unique term of valuation
$m$ in \eqref{general-semiroot-expansion}. Since the leading term
$f_g^{n_g}$ has value $D>m$, the coefficient of $z^m$ could 
not cancel after the $z$-substitution, contradicting $F(x,y)=0$.
Thus, 
\begin{equation}\label{general-semiroot-bound-by-D}
C_M\neq0\ \Longrightarrow\ \nu(M)\geq D.
\end{equation}

It remains only to compare $D$ with $2\de$, which is routine.
Now:
\begin{equation}\label{general-semiroot-conductor-formula}
 2\delta
 =\sum_{i=1}^g(n_i-1)\beta_i-\beta_0+1.
\end{equation}
Therefore
\begin{align}
 D-2\delta
 &=n_g\beta_g-
 \sum_{i=1}^g(n_i-1)\beta_i+\beta_0-1 \notag\\
 &=\beta_g+\beta_0-1
 -\sum_{i=1}^{g-1}(n_i-1)\beta_i.
 \label{general-semiroot-D-minus-conductor}
\end{align}
For $i<g$, the characteristic values satisfy
\(
 \beta_{i+1}>n_i\beta_i,
\)
and, therefore, 
\(
 (n_i-1)\beta_i<\beta_{i+1}-\beta_i.
\)
If $g\geq2$, we arrive at:
\[
 \sum_{i=1}^{g-1}(n_i-1)\beta_i
 <\sum_{i=1}^{g-1}(\beta_{i+1}-\beta_i)
 =\beta_g-\beta_1,
\]
which results in 
%Substitution in \eqref{general-semiroot-D-minus-conductor} yields
\begin{equation*}%\label{general-semiroot-strict-bound}
 D-2\delta>\beta_0+\beta_1-1>0.
\end{equation*}
When $g=1$, the sum in
\eqref{general-semiroot-D-minus-conductor} is empty and
\[
 D-2\delta=\beta_0+\beta_1-1>0.
\]
Thus \(
 D=n_g\beta_g>2\delta.
\) in all cases, and every
generalized monomial occurring in $F$, including
$f_g^{n_g}$, has valuation greater than $2\delta$ due to
\eqref{general-semiroot-bound-by-D}.  Hence $F\in J$. \sq

\comment{
\begin{example}\label{ex:general-semiroot-smallest-gap}
For the cusp
\[
 x=z^3,\ 
 y=z^2,\ 
 F=x^2-y^3,
\]
we have $D=\nu(x^2)=6$ and $2\delta=2$.  Thus
\[
 D-2\delta=4.
\]
The leading term $x^2$ and the unique standard remainder monomial
$-y^3$ have the common value $6$.  This is the smallest possible gap
for a singular plane branch: when $g=1$ the formula above gives
$D-2\delta=\beta_0+\beta_1-1\geq4$, while additional characteristic
values make the inequality strict.
\end{example}

\begin{remark}\label{rem:general-semiroot-differences}
The formal standard expansion and its exponent bounds alone do not
exclude monomials of small valuation.  What eliminates them from the
defining equation is the combination of two facts: standard remainder
valuations are pairwise distinct, and $F(x(z),y(z))=0$.  After this
elimination, the purely combinatorial inequality
$n_g\beta_g>2\delta$ completes the proof.  This separation is useful:
the first theorem proves distinctness directly by congruences, while in
the general theorem it is supplied by the standard semiroot--Ap\'ery
theory.
\end{remark}
}

\vskip 0.2cm
{\bf Does $\Ga$ always determine the slice?}
The following example demonstrates that the
answer is ``no". It suffices to take $\F=\Q$.
For 
\( \o=\Q[[z]],\ \ \a=\Q[[X,Y]]\), let $ y=z^{12}$,
$\r_i=\Q[[x_i,y]]\subset\o$, where 
\begin{align}\label{two-parametrizations}
 &x_0=z^{18}+z^{20}+z^{55},\ 
 x_1=z^{18}+z^{20}-\frac12z^{22}+z^{55}.
 \end{align}
The Puiseux characteristic exponents of both branches are
\begin{equation*}%\label{common-characteristic-exponents}
 (\beta_0,\beta_1,\beta_2,\beta_3)=(12,18,20,55).\notag
\end{equation*}
Indeed, the successive gcd's are
\begin{equation*}%\label{gcd-chain}
 e_0=12,\ \ e_1=\gcd(12,18)=6,\ \ e_2=2,\ \ e_3=1.\notag
\end{equation*}
The additional exponent $22$ occurring in $x_1$ does not change 
this.
%chain, since $\gcd(12,18,20,22)=2$.

Let $n_i=e_{i-1}/e_i$.  Then
\(
 (n_1,n_2,n_3)=(2,3,2).
\) and the minimal generators of $\Ga$ are  
$\overline\beta_0=12,\ \overline\beta_1=18$,
\begin{align*}%\label{semigroup-generators-calculation}
\overline\beta_2
 &=n_1\overline\beta_1+\beta_2-\beta_1
   =2\cdot18+20-18=38,\notag\\
 \overline\beta_3
 &=n_2\overline\beta_2+\beta_3-\beta_2
   =3\cdot38+55-20=149.
\end{align*}
Thus,
$\Ga(\r_0)=\Ga(\r_1)=\lan 0,12,18,38,149\ran$, and
the conductor $(z^c)$ is the same for both:\, 
%\begin{align}\label{conductor-calculation}
 \(c=\sum_{i=1}^{3}(n_i-1)\overline\beta_i
      -\overline\beta_0+1
 =(2-1)18+(3-1)38+(2-1)149-12+1=232 =2\de\)
for $\de=116/.$ We will prove that the corresponding
Artin algebras  $A_i=\r_i/(z^{2\de})$ are non-isomorphic.
%A direct enumeration of $\lan12,18,38,149\ran$ independently gives
%largest gap $231$, conductor $232$, and exactly $116$ gaps.

Let $\mathfrak{m}_i$ be the maximal ideal of $A_i$ and
\(
 \Omega_i
\)
the modules of K\"ahler differentials $\Om^1_{\r_i}$.
Define the maximal-ideal Hilbert series by
\begin{equation}\label{omega-hilbert-series-definition}
 H_{\Omega_i}(u)
 =\sum_{j\geq0}h_{i,j}u^j,
 \ \ h_{i,j}=dim_{\Q}
 \frac{\mathfrak{m}_i^j\Omega_i}{\mathfrak{m}_i^{j+1}\Omega_i}.
\end{equation}
The exact linear algebra calculation gives that 
$\sum_j h_{i,j}=119$ for $i=0,1$, but  
$H_{\Omega_1}(u)-H_{\Omega_0}(u)=u^{15}-u^{16}$
and 
\(
 dim_{\Q}\operatorname{Der}_{\Q}(A_0,A_0)=214
 \neq215=
 dim_{\Q}\operatorname{Der}_{\Q}(A_1,A_1)
\), which proves that these rings are non-isomorphic.

\comment{
We record the exact finite linear-algebra calculation.  Let
$g_{i,1},\ldots,g_{i,12}$ be the Gr\"obner basis obtained above.  In
the $116$-element standard-monomial basis of $A_i$, the presentation
of the differential module is
\begin{equation}\label{omega-presentation}
 A_i^{12}\longrightarrow A_i^2\longrightarrow\Omega_i\longrightarrow0,
\ \ (a_r)_r\longmapsto
 \sum_{r=1}^{12}a_r
 \left(\partial_Xg_{i,r},\partial_Yg_{i,r}\right).
\end{equation}
All products in \eqref{omega-presentation} are reduced exactly modulo
$J_i$.  The derivations are computed independently as the kernel of
\begin{align}\label{derivation-kernel}
 A_i^2&\longrightarrow A_i^{12},\notag\\
 (u,v)&\longmapsto
 \left(
 (\partial_Xg_{i,r})u+(\partial_Yg_{i,r})v
 \right)_{r=1}^{12}.
\end{align}
Finally, adjoining to the image in \eqref{omega-presentation} the two
copies of $\m_i^j\subset A_i$ gives the filtration in
\eqref{omega-hilbert-series-definition}.

Exact Gaussian elimination over $\Q$ gives
\begin{center}
\begin{tabular}{c|c|c|c|c}
 & $dim A_i$ & rank in \eqref{omega-presentation}
 & $dim\Omega_i$ & $dim\operatorname{Der}_{\Q}(A_i,A_i)$\\ \hline
$A_0$ & $116$ & $33$ & $199$ & $214$\\
$A_1$ & $116$ & $33$ & $199$ & $215$
\end{tabular}
\end{center}
}

\section{\bf Multibranch weights}\label{app:multweights}
Continuing Sections \ref{sec:multi},\ref{sec:sup-links},  
 \(\a=\F_q[[x,y]],\ \ 
 F=\prod_{i=1}^{\kappa}F_i,\ \ 
 \r=\a/(F) \), 
where $F_i$ are pairwise non-proportional irreducible polynomials.
power series. The ring 
\(
 \o=\bigoplus_{i=1}^{\kappa}\o_i e_i,\ \ 
 \o_i=\F_q[[z_i]]
\)
for the idempotents $e_i (1\le i\le \kappa)$ is
the normalization of $\r$.  We identify $\r$ with its image in
$\o$.  Let $\de=dim_{\F_q}(\o/\r)$,
 $\de_i$ be the delta-invariant of the $i$-th branch. 
Then
\begin{equation}\label{eq:intersection-number}
 \de
 =\sum_{i=1}^{\kappa}\de_i+
 \sum_{1\leq i<j\leq\kappa}\iota_{ij} \text{ for }
 \iota_{ij}
 =dim_{\F_q}\frac{\a}{(F_i,F_j)}
 \ \ (i<j).
\end{equation}

The conductor of $\r$ is $\mathfrak{c}=\r:\o
=\bigoplus_{i=1}^{\kappa}z_i^{n_i}\o_i e_i,$
where $n_i=2\de_i+\sum_{j\ne i}\iota_{ij}$.
See 
\cite{Delgado}, and Section 7.1 in \cite{ChQ}
for some basics.

We will need the following {\sf\em preconductor} and
{\sf\em boundary space}:
\begin{align}
 &K\equal\r:\mathfrak{m}_{\o}
 =\oplus_{i=1}^{\kappa}z_i^{n_i-1}\o_i e_i
 \label{eq:K} \text{ for } 
\mathfrak{m}_{\o}=\oplus_{i=1}^{\kappa}z_i\o_i e_i\\
\label{eq:E}
 &E\equal K/\mathfrak{c}=\oplus_{i=1}^{\kappa}E_i,\ \ 
 E_i=z_i^{n_i-1}\o_i e_i/
 z_i^{n_i}\o_i e_i\, \simeq\, \F_q.
\end{align}
Let $\epsilon_i$ be the class of $z_i^{n_i-1}e_i$ in $E_i$. We set: 
\begin{equation}\label{eq:boundary-map}
 \partial:K/\mathfrak{c}\longrightarrow E,\ \ 
 \partial(f)=\sum_{i=1}^{\kappa}\alpha_i\epsilon_i
\text{ for }
f_i=\alpha_i z_i^{n_i-1}+O(z_i^{n_i}).
\end{equation}

We mainly  follow now classical \cite{Rosenlicht1952}. 
Choose
a generator $\omega$ of the dualizing module of $\r$
(we only need that it is Gorenstein). Then
\begin{equation}\label{eq:omega-leading}
 \omega_i=
 \left(\lambda_i z_i^{-n_i}
 +O(z_i^{-n_i+1})\right)dz_i,
 \  \lambda_i\in\F_q^* \text{ (on the normalization)}.
\end{equation}
The regular differentials are characterized by the
{\sf\em residue-balancing condition} on
the normalization, which is
\begin{equation}\label{eq:residue-relation}
 \sum_{i=1}^{\kappa}
 \operatorname {Res}_{z_i=0}(f_i\omega_i)
 =\sum_{i=1}^{\kappa}\lambda_i\alpha_i=0
 \ \ \text{for }f\in\r\cap K.
\end{equation}

\begin{proposition}\label{prop:P}
The image
$P\equal \partial(\r\cap K)=(\r\cap K)/\mathfrak{c} \subset E$
is 
\begin{equation}\label{eq:P-equation}
 P=
 \Bigl\{\,\sum_i\alpha_i\epsilon_i\in E
 \ \mid\ 
 \sum_i\lambda_i\alpha_i=0\,\Bigr\}.
\end{equation}
In particular, $dim_{\F_q}P=\kappa-1$, and
$P\cap E_i=0$ for every $i$.
\end{proposition}
{\it Proof.}
Equation \eqref{eq:residue-relation} gives the inclusion of
$P$ into the hyperplane \eqref{eq:P-equation}. 
The {\sf\em reciprocity map} for $\r$ gives
\(
 dim_{\F_q}(\r\cap K)/\mathfrak{c}=\kappa-1,
\)
which is the  boundary form of the symmetry theorem in
\cite{Delgado}.  The two spaces therefore coincide.  Every
$\lambda_i$ is nonzero by \eqref{eq:omega-leading}, so this hyperplane
contains no coordinate line. \sq

Finally, for every $i<j$, the corresponding
{\sf\em boundary line} is 
the following one-dimensional subspace of
$(\r\cap K)/\mathfrak{c}\subset\r/\mathfrak{c}$:
\begin{equation}\label{eq:Lij}
 L_{ij}\equal P\cap(E_i\oplus E_j)
 =\F_q(\lambda_j\epsilon_i-\lambda_i\epsilon_j).
\end{equation}

%{\bf Lifting to \mathversion{bold}{$\a$}.}
% and defining
%\texorpdfstring{$\mathscr B_{ij}$}{Bij}.}
\def\sB{{\mathscr B}}
\def\sI{{\mathscr I}}
Let
\(
 \pi:\a\longrightarrow\r
\)
and 
$\sC=\pi^{-1}(\mathfrak{c})\subset\a
$
be the lifted conductor.  Choose any lifts $b_{ij}$ of 
$\lambda_j\epsilon_i-\lambda_i\epsilon_j$ to $\a$ modulo
$\mathfrak{c}$; it spans $L_{ij}$.

Since $(F)=\ker\pi\subset\sC$, the classes of $b_{ij}$ modulo 
$\sC$ are
independent of the lifts.
Their $\pi$-images are characterized by the
conditions:
\begin{align}
 (b_{ij})_i
 =\alpha_i z_i^{n_i-1}\!\!&\mod{z_i^{n_i}\o_i},\  
 (b_{ij})_j
 =\alpha_j z_j^{n_j-1}\!\!\mod{z_j^{n_j}\o_j},
 \label{eq:bij-conditions}\\
 &(b_{ij})_{k}
 \in z_\ell^{n_\ell}\o_kl
 \ \ (k\notin\{i,j\}).
\notag
\end{align}
Then $\lambda_i\alpha_i+\lambda_j\alpha_j=0$ for
$ \alpha_i\alpha_j\ne0,$ which fixes
$\{\la_i,\la_j\}\mod \F_q^*$.

The induced lifted line is 
\begin{equation}\label{eq:Lij-lift}
 \widetilde L_{ij}
 =\F_q\bigl(b_{ij}+\sC\bigr)\subset\a/\sC.
\end{equation}
The canonical isomorphism $\a/\sC\simeq\r/\mathfrak{c}$ identifies
$\widetilde L_{ij}$ with $L_{ij}$; in particular,
$dim_{\F_q}\widetilde L_{ij}=1$. This gives an alternative
approach to  $b_{ij}$. 

\comment{
\begin{remark}\label{rem:no-individual-lifts}
The individual classes $\epsilon_i$ belong to $E$ but not to $P$.
Consequently there is generally no $b_i\in\a$ whose normalization is
$z_i^{n_i-1}e_i$ modulo $\mathfrak{c}$.  Such elements exist in the normalization
or the branch-localized ring, but only the compatible differences
\[
 \lambda_jz_i^{n_i-1}e_i-\lambda_i z_j^{n_j-1}e_j
\]
descend to $\r$ and admit lifts to $\a$.
\end{remark}
}

We arrive at
\(
 \sB_{ij}\equal \sC+\F_q b_{ij}\!\mod \sC,
\) one-dimensional ideals in $\a/\sC$ due to  
$\mathfrak{m}_\a\, b_{ij}\subset \sC$, which can be identified
with $L_{ij}$.

\comment{
The notation $\sB_{ij}$ in this note always means the minimal ideal
\eqref{eq:Bij}, not a larger ideal associated with the full
two-branch contact algebra.  If such a larger pairwise ideal is used,
only its projection to the last layer $K/\mathfrak{c}$ is relevant here.

\begin{remark}\label{rem:no-equation-recovery}
The ideals $\sC$ and $\sB_{ij}$ are
$\m_{\a}$-primary: their quotients in $\a$ are finite-dimensional.
They record only the finite conductor-boundary arrangement.  Although
explicit representatives $b_{ij}$ may contain products of some
$F_\ell$, the boundary quotients \eqref{eq:Bijbar} do not canonically
recover the height-one ideal $(F_1\cdots F_\kappa)$.
\end{remark}
}

\vskip 0.2cm
{\bf Recovery of \mathversion{bold}{$w_M$}.}
For standard  $M\subset\o$ 
such that $M\supset\r$, let
\(
 I=M^\hollowstar=\r:M.\) Then $ \mathfrak{c}\subset I\subset\r,$ and
we set $\sI\equal\pi^{-1}(\i) \sC.$
For $U\subset[\kappa]=\{1,\ldots,\kappa\}$, let:
\begin{equation}\label{eq:BU}
 \sB_U=
 \sum_{i<j,\  i,j\in U}\sB_{ij}=P_U\equal P\cap E_U
\text{ for } E_U=\oplus_{i\in U}E_i.
\end{equation}
Also, $\sB=\sB_{[\kappa]}=\sum_{i<j}\sB_{ij}=P.$ Indeed,
the lines $L_{ij}$ span $P$, and those with
$i,j\in U$ span the hyperplane $P\cap E_U$. Given the lift 
$\sI$  of $I=M^\hollowstar$ considered modulo $\sC$, let
$W_I\equal \sI\cap\sB
 = I\cap K \!\mod\mathfrak{c}\subset P$, where we use the
canonical identifications above. Finally, we define:
\[r_I\!=\!
 dim_{\F_q}\sI\!\cap\!\sB\!=\!dim_{\F_q} W_I, \ \,
 s_U(I)\!=\!
 dim_{\F_q}\sI\!\cap\!\sB_U\!=\!
dim_{\F_q} W_I\!\cap\! E_U.
\]

Recall that 
\[
 \overline M=(M+\mathfrak{m}_{\o})/\mathfrak{m}_{\o},\ \ 
 \upsilon_M=dim_{\F_q}\overline M,\ \,
 i_{\overline M}
 =
 \left|
 \bigl(\overline M\cap(\F_q^*)^\kappa\bigr)/\F_q^*
 \right|.
\]

\begin{theorem}\label{thm:weight}
For $U\subset[\kappa]$, set
$D_U(I)=\kappa-r_I-|U|+s_U(I).$
Then
\begin{align}\label{eq:upsilon}
 &\upsilon_M=D_{\varnothing}(I)=\kappa-r_I,\ 
 i_{\overline M}
 =\frac1{q-1}
 \sum_{U\subset[\kappa]}(-1)^{|U|}q^{D_U(I)}.
\end{align}
Accordingly,
\begin{equation}\label{eq:weight}
 w_M=
 \frac{(q-1)^{\kappa-1}}
 {q^{\kappa-\upsilon_M}\,i_{\overline M}}=
 \frac{(q-1)^\kappa}
 {q^{r_I}
 \displaystyle\sum_{U\subset[\kappa]}
 (-1)^{|U|}
 q^{\,\kappa-r_I-|U|+s_U(I)}}.
\end{equation}
Thus the weight $w_M$ from
(\ref{M-one}) is recovered entirely from $I=M^\hollowstar$ through the
intersections of $\sI$ with sums of the boundary ideals
$\sB_{ij}$.
\end{theorem}
{\it Proof.} One has:
\(
 dim\frac{E}{W_I+E_U}
 =\kappa-r_I-|U|+s_U(I)=D_U(I).
\)
Gorenstein duality identifies this dimension with the dimension of
the subspace of $\overline M$ on which all coordinates in $U$ vanish.
Taking $U=\varnothing$ gives \eqref{eq:upsilon}.  
Then we employ the Inclusion--Exclusion Principle 
over the coordinate hyperplanes, and  
substitute into
{\large \(
 w_M=
 \frac{(q-1)^{\kappa-1}}
 {q^{\kappa-\upsilon_M}\,i_{\overline M}}.
\)}
\sq

For instance, let $\kappa=3$ and 
\(
 \varepsilon_{ij}(I)
 \equal dim_{\F_q}\sI\cap\sB_{ij}\in\{0,1\}.
\)
Then
\[
 D_{\varnothing}=3-r_I,\ \ 
 D_i=2-r_I,\ \ 
 D_{ij}=1-r_I+\varepsilon_{ij},\ \ 
 D_{123}=0, \text{ and }
\]
\begin{equation}\label{eq:weight-three}
 \displaystyle
 w_M=
 \frac{(q-1)^3}
 {q^{r_I}\left(
 q^{3-r_I}-3q^{2-r_I}
 +\sum_{i<j}q^{1-r_I+\varepsilon_{ij}}-1
 \right)}.
\end{equation}
%The number $r_\i$  requires the intersection with the full sum $\sB$.
%The three individual binary intersections alone do not determine the
%weight.

\vskip 0.2cm
{\bf Torus links.} Let
$F_i=x^{\ss_i}+ c_i y^{\rr_i},\ \ 
 \gcd(\rr_i,\ss_i)=1,$
under the assumption that all $F_i$  are irreducible
over $\F_q$ and pairwise non-proportional. Then
\(
 x_i=z_i^{\rr_i},\ \ y_i=z_i^{\ss_i},\ \ 
 2\de_i=(\rr_i-1)(\ss_i-1), 
\) and 
 $\iota_{ij}=\min(\rr_i\ss_j,\rr_j\ss_i)$ for $i\ne j$.
Define the monomials
\begin{equation}\label{eq:hij}
 h_{ij}=
 \begin{cases}
 x^{\ss_i-1}y^{\rr_j-1},
 &\ss_i\rr_j\le \ss_j\rr_i,\\[2mm]
 x^{\ss_j-1}y^{\rr_i-1},
 &\ss_j\rr_i<\ss_i\rr_j,
 \end{cases}
\end{equation}
where
\(
 \nu_i(h_{ij})=2\de_i+\iota_{ij}-1,\ 
 \nu_j(h_{ij})=2\de_j+\iota_{ij}-1.
\)

Setting
$G_{ij}=\prod_{k\notin\{i,j\}}F_k$, we can 
pick $b_{ij}=G_{ij}h_{ij}$. 
Indeed,
\[
 \nu_i(G_{ij})
 =\sum_{k\notin\{i,j\}}\iota_{ik},
\]
and similarly on branch $j$, whereas $G_{ij}$ vanishes on every
other branch.  Hence \eqref{eq:bij-conditions} holds.

The terms $h_{ij}$ are monomial. However, 
$b_{ij}$ are monomial only for $\kappa=2$ (if $c_i\neq 0$); $G_{ij}$ do not 
occur then. Generally, $F_i$ can be with lower terms. The corresponding
singularities may have non-trivial deformations. Thus, we obtain 
the following corollary:

\begin{corollary}
The superpolynomials
$\h^{mot}$ and $\sH^{inst}$ are {\sf\em topological invariants}
for   $\kappa=2$ if $\Gamma_1$ and $\Gamma_2$ are with $3$ generators
(for non-proportional $F_{1,2}$ as above). \sq
\end{corollary}
\vskip 0.2cm

{\bf The double-trefoil example.}
Assume $\operatorname {char}\F_q\ne2$ and put
\(
 F_1=x^2-y^3,\ \ F_2=x^2+y^3.
\)
Both branches have semigroup $\langle 0,2,3\rangle$ and
$\de_1=\de_2=1$.  Their intersection number is
\(
 \iota_{12}=dim_{\F_q}\frac{\a}{(x^2,y^3)}=6,
\)
and
\(
 n_1=n_2=8,\  \de=8.
\)
The lifted conductor and the boundary class are
\begin{equation}\label{eq:double-C}
 \sC=(x^3,x^2y,xy^3,y^4),\ 
 b_{12}=xy^2.
\end{equation}
Indeed, $xy^2$ has valuation $7=n_i-1$ on both branches.  Thus,
\( \sB_{12}=\sC+(xy^2)
 =\lan x^3,x^2y,xy^2,y^4\ran\) and 
the boundary space $P=L_{12}$ is one-dimensional. One has:
\begin{equation}\label{eq:double-weights}
 w_M=
 \begin{cases}
 1,
 &\sI\cap\sB_{12}=\sC,\\[1mm]
 (q-1)/q,
 &\sI\cap\sB_{12}=\sB_{12}.
 \end{cases}
\end{equation}

To compare this case with the example of double trefoil
from \cite{ChQ}, let $\varrho=2$. Recall that
\(
 deg(I)=dim_{\F_q}(\r/I)=dim_{\F_q}(\a/\sI),
 \ \ \deg(M)=\de-deg(I)
\). We obtain the following table:
\[
\begin{array}{c|rrrrrr}
 deg(\sI)&1&2\ \ \ &3\ \ \ \,&4\ &5\ \ \ \,&6\ \ \\ \hline
 w_M=(q-1)/q
 &1&q+1&q(q+1)&2q^2&q(q-1)&0\ \ \\
 w_M=1
 &0&0\ \ \ &0\ \ \ \,&0\ &0\ \ \ \, &q^2.
\end{array}
\]
After omitting the common factor $1+q\aa$, their weighted instanton
contribution is
\begin{align}
 \widetilde H_{\varrho=2}^{\mathrm{inst}}(q,t)
 ={}&q^2t^6+\frac{q-1}{q}\bigl(
 t+(q+1)t^2+q(q+1)t^3\notag\\
 &\hspace{27mm}{}+2q^2t^4+q(q-1)t^5\bigr).
 \label{eq:double-Hinst}
\end{align}
Multiplication by $(qt)^8$ after $t\mapsto 1/(qt)$ gives
\begin{align*}
 &\sum_{\varrho(M)=2}
 (qt)^{\deg(M)}
 \left((q-1)/q\right)^{2-\upsilon_M}
=(qt)^2q^2+\\
\frac{q-1}{q}\bigl(
 (qt)^7&+(qt)^6(q+1)+(qt)^5 q(q+1)
+2(qt)^4q^2+(qt)^3q(q-1)\bigr),
\end{align*}
which is exactly the formula right before Section 7.2 in
\cite{ChQ}.
\vskip 0.2cm

{\bf The Hopf three-link.}
Take $3$ distinct lines, for example
\(
 F_1=x,\ \ F_2=y,\ \ F_3=x-y.
\)
The corresponding link is the positive Hopf $3$-link $T(3,3)$.
All branches are smooth and all three intersection numbers equal
$1$. Thus, 
\(
 \de=3,\ \ n_1=n_2=n_3=2,\ \ 
 \sC=(x,y)^2.
\)
The boundary lifts (modulo $\sC$) can be chosen as
\begin{equation}\label{eq:hopf-Bij}
 \sB_{12}=\sC+\F_q\,(x-y),\ \ 
 \sB_{13}=\sC+\F_q\,y,\ \ 
 \sB_{23}=\sC+\F_q\,x.
\end{equation}
They are $3$ lines in 
\(
 P=\mathfrak{m}_{\a}/\mathfrak{m}_{\a}^2\simeq\F_q^2,
\)
and
\(
 \sB=\sB_{12}+\sB_{13}+\sB_{23}=\mathfrak{m}_{\a}.
\)

The ideals between $\sC$ and $\a$ are
\[
 \a,\ \ \mathfrak{m}_{\a},\ \ \sC+L
 \ \ (L\in\mathbb P(P)),\ \ \sC.
\]
Recall that we identify $L$ with their lifts to 
$P=\mathfrak{m}_\a/\sC$.
There are three distinguished lines
$\sB_{12},\sB_{13},
\sB_{23}$ and $q-2$ remaining generic lines, $L_{gen}$, in $P$.
 Formula
\eqref{eq:weight-three} gives
\[
\begin{array}{c|c|c|c|c|c}
 \sI\subset \a&\#&deg(\sI)&\varrho(\sI)&W_I&w_M\\ \hline
 \a&1&0&1&P&(q-1)^2/q^2\\
 \mathfrak{m}_{\a}&1&1&2&P&(q-1)^2/q^2\\
 \sC+(L_{gen})&q-2&2&2&L_{\rm gen}&
 (q-1)^2/[q(q-2)]\\
 \sB_{ij}&3&2&2&L_{ij}&(q-1)/q\\
 \sC&1&3&3&0&1.
\end{array}
\]
The corresponding {\sf\em weighted} instanton sum is
\begin{align}
 H_{3,3}^{inst}(q,t,\aa)
 ={}&\frac{(q-1)^2}{q^2}
 +\frac{(q-1)^2}{q^2}t(1+q\aa)\notag\\
 &+\frac{(q-1)(q+2)}qt^2(1+q\aa)
 +t^3(1+qa)(1+q^2\aa).
 \label{eq:hopf-Hinst}
\end{align}
Finally,
\begin{align}
 &\mathcal H_{3,3}^{mot}(q,t,\aa)
 ={}(qt)^3H_{3,3}^{inst}
 \left(q,\frac1{qt},\aa\right)\label{eq:hopf-Hmot}\\
 =(1\!+\!q\aa)(1\!+\!q^2a)+\,& q(q\!-\!1)^2t^3
 +(1\!+\!q\aa)\left((q\!-\!1)(q\!+\!2)t+
(q\!-\!1)^2t^2\right)\notag.
 \end{align}
This agrees with the direct calculation from the families of standard
modules for the Hopf $3$-link. See formula (7.56) in \cite{ChQ}.

\section{\bf Volume conjecture for hooks}\label{sec:inst-count}
\newcommand{\GL}{\operatorname{GL}}
\newcommand{\coker}{\operatorname{coker}}
\newcommand{\soc}{\operatorname{soc}}
\newcommand{\qqbinom}[2]{\genfrac{[}{]}{0pt}{}{#1}{#2}_{q}}
\newcommand{\Vol}{\operatorname{Vol}}
\newcommand{\Hdaha}{\mathfrak H^{\rm daha}}
\newcommand{\Hinst}{\mathscr H^{\rm inst}}
Let
\(
 \a\!=\!\F_q[[x,y]],\   \Sigma_\lambda(n,r)\!=\!
 \{M\mid I_\lambda^{\oplus n}\subset M\subset\a^n,
              \varrho(M)=r\},
\)
where  $\varrho(M)=dim_{\F_q}M/\mathfrak{m}_\a M$ for
 $\mathfrak{m}_\a=(x,y).$ For better readability, the modules  
will be denoted by  $M$, instead of $\mathcal{M}$. Let us state
the formula
for $|\Sigma_\lambda(n,r)|$ for every hook
\(
 \lambda=(b,1^c),\ \  b\ge2,\  c\ge1,
\) 
when 
\begin{equation*}
 I_\lambda=(x^{c+1},xy,y^b),\ \ 
 \a_\lambda=\a/I_\lambda
 =\F_q[x,y]/(x^{c+1},xy,y^b).
\end{equation*}

For $0\le j\le a$, we will  constantly use 
\(
 \qqbinom{a}{j}=
 \prod_{i=0}^{j-1}\frac{q^{a-i}-1}{q^{j-i}-1}, 
  \qqbinom{a}{j}=0\ (j>a), \text{ and }
 |\GL_a(\F_q)|\!=\!\prod_{i=0}^{a-1}(q^a-q^i),
  |\GL_0(\F_q)|\!=\!1\).
For $0\le \ell\le k$, $\pi_k(\ell)$ will be the proportion, among all
$k\times k$ matrices over $\F_q$, of those with the 
kernel of dimension $\ell$:
\( \pi_k(\ell)=q^{-k^2}
 \frac{\prod_{i=0}^{k-\ell-1}(q^k-q^i)^2}
      {\prod_{i=0}^{k-\ell-1}(q^r-q^i)}\).
Empty products are $1$.
For an integer $d\ge1$, define $\Pi_d(k,j;n)$ recursively by
{\small 
\begin{equation}\label{eq:Md-rec}
 \Pi_1(k,j;n)=q^{k(n-j)}\qqbinom{k}{j}, \
 \Pi_d(k,j;n)= q^{nk}\sum_{\ell=j}^k
     \pi_k(\ell)\,\Pi_{d-1}(\ell,j;n)
\end{equation}
}

\noindent
for  $d\ge2$, and  put
\( c_{0,n}(q)=1,\ 
 c_{j,n}(q)=\prod_{i=0}^{j-1}(q^n-q^i)\ (j\ge1)\). We set:
%For $\lambda=(b,1^c)$ set
\begin{equation}\label{eq:Psi-rec}
 \Psi_{k,n}^{\lambda}(q)=
 \sum_{j=0}^k c_{j,n}(q)
 \frac{\Pi_c(k,j;n)\Pi_{b-1}(k,j;n)}{\qqbinom{k}{j}},
 \ \  \Psi_{0,n}^{\lambda}=1.
\end{equation}

Let $G_{e,k}(q)$ be the number of $g\in\GL_e(\F_q)$ with
$dim\ker(g-1)=k$:
\begin{equation}\label{eq:Gek}
 G_{e,k}(q)\!=\!\qqbinom{e}{k}
 \,\sum_{h=0}^{e-k}\qqbinom{e\!-\!k}{h}(-1)^h
 q^{\binom h2\!+\!(k+h)(e\!-\!k-h)}
 |\GL_{e\!-\!k-h}(\F_q)|.
\end{equation}
Finally,
 \( F_e^{(n)}(q)=
 \frac{1}{|\GL_e(\F_q)|}
 \sum_{k=0}^e G_{e,k}(q)\Psi_{k,n}^{\lambda}(q),
 \ \  F_{-1}^{(n)}=0.\)

\begin{theorem}\label{thm:full-rec}
For every hook $\lambda=(b,1^c)$ and every $e\ge0$,
\begin{equation}\label{eq:Sigma-final}
 |\Sigma_\lambda(n,n+e)|
 =F_e^{(n)}(q)-F_{e-1}^{(n)}(q).
\end{equation}
%for $F_e^{(n)}$ as above.
\end{theorem}

\vskip 0.2cm
{\bf Why the finite algebra $\a_\lambda$ detects $\varrho(M)$?}
The definition of $\varrho(M)$ is made in $\a$, while the recurrence 
uses only $\a_\lambda$.  To explain this, let
\(
 Q_M=\a^n/M,
\)
which is a finite $\a_\lambda$-module, 
\(
 D(V)= \operatorname{Hom}_{\F_q}(V,\F_q)\), and 
\(L_M=D(Q_M)\subset D(\a_\lambda)^n.
\)

\begin{lemma}\label{lem:rho-dual}
One has:
 \(\varrho(M)=n+\mu(L_M),\ \ 
 \mu(L)=dim_{\F_q}L/\mathfrak{m}_{\a_\lambda}L.\)
\end{lemma}
{\it Proof.}
Apply $\F_q\otimes_\a-$ to
$0\to M\to\a^n\to Q_M\to0$.  Since $\a^n$ is free,
\(
 \varrho(M)=n-dim H_0(x,y;Q_M)+dim H_1(x,y;Q_M),
\)
where $H_i(x,y;Q_M)$ is Koszul homology.  The Koszul complex has terms
$Q_M,Q_M^2,Q_M$, hence its Euler characteristic is zero and
\(
 -dim H_0+dim H_1=dim H_2.
\)
Moreover,
\(
 H_2(x,y;Q_M)=\{v\in Q_M\mid xv=yv=0\}=\soc(Q_M).
\)
Under the perfect pairing $Q_M\times D(Q_M)\to\F_q$, the dual of
$L_M/\mathfrak{m}_{\a_\lambda}L_M$ is $\soc(Q_M)$.  Therefore
$dim H_2=\mu(L_M)$, proving the required. \sq

Thus the excess $\varrho(M)-n$ is exactly the minimal number of
generators of the finite $\a_\lambda$-module $L_M$.  This is the reason
that $\a_\lambda$, although it is only the quotient $\a/I_\lambda$, is
sufficient for the rank count.

Let $F_e^{(n)}(q)$ denote the number of $\a_\lambda$-submodules of
$D(\a_\lambda)^n$ generated by at most $e$ elements.  Lemma
\ref{lem:rho-dual} gives
\begin{equation}\label{eq:Sigma-F}
 |\Sigma_\lambda(n,n+e)|=F_e^{(n)}(q)-F_{e-1}^{(n)}(q).
\end{equation}
let us obtain the formula for $F_e^{(n)}$ stated above.
\vskip 0.2cm

{\bf Burnside reduction.}
Setting 
\(
 X_e=\operatorname{Hom}_{\a_\lambda}(\a_\lambda^e,D(\a_\lambda)^n),
\)
the group $\GL_e(\a_\lambda)$ acts here by precomposition.  Its
orbits are the images of the maps in $X_e$, hence are exactly the
submodules generated by at most $e$ elements.  Indeed, two generating
$e$-tuples of the same finite module are related by an invertible
change of generators over the local ring $\a_\lambda$; modulo the
maximal ideal, this is ordinary linear algebra. Then the resulting
invertible matrix lifts to $\GL_e(\a_\lambda)$ by Nakayama's lemma.

For $g\in\GL_e(\a_\lambda)$, a map $f\in X_e$ is fixed by $g$ exactly
when $f(g-1)=0$, equivalently when $f$ factors through
$\coker(g-1)$.  For every finite $\a_\lambda$-module $C$,
\(\operatorname{Hom}_{\a_\lambda}(C,D(\a_\lambda))\simeq D(C)\)
via
\[
 f\longmapsto \ell=(c\mapsto f(c)(1)),
\text{ with inverse }
 \ell\longmapsto\bigl(c\mapsto[b\mapsto\ell(bc)]\bigr).
\]
Thus,
\(
 |\operatorname{Hom}_{\a_\lambda}(C,D(\a_\lambda)^n)\,|
 \!=\!q^{n\,dim_{\F_q}C}
\) and 
Burnside's lemma gives
\begin{equation}\label{eq:Burnside}
 F_e^{(n)}(q)=
 \frac1{|\GL_e(\a_\lambda)|}
 \sum_{g\in\GL_e(\a_\lambda)}
 q^{n\,dim_{\F_q}(\coker(g-1))}.
\end{equation}

Reduction modulo
$\mathfrak{m}_{\a_\lambda}$ gives $\GL_e(\a_\lambda)\to\GL_e(\F_q)$, and every
$\bar g\in\GL_e(\F_q)$ has exactly $|\mathfrak{m}_{\a_\lambda}|^{e^2}$ lifts.  Fix
$\bar g$ and put
\(
 k=dim_{\F_q}\ker(\bar g-1)
\). Then choose $P_0,Q_0\in\GL_e(\F_q)$ such that

\[ P_0(\bar g-1)Q_0=
 \begin{pmatrix}\mathbf I_{e-k}&0\\0&0\end{pmatrix} 
\text{ for the matrix unit } \mathbf I,
\]
and lift $P_0,Q_0$ to $P,Q\in\GL_e(\a_\lambda)$.  For a lift $g$ of
$\bar g$, write
\begin{equation*}
 P(g-1)Q=
 \begin{pmatrix}A&B\\ C&D\end{pmatrix}, \text { where }
\end{equation*}
$A$ is  $\mathbf I_{e-k}$ modulo $\mathfrak{m}_{\a_\la}$ (invertible), 
$B,C,D$ have entries in $\mathfrak{m}_{\a_\lambda}$. Using
\begin{equation}\label{eq:block-g}
 \begin{pmatrix}\mathbf I&0\\-CA^{-1}&\mathbf I\end{pmatrix}\ 
P(g-1)Q
 \begin{pmatrix}\mathbf I&-A^{-1}B\\0&\mathbf I\end{pmatrix}
=
 \begin{pmatrix}A&0\\0&D-CA^{-1}B\end{pmatrix},
\end{equation}
we conclude that:\,
\begin{equation}\label{eq:coker-K}
 \coker(g-1)\simeq\coker K \text{\, for \,}
   K=D-CA^{-1}B\in\operatorname{Mat}_k(\mathfrak{m}_{\a_\lambda}).
\end{equation}
For fixed $A,B,C$, the map
$D\mapsto D-CA^{-1}B$ is a translation of the additive group
$\operatorname{Mat}_k(\mathfrak{m}_{\a_\lambda})$.  Hence, 
among all lifts of the fixed
$\bar g$, every $K\in\operatorname{Mat}_k(\mathfrak{m}_{\a_\lambda})$ occurs with the same
multiplicity.  Thus the average of the summand in 
\eqref{eq:Burnside}
over the lifts of $\bar g$ is
\begin{equation}\label{eq:Psi-definition}
 \Psi_{k,n}^{\lambda}(q)=
 \frac1{|\mathfrak{m}_{\a_\lambda}|^{k^2}}
 \sum_{K\in\operatorname{Mat}_k(\mathfrak{m}_{\a_\lambda})}
 q^{n\,dim_{\F_q}(\coker K)}.
\end{equation}
It depends on $\bar g$ only through $k=dim_{\F_q}(\ker(\bar g-1))$.  
Grouping
\eqref{eq:Burnside} according to $k$ gives
\begin{equation}\label{eq:F-Psi}
 F_e^{(n)}(q)=
 \frac1{|\GL_e(\F_q)|}
 \sum_{k=0}^eG_{e,k}(q)\Psi_{k,n}^{\lambda}(q).
\end{equation}
Formula \eqref{eq:Gek} then follows by, first, choosing the $k$-dimensional
fixed space and then applying M\"obius inversion in the lattice of
subspaces containing it.
\vskip 0.2cm

{\bf $\a_\lambda$ for a hook.} Setting
\(
 A_x=\F_q[x]/(x^{c+1}),\ \ 
 A_y=\F_q[y]/(y^b),
\)
every element of
\(
 \a_\lambda=\F_q[x,y]/(x^{c+1},xy,y^b)
\)
has a unique representative 
\(a+x f(x)+y g(y)\) for \( a\in\F_q,\ 
xf(x)\in x\F_q[x]/(x^{c+1})\) and \(
yg(y)\in y\F_q[y]/(y^b)\).  Equivalently, the map
\[
 a+xf(x)+yg(y)\longmapsto
 \bigl(a+xf(x),\ a+yg(y)\bigr)
\]
gives the explicit ring isomorphism
\begin{equation}\label{eq:B-explicit}
 \a_\lambda\simeq
 \{(f,g)\in A_x\times A_y\mid f(0)=g(0)\}.
\end{equation}
In particular, 
$ \mathfrak{m}_{\a_\lambda}=
 x\F_q[x]/(x^{c+1})\oplus
 y\F_q[y]/(y^b)$
as a vector space, and 
\begin{equation}\label{eq:xy-zero}
 \bigl(x\F_q[x]/(x^{c+1})\bigr)
 \bigl(y\F_q[y]/(y^b)\bigr)=0,
\end{equation}
which is simply the relation $xy=0$ in $\a_\lambda$.

This decomposition is the key point in the recurrence.  Every matrix
$K\in\operatorname{Mat}_k(\mathfrak{m}_{\a_\lambda})$ has then a unique decomposition
\begin{equation}\label{eq:Kxy}
 K=K_x+K_y=xP(x)+yQ(y),
\end{equation}
where
\(
 P(x)\in\operatorname{Mat}_k(\F_q[x]/(x^c)),\ \ 
 Q(y)\in\operatorname{Mat}_k(\F_q[y]/(y^{b-1})).
\)
We call $K_x=xP(x)$ the $x$-part and $K_y=yQ(y)$ the $y$-part of $K$.
The relation \eqref{eq:xy-zero} means that the two parts do not create
mixed products, but they are still coupled by the common residue space
$\a_\lambda/\mathfrak{m}_{\a_\lambda}\simeq\F_q$.  
Thus, our description has two natural stages.

First, we solve the matrix problem for a single truncated polynomial
ring $\F_q[z]/(z^{d+1})$ with matrices divisible by $z$.  Besides the
cokernel dimension, we must retain the image in $\F_q^k$ of the kernel
under reduction modulo $z$.  This extra subspace is essential: it is
exactly the information needed when the $x$-part and $y$-part are put
back together. Using this calculation for an arbitrary depth
$d$, the general hook case is obtained by taking $d=c$ for the
$x$-part and $d=b-1$ for the $y$-part.  The only new datum is then the
intersection of the two resulting subspaces of the same residue space
$\F_q^k$. 
\vskip 0.2cm

{\bf The truncated-polynomial calculation.}
For $d\ge1$, let
\(
 R_d=\F_q[z]/(z^d),\ \  A_d=\F_q[z]/(z^{d+1}).
\)
For
\(
 P=P_0+zP_1+\cdots+z^{d-1}P_{d-1}\in\operatorname{Mat}_k(R_d),
\)
we set
\(
 T_d(P)=zP:A_d^k\longrightarrow A_d^k\), and introduce the
space 
$U_d(P)=\operatorname{im}\bigl(\ker T_d(P)\longrightarrow
                 A_d^k/zA_d^k\simeq\F_q^k\bigr)$,
which will be needed
later to impose the common-residue condition between the $x$-part and
the $y$-part of \eqref{eq:B-explicit}. Let
\begin{align}
 D_d(P)&=dim_{\F_q}\coker T_d(P)\label{eq:Dd},\ 
 u_d(P)=dim_{\F_q}U_d(P).\notag
\end{align}
We will now prove \eqref{eq:Md-rec}. For $0\le j\le k$, define
\begin{equation}\label{eq:M-definition}
 \Pi_d(k,j;n)=q^{-dk^2}
 \sum_{P\in\operatorname{Mat}_k(R_d)}
 q^{n(D_d(P)-u_d(P))}
 \qqbinom{u_d(P)}{j}.
\end{equation}

If $d=1$, then $P=P_0\in\operatorname{Mat}_k(\F_q)$.  If
The image of $zP_0$ has dimension $k-\ell$ for
$\ell=dim\ker P_0$, and 
\(
 D_1(P_0)=k+\ell,\ \  u_1(P_0)=\ell.
\)
Thus, $q^{n(D_1-u_1)}=q^{nk}$. Then  count pairs $(P_0,W)$ with
$W\subset\ker P_0$ and $dim W=j$.  There are $\qqbinom{k}{j}$ choices
for $W$, and for fixed $W$ there are $q^{k(k-j)}$ matrices $P_0$ which
vanish on $W$.  Hence
\(
 q^{-k^2}\sum_{P_0}\qqbinom{dim\ker P_0}{j}
 =q^{-kj}\qqbinom{k}{j},
\)
which gives the first equation in \eqref{eq:Md-rec}.

Now consider $d\ge2$ and write
\[
 P(z)=P_0+zP_1+\cdots+z^{d-1}P_{d-1}.
\]
Fix $P_0$ and put $\ell=dim\ker P_0$. The operation
from (\ref{eq:block-g}) put $P_0$ into
\(
 \begin{pmatrix}I_{k-\ell}&0\\0&0\end{pmatrix}.
\). Let \(
 P(z)=\begin{pmatrix}A(z)&B(z)\\C(z)&D(z)\end{pmatrix},
\), 
where $A(0)=\mathbf I_{k-\ell}$ and $B(0)=C(0)=D(0)=0$.
Applying the same operation to $P(z)$ does not change
$D_d(P)$ or $u_d(P)$ and transforms $P(z)$ into
\[
 \begin{pmatrix}A(z)&0\\0&S(z)\end{pmatrix}\ \text{ for } \ 
 S(z)=D(z)-C(z)A(z)^{-1}B(z).
\]
Since $S(0)=0$, there is a unique
\[
 P'(z)\in\operatorname{Mat}_{\ell}(R_{d-1})
 \ \text{ with }\  S(z)=zP'(z).
\]
For fixed $A,B,C$, the map $D\mapsto S$ is a matrix translation, so every
$P'\in\operatorname{Mat}_{\ell}(R_{d-1})$ occurs with the same
multiplicity when $P_1,\ldots,P_{d-1}$ vary.  This is
the key step.

On the first $k-\ell$ coordinates, $T_d(P)$ is equivalent to
multiplication by $z$, whose cokernel has dimension $k-\ell$.  On the
last $\ell$ coordinates, it is equivalent to multiplication by
$z^2P'$.  The constant term in $A_d^{\ell}$ contributes $\ell$
dimensions to the cokernel, and after removing this constant 
layer the
remaining quotient is the cokernel of
\(
 zP':A_{d-1}^{\ell}\longrightarrow A_{d-1}^{\ell}.
\)
Therefore
%\begin{equation}\label{eq:D-reduction}
\( D_d(P)=k+D_{d-1}(P').\)
%\end{equation}

The kernel of multiplication by $z$ on the first block reduces to zero
modulo $z$.  On the second block, reduction modulo $z$ is exactly the
residue image of the kernel of $zP'$; hence
%\begin{equation}\label{eq:U-reduction}
\( U_d(P)=U_{d-1}(P'),\ \  u_d(P)=u_{d-1}(P').\)
%\end{equation}
There are $q^{k^2}\pi_k(\ell)$ possibilities for $P_0$ with kernel
dimension $\ell$.  Substituting $D_d(P)$ and $u_d(P)$
into
\eqref{eq:M-definition} gives
\(
 \Pi_d(k,j;n)=q^{nk}\sum_{\ell=j}^k
 \pi_k(\ell)\Pi_{d-1}(\ell,j;n),
\)
which is \eqref{eq:Md-rec}.
\vskip 0.2cm

{\bf Combining the $x$- and $y$-parts.}
Begin with the (unique) decomposition \eqref{eq:Kxy}, and put
\[
 U=U_c(P),\   V=U_{b-1}(Q),\ 
u=dim U,\   v=dim V,\   h=dim(U\cap V).
\]
Recall from \eqref{eq:B-explicit} that
%\begin{equation}%\label{eq:Bk-explicit}
\( \a_\lambda^k=
 \{(a,b)\in A_x^k\times A_y^k\mid \bar a=\bar b\in\F_q^k\}.
\)
%\end{equation}
The action of $K=xP+yQ$ is
\(
 (a,b)\longmapsto(xP(a),yQ(b)).
\)
Hence
\[
 \ker K=
 \{(a,b)\mid xP(a)=0,\ yQ(b)=0,\ \bar a=\bar b\}.
\]
Consider
\(
 \ker(xP)\times\ker(yQ)\longrightarrow\F_q^k,
 \ \ (a,b)\longmapsto\bar a-\bar b.
\)
Its image is $U+V$ and its kernel is $\ker K$.  Since $xP$, $yQ$, and
$K$ are square endomorphisms of finite-dimensional $\F_q$-vector
spaces,
\(
 dim\ker(xP)=D_c(P), \ 
 dim\ker(yQ)=D_{b-1}(Q),\ 
 dim\ker K=dim\coker K.
\)
Therefore
\begin{equation}\label{eq:global-coker}
 dim_{\F_q}\coker K
 =D_c(P)+D_{b-1}(Q)-u-v+h.
\end{equation}
This formula is the point at which the two truncated-polynomial
calculations are combined.  Each of them has already produced the
corresponding 
cokernel dimension and a subspace in the same residue space $\F_q^k$;
the general case depends only on the dimension of their
intersection.

Fix a subspace $U\subset\F_q^k$ with $dim\, U=u$.  The number of 
subspaces
$V\subset\F_q^k$ with $dim\, V=v$ and $dim(U\cap V)=h$ is
%\begin{equation}%\label{eq:inter-number}
$q^{(u-h)(v-h)}
 \qqbinom{u}{h}\qqbinom{k-u}{v-h}.$
%\end{equation}
Indeed, choose $U\cap V$, then the image of $V$ in $\F_q^k/U$, and
finally the linear map which lifts this image modulo $U\cap V$.

To establish the recurrence we need to count triples
$(W,U,V)$ with $dim\, W=j$ and $W\subset U\cap V$, which give the
following formula:
\begin{equation}\label{eq:intersection-moment}
 \frac{1}{\qqbinom{k}{u}\qqbinom{k}{v}}
 \sum_{\substack{dim U=u\\dim V=v}}
 \qqbinom{dim(U\cap V)}{j}
 =\frac{\qqbinom{u}{j}\qqbinom{v}{j}}{\qqbinom{k}{j}}.
\end{equation}
Also, the $q$-Newton identity gives
%\begin{equation}\label{eq:q-Newton}
\( q^{nh}=\sum_{j=0}^h
 \left(\prod_{i=0}^{j-1}(q^n-q^i)\right)
 \qqbinom{h}{j},\)
%\end{equation}
where the coefficients in parentheses
are $c_{j,n}(q)$ from (\ref{eq:Psi-rec}).

For fixed $u$, the total weight in \eqref{eq:M-definition} 
is the same
for every $u$-dimensional subspace $U$.  Indeed,
$P\mapsto PH$ for $H\in\GL_k(\F_q)$ is a bijection of 
the matrices $P$
preserving $D_c(P)$ and $u_c(P)$, and sending $U_c(P)$ to
$H^{-1}U_c(P)$. Similarly, for $Q$ and $V$.

Now Substitute \eqref{eq:global-coker} into 
\eqref{eq:Psi-definition}, and use
 \eqref{eq:intersection-moment}.  The sums over
$P$ and $Q$ separate and are exactly the two factors in
\eqref{eq:M-definition}.  We arrive at
\[
 \Psi_{k,n}^{\lambda}(q)=
 \sum_{j=0}^k c_{j,n}(q)
 \frac{\Pi_c(k,j;n)\Pi_{b-1}(k,j;n)}{\qqbinom{k}{j}},
\]
which is \eqref{eq:Psi-rec}. Combined with \eqref{eq:F-Psi},
\eqref{eq:Sigma-F}, this proves Theorem \ref{thm:full-rec}.
\vskip 0.2cm

{\bf Polynomiality.}
As a byproduct, we obtain the following corollary.
\begin{corollary}\label{cor:poly}
For every hook $\lambda=(b,1^c)$ and all fixed $n,r$,
\begin{equation}\label{eq:poly}
 |\Sigma_\lambda(n,r)|\in\Q[q].
\end{equation}
%In general the conclusion is $\Q[q]$, not $\Z[q]$.
\end{corollary}
{\it Proof.}
The recurrence first shows that $|\Sigma_\lambda(n,r)|$ is a rational
function of $q$ with rational coefficients.  Write it in lowest terms
as $P(q)/D(q)$ with coprime $P,D\in\Z[q]$, after clearing a constant
denominator.  For every prime power $q_0$, this ratio is a natural number.
Therefore this ratio is a polynomial.
The coefficients need not be integral;
for example $(q^2-q)/2$ is integer-valued at every integer $q$.
They are conjectured to be from $\Z_+$, but we have
proved this only
over $\Q$. \sq
\vskip 0.2cm

{\bf Volume Conjecture.}
The recurrence and Corollary~\ref{cor:poly} are finite-field 
statements;
no root of unity has been used so far. Recall that
\begin{equation}\label{eq:Sigma-n-root}
 \Sigma_\lambda^{(n)}(q)=
 1+\sum_{r>n}(1-q)(1-q^2)\cdots(1-q^{r-n})
       |\Sigma_\lambda(n,r)|.
\end{equation}
Since every coefficient is a polynomial in $q$, this expression can be
evaluated at any complex number.  Put
\begin{equation}\label{eq:qbar-root}
 N=n+1,\ \  \widehat q=\exp(2\pi\imath/N),\ \ 
 \Sigma_\lambda^\intercal=
 \Sigma_\lambda^{(n)}(q)\big|_{q\mapsto \widehat q}\,.
\end{equation}

Let us state  Theorem \ref{thm:volconj} based on 
the conjectural connections of $\sH^{inst}_\la$ with
$\mathfrak{H}^{daha}_\la$ and the conjectural  connection of the
latter with the corresponding HOMFLY-PT polynomials. 

\comment{
\begin{equation}\label{eq:instanton-at-root}
 J_N(K_\lambda;\widehat q)\sim
 \Hinst_\lambda(\widehat q,\widehat q^{\,-n-1},-\widehat q^{\,2};n)
 =\Hinst_\lambda(\widehat q,1,-\widehat q^{\,2};n),
\end{equation}
where $\sim$ means equality up to a power of $\widehat q$.  The last equality
uses only $N=n+1$.  For a module with $\varrho(M)=r$, the factor in the
last instanton sum is
\[
 (1-\widehat q\,)(1-\widehat q^{\,2})\cdots(1-\widehat q^{\,r-n}).
\]
Therefore \eqref{eq:Sigma-n-root} gives that  $J_N(K_\lambda;\widehat q\,)$
is $ \Sigma_\lambda^\intercal$ up to a power of $\widehat q$. We arrive
at the following Theorem.}
%\begin{equation}\label{eq:Jones-Sigma-root}
% \Sigma_\lambda^\intercal\sim.
%\end{equation}

\begin{theorem}[Instanton Volume Conjecture]%\label{cor:volume}
Assuming the all-order Kashaev asymptotic
in the normalization of Dimofte--Garoufalidis:\,
\begin{equation*}%\label{eq:inst-vol-root}
 |\Sigma_\lambda^\intercal|\sim
 |\de_K|^{-1/2}N^{3/2}
 \exp\!\left(\frac{V(K)N}{2\pi}\right)
 \left|\exp\!\left(
 \sum_{m\ge2}S_{K,m}\left(\frac{2\pi\imath}{N}\right)^{m-1}
 \right)\right|,
\end{equation*}
where $K=K_\lambda$, $V(K)=\Vol(S^3\setminus K)$,
$\de_K$ is the one-loop torsion, and $S_{K,m}$ are the higher-loop
invariants.  For two-loops, the coefficient denoted $S_{K,2}$ here is
$S^{\rm DG}_{2,1}$ in \cite{DiG}.
\end{theorem}
\vskip 0.2cm

{\bf Numerical hook tests.}
We record here only basic
numerical checks. We mostly calculated for diagrams
$(m,1)$. The corresponding Coxeter knots are
for the $3$--braids
\(
 T_1^{2(m-1)}T_2T_1^2T_2T_1T_2,
\)
whose closures are the pretzel knots $P(-2,3,2m-1)$, up to the usual
mirror/orientation convention.  Thus, the cases $m=4,5$ below are
$P(-2,3,7)$ and $P(-2,3,9)$.  The third example, $(3,1,1)$, 
is included to
emphasize that the large-$N$ recurrence is not restricted to the 
subfamily $(m,1)$.

We will provide numerical data for 
$V_N=\frac{2\pi}{N}\log|\Sigma_\lambda^\intercal|$,
\begin{equation}%\label{eq:VN-def}
\mathcal C_N=V_N-\frac{3\pi\log N}{N}%\label{eq:C0}\\
 \text{ and }\mathcal C_N^1=\mathcal C_N-
 \frac{2\pi\log|\de_K|^{-1/2}}{N}.\label{eq:C1}
\end{equation}
Thus $\mathcal C_N$ removes the universal factor $N^{3/2}$ and
$\mathcal C_N^1$ also removes the exact one-loop factor, assuming 
that they are known. These
assumptions are part of the Volume Conjecture, but result in
much faster convergence. The claim about $N^{3/2}$ is the
most helpful to accelerate the calculations. We 
clearly see $N^{3/2}$ numerically and can obtained sufficiently good
estimates for $V_N, C_N$ and $C^1_N$ without this assumption. 

When the conjectured absolute value
of the two-loop term is independently known, we also put it in 
\begin{equation}\label{eq:C2}
 \mathcal C_N^2=\mathcal C_N^1-
 \frac{4\pi^2|\Im S_{K,2}|}{N^2}.
\end{equation}

\comment{
Only the modulus is used here; hence $|\Sigma_\lambda^\intercal|$
determines $|\Im S_{K,2}|$, not its sign or its real part.  
The sign of
the $N^{-2}$ term itself depends on the root/conjugation convention; in
the convention \eqref{eq:qbar-root} the two calculations below 
have the
sign displayed in \eqref{eq:C2}.
}

%For $\lambda=(4,1)$, the recurrence uses depths $1$ and $3$.  The
%root-of-unity sequence agrees with the previous $(3,2)$ realization 
%of $K12n242=P(-2,3,7)$.  

\vskip 0.2cm
{\sf\em The case of $(4,1)$.} 
The data used for comparison are
\begin{align}
 V(K)&=2.828122088330783162763898809276635\ldots,\label{eq:V41}\\
 \de_K&\!=\!-2(3\xi-2)\xi^{-2} \text{\, for\,\, }  \xi^3-\xi^2+1=0,\ \,
 \Im\xi<0,
\label{eq:delta41}\\
 |\de_K|^{-1/2}&=0.534053998058845895612015890703009\ldots .
\end{align}
Dimofte--Garoufalidis explicitly tabulate
\begin{align}
 S^{\rm DG}_{2,1}&=
 \frac{-73-1524\xi-879\xi^2}{25392}\label{eq:S41}\\
 &=-0.0629833007949880\ldots
   +0.0899553605974326\ldots\,\imath.\notag
\end{align}

At $N=1009$ the hook recurrence gives
\begin{align*}
 V_{1009}&=2.88882658986931655839\ldots,\\
 \mathcal C_{1009}&=2.82421954994389413487\ldots,\\
 \mathcal C_{1009}^1&=2.82812557600191030600\ldots,\\
 \mathcal C_{1009}^2&=2.82812208777720970163\ldots .
\end{align*}
The last number differs from \eqref{eq:V41} by about
$5.5\cdot10^{-10}$.  Thus, the same recurrence that proves
Corollary \ref{cor:poly} also reproduces the known volume, torsion, and
two-loop data for this knot.
\vskip 0.2cm

{\sf\em The case of $(5,1)$.}
% the two depths are $1$ and $4$, and
%$K=14n6022=P(-2,3,9)$, up to the usual mirror convention.  
In
Sec.~6.4 of \cite{DiG}, Dimofte--Garoufalidis give the trace field
\begin{equation}\label{eq:E51}
 E=\Q(\alpha),\ \ 
 \alpha^5-\alpha^4+\alpha^3+2\alpha^2-2\alpha+1=0,
\end{equation}
with geometric embedding
$\alpha=0.4425377456\ldots-0.4544788919\ldots\imath\,$, and the exact
one-loop element
\begin{equation}\label{eq:delta51}
 \de_K=-4\alpha^4+10\alpha^3-10\alpha^2+2\alpha+14,
\end{equation}
whose norm is $-2^7\cdot733=-93824$.  Numerically,
\begin{equation}\label{eq:delta51num}
 |\de_K|^{-1/2}=0.268906323397755142816081854844969\ldots .
\end{equation}
A high-precision SnapPy computation for \texttt{K14n6022} gives
\begin{equation}\label{eq:V51}
 V(K)=4.124903251807676155505899848869378\ldots .
\end{equation}

At $N=1009$ we obtain
\begin{align*}
 V_{1009}&=4.18133570499482290973\ldots,\\
 \mathcal C_{1009}&=4.11672866506940048621\ldots,\\
 \mathcal C_{1009}^1&=4.12490734353796281936\ldots .
\end{align*}
The last number is already within $4.1\cdot10^{-6}$ of
\eqref{eq:V51}.  Since an explicit algebraic $S^{\rm DG}_{2,1}$ is not
tabulated for $P(-2,3,9)$ in Sec.~6.4 of \cite{DiG}, the two-loop 
term is
our prediction. See there for the
Neumann--Zagier/Feynman-diagram construction;
their Table~2 records
numerical Kashaev asymptotics for this knot through five loops.

For this prediction, we insert the independently known
$V(K)$ (only numerically) and
$|\de_K|^{-1/2}$, and define
\begin{equation}\label{eq:TN}
 T_N=\frac{N^2}{4\pi^2}\left(
 V_N-V(K)-\frac{3\pi\log N+2\pi\log|\de_K|^{-1/2}}{N}
 \right).
\end{equation}
To accelerate, we use
the following ``three-point extrapolation".
The all-order expansion has the form
\[
 T_N=s+\frac{a_1}{N}+\frac{a_2}{N^2}+O(N^{-3}),
 \ \  s=|\Im S_{K,2}|
\]
for the modulus-visible value.  The extrapolation 
means simply solving the three linear
equations at $N=251,509,1009$ for $s,a_1,a_2$; it is not a 
statistical fit.  For $(4,1)$ this gives
\[
 s=0.0899553366507634\ldots,
\]
very close to the exact \cite{DiG} value
$0.0899553605974326\ldots$ in \eqref{eq:S41}.  Applying the identical
procedure to $(5,1)$ gives the prediction
\begin{equation}\label{eq:S51-pred}
 |\Im S^{\rm DG}_{2,1}(P(-2,3,9))|
 =0.105526149690541\ldots .
\end{equation}
The single value $N=1009$ gives $0.105518739472406\ldots$ before 
employing the $3$-point 
extrapolation.
\vskip 0.2cm

{\sf\em The case of
 $(3,1,1)$.} % both truncation depths are equal to $2$:
In this case, \(
 \a_{(3,1,1)}=\F_q[x,y]/(x^3,xy,y^3).
\)
The general two-part mechanism is used in the recurrence,
with neither part of depth $1$.  It corresponds to $K12n725$.  At
$N=1009$, with $2304$-bit precision, the recurrence gives
\begin{align*}
 |\Sigma_{(3,1,1)}^\intercal|
 &=3.29984216693979425657\cdot10^{416},\\
 V_{1009}&=5.97225963005040006892\ldots,
\end{align*}
and the internal identity error is $7.23\cdot10^{-180}$.  Using
\[
 V(K12n725)=5.916745735182788\ldots
\]
as the volume, the one-loop constant inferred from a single
$N$ is
\begin{equation*}%\label{eq:C311-single}
 \exp\!\left(\frac{N(V_N-V)-3\pi\log N}{2\pi}\right)
 =0.232179626426\ldots\ \ (N=1009).
\end{equation*}
The values at $N=251,521,1009$ have an expansion in powers of $1/N$.
Solving the three equations
\(
 C_N=C+\frac{c_1}{N}+\frac{c_2}{N^2}
\)
gives
\begin{equation*}%\label{eq:delta311-pred}
 C=|\de_{K12n725}|^{-1/2}\simeq0.2318964998\ldots,
 \ \ 
 |\de_{K12n725}|\simeq18.59566\ldots .
\end{equation*}
This is a prediction from the instanton sequence, not an independent
algebraic torsion input.

After inserting this extrapolated $C$, the coefficient of $N^{-2}$ can
be treated in the same way.  For $N=251,521,1009$ the quantities
$N^2(\mathcal C_N^1-V)$ extrapolate quadratically in $1/N$ to
\(
 7.7319489\ldots .
\)
Consequently the modulus-visible two-loop prediction is
\begin{equation}\label{eq:S311-pred}
 |\Im S^{\rm DG}_{2,1}(K12n725)|
 \simeq\frac{7.7319489\ldots}{4\pi^2}
 =0.19585255\ldots .
\end{equation}
As for $P(-2,3,9)$, this is not quoted from an explicit algebraic value
in \cite{DiG}; it is the part of the two-loop coefficient extracted from
$|\Sigma^\intercal|$.

The $N=1009$ run for $(3,1,1)$ took $5862$ seconds, essentially the
same computational scale as the $(5,1)$ run.  This is consistent with the
$O(N^3)$ recurrence for fixed hook depths and is useful evidence that the
large-color computation is genuinely a general hook calculation, rather
than a special feature of hooks with one depth equal to $1$.
\comment{
The raw estimator $V_N$ converges slowly because it still contains the
terms $3\pi\log N/N$ and $2\pi\log|\de_K|^{-1/2}/N$.  Comparing several
values of $N$ separates these different powers of $1/N$ instead of
leaving them inside the volume error.  This is why the one-loop and
two-loop coefficients can be recovered accurately while the raw
single-$N$ volume is still visibly displaced.  The two-loop coefficient
is nevertheless more sensitive to the next $N^{-3}$ term; the $(4,1)$
comparison above provides a calibration of this effect.
}

\end{appendices}

\end{document}